\documentclass[12pt,leqno]{article}
\usepackage{latexsym,amsmath,amsfonts,amssymb,amsthm,amscd}
\usepackage{preprint}
\usepackage{mathrsfs}
\usepackage{comment}
\usepackage{bm}
\usepackage{cite}
\usepackage{color}
\usepackage{graphicx}

\theoremstyle{definition}
\newtheorem{defn}{Definition}[section]
\newtheorem{exmp}[defn]{Example}

\theoremstyle{plain}
\newtheorem{thm}[defn]{Theorem}
\newtheorem{prop}[defn]{Proposition}
\newtheorem{lem}[defn]{Lemma}
\newtheorem{cor}[defn]{Corollary}

\theoremstyle{remark}
\newtheorem*{rem}{Remark}

\newcounter{claim}
\renewcommand{\thesection}{\arabic{section}}
\renewcommand{\thesubsection}{\thesection.\arabic{subsection}}

\DeclareMathOperator{\re}{Re}
\DeclareMathOperator{\im}{Im}

\DeclareMathOperator{\Aut}{Aut}
\DeclareMathOperator{\id}{id}
\DeclareMathOperator{\ord}{ord}
\DeclareMathOperator{\Ext}{Ext}
\DeclareMathOperator{\sgn}{sgn}
\DeclareMathOperator{\Int}{Int}
\DeclareMathOperator{\Ker}{Ker}
\DeclareMathOperator{\card}{card}
\DeclareMathOperator{\ob}{ob}
\DeclareMathOperator{\Cont}{Cont}
\DeclareMathOperator{\TEmb}{TEmb}
\DeclareMathOperator{\QCEmb}{QCEmb}
\DeclareMathOperator{\CEmb}{CEmb}
\DeclareMathOperator{\Homeo}{Homeo}
\DeclareMathOperator{\QCHomeo}{QCHomeo}
\DeclareMathOperator{\CHomeo}{CHomeo}

\title{Closed continuations of Riemann surfaces}
\author{Makoto {\sc Masumoto} and Masakazu {\sc Shiba}}
\date{}
\begin{document}
\maketitle
\begin{quote}
{\bf Abstract.} Any open Riemann surface $R_{0}$ of finite genus $g$ can be 
conformally embedded into a closed Riemann surface of the same genus, that is, 
$R_{0}$ is realized as a subdomain of a closed Riemann surface of genus $g$. We 
are concerned with the set $\mathfrak{M}(R_{0})$ of such closed Riemann 
surfaces. We formulate the problem in the Teichm\"{u}ller space setting to 
investigate geometric properties of $\mathfrak{M}(R_{0})$. We show, among other 
things, that $\mathfrak{M}(R_{0})$ is a closed Lipschitz domain homeomorphic to 
a closed ball provided that $R_{0}$ is nonanalytically finite. 
\end{quote}

\section{Introduction}
\setcounter{equation}{0}
\label{sec:introduction}

Let $R_{0}$ be an open Riemann surface of finite genus $g$. If $g=0$, then the 
general uniformization theorem assures us that $R_{0}$ is conformally 
equivalent to a domain on the Riemann sphere 
$\hat{\mathbb{C}}:=\mathbb{C} \cup \{\infty\}$. In 1928 Bochner generalized it 
to the cases of higher genera in his study \cite{Bochner1928} on continuations 
of Riemann surfaces, showing that $R_{0}$ can be realized as a subdomain of a 
closed Riemann surface of the same genus $g$. While a closed Riemann surface of 
genus zero is essentially unique, this is not the case for closed Riemann 
surfaces of positive genus so that there may be two or more closed Riemann 
surfaces of genus $g$ with a domain conformally equivalent to $R_{0}$ provided 
that $g>0$. It is then natural to ask into which closed Riemann surfaces of 
genus $g$ the Riemann surface $R_{0}$ can be conformally embedded. 

Heins \cite{Heins1953} tackled the problem for $g=1$, and proved in 1953 that 
the set of closed Riemann surfaces of genus one which are continuations of 
$R_{0}$ is relatively compact in the moduli space of genus one. Four years 
later Oikawa \cite{Oikawa1957} formulated the problem in the context of 
Teichm\"{u}ller spaces, and discovered that the set of marked closed Riemann 
surfaces of genus $g$ into which $R_{0}$ can be mapped by a homotopically 
consistent conformal embedding is compact and connected in the Teichm\"{u}ller 
space of genus $g$. In 1987 the second author \cite{Shiba1987} improved 
Oikawa's result in the case of genus one by deducing that Oikawa's set is in 
fact a closed disk, which may degenerate to a singleton, with respect to the 
Teichm\"{u}ller distance. The authors then have been developing the theory 
mainly in the framework of Torelli spaces (see, for example, Masumoto 
\cite{Masumoto1999,Masumoto2000}, Schmieder-Shiba 
\cite{SchS1997,SchS1998,SchS2001} and Shiba 
\cite{Shiba1988,Shiba1989,Shiba2001,Shiba2004,Shiba2019}). 
For other results related to conformal embeddings of Riemann surfaces see 
Bourque \cite{Bourque2018}, Earle-Marden \cite{EM1978}, Fehlmann-Gardiner 
\cite{FG1995}, Hamano \cite{Hamano2018}, Hamano-Shiba-Yamaguchi \cite{HSY2017}, 
Horiuchi-Shiba \cite{HS1994}, Ioffe \cite{Ioffe1975}, Ito-Shiba 
\cite{IS1999,IS2001}, Kahn-Pilgrim-Thurston \cite{KPT2022}, Masumoto 
\cite{Masumoto2007,Masumoto2011}, Sasai \cite{Sasai2005,Sasai2006}, Shiba 
\cite{Shiba1984,Shiba1993} and Shiba-Shibata \cite{SS1985,SS1987}. Applications 
of conformal embedding theory to hyperbolic geometry and holomorphic mappings 
are found in Bourque \cite{Bourque2016} and Masumoto 
\cite{Masumoto1997,Masumoto2000,Masumoto2013,Masumoto2016,Masumoto2019}. 

In the present paper we address the problem for finite open Riemann surfaces of 
positive genus in the Teichm\"{u}ller space setting. One of our main purposes 
is to generalize our previous results in \cite{Shiba1987} to the cases of 
higher genera. 

To formulate the problem we introduce the category $\mathfrak{F}_{g}$ of marked 
Riemann surfaces of finite genus $g$ and conformally compatible continuous 
mappings in \S\ref{sec:categorical_definition}. Its class of objects consists 
of equivalence classes ${\bm S}=[S,\eta]$, where $S$ is a Riemann surface of 
genus $g$, not necessarily closed, and $\eta$ is a sense-preserving 
homeomorphism of $\dot{\Sigma}_{g}$ into $S$, where $\dot{\Sigma}_{g}$ is a 
fixed surface obtained from a closed oriented topological surface $\Sigma_{g}$ 
of genus $g$ by deleting one point (see 
Definition~\ref{defn:marked_Riemann_surface}). A continuous mapping of $\bm S$ 
into another marked Riemann surface ${\bm S}'=[S',\eta'] \in \mathfrak{F}_{g}$ 
is an equivalence class ${\bm f}=[f,\eta,\eta']$, where $f$ is a continuous 
mapping of $S$ into $S'$ (see Definition~\ref{defn:continuous_mapping}). The 
composition ${\bm f}' \circ {\bm f}$ is meaningful only for conformally 
compatible continuous mappings (see 
Definition~\ref{defn:conformal_compatible_continuous_mapping}). Also, a 
quadratic differential on ${\bm S}$ is an equivalence class 
${\bm \varphi}=[\varphi,\eta]$, where $\varphi$ is a quadratic differential on 
$S$ (see \S\ref{sec:closed_regular_self-welding}). 

The Teichm\"{u}ller space $\mathfrak{T}_{g}$ of genus $g$ is defined to be the 
full-subcategory of $\mathfrak{F}_{g}$ whose class of objects is composed of 
all marked closed Riemann surfaces of genus $g$. Equipped with the 
Teichm\"{u}ller distance, $\mathfrak{T}_{g}$ is a metric space homeomorphic to 
$\mathbb{R}^{2d_{g}}$, where $d_{g}=\max\{g,3g-3\}$. In fact, as is well-known, 
$\mathfrak{T}_{g}$ is a $d_{g}$-dimensional complex manifold biholomorphic to a 
bounded domain in $\mathbb{C}^{d_{g}}$. It is quite new to understand the 
Teichm\"{u}ller space $\mathfrak{T}_{g}$ as a full-subcategory of 
$\mathfrak{F}_{g}$. This setting fits our theory because we are required to 
deal with not only quasiconformal mappings between marked closed Riemann 
surfaces but also continuous mappings of marked compact bordered Riemann 
surfaces into marked compact Riemann surfaces. 

Let ${\bm R}_{0}=[R_{0},\theta_{0}]$ be a marked finite open Riemann surface 
of positive genus $g$; the fundamental group of $R_{0}$ is finitely generated, 
or equivalently, $R_{0}$ has finitely many handles and boundary components in 
the sense of Ker\'{e}kj\'{a}rt\'{o}-Sto\"{\i}low (for the definition of a 
boundary component, see \cite[I.36B]{AS1960}). It is said to be analytically 
finite if $R_{0}$ is conformally equivalent to a Riemann surface obtained from 
a closed Riemann surface by removing finitely many points. For 
${\bm R} \in \mathfrak{T}_{g}$ let $\CEmb_{\mathrm{hc}}({\bm R}_{0},{\bm R})$ 
be the set of homotopically consistent conformal embeddings of ${\bm R}_{0}$ 
into $\bm R$ (see 
Definition~\ref{defn:homotopically_consistent_continuous_mapping}). By a closed 
continuation of ${\bm R}_{0}$ we mean a pair $({\bm R},{\bm \iota})$, where 
${\bm R} \in \mathfrak{T}_{g}$ and 
${\bm \iota} \in \CEmb_{\mathrm{hc}}({\bm R}_{0},{\bm R})$. We are concerned 
with the set $\mathfrak{M}({\bm R}_{0})$ of ${\bm R} \in \mathfrak{T}_{g}$ for 
which $\CEmb_{\mathrm{hc}}({\bm R}_{0},{\bm R}) \neq \varnothing$. 

\begin{thm}
\label{thm:main:M(R_0)}
Let ${\bm R}_{0}$ be a marked finite open Riemann surface of positive genus 
$g$. Then\/ $\mathfrak{M}({\bm R}_{0})$ is either a singleton or a closed 
Lipschitz domain homeomorphic to a closed ball in\/ $\mathbb{R}^{2d_{g}}$. The 
former case occurs if and only if ${\bm R}_{0}$ is analytically finite. 
\end{thm}

The theorem is proved in the final section, where we show a stronger assertion 
that there is a homeomorphism of $\mathfrak{T}_{g}$ onto $\mathbb{R}^{2d_{g}}$ 
that maps $\mathfrak{M}({\bm R}_{0})$ onto a point or a closed ball. Being a 
closed Lipschitz domain, $\mathfrak{M}({\bm R}_{0})$ satisfies inner and outer 
cone conditions unless ${\bm R}_{0}$ is analytically finite. Actually, it 
satisfies an outer ball condition. If $g=1$, then its boundary 
$\partial\mathfrak{M}({\bm R}_{0})$ is smooth and a sphere (or rather circle) 
with respect to the Teichm\"{u}ller distance. On the other hand, if $g>1$, then 
the boundary is nonsmooth for some ${\bm R}_{0}$ and hence cannot be a sphere 
with respect to the Teichm\"{u}ller distance. 

For $K \geqq 1$ let $\mathfrak{M}_{K}({\bm R}_{0})$ denote the set of 
${\bm R} \in \mathfrak{T}_{g}$ into which ${\bm R}_{0}$ can be mapped by a 
homotopically consistent $K$-quasiconformal embedding. Thus 
$\mathfrak{M}_{1}({\bm R}_{0})=\mathfrak{M}({\bm R}_{0})$. If $K>1$, then 
$\mathfrak{M}_{K}({\bm R}_{0})$ is the closed $(\log K)/2$-neighborhood of 
$\mathfrak{M}({\bm R}_{0})$. 

\begin{thm}
\label{thm:main:M_K(R_0)}
Let ${\bm R}_{0}$ be a marked finite open Riemann surface of positive genus 
$g$. If $K>1$, then $\mathfrak{M}_{K}({\bm R}_{0})$ is a closed domain 
homeomorphic to a closed ball in $\mathbb{R}^{2d_{g}}$ and has a 
$C^{1}$-boundary. 
\end{thm}

Actually, there is a homeomorphism of $\mathfrak{T}_{g}$ onto 
$\mathbb{R}^{2d_{g}}$ such that $\mathfrak{M}_{K'}({\bm R}_{0})$, $K' \geqq K$, 
are mapped onto concentric closed balls. Theorem~\ref{thm:main:M_K(R_0)} is 
{\em not\/} a consequence of Theorem~\ref{thm:main:M(R_0)} but a tool for the 
proof of Theorem~\ref{thm:main:M(R_0)}. 

For the proofs of Theorems~\ref{thm:main:M(R_0)} and~\ref{thm:main:M_K(R_0)} we 
introduce a procedure called self-weldings of marked compact bordered Riemann 
surfaces ${\bm S}=[S,\eta]$, which is also useful to construct various examples 
of closed continuations. Let $A_{+}({\bm S})$ denote the set of nonzero 
holomorphic quadratic differentials ${\bm \varphi}=[\varphi,\eta]$ on $\bm S$ 
with $\varphi \geqq 0$ along the border $\partial S$. We use 
${\bm \varphi} \in A_{+}({\bm S})$ to identify arcs on the border. This 
operation gives rise to a continuous mapping $\bm \kappa$ of $\bm S$ onto a 
marked closed Riemann surface $\bm R$ holomorphic and injective on the interior 
${\bm S}^{\circ}$ of $\bm S$ together with a meromorphic quadratic differential 
$\bm \psi$ on $\bm R$. The pair $\langle{\bm R},{\bm \kappa}\rangle$ is 
referred to as a closed self-welding of $\bm S$, and the quadratic 
differentials $\bm \varphi$ and $\bm \psi$ are called a welder of 
$\langle{\bm R},{\bm \kappa}\rangle$ and the co-welder of $\bm \varphi$, 
respectively. Important are the self-weldings with holomorphic co-welders. Such 
a self-welding is said to be regular, or $\bm \varphi$-regular if we need to 
refer to the welder $\bm \varphi$. 

If ${\bm R}_{0}$ is nonanalytically finite, that is, it is finite but not 
analytically finite, then it is considered as the interior of a marked compact 
bordered Riemann surface $\bm S$, possibly, with finitely many points deleted. 
Each closed self-welding $\langle{\bm R},{\bm \kappa}\rangle$ of $\bm S$ yields 
a closed continuation $({\bm R},{\bm \kappa}|_{{\bm R}_{0}})$ of ${\bm R}_{0}$, 
called a closed self-welding continuation of ${\bm R}_{0}$. A welder of the 
self-welding is also referred to as a welder of the continuation. If 
$\langle{\bm R},{\bm \kappa}\rangle$ is $\bm \varphi$-regular, then 
$({\bm R},{\bm \kappa}|_{{\bm R}_{0}})$ is said to be $\bm \varphi$-regular. If 
this is the case, then ${\bm \kappa}|_{{\bm R}_{0}}$ is a Teichm\"{u}ller 
conformal embedding of ${\bm R}_{0}$ into $\bm R$, and $\bm \varphi$ is an 
initial quadratic differential of ${\bm \kappa}|_{{\bm R}_{0}}$. 

Let $A_{+}({\bm R}_{0})$ be the space of nonzero holomorphic quadratic 
differentials on ${\bm R}_{0}$ that can be extended to elements in 
$A_{+}({\bm S})$. Denote by $A_{L}({\bm R})$ the set of 
${\bm \varphi}=[\varphi,\theta_{0}]  \in A_{+}({\bm R}_{0})$ such that for any 
horizontal trajectory $a$ of $\varphi$ on $\partial S$ its $\varphi$-length is 
at most the half of the $\varphi$-length of the component of $\partial S$ 
including $a$. Clearly, $A_{L}({\bm R}_{0})$ is a proper subset of 
$A_{+}({\bm R}_{0})$. 

\begin{thm}
\label{thm:main:A_L(R_0)}
Let ${\bm R}_{0}$ be a marked nonanalytically finite open Riemann surface of 
positive genus $g$. An element of $A_{+}({\bm R}_{0})$ induces a closed 
regular self-welding continuation of ${\bm R}_{0}$ if and only if it belongs to 
$A_{L}({\bm R}_{0})$. 
\end{thm}

The theorem shows that the procedure of closed regular self-welding provides us 
with an explicit method of obtaining all Teichm\"{u}ller conformal embeddings 
of ${\bm R}_{0}$. Their importance is clarified in the next theorem. 

For ${\bm \varphi} \in A_{L}({\bm R}_{0})$ and 
${\bm R} \in \mathfrak{M}({\bm R}_{0})$ let 
$\CEmb_{\bm \varphi}({\bm R}_{0},{\bm R})$ be the set of 
${\bm \iota} \in \CEmb_{\mathrm{hc}}({\bm R}_{0},{\bm R})$ such that 
$({\bm R},{\bm \iota})$ is a closed ${\bm \varphi}$-regular self-welding 
continuation of ${\bm R}_{0}$; it can be empty. Set 
$\CEmb_{\bm \varphi}({\bm R}_{0})=
\bigcup_{{\bm R} \in \mathfrak{M}({\bm R}_{0})}
\CEmb_{\bm \varphi}({\bm R}_{0},{\bm R})$. If 
$\card\CEmb_{\bm \varphi}({\bm R}_{0})>1$, then $\bm \varphi$ is called 
exceptional. Let $A_{E}({\bm R}_{0})$ be the set of exceptional quadratic 
differentials in $A_{L}({\bm R}_{0})$. Define 
$\CEmb_{L}({\bm R}_{0})=\bigcup_{{\bm \varphi} \in A_{L}({\bm R}_{0})}
\CEmb_{\bm \varphi}({\bm R}_{0})$ and 
$\CEmb_{E}({\bm R}_{0})=\bigcup_{{\bm \varphi} \in A_{E}({\bm R}_{0})}
\CEmb_{\bm \varphi}({\bm R}_{0})$. 

For ${\bm \varphi} \in A_{L}({\bm R}_{0})$ let 
$\mathfrak{M}_{\bm \varphi}({\bm R}_{0})$ denote the set of 
${\bm R} \in \mathfrak{M}({\bm R}_{0})$ for which 
$\CEmb_{\bm \varphi}({\bm R}_{0},{\bm R}) \neq \varnothing$. If 
${\bm \varphi} \in A_{L}({\bm R}_{0}) \setminus A_{E}({\bm R}_{0})$, then 
$\mathfrak{M}_{\varphi}({\bm R}_{0})$ is a singleton. Set 
$\mathfrak{M}_{L}({\bm R}_{0})=\bigcup_{{\bm \varphi} \in A_{L}({\bm R}_{0})}
\mathfrak{M}_{\bm \varphi}({\bm R}_{0})$ and 
$\mathfrak{M}_{E}({\bm R}_{0})=\bigcup_{{\bm \varphi} \in A_{E}({\bm R}_{0})}
\mathfrak{M}_{\bm \varphi}({\bm R}_{0})$. 

\begin{thm}
\label{thm:main:CEmb(R_0,R)}
Let ${\bm R}_{0}$ be a marked nonanalytically finite open Riemann surface of 
positive genus $g$. 
\begin{list}{{\rm (\roman{claim})}}{\usecounter{claim}
\setlength{\topsep}{0pt}
\setlength{\itemsep}{0pt}
\setlength{\parsep}{0pt}
\setlength{\labelwidth}{\leftmargin}}
\item The boundary $\partial\mathfrak{M}({\bm R}_{0})$ coincides with\/ 
$\mathfrak{M}_{L}({\bm R}_{0})$. 

\item If ${\bm \varphi} \in A_{L}({\bm R}_{0})$ and 
${\bm R} \in \mathfrak{M}_{\bm \varphi}({\bm R}_{0})$, then 
$\CEmb_{\bm \varphi}({\bm R}_{0},{\bm R})=
\CEmb_{\mathrm{hc}}({\bm R}_{0},{\bm R})$. 

\item  If $g \geqq 3$, then $\mathfrak{M}_{E}({\bm R}_{0})$ has a nonempty 
interior with respect to the relative topology on 
$\partial\mathfrak{M}({\bm R}_{0})$. 
\end{list}
\end{thm}

We examine the behavior of the extremal length 
function to prove~(i) and~(ii) though they also follow from Bourque 
\cite{Bourque2016} and Kahn-Pilgrim-Thurston \cite{KPT2022}. The set 
$A_{E}({\bm R}_{0})$ is nowhere dense in $A_{L}({\bm R}_{0})$ 
(Theorem~\ref{thm:exceptional_border_component}). Nevertheless, quadratic 
differentials in $A_{E}({\bm R}_{0})$ yield abundant boundary points of 
$\mathfrak{M}({\bm R}_{0})$ through closed regular self-weldings, as 
assertion~(iii) claims. If $g=1$, then $A_{E}({\bm R}_{0})=\varnothing$. 
If $g=2$, then there are examples of ${\bm R}_{0}$ for which 
$\mathfrak{M}_{E}({\bm R}_{0})$ has a nonempty interior in 
$\partial\mathfrak{M}({\bm R}_{0})$. 

Let $\mathscr{MF}(\Sigma_{g})$ be the space of measured foliations on 
$\Sigma_{g}$. The horizontal foliation of a holomorphic quadratic differential 
${\bm \psi}=[\psi,\theta]$ on a marked closed Riemann surface 
${\bm R}=[R,\theta]$ determines an element $\mathcal{H}_{\bm R}({\bm \psi})$ of 
$\mathscr{MF}(\Sigma_{g})$ through $\theta$. This defines a homeomorphism 
$\mathcal{H}_{\bm R}$ of the space $A({\bm R})$ of holomorphic quadratic 
differentials on $\bm R$ onto $\mathscr{MF}(\Sigma_{g})$ by 
Hubbard-Masur \cite{HM1979}. 

Now, for ${\bm \varphi} \in A_{L}({\bm R}_{0})$, taking a closed 
$\bm \varphi$-regular self-welding continuation $({\bm R},{\bm \iota})$ of 
${\bm R}_{0}$ and denoting by $\bm \psi$ the co-welder of $\bm \varphi$, we 
set $\mathcal{H}_{{\bm R}_{0}}({\bm \varphi})=\mathcal{H}_{\bm R}({\bm \psi})$. 
Then $\mathcal{H}_{{\bm R}_{0}}$ is a well-defined mapping of 
$A_{L}({\bm R}_{0})$ into $\mathcal{MF}(\Sigma_{g}) \setminus \{0\}$. 

\begin{thm}
\label{thm:main:measured_foliation}
If ${\bm R}_{0}$ is a marked nonanalytically finite open Riemann surface of 
positive genus $g$, then $\mathcal{H}_{{\bm R}_{0}}$ is a homeomorphism of 
$A_{L}({\bm R}_{0})$ onto $\mathscr{MF}(\Sigma_{g}) \setminus \{0\}$. 
\end{thm}

The extremal length function $\Ext$ is a nonnegative continuous function on 
$\mathfrak{T}_{g} \times \mathscr{MF}(\Sigma_{g})$. Specifically, if 
$({\bm R},\mathcal{F}) \in \mathfrak{T}_{g} \times \mathscr{MF}(\Sigma_{g})$, 
then $\Ext({\bm R},\mathcal{F})=\Ext_{\mathcal{F}}({\bm R})$ is exactly 
$\|\mathcal{H}_{\bm R}^{-1}(\mathcal{F})\|_{\bm R}$, where 
$\|{\bm \varphi}\|_{\bm S}$ stands for the $L^{1}$-norm of a quadratic 
differential $\bm \varphi$ on a marked Riemann surface $\bm S$. 

For $\mathcal{F} \in \mathscr{MF}(\Sigma_{g}) \setminus \{0\}$ let 
$\mathfrak{M}_{\mathcal{F}}({\bm R}_{0})$ denote the set of points of 
$\mathfrak{M}({\bm R}_{0})$ at which the restriction of $\Ext_{\mathcal{F}}$ to 
$\mathfrak{M}({\bm R}_{0})$ attains its maximum. It is nonempty since 
$\mathfrak{M}({\bm R}_{0})$ is compact. 

\begin{thm}
\label{thm:main:Ext}
Let ${\bm R}_{0}$ be a marked nonanalytically finite open Riemann surface of 
positive genus $g$. 
\begin{list}{{\rm (\roman{claim})}}{\usecounter{claim}
\setlength{\topsep}{0pt}
\setlength{\itemsep}{0pt}
\setlength{\parsep}{0pt}
\setlength{\labelwidth}{\leftmargin}}
\item Let  $\mathcal{F} \in \mathscr{MF}(\Sigma_{g}) \setminus \{0\}$, and set 
${\bm \varphi}=\mathcal{H}_{{\bm R}_{0}}^{-1}(\mathcal{F})$. Then 
$\mathfrak{M}_{\mathcal{F}}({\bm R}_{0})=
\mathfrak{M}_{\bm \varphi}({\bm R}_{0})$ and 
\begin{equation}
\label{eq:Ext:maximum}
\max\Ext_{\mathcal{F}}(\mathfrak{M}({\bm R}_{0}))=
\max\Ext_{\mathcal{F}}(\partial\mathfrak{M}({\bm R}_{0}))=
\|{\bm \varphi}\|_{{\bm R}_{0}}. 
\end{equation}

\item If $g \geqq 3$, then there exists 
$\mathcal{F} \in \mathscr{MF}(\Sigma_{g}) \setminus \{0\}$ such that 
$\card\mathfrak{M}_{\mathcal{F}}({\bm R}_{0})=2^{\aleph_{0}}$ while 
$\card\CEmb_{\mathrm{hc}}({\bm R}_{0},{\bm R})=1$  for all 
${\bm R} \in \mathfrak{M}_{\mathcal{F}}({\bm R}_{0})$. 

\item Let ${\bm R} \in \mathfrak{T}_{g}$. Then 
${\bm R} \in \mathfrak{M}({\bm R}_{0})$ if and only if 
\begin{equation}
\label{eq:extremal_length:description_of_M(R_0)}
\Ext_{\mathcal{F}}({\bm R}) \leqq
\max\Ext_{\mathcal{F}}(\partial\mathfrak{M}({\bm R}_{0}))
\end{equation}
for all $\mathcal{F} \in \mathscr{MF}(\Sigma_{g}) \setminus \{0\}$. 
\end{list}
\end{thm}

Kahn-Pilgrim-Thurston \cite[Theorem~1]{KPT2022} gives a necessary 
and sufficient condition for ${\bm R} \in \mathfrak{T}_{g}$ to belong to 
$\mathfrak{M}({\bm R}_{0})$ in terms of stretch factors. A stretch factor of a 
topological embedding of ${\bm R}_{0}$ into $\bm R$ is defined with the 
extremal lengths of simple closed multi-curves on ${\bm R}_{0}$. 
Theorem~\ref{thm:main:Ext}~(iii) tells us which multi-curves we should take 
into account to calculate the stretch factor. Specifically, the authors of 
\cite{KPT2022} employ the Jenkins-Strebel quadratic differentials in 
$A_{+}({\bm R}_{0})$ to obtain the stretch factors. Our theorem asserts that 
those in $A_{L}({\bm R}_{0})$ are sufficient. 

Besides Theorems~\ref{thm:main:A_L(R_0)} -- \ref{thm:main:Ext}, the 
Teichm\"{u}ller geodesic rays induced from the co-welders of elements of 
$A_{L}({\bm R}_{0})$ play an important role in the proof of 
Theorem~\ref{thm:main:M_K(R_0)}. We then apply analytic properties of the 
extremal length function to establish Theorem~\ref{thm:main:M(R_0)}. Existence 
of linearly independent elements ${\bm \varphi}_{j} \in A_{L}({\bm R}_{0})$, 
$j=1,2$, with $\mathfrak{M}_{{\bm \varphi}_{1}}({\bm R}_{0}) \cap
\mathfrak{M}_{{\bm \varphi}_{2}}({\bm R}_{0}) \neq \varnothing$ is one of the 
biggest obstacles in the proof of Theorem~\ref{thm:main:M(R_0)}. 

{\bf Acknowledgments.} This research is supported in part by JSPS KAKENHI 
Grant Numbers 18K03334 and 22K03556. The authors express sincere thanks to 
Hideki Miyachi, a brilliant expert on Teichm\"{u}ller theory, for his 
invaluable comments and suggestions. Without his help they could not finish the 
article. They are also grateful to Shuhei Masumoto and Shingo Okuyama for 
formulating our theory within the scheme of category theory. The two colleagues 
help the authors write a clearer paper. Last but not least, the authors thank 
Ken-ichi Sakan for his continuing encouragement. 

\section{Preliminaries}
\setcounter{equation}{0}
\label{sec:preliminaries}

In the present article we are concerned with conformal embeddings of a Riemann 
surface into another. We begin with confirming terminology and notation. 

Let $X$ and $Y$ be topological spaces. Denote by $\Cont(X,Y)$ the set of 
continuous mappings of $X$ into $Y$; we avoid employing the usual notation 
$C(X,Y)$ so as not to confuse it with the set of conformal mappings in the case 
where $X$ and $Y$ are Riemann surfaces. The subset of homeomorphisms of $X$ 
onto $Y$ is denoted by $\Homeo(X,Y)$. A {\em topological embedding\/} of $X$ 
into $Y$ is, by definition, a mapping $f$ of $X$ into $Y$ for which the 
correspondence $x \mapsto f(x)$ defines an element of $\Homeo(X,f(X))$, where 
$f(X)$ is endowed with the relative topology induced from the topology of $Y$. 
We denote by $\TEmb(X,Y)$ the set of topological embeddings of $X$ into $Y$. 
If $X$ and $Y$ are oriented surfaces, then $\Homeo^{+}(X,Y)$ and 
$\TEmb^{+}(X,Y)$ denote the sets of sense-preserving elements in $\Homeo(X,Y)$ 
and $\TEmb(X,Y)$, respectively. 

For $f_{k} \in \Cont(X,Y)$, $k=0,1$, we write $f_{0} \simeq f_{1}$ if $f_{0}$ 
is homotopic to $f_{1}$. Also, for two loops $c_{0}$ and $c_{1}$ on $X$ 
we use the notation $c_{0} \simeq c_{1}$ to mean that $c_{0}$ is freely 
homotopic to $c_{1}$. This usage of $\simeq$ is consistent with the previous 
one as we can regard $c_{0}$ and $c_{1}$ as continuous mappings of the unit 
circle into $X$. 

A {\em surface\/} means a connected $2$-manifold whose topology has a countable 
base. A surface with boundary is often called a {\em bordered\/} surface or a 
surface with {\em border\/}. As is well-known, every surface is triangulable, 
and can be endowed with conformal structure, or complex structure provided that 
it is orientable. 

A {\em Riemann surface\/} is a connected complex manifold of dimension one (see 
Ahlfors-Sario \cite[II.1E]{AS1960} and Strebel 
\cite[Definition~1.1]{Strebel1984}), and a {\em bordered Riemann surface\/} is 
a connected 1-dimensional complex manifold with boundary (see 
\cite[II.3A]{AS1960} and \cite[Definition~1.2]{Strebel1984}). By abuse of 
language we refer to bordered Riemann surfaces also as Riemann surfaces. Thus, 
when we speak of a Riemann surface, it may be a bordered Riemann surface. A 
bordered Riemann surface is sometimes called a {\em Riemann surface with 
border}. A {\em Riemann surface without border\/} means a Riemann surface which 
is not a bordered Riemann surface. A Riemann surface without border is called 
{\em closed\/} (resp.\ {\em open}) if it is compact (resp.\ noncompact). Thus a 
compact Riemann surface is either a closed Riemann surface or a compact 
bordered Riemann surface. By a {\em torus\/} we mean a closed Riemann surface 
of genus one. 

A holomorphic mapping of a Riemann surface into another is called a 
{\em conformal embedding\/} if it is a topological embedding at the same time, 
where holomorphic mappings of a bordered Riemann surface are supposed to be 
analytic on the border. A biholomorphic mapping of a Riemann surface onto 
another is also referred to as a {\em conformal homeomorphism}. Conformal 
homeomorphisms are conformal embeddings. Two Riemann surfaces are said to be 
{\em conformally equivalent\/} to each other if there is a conformal 
homeomorphism of one onto the other. 

\begin{exmp}
\label{exmp:conformal_embedding}
Set $S=\{z \in \mathbb{C} \mid 1<|z|<2,0 \leqq \arg z<\pi\}$; it is a bordered 
Riemann surface. Though the holomorphic mapping $f:S \to \mathbb{C}$ defined by 
$f(z)=z^{2}$ is injective, it is not a conformal embedding since $f$ is not a 
homeomorphism of $S$ onto its image $f(S)=\{w \in \mathbb{C} \mid 1<|w|<4\}$. 
\end{exmp}

For Riemann surfaces $S$ and $R$ the sets of conformal embeddings of $S$ into 
$R$ and conformal homeomorphisms of $S$ onto $R$ are denoted by $\CEmb(S,R)$ 
and $\CHomeo(S,R)$, respectively. By a {\em conformal automorphism\/} of $S$ we 
mean a conformal homeomorphism of $S$ onto itself. Let $\Aut(S)$ be the group 
of conformal automorphisms of $S$. Thus $\Aut(S)=\CHomeo(S,S)$. Its subgroup 
consisting of $\kappa \in \Aut(S)$ satisfying $\kappa \simeq \id_{S}$ is 
denoted by $\Aut_{0}(S)$. 

\begin{exmp}
\label{exmp:conformal_automorphism:complex_plane}
Conformal automorphisms of $\mathbb{C}$ are of the form 
$\kappa_{a,b}:z \mapsto az+b$, where 
$(a,b) \in (\mathbb{C} \setminus \{0\}) \times \mathbb{C}$. Since 
$(\mathbb{C} \setminus \{0\}) \times \mathbb{C}$ is arcwise connected, we know 
that $\Aut_{0}(\mathbb{C})=\Aut(\mathbb{C})$. The translations $\kappa_{1,b}$, 
$b \in \mathbb{C}$, form a subgroup $\Aut_{\mathrm{tr}}(\mathbb{C})$. 
\end{exmp}

\begin{exmp}
\label{exmp:conformal_automorphism:torus}
For each $\tau \in \mathbb{H}$ let $\Gamma_{\tau}$ denote the additive subgroup 
of $\mathbb{C}$ generated by $1$ and $\tau$. We endow the quotient group 
$T_{\tau}:=\mathbb{C}/\Gamma_{\tau}$ with conformal structure so that the 
natural projection $\Pi_{\tau}:\mathbb{C} \to T_{\tau}$ is a holomorphic 
universal covering map. Then $T_{\tau}$ is a torus. As is well-known, each 
torus is conformally equivalent to some $T_{\tau}$. For later use we introduce 
a few more notations. We denote by $A_{\tau}$ and $B_{\tau}$ the simple loops 
on $T_{\tau}$ defined by $A_{\tau}(t)=\Pi_{\tau}(t)$, 
$B_{\tau}(t)=\Pi_{\tau}(t\tau)$, $t \in [0,1]$. For any $\beta \in T_{\tau}$ 
the translation $\iota_{\beta}:p \mapsto p+\beta$ is a conformal automorphism 
of $T_{\tau}$ for which the image loops $(\iota_{\beta})_{*}A_{\tau}$ and 
$(\iota_{\beta})_{*}B_{\tau}$ are freely homotopic to $A_{\tau}$ and 
$B_{\tau}$, respectively. If $b \in \Pi_{\tau}^{-1}(\beta)$, then we have 
$\Pi_{\tau} \circ \kappa_{1,b}=\iota_{\beta} \circ \Pi_{\tau}$. The set 
$\Aut_{\mathrm{tr}}(T_{\tau})$ of translations of $T_{\tau}$ is a subgroup of 
$\Aut_{0}(T_{\tau})$. In fact, $\Aut_{\mathrm{tr}}(T_{\tau})$ is identical with 
$\Aut_{\mathrm{0}}(T_{\tau})$ (see 
Corollary~\ref{cor:conformal_automorphism:homotopic_to_id}~(ii)). 
\end{exmp}

\begin{defn}[continuation]
\label{defn:continuation}
Let $S$ be a Riemann surface. A {\em continuation\/} of $S$ is, by definition, 
a pair $(R,\iota)$ where $R$ is a Riemann surface and $\iota$ is a conformal 
embedding of $S$ into $R$. 
\end{defn}

Two continuations $(R_{j},\iota_{j})$, $j=1,2$, of $S$ are defined to be 
equivalent to each other if there is $\kappa \in \CHomeo(R_{1},R_{2})$ such 
that $\iota_{2}=\kappa \circ \iota_{1}$. A continuation $(R,\iota)$ of $S$ is 
called {\em closed\/} (resp.\ {\em open, compact}) if $R$ is closed (resp.\ 
open, compact). If $\iota(S)$ is dense in $R$, then we say that $(R,\iota)$ is 
a {\em dense\/} continuation of $S$. A continuation $(R,\iota)$ of $S$ is said 
to be {\em genus-preserving\/} if each connected component of 
$R \setminus \iota(S)$ is topologically embedded into $\mathbb{C}$. In the case 
where $S$ is of finite genus, a continuation $(R,\iota)$ of $S$ is 
genus-preserving if and only if $R$ is of the same genus as $S$. Observe that 
if $S$ is a bordered Riemann surface, then there is a genus-preserving open 
continuation $(R,\iota)$ of $S$ such that $\iota(S)$ is a retract of $R$. If 
$(R,\iota)$ is a continuation of $S$ and $(R',\iota')$ is a continuation of 
$R$, then $(R',\iota' \circ \iota)$ is a continuation of $S$. 

\begin{rem}
In \cite[III.3D]{RS1968} Rodin and Sario call a continuation $(R,\iota)$ of 
$S$ compact if $R$ is closed. We employ the term ``closed continuation" to 
remain consistent. In \cite{Shiba2019} genus-preserving closed continuations 
are referred to as closings. 
\end{rem}

Let $f$ be a quasiconformal embedding of a Riemann surface $S$ without border 
into a Riemann surface. If $z$ is a local parameter around a point $p \in S$ 
and $w$ is a local parameter around $f(p)$, then with respect to these 
parameters $f$ is regarded as a complex function $w=f(z)$. The quantity 
$\bar{\partial}f/\partial f=(f_{\bar{z}}\,d\bar{z})/(f_{z}\,dz)$ does not 
depend on a particular choice of $w$ and defines a measurable $(-1,1)$-form 
$\mu_{f}$ on $S$ called the {\em Beltrami differential\/} of $f$. Then 
$|\mu_{f}|$ is a measurable function on $S$, whose $L^{\infty}$-norm 
$\|\mu_{f}\|_{\infty}$ is less than one. The {\em maximal dilatation} $K(f)$ of 
$f$ is defined by $K(f)=(1+\|\mu_{f}\|_{\infty})/(1-\|\mu_{f}\|_{\infty})$. 
Denote by $\QCEmb(S,R)$ and $\QCHomeo(S,R)$ the sets of quasiconformal 
embeddings of $S$ into $R$ and quasiconformal homeomorphisms of $S$ onto $R$, 
respectively. 

Let $h_{j} \in \QCHomeo(S,R_{j})$, $j=1,2$, where $S$ and $R_{j}$ are Riemann 
surfaces without border. If $\mu_{h_{1}}=\mu_{h_{2}}$ almost everywhere on $S$, 
then $h_{2} \circ h_{1}^{-1} \in \CHomeo(R_{1},R_{2})$. 
For any measurable $(-1,1)$-form $\mu$ on $S$ with $\|\mu\|_{\infty}<1$ there 
is a quasiconformal homeomorphism $h$ of $S$ onto a Riemann surface without 
border for which $\mu_{h}=\mu$ almost everywhere on $S$. 

It is convenient to define quasiconformal embeddings also on bordered Riemann 
surfaces. Let $S$ be a bordered Riemann surface. A topological embedding $f$ of 
$S$ into another Riemann surface $R$ is called {\em quasiconformal\/} if there 
are continuations $(S',\iota')$ and $(R',\kappa')$ of $S$ and $R$, 
respectively, where $S'$ and $R'$ are Riemann surfaces without border, together 
with $f' \in \QCEmb(S',R')$ such that $\kappa' \circ f=f' \circ \iota'$. We can 
then speak of the maximal dilatation $K(f)$. Note that $K(f') \geqq K(f)$. Each 
component of the border of $S$ is mapped by $f$ onto a simple arc or loop of 
vanishing area. 

Let $K$ be a real constant with $K \geqq 1$. A quasiconformal embedding is 
called {\em $K$-quasiconformal\/} if its maximal dilatation does not exceed 
$K$. Conformal embeddings are $1$-quasiconformal, and vice versa. 

A subset $S'$ of a Riemann surface $S$ is called a {\em subsurface\/} of $S$ if 
it carries its own conformal structure such that the inclusion mapping of $S'$ 
into $S$ is a conformal embedding of $S'$ into $S$. Each component of the 
border of $S'$, if any, is an analytic curve on $S$. Subdomains of $S$, or 
nonempty connected open subsets of $S$, are subsurfaces of $S$. 

Let $S$ be a bordered Riemann surface. We denote its border by $\partial S$, 
and set $S^{\circ}=S \setminus \partial S$, called the {\em interior\/} of $S$. 
Note that $S^{\circ}$ is a subsurface of $S$ and that $\partial S$, which is 
also a topological boundary of $S^{\circ}$, is a 1-dimensional real analytic 
submanifold of $S$. If $\iota$ denotes the inclusion mapping of $S^{\circ}$ 
into $S$, then we call $(S,\iota)$ the {\em canonical continuation\/} of 
$S^{\circ}$. In the case where $S$ is a Riemann surface without border, we 
define $\partial S=\varnothing$ and $S^{\circ}=S$ for convenience, and call 
$S^{\circ}$ the interior of $S$. 

\begin{defn}[finite Riemann surface]
\label{defn:finite_Riemann_surface}
A Riemann surface is called {\em finite\/} if its fundamental group is finitely 
generated. 
\end{defn}

A Riemann surface $S$ is said to be {\em analytically\/} finite if there is a 
closed continuation $(R,\iota)$ of $S$ such that $R \setminus \iota(S)$ is a 
finite set, that is, $\card(R \setminus \iota(S))<\aleph_{0}$, where $\card E$ 
denotes the cardinal number of a set $E$. A Riemann surface is said to be a 
{\em nonanalytically\/} finite if it is finite but not analytically finite. 

Analytically finite Riemann surfaces are finite Riemann surfaces without 
border. Closed Riemann surfaces are analytically finite. If $S$ is an 
analytically finite Riemann surface and $h$ is a quasiconformal homeomorphism 
of $S$ onto another Riemann surface $R$, then $R$ is also analytically finite. 

\begin{defn}[natural continuation]
\label{defn:natural_continuation}
Let $S$ be a finite open Riemann surface. A compact continuation 
$(\breve{S},\breve{\iota})$ of $S$ is called {\em natural\/} if 
$(\breve{S})^{\circ} \setminus \breve{\iota}(S)$ contains at most finitely many 
points. 
\end{defn}

Clearly, $(\breve{S},\breve{\iota})$ is a dense continuation of $S$. To obtain 
a natural compact continuation of $S$ we have only to consider the case where 
$S$ is nonanalytically finite. Then the universal covering Riemann surface of 
$S$ is conformally equivalent to the upper half plane 
$\mathbb{H}:=\{z \in \mathbb{C} \mid \im z>0\}$. The covering transformation 
group $\Gamma$ is a torsion-free Fuchsian group of the second kind keeping 
$\mathbb{H}$ invariant, and $\mathbb{H}/\Gamma$ is conformally equivalent to 
$S$. Let $\Omega(\Gamma)$ denote the region of discontinuity of $\Gamma$, and 
set $S'=(\bar{\mathbb{H}} \cap \Omega(\Gamma))/\Gamma$, where 
$\bar{\mathbb{H}}$ is the closure of $\mathbb{H}$ in the extended complex plane 
$\hat{\mathbb{C}}:=\mathbb{C} \cup \{\infty\}$. If $S'$ is compact, then we set 
$\breve{S}=S'$. Otherwise, $S'$ has finitely many punctures and we let 
$\breve{S}$ be the Riemann surface $S'$ with the punctures filled in. In any 
case we have a conformal embedding $\breve{\iota}$ of $S$ into $\breve{S}$ to 
obtain a desired continuation $(\breve{S},\breve{\iota})$. As is easily 
verified, two natural compact continuations of $S$ are equivalent to each 
other. 

\begin{exmp}
\label{exmp:natural_compactification}
Let $\mathbb{D}$ and $S$ be the open unit disk and the open square 
$(-1,1) \times (-1,1)$ in the complex plane $\mathbb{C}$, respectively. They 
are subsurfaces of $\mathbb{C}$. M\"{o}bius transformations of the closure 
$\bar{\mathbb{D}}$ onto $\bar{\mathbb{H}}$ make $\bar{\mathbb{D}}$ a compact 
bordered Riemann surface so that the inclusion mapping $\iota$ of 
$\bar{\mathbb{D}}$ into $\mathbb{C}$ is a conformal embedding. In other words, 
$(\mathbb{C},\iota)$ is a continuation of $\bar{\mathbb{D}}$. Thus 
$\bar{\mathbb{D}}$ is a subsurface of $\mathbb{C}$. Take 
$\breve{\iota} \in \CEmb(S,\bar{\mathbb{D}})$ such that 
$\breve{\iota}(S)=\mathbb{D}$. Then $(\bar{\mathbb{D}},\breve{\iota})$ is a 
natural compact continuation of $S$. The embedding $\breve{\iota}$ is extended 
to a homeomorphism of the topological closure $\bar{S}=[-1,1] \times [-1,1]$ 
onto $\bar{\mathbb{D}}$. Its inverse defines a topological embedding $h$ of 
$\bar{\mathbb{D}}$ into $\mathbb{C}$, but $(\mathbb{C},h)$ is {\em not\/} a 
continuation of $\bar{\mathbb{D}}$, for, $h$ is not analytic at four points on 
the border $\partial\mathbb{D}$. Obviously, $\bar{S} \setminus \{\pm 1 \pm i\}$ 
is a subsurface of $\mathbb{C}$ though $\bar{S}$ is not. Also, if 
$\dot{S}=S \setminus \{0\}$, then $(\bar{\mathbb{D}},\breve{\iota}|_{\dot{S}})$ 
is a natural compact continuation of $\dot{S}$. 
\end{exmp}

\section{Categorical definition of Teichm\"{u}ller spaces}
\setcounter{equation}{0}
\label{sec:categorical_definition}

Let $R_{0}$ be a Riemann surface of finite genus. Then it is conformally 
embedded into a closed Riemann surface of the {\em same\/} genus by Bochner 
\cite[Satz~V]{Bochner1928}, which gives rise to a genus-preserving closed 
continuation of $R_{0}$. We are interested in the set of closed Riemann 
surfaces of the same genus as $R_{0}$ into which $R_{0}$ can be conformally 
embedded, and will work in the context of Teichm\"{u}ller theory. 

There are several equivalent definitions of Teichm\"{u}ller spaces for closed 
Riemann surfaces. In the present article we adopt the following definition. The 
point is how we should define continuous mappings between marked Riemann 
surfaces. 

We start with introducing the {\em category\/ $\mathfrak{F}_{g}$ of marked 
Riemann surfaces of genus $g$ and conformally compatible continuous mappings}. 
Fix a closed oriented surface $\Sigma_{g}$ of positive genus $g$, and remove 
one point from $\Sigma_{g}$ to obtain a once-punctured closed oriented surface 
$\dot{\Sigma}_{g}$. For the sake of definiteness, using the notations in 
Example~\ref{exmp:conformal_automorphism:torus}, we set 
$\Sigma_{1}=T_{i}=T_{\sqrt{-1}}$ and 
$\dot{\Sigma}_{1}=\Sigma_{1} \setminus \{\Pi_{i}((1+i)/2)\}$. 

\begin{defn}[handle mark]
\label{defn:handle_mark}
Let $g$ be a positive integer, and let $S$ be an oriented surface of genus $g$ 
or higher, possibly, with border. By a {\em $g$-handle mark\/} of $S$ we mean 
an element of $\TEmb^{+}(\dot{\Sigma}_{g},S)$. 
\end{defn}

If $\eta$ is a $g$-handle mark of $S$, then 
$\eta(\dot{\Sigma}_{g}) \subset S^{\circ}$. Thus $\eta$ is also considered as a 
$g$-handle mark of $S^{\circ}$. 

\begin{prop}
\label{prop:1-handle_mark:conformal_automorphism}
Let $\chi$ be a\/ $1$-handle mark of a Riemann surface $S$ of positive genus, 
and assume that $\kappa \in \Aut(S)$ satisfies $\kappa \circ \chi \simeq \chi$. 
\begin{list}{{\rm (\roman{claim})}}{\usecounter{claim}
\setlength{\topsep}{0pt}
\setlength{\itemsep}{0pt}
\setlength{\parsep}{0pt}
\setlength{\labelwidth}{\leftmargin}}
\item If $S$ is not a torus, then $\kappa=\id_{S}$. 

\item If $S$ is a torus, then $\kappa \in \Aut_{\mathrm{tr}}(S)$. 
\end{list}
\end{prop}

\begin{proof}
(i) The interior $S^{\circ}$ carries a hyperbolic metric, that is, a complete 
conformal metric of constant curvature $-4$. Let $A$ and $B$ be the hyperbolic 
geodesic loops on $S^{\circ}$ that are freely homotopic to 
$\chi_{*}A_{\sqrt{-1}}$ and $\chi_{*}B_{\sqrt{-1}}$, respectively (for the 
definition of $A_{\sqrt{-1}}$ and $B_{\sqrt{-1}}$ see 
Example~\ref{exmp:conformal_automorphism:torus}). Then $\kappa_{*}A \simeq A$ 
and $\kappa_{*}B \simeq B$ due to $\kappa \circ \chi \simeq \chi$. Since 
$\kappa$ is isometric with respect to the hyperbolic metric, $\kappa_{*}A$ and 
$\kappa_{*}B$ are also hyperbolic geodesic loops and hence trace $A$ and $B$ in 
the same direction, respectively. Because $A$ and $B$ have exactly one point, 
say $p$, in common, we know that $\kappa(p)=p$, which implies $\kappa$ 
coincides with $\id_{S}$ on $A$ and $B$. Therefore $\kappa=\id_{S}$ by the 
identity theorem. 

(ii) We may assume that $S$ is identical with some $T_{\tau}$ in 
Example~\ref{exmp:conformal_automorphism:torus}. The conformal automorphism 
$\kappa$ of $T_{\tau}$ is lifted to a conformal automorphism $\tilde{\kappa}$ 
of $\mathbb{C}$. The 1-form $\tilde{\omega}:=dz$ on $\mathbb{C}$ is projected 
to a holomorphic 1-form $\omega$ on $T_{\tau}$. Since $\kappa_{*}A_{\tau}$ is 
freely homotopic to $A_{\tau}$, the pull-back $\kappa^{*}\omega$ of $\omega$ 
via $\kappa$ has the same period along $A_{\tau}$ as $\omega$: 
$$
\int_{A_{\tau}} \kappa^{*}\omega=\int_{\kappa_{*}A_{\tau}} \omega=
\int_{A_{\tau}} \omega.
$$
As $T_{\tau}$ is a torus, we obtain $\kappa^{*}\omega=\omega$ and hence 
$\tilde{\kappa}^{*}\tilde{\omega}=\tilde{\omega}$, or equivalently, 
$\tilde{\kappa}(z)=z+C$, $z \in \mathbb{C}$, for some constant $C$. 
Consequently, $\kappa$ belongs to $\Aut_{\mathrm{tr}}(T_{\tau})$. 
\end{proof}

\begin{cor}
\label{cor:conformal_automorphism:homotopic_to_id}
Let $S$ be a Riemann surface of positive genus. 
\begin{list}{{\rm (\roman{claim})}}{\usecounter{claim}
\setlength{\topsep}{0pt}
\setlength{\itemsep}{0pt}
\setlength{\parsep}{0pt}
\setlength{\labelwidth}{\leftmargin}}
\item If $S$ is not a torus, then $\Aut_{0}(S)$ is trivial. 

\item If $S$ is a torus, then $\Aut_{0}(S)=\Aut_{\mathrm{tr}}(S)$. 
\end{list}
\end{cor}

\begin{proof}
Choose a $1$-handle mark $\chi$ of $S$. If $\kappa \in \Aut_{0}(S)$, then 
$\kappa \circ \chi \simeq \chi$. Thus the corollary follows at once from 
Proposition~\ref{prop:1-handle_mark:conformal_automorphism}. 
\end{proof}

\begin{cor}
\label{cor:conformal_homeomorphism:uniqueness}
Let $\chi_{1}$ be a\/ $1$-handle mark of a Riemann surface $S_{1}$ of positive 
genus, and suppose that $\kappa_{0},\kappa_{1} \in \CHomeo(S_{1},S_{2})$ 
satisfy $\kappa_{0} \circ \chi_{1} \simeq \kappa_{1} \circ \chi_{1}$. 
\begin{list}{{\rm (\roman{claim})}}{\usecounter{claim}
\setlength{\topsep}{0pt}
\setlength{\itemsep}{0pt}
\setlength{\parsep}{0pt}
\setlength{\labelwidth}{\leftmargin}}
\item If $S_{1}$ is not a torus, then $\kappa_{0}=\kappa_{1}$. 

\item If $S_{1}$ is a torus, then $\kappa_{0} \simeq \kappa_{1}$. 
\end{list}
\end{cor}

\begin{defn}[marked Riemann surface]
\label{defn:marked_Riemann_surface}
Consider all pairs $(S,\eta)$, where $S$ is a Riemann surface $S$ of genus $g$ 
and $\eta$ is a $g$-handle mark of $S$. We say that $(S_{1},\eta_{1})$ is 
equivalent to $(S_{2},\eta_{2})$ if there is $\kappa \in \CHomeo(S_{1},S_{2})$ 
such that $\kappa \circ \eta_{1} \simeq \eta_{2}$. We call each equivalence 
class ${\bm S}=[S,\eta]$ a {\em marked\/} Riemann surface of genus $g$. 
\end{defn}

A marked Riemann surface ${\bm S}=[S,\eta]$ is called a marked {\em bordered\/} 
Riemann surface or a marked Riemann surface {\em with border\/} if $S$ is a 
bordered Riemann surface. Otherwise, $\bm S$ is referred to as a marked Riemann 
surface {\em without border}. If $S$ is closed (resp.\ open, compact), then 
$\bm S$ is called {\em closed\/} (resp.\ {\em open}, {\em compact}). We also 
say that $\bm S$ is a marked {\em closed\/} (resp.\ {\em open}, {\em compact}) 
Riemann surface. The meaning of a marked torus is obvious. If $S$ is finite 
(resp.\ analytically finite, nonanalytically finite), then $\bm S$ is said to 
be {\em finite\/} (resp.\ {\em analytically finite}, {\em nonanalytically 
finite}). Since $\eta(\dot{\Sigma}_{g}) \subset S^{\circ}$, regarding $\eta$ as 
a topological embedding of $\dot{\Sigma}_{g}$ into $S^{\circ}$, we obtain a 
marked Riemann surface ${\bm S}^{\circ}=[S^{\circ},\eta]$ without border, which 
will be called the {\em interior\/} of $\bm S$. 

\begin{rem}
Let $(S,\eta),(S',\eta') \in {\bm S}$. Suppose that 
$\kappa_{0},\kappa_{1} \in \CHomeo(S,S')$ satisfy 
$\kappa_{k} \circ \eta \simeq \eta'$ for $k=0,1$. Choose a $1$-handle mark 
$\chi$ of $\dot{\Sigma}_{g}$ to obtain a $1$-handle mark $\eta \circ \chi$ of 
$S$. Since 
$\kappa_{0} \circ (\eta \circ \chi) \simeq \kappa_{1} \circ (\eta \circ \chi)$, 
it follows from Corollary~\ref{cor:conformal_homeomorphism:uniqueness} that 
$\kappa_{0} \simeq \kappa_{1}$ if $\bm S$ is a marked torus and that 
$\kappa_{0}=\kappa_{1}$ otherwise. 
\end{rem}

\begin{exmp}
\label{exmp:marked_torus}
We have obtained a torus $T_{\tau}$ for each $\tau \in \mathbb{H}$ in 
Example~\ref{exmp:conformal_automorphism:torus}. The homeomorphism 
$z \mapsto \{(\tau+i)z-(\tau-i)\bar{z}\}/2i$ of $\mathbb{C}$ onto itself 
induces $\eta_{\tau} \in \Homeo^{+}(\Sigma_{1},T_{\tau})$ and a marked torus 
${\bm T}_{\tau}:=[T_{\tau},\dot{\eta}_{\tau}]$, where 
$\dot{\eta}_{\tau}=\eta_{\tau}|_{\dot{\Sigma}_{1}}$. 
\end{exmp}

We show two propositions to verify that the marked closed Riemann surfaces 
defined above are essentially the same as the known ones. The second 
proposition also plays a fundamental role in our investigations on closed 
continuations. 

\begin{prop}
\label{prop:homotopic_extension}
Let $D$ be a subdomain of\/ $\mathbb{D}$ such that\/ $\mathbb{D} \setminus D$ 
is compact, and set $\tilde{D}=D \cup \partial\mathbb{D}$. For each 
$f \in \TEmb(\tilde{D},\bar{\mathbb{D}})$ with 
$f(\partial\mathbb{D})=\partial\mathbb{D}$ there are 
$h \in  \Homeo(\bar{\mathbb{D}},\bar{\mathbb{D}})$ and 
$H \in \Cont(\tilde{D} \times [0,1],\bar{\mathbb{D}})$ such that $H(z,0)=h(z)$, 
$H(z,1)=f(z)$ and $H(\zeta,t)=f(\zeta)$ for $z \in \tilde{D}$, 
$\zeta \in \partial\mathbb{D}$ and $t \in [0,1]$. 
\end{prop}

\begin{proof}
For $r \in [0,1]$ and $\zeta \in \partial\mathbb{D}$ set 
$h(r\zeta)=rf(\zeta)$. Then $h$ defines a homeomorphism of $\bar{\mathbb{D}}$ 
onto itself with $h=f$ on $\partial\mathbb{D}$. The correspondence 
$$
H:\tilde{D} \times [0,1] \ni (z,t) \mapsto (1-t)h(z)+tf(z) \in \bar{\mathbb{D}}
$$
possesses the desired properties. 
\end{proof}

\begin{prop}
\label{prop:homotopically_consistent}
Let $S_{1}$ and $S_{2}$ be oriented surfaces of genus $g$ possibly with border, 
and let $f_{0}$ and $f_{1}$ be topological embeddings of $S_{1}$ into $S_{2}$. 
If $S_{2}$ is closed and $f_{0} \circ \eta_{1} \simeq f_{1} \circ \eta_{1}$ for 
a $g$-handle mark $\eta_{1}$ of $S_{1}$, then $f_{0} \simeq f_{1}$. 
\end{prop}

\begin{proof}
Take $h \in \TEmb^{+}(\bar{\mathbb{D}},\Sigma_{g})$ for which 
$\Sigma_{g} \setminus \dot{\Sigma}_{g} \subset h(\mathbb{D})$. The simple loop 
$\gamma_{1}:=\eta_{1} \circ h(\partial\mathbb{D})$ divides $S_{1}$ into two 
surfaces, one of which, say $S'_{1}$, is homeomorphic to $\dot{\Sigma}_{g}$ 
while the other, say $D_{1}$, is planar, that is, of genus zero. Since 
$f_{0} \circ \eta_{1} \simeq f_{1} \circ \eta_{1}$, there is 
$F' \in \Cont(S'_{1} \times I,S_{2})$, where $I=[0,1]$, such that 
$F'(\,\cdot\,,t)=f_{t}|_{S'_{1}}$ for $t=0,1$. 

Introducing a  conformal structure into each $S_{j}$, let 
$\Pi_{j}:\mathbb{U}_{S_{j}} \to S_{j}$ be a holomorphic universal covering map 
with covering transformation group $\Gamma_{j}:=\Aut(\Pi_{j})$, where 
$\mathbb{U}_{S_{2}}=\mathbb{H}$ or $\mathbb{C}$. Fix a point 
$p_{0} \in \Sigma_{g} \setminus h(\bar{\mathbb{D}})$, and choose 
$z_{0} \in \Pi_{1}^{-1}(\eta_{1}(p_{0}))$. Let $\tilde{S}'_{1}$ denote the 
component of $\Pi_{1}^{-1}(S'_{1})$ containing $z_{0}$, and let $\Gamma'_{1}$ 
be the subgroup of $\Gamma_{1}$ consisting of covering transformations leaving 
$\tilde{S}'_{1}$ invariant. The restriction 
$\Pi'_{1}:=\Pi_{1}|_{\tilde{S}'_{1}}$ of $\Pi_{1}$ to $\tilde{S}'_{1}$ is a 
holomorphic covering map of $S'_{1}$, and $\Gamma'_{1}$ acts on 
$\tilde{S}'_{1}$ as the covering transformation group $\Aut(\Pi'_{1})$. The 
embedding $\eta_{1}$ induces an injective group homomorphism $(\eta_{1})_{*}$ 
of the fundamental group $\pi_{1}(\dot{\Sigma}_{g},p_{0})$ of 
$\dot{\Sigma}_{g}$ with base point $p_{0}$ into 
$\pi_{1}(S_{1},\eta_{1}(p_{0}))$. The latter fundamental group is canonically 
isomorphic to $\Gamma_{1}$ with $\Gamma'_{1}$ corresponding to 
$(\eta_{1})_{*}(\pi_{1}(\dot{\Sigma}_{g},p_{0}))$. Let $\Delta_{1}$ be the 
subset of $\Gamma_{1}$ assigned to the elements of 
$\pi_{1}(S_{1},\eta_{1}(p_{0}))$ represented by loops freely homotopic to loops 
in $D_{1}$. 

Any loop in $\tilde{S}'_{1}$ with initial point $z_{0}$ is trivial in 
$\mathbb{U}_{S_{1}}$ and hence is projected to a loop trivial {\em in\/} 
$S_{1}$. Thus the induced group homomorphism 
$(F' \circ (\Pi'_{1} \times \id_{I}))_{*}$ maps the fundamental group 
$\pi_{1}(\tilde{S}'_{1} \times I,(z_{0},0))$ to the trivial subgroup of 
$\pi_{1}(S_{2},f_{0} \circ \eta_{1}(p_{0}))$. Therefore, 
there is $\tilde{F}' \in \Cont(\tilde{S}'_{1} \times I,\mathbb{U}_{S_{2}})$ 
such that $\Pi_{2} \circ \tilde{F}'=F' \circ (\Pi'_{1} \times \id_{I})$. Set 
$\tilde{f}'_{t}=\tilde{F}'(\,\cdot\,,t)$ for $t \in I$. Since $\Gamma_{2}$ is 
discrete, there is a group homomorphism $\rho':\Gamma'_{1} \to \Gamma_{2}$ such 
that $\tilde{f}'_{t} \circ \gamma=\rho'(\gamma) \circ \tilde{f}'_{t}$ for 
$t \in I$. Note that $\rho'$ does not depend on $t$. 

For $k=0,1$ there is 
$\tilde{f}_{k} \in \Cont(\mathbb{U}_{S_{1}},\mathbb{U}_{S_{2}})$ together with 
a group homomorphism $\rho_{k}:\Gamma_{1} \to \Gamma_{2}$ such that 
$f_{k} \circ \Pi_{1}=\Pi_{2} \circ \tilde{f}_{k}$ and 
$\tilde{f}_{k} \circ \gamma=\rho_{k}(\gamma) \circ \tilde{f}_{k}$ for 
$\gamma \in \Gamma_{1}$. Since $f_{k}$ is identical with $F'(\,\cdot\,,k)$ on 
$S'_{1}$, we can choose $\tilde{f}_{k}$ so that 
$\tilde{f}_{k}(z_{0})=\tilde{f}'_{k}(z_{0})$. Then $\rho_{k}$ coincides with 
$\rho'$ on $\Gamma'_{1}$. Observe that $\Delta_{1} \subset \Ker\rho_{k}$ as 
$S_{2}$ is closed. Since $\Gamma_{1}$ is generated by $\Gamma'_{1}$ and 
$\Delta_{1}$, we infer that $\rho_{0}=\rho_{1}$ and hence $f_{0} \simeq f_{1}$ 
(see, for example, Ahlfors \cite[Lemma on p.119]{Ahlfors1987} or Bourque 
\cite[Lemma~2.9]{Bourque2018}). This completes the proof. 
\end{proof}

In Teichm\"{u}ller theory closed Riemann surfaces of genus $g$ have been marked 
with sense-preserving homeomorphisms of $\Sigma_{g}$, not $\dot{\Sigma}_{g}$. 
Our definition leads us to essentially the same space of marked closed Riemann 
surfaces. In fact, if $\eta_{0}$ is a $g$-handle mark of a closed Riemann 
surface $S$ of genus $g$, then it follows from 
Proposition~\ref{prop:homotopic_extension} that there is 
$\eta \in \Homeo^{+}(\Sigma_{g},S)$ such that 
$(S,\dot{\eta})$ is equivalent to $(S,\eta_{0})$, where 
$\dot{\eta}=\eta|_{\dot{\Sigma}_{g}}$. Moreover, let $\dot{\eta}'$ be a 
$g$-handle mark of a closed Riemann surface $S'$ of genus $g$, and assume that 
it is extended to a homeomorphism $\eta'$ of $\Sigma_{g}$ onto $S'$. If 
$(S',\dot{\eta}')$ is equivalent to $(S,\dot{\eta})$, then 
$\kappa \circ \dot{\eta} \simeq \dot{\eta}'$ for some 
$\kappa \in \CHomeo(S,S')$. The inclusion mapping 
$\iota:\dot{\Sigma}_{g} \to \Sigma_{g}$ is a $g$-handle mark of $\Sigma_{g}$. 
Since $(\kappa \circ \eta) \circ \iota \simeq \eta' \circ \iota$, 
Proposition~\ref{prop:homotopically_consistent} implies that 
$\kappa \circ \eta \simeq \eta'$. 

\begin{defn}[continuous mapping]
\label{defn:continuous_mapping}
Let ${\bm S}_{j}$, $j=1,2$, be marked Riemann surfaces of genus $g$, and 
consider all triples $(f,\eta_{1},\eta_{2})$, where 
$(S_{j},\eta_{j}) \in {\bm S}_{j}$ and $f \in \Cont(S_{1},S_{2})$. Two such 
triples $(f,\eta_{1},\eta_{2})$ and $(f',\eta'_{1},\eta'_{2})$, where 
$(S'_{j},\eta'_{j}) \in {\bm S}_{j}$ and $f' \in \Cont(S'_{1},S'_{2})$, are 
said to be equivalent to each other if $f' \circ \kappa_{1}=\kappa_{2} \circ f$ 
for some $\kappa_{j} \in \CHomeo(S_{j},S'_{j})$ with 
$\kappa_{j} \circ \eta_{j} \simeq \eta'_{j}$. We call each equivalence class 
${\bm f}=[f,\eta_{1},\eta_{2}]$ a {\em continuous mapping\/} of ${\bm S}_{1}$ 
into ${\bm S}_{2}$ and use the notation ${\bm f}:{\bm S}_{1} \to {\bm S}_{2}$. 
\end{defn}

If $f$ has some conformally invariant properties in addition, then $\bm f$ is 
said to possess the same properties. For example, if $f$ is a quasiconformal 
embedding of $S_{1}$ into $S_{2}$, then $\bm f$ is called a quasiconformal 
embedding of ${\bm S}_{1}$ into ${\bm S}_{2}$. 

To compose continuous mappings of a marked Riemann surface into another we need 
to place restrictions on the mappings. In practice they target continuous 
mappings of marked tori. 

\begin{defn}[conformally compatible continuous mapping]
\label{defn:conformal_compatible_continuous_mapping}
Let $S_{1}$ and $S_{2}$ be Riemann surfaces of positive genus. Then 
$f \in \Cont(S_{1},S_{2})$ is said to be {\em conformally compatible\/} if for 
any $\kappa_{1} \in \Aut_{0}(S_{1})$ there is $\kappa_{2} \in \Aut_{0}(S_{2})$ 
such that $f \circ \kappa_{1}=\kappa_{2} \circ f$. 
\end{defn}

If $S_{1}$ is not a torus, then every continuous mapping of $S_{1}$ into 
$S_{2}$ is conformally compatible as $\Aut_{0}(S_{1})=\{\id_{S_{1}}\}$. If 
$S_{1}$ is a torus while $S_{2}$ is not, then every conformally compatible 
continuous mapping of $S_{1}$ into $S_{2}$ is constant because 
$\Aut_{0}(S_{1})=\Aut_{\mathrm{tr}}(S_{1})$ acts transitively on $S_{1}$ while 
$\Aut_{0}(S_{2})$ is trivial. To determine conformally compatible continuous 
mappings of $S_{1}$ into $S_{2}$ in the case where both $S_{1}$ and $S_{2}$ are 
tori, take holomorphic universal covering mappings 
$\Pi_{j}:\mathbb{C} \to S_{j}$, $j=1,2$. Suppose that 
$f \in \Cont(S_{1},S_{2})$ is conformally compatible. Let 
$\tilde{f} \in \Cont(\mathbb{C},\mathbb{C})$ be a lift of $f$. Thus we have 
$\Pi_{2} \circ \tilde{f}=f \circ \Pi_{1}$. For any 
$\tau_{1} \in \Aut_{\mathrm{tr}}(\mathbb{C})$ there is 
$\kappa_{1} \in \Aut_{\mathrm{tr}}(S_{1})=\Aut_{0}(S_{1})$ such that 
$\kappa_{1} \circ \Pi_{1}=\Pi_{1} \circ \tau_{1}$. 
Since $f$ is conformally compatible, we find 
$\kappa_{2} \in \Aut_{0}(S_{2})=\Aut_{\mathrm{tr}}(S_{2})$ for which 
$f \circ \kappa_{1}=\kappa_{2} \circ f$, which leads us to 
$$
\Pi_{2} \circ \tilde{f} \circ \tau_{1}=f \circ \Pi_{1} \circ \tau_{1}=
f \circ \kappa_{1} \circ \Pi_{1}=\kappa_{2} \circ f \circ \Pi_{1}=
\kappa_{2} \circ \Pi_{2} \circ \tilde{f}=\Pi_{2} \circ \tau_{2} \circ \tilde{f}
$$
for some $\tau_{2} \in \Aut_{\mathrm{tr}}(\mathbb{C})$. Therefore, we have 
$\tilde{f}(z+a)=\tilde{f}(z)+l(a)$ for $z,a \in \mathbb{C}$, where 
$l(a)=\tilde{f}(a)-\tilde{f}(0)$. Since $l(a_{1}+a_{2})=l(a_{1})+l(a_{2})$, it 
follows that $l(ca)=cl(a)$ for $c \in \mathbb{Q}$ and hence for 
$c \in \mathbb{R}$ by continuity. Consequently, $l$ is an $\mathbb{R}$-linear 
mapping of $\mathbb{C}$ into itself, and $\tilde{f}$ is an $\mathbb{R}$-affine 
transformation. For any $\gamma$ in the covering transformation group 
$\Aut(\Pi_{1})$ of $\Pi_{1}$ there is $\rho(\gamma) \in \Aut(\Pi_{2})$ such 
that $\tilde{f} \circ \gamma=\rho(\gamma) \circ \tilde{f}$. Conversely, if 
$\tilde{f}:\mathbb{C} \to \mathbb{C}$ is an $\mathbb{R}$-affine mapping 
possessing the last property, then it induces a conformally compatible element 
$f \in \Cont(S_{1},S_{2})$ such that $\Pi_{2} \circ \tilde{f}=f \circ \Pi_{1}$. 
In particular, the inverse of a conformally compatible homeomorphism of a torus 
onto another is conformally compatible. Note that whether $S_{1}$ is a torus or 
not, conformal embeddings are conformally compatible. 

Let ${\bm S}_{j}$, $j=1,2$, be marked Riemann surfaces of genus $g$, and let 
$\bm f$ be a continuous mapping of ${\bm S}_{1}$ into ${\bm S}_{2}$. Take 
representatives $(f,\eta_{1},\eta_{2}),(f',\eta'_{1},\eta'_{2}) \in {\bm f}$, 
where $(S_{j},\eta_{j}),(S'_{j},\eta'_{j}) \in {\bm S}_{j}$. If $f$ is 
conformally compatible, then so is $f'$. Hence we can speak of a 
{\em conformally compatible\/} continuous mapping of a marked Riemann surface 
into another. 

\begin{defn}[composition]
\label{defn:composition}
Let ${\bm f}_{1}:{\bm S}_{1} \to {\bm S}_{2}$ and 
${\bm f}_{2}:{\bm S}_{2} \to {\bm S}_{3}$ be conformally compatible continuous 
mappings. Take $(f_{1},\eta_{1},\eta_{2}) \in {\bm f}_{1}$ and 
$(f_{2},\eta'_{2},\eta_{3}) \in {\bm f}_{2}$, where 
$(S_{j},\eta_{j}) \in {\bm S}_{j}$, $j=1,2,3$, and 
$(S'_{2},\eta'_{2}) \in {\bm S}_{2}$. Define a continuous mapping 
${\bm f}_{2} \circ {\bm f}_{1}$ of ${\bm S}_{1}$ into ${\bm S}_{3}$ to be the 
equivalence class $[f_{2} \circ \kappa_{22'} \circ f_{1},\eta_{1},\eta_{3}]$, 
where  $\kappa_{22'} \in \CHomeo(S_{2},S'_{2})$ with 
$\kappa_{22'} \circ \eta_{2} \simeq \eta'_{2}$. 
\end{defn}

It is easy to verify that ${\bm f}_{2} \circ {\bm f}_{1}$ is well-defined, that 
is, that $(f_{2} \circ \kappa_{22'} \circ f_{1},\eta_{1},\eta_{3})$ represents 
a unique continuous mapping of ${\bm S}_{1}$ into ${\bm S}_{3}$. Actually, 
${\bm f}_{2} \circ {\bm f}_{1}$ is well-defined if ${\bm f}_{2}$ is conformally 
compatible. We do not need to require ${\bm f}_{1}$ to be conformally 
compatible. 

Compositions of conformally compatible continuous mappings of marked Riemann 
surfaces are again conformally compatible. Also, the binary operation $\circ$ 
is associative. 

In general, for a marked Riemann surface $\bm S$ of genus $g$ let 
${\bm 1}_{\bm S}$ denote the conformal automorphism of $\bm S$ represented by 
$(\id_{S},\eta,\eta)$, where $(S,\eta) \in {\bm S}$. Then 
$\mathbf{1}_{{\bm S}_{2}} \circ {\bm f}={\bm f} \circ \mathbf{1}_{{\bm S}_{1}}=
{\bm f}$ for any conformally compatible continuous mappings $\bm f$ of 
${\bm S}_{1}$ into ${\bm S}_{2}$. 

We are now ready to define the category $\mathfrak{F}_{g}$. Its class 
$\ob(\mathfrak{F}_{g})$ of objects consists of all marked Riemann surfaces of 
genus $g$. The hom-set $\Cont({\bm S},{\bm R})$ of morphisms from $\bm S$ to 
$\bm R$ is composed of all conformally compatible continuous mappings of 
$\bm S$ into $\bm R$. The {\em Teichm\"{u}ller space\/} $\mathfrak{T}_{g}$ of 
genus $g$ is, by definition, the full subcategory of $\mathfrak{F}_{g}$ whose 
class of objects is exactly the set of marked closed Riemann surfaces of genus 
$g$. For the sake of simplicity we abbreviate $\ob(\mathfrak{F}_{g})$ and 
$\ob(\mathfrak{T}_{g})$ to $\mathfrak{F}_{g}$ and $\mathfrak{T}_{g}$, 
respectively. 

\begin{rem}
The set $\Cont({\bm S},{\bm R})$ is {\em not\/} the set of continuous mappings 
of $\bm S$ into $\bm R$. Every element of $\Cont({\bm S},{\bm R})$ should be 
conformally compatible. 
\end{rem}

Denote by $\TEmb({\bm S},{\bm R})$ (resp.\ $\QCEmb({\bm S},{\bm R})$, 
$\CEmb({\bm S},{\bm R})$) the set of conformally compatible topological (resp.\ 
quasiconformal, conformal) embeddings of $\bm S$ into $\bm R$. The set of 
sense-preserving elements in $\TEmb({\bm S},{\bm R})$ is denoted by 
$\TEmb^{+}({\bm S},{\bm R})$. Also, let $\Homeo({\bm S},{\bm R})$ denote the 
set of surjective elements in $\TEmb({\bm S},{\bm R})$, and set 
$$
\Homeo^{+}({\bm S},{\bm R})=
\Homeo({\bm S},{\bm R}) \cap \TEmb^{+}({\bm S},{\bm R}).
$$
Replacing ``Emb" with ``Homeo" in the notations for the other classes above, we 
obtain notations for the subset of homeomorphisms of $\bm S$ onto $\bm R$ in 
the classes under consideration. Thus $\QCHomeo({\bm S},{\bm R})$ means the 
set of conformally compatible quasiconformal homeomorphisms of $\bm S$ onto 
$\bm R$. 

\begin{defn}[homotopically consistent embedding]
\label{defn:homotopically_consistent_continuous_mapping}
A sense-preserving topological embedding ${\bm f}:{\bm S}_{1} \to {\bm S}_{2}$ 
is called {\em homotopically consistent\/} if 
$f \circ \eta_{1} \simeq \eta_{2}$, where $(S_{j},\eta_{j}) \in {\bm S}_{j}$, 
$j=1,2$, and $(f,\eta_{1},\eta_{2}) \in {\bm f}$. 
\end{defn}

This definition does not depend on a particular choice of representatives. If 
${\bm f}_{1} \in \TEmb^{+}({\bm S}_{1},{\bm S}_{2})$ and 
${\bm f}_{2} \in \TEmb^{+}({\bm S}_{2},{\bm S}_{3})$ are homotopically 
consistent, then so is ${\bm f}_{2} \circ {\bm f}_{1}$. If 
$\mathrm{X}({\bm S},{\bm R})$ is a subset of $\TEmb^{+}({\bm S},{\bm R})$, then 
we denote by $\mathrm{X}_{\mathrm{hc}}({\bm S},{\bm R})$ the subset of 
homotopically consistent elements in $\mathrm{X}({\bm S},{\bm R})$. For 
example, $\QCEmb_{\mathrm{hc}}({\bm S},{\bm R})$ means the set of homotopically 
consistent and conformally compatible quasiconformal embeddings of $\bm S$ into 
$\bm R$. 

Let ${\bm f}_{0}=[f_{0},\eta_{10},\eta_{20}]$ and 
${\bm f}_{1}=[f_{1},\eta_{11},\eta_{21}]$ be continuous mappings of 
${\bm S}_{1}$ into ${\bm S}_{2}$. There are 
$\kappa_{j} \in \CHomeo(S_{j0},S_{j1})$, $j=1,2$, with 
$\kappa_{j} \circ \eta_{j0} \simeq \eta_{j1}$, where 
${\bm S}_{j}=[S_{j0},\eta_{j0}]=[S_{j1},\eta_{j1}]$. If 
$f_{1} \circ \kappa_{1} \simeq \kappa_{2} \circ f_{0}$, then we say that 
${\bm f}_{1}$ is {\em homotopic\/} to ${\bm f}_{0}$ and write 
${\bm f}_{1} \simeq {\bm f}_{0}$. Again, this definition does not depend 
on a particular choice of representatives. 

A continuous mapping of ${\bm S}_{1}$ into ${\bm S}_{2}$ is said to be 
{\em homotopically consistent\/} if it is homotopic to a homotopically 
consistent topological embedding of ${\bm S}_{1}$ into ${\bm S}_{2}$. If 
${\bm f}_{j} \in \Cont({\bm S}_{j},{\bm S}_{j+1})$, $j=1,2$, are homotopically 
consistent, then so is ${\bm f}_{2} \circ {\bm f}_{1}$. We use the notation 
$\Cont_{\mathrm{hc}}({\bm S},{\bm R})$ to express the set of homotopically 
consistent elements in $\Cont({\bm S},{\bm R})$. 

\begin{defn}[continuation]
\label{defn:continuation:marked_surface}
By a {\em continuation\/} of ${\bm S} \in \mathfrak{F}_{g}$ we mean a pair 
$({\bm R},{\bm \iota})$ where ${\bm R} \in \mathfrak{F}_{g}$ and 
${\bm \iota} \in \CEmb_{\mathrm{hc}}({\bm S},{\bm R})$. 
\end{defn}

A continuation $({\bm R},{\bm \iota})$ is called {\em closed\/} (resp.\ 
{\em open, compact}) if $\bm R$ is closed (resp.\ open, compact). We say that 
$({\bm R},{\bm \iota})$ is {\em dense\/} if so is $(R,\iota)$, where 
$(S,\eta) \in {\bm S}$, $(R,\theta) \in {\bm R}$ and 
$(\iota,\eta,\theta) \in {\bm \iota}$. If $({\bm R},{\bm \iota})$ and 
$({\bm R}',{\bm \iota}')$ are continuations of $\bm S$ and $\bm R$, 
respectively, then $({\bm R}',{\bm \iota}' \circ {\bm \iota})$ is a 
continuation of $\bm S$. 

Let ${\bm S}=[S,\eta] \in \mathfrak{F}_{g}$. If $(R,\iota)$ is a 
genus-preserving continuation of $S$, then, setting 
${\bm R}=[R,\iota \circ \eta]$ and ${\bm \iota}=[\iota,\eta,\iota \circ \eta]$, 
we obtain a continuation $({\bm R},{\bm \iota})$ of $\bm S$. In particular, if 
$S$ is a subsurface of $R$ and $\iota:S \to R$ is the inclusion mapping, then 
we call $({\bm R},{\bm \iota})$ the {\em inclusion continuation\/} of $\bm S$. 
If this is the case, then for 
${\bm f} \in \Cont({\bm R},{\bm R}')$ we call ${\bm f} \circ {\bm \iota}$ the 
{\em restriction of $\bm f$ to $\bm S$} and denote it by ${\bm f}|_{\bm S}$. 

Let ${\bm S}=[S,\eta]$ be a marked finite open Riemann surface of genus $g$, 
and let $(\breve{S},\breve{\iota})$ be a natural compact continuation of $S$. 
Setting $\breve{\bm S}=[\breve{S},\breve{\iota} \circ \eta]$ and 
$\breve{\bm \iota}=[\breve{\iota},\eta,\breve{\iota} \circ \eta]$, we obtain a 
continuation $(\breve{\bm S},\breve{\bm \iota})$ of $\bm S$, which is referred 
to as the {\em natural compact continuation\/} of $\bm S$. 

Let ${\bm S}_{1},{\bm S}_{2} \in \mathfrak{F}_{g}$. For quasiconformal 
embeddings $\bm f$ of ${\bm S}_{1}$ into ${\bm S}_{2}$ we can speak of their 
maximal dilatations $K({\bm f})$. If ${\bm S}_{1}$ belongs to 
$\mathfrak{T}_{g}$, then so does ${\bm S}_{2}$, and their {\em Teichm\"{u}ller 
distance\/} $d_{T}({\bm S}_{1},{\bm S}_{2})$ is defined to be 
$\inf(\log K({\bm f}))/2$, where the infimum is taken over all homotopically 
consistent quasiconformal homeomorphisms $\bm f$ of ${\bm S}_{1}$ onto 
${\bm S}_{2}$. The complete metric space $\mathfrak{T}_{g}$ is in fact known to 
be a complex manifold homeomorphic to $\mathbb{R}^{2d_{g}}$ and biholomorphic 
to a bounded domain in $\mathbb{C}^{d_{g}}$, where $d_{g}=\max\{g,3g-3\}$. 

Now, for ${\bm R}_{0}=[R_{0},\theta_{0}] \in \mathfrak{F}_{g}$ we are concerned 
with the set $\mathfrak{M}({\bm R}_{0})$ of ${\bm R} \in \mathfrak{T}_{g}$ such 
that $\CEmb_{\mathrm{hc}}({\bm R}_{0},{\bm R}) \neq \varnothing$. It follows 
from Bochner \cite[Satz~V]{Bochner1928} that $\mathfrak{M}({\bm R}_{0})$ is 
nonempty. 

\begin{exmp}
\label{exmp:g=1:compact}
Under the notations in Example~\ref{exmp:marked_torus} the correspondence 
$\tau \mapsto {\bm T}_{\tau}$ is a biholomorphism of $\mathbb{H}$ onto 
$\mathfrak{T}_{1}$. If we identify $\mathfrak{T}_{1}$ with $\mathbb{H}$ through 
the biholomorphism, then the Teichm\"{u}ller distance on $\mathfrak{T}_{1}$ 
coincides with the distance on $\mathbb{H}$ induced by the hyperbolic metric 
$|dz|/(2\im z)$. Now, let ${\bm R}_{0}$ be a marked open Riemann surface of 
genus one. Then $\mathfrak{M}({\bm R}_{0})$ is a closed disk or a singleton in 
$\mathfrak{T}_{1}$ by \cite[Theorem~5]{Shiba1987}. It degenerates to a 
singleton if and only if $R_{0} \in O_{AD}$ (see \cite[Theorem~6]{Shiba1987}). 
Note that a finite open Riemann surface belongs to $O_{AD}$ if and only if it 
is analytically finite. 
\end{exmp}

\section{Self-weldings with positive quadratic differentials}
\setcounter{equation}{0}
\label{sec:self-welding}

We introduce an operation called a self-welding, which brings us 
genus-preserving closed continuations of finite open Riemann surfaces. Though 
the procedure is plain, the harvest is rich. The current and next sections are 
devoted to developing an elementary theory of self-weldings. 

Let $\varphi=\varphi(z)\,dz^{2}$ be a quadratic differential on a Riemann 
surface $S$; it is an assignment of a function $\varphi(z)$ of $z$ to each 
local coordinate $z$ so that $\varphi(z)\,(dz)^{2}$ is invariant under 
coordinate changes. Then $|\varphi|:=|\varphi(z)|\,dx \wedge dy$, $z=x+iy$, 
defines a 2-form on $S$. If $\varphi$ is measurable, that is, $\varphi(z)$ is 
measurable for each $z$, then set 
$$
\|\varphi\|_{E}=\iint_{E} |\varphi|
$$
for measurable subsets $E$ of $S$. If $\|\varphi\|_{E}$ is finite, then 
$\varphi$ is said to be {\em integrable\/} over $E$. Also, 
$\sqrt{|\varphi|\,}:=\sqrt{|\varphi(z)|\,}\,|dz|$ and 
$|\im\sqrt{\varphi\,}|:=|\im(\sqrt{\varphi(z)\,}\,dz)|=
|\im(\sqrt{\varphi(z)\,})\,dx+\re(\sqrt{\varphi(z)\,})\,dy|$ are invariant 
under coordinate changes. The {\em $\varphi$-length} $L_{\varphi}(c)$ and the 
{\em $\varphi$-height} $H_{\varphi}(c)$ of a curve $c$ on $S$ are defined by 
$$
L_{\varphi}(c)=\int_{c} \sqrt{|\varphi|\,} \quad \text{and} \quad
H_{\varphi}(c)=\int_{c} |\im\sqrt{\varphi\,}|,
$$
respectively, provided that the integrals are meaningful. Otherwise, we just 
set $L_{\varphi}(c)=\infty$ or $H_{\varphi}(c)=\infty$. Note that 
$\|\varphi\|_{E}$ is the area of $E$ with respect to the area element 
corresponding to the length element $\sqrt{|\varphi|\,}$. 

Let $(R,\iota)$ be a continuation of a Riemann surface $S$, and let $\psi$ be a 
quadratic differential on $R$. For $p \in S$ take local coordinates $z$ and 
$w$ around $p$ and $\iota(p)$, respectively, and consider $\iota$ as a 
holomorphic function $w=\iota(z)$. Set 
$\varphi(z)=\psi(\iota(z))\iota'(z)^{2}$, where $\psi(w)$ is the function 
assigned to the local coordinate $w$ by the quadratic differential $\psi$. Then 
assigning $\varphi(z)$ to $z$ defines a quadratic differential on $S$ called 
the {\em pull-back\/} of $\psi$ by $\iota$ to be denoted by $\iota^{*}\psi$. 

Let $\iota \in \CEmb(S,R)$. Considering $\iota$ as a conformal homeomorphism 
$S$ onto $\iota(S)$, we can speak of the pull-back of a quadratic differential 
$\varphi$ on $S$  by $\iota^{-1}$. It is a quadratic differential on 
$\iota(S)$, which will be denoted by $\iota_{*}\varphi$. 

A quadratic differential $\varphi=\varphi(z)\,dz^{2}$ is called 
{\em meromorphic\/} (resp.\ {\em holomorphic}) if $\varphi(z)$ is meromorphic 
(resp.\ holomorphic) for each local coordinate $z$. If $S$ is a bordered 
Riemann surface and $\varphi$ is a meromorphic (resp.\ holomorphic) quadratic 
differential on $S$, then for any open continuation $(R,\iota)$ of $S$ there is 
a meromorphic (resp.\ holomorphic) quadratic differential $\psi$ on a 
neighborhood of $\iota(S)$ such that the pull-back $\iota^{*}\psi$ coincides 
with $\varphi$ on $S$. 

Let $M(S)$ be the complex vector space of meromorphic quadratic differentials 
on $S$. Denote by $A(S)$ the subspace composed of holomorphic quadratic 
differentials on $S$. For nonzero $\varphi \in M(S)$ the algebraic degree of 
$\varphi$ at $p \in S$ is denoted by $\ord_{p}\varphi$. If 
$\partial S \neq \varnothing$ and $p \in \partial S$, then, taking $(R,\iota)$ 
and $\psi$ as in the preceding paragraph, we understand that 
$\ord_{p}\varphi=\ord_{\iota(p)}\psi$. By a {\em critical\/} point of 
$\varphi$ we mean a point $p$ for which $\ord_{p}\varphi \neq 0$. It is a 
{\em zero\/} of $\varphi$ of order $\ord_{p}\varphi$ if $\ord_{p}\varphi>0$ 
while it is a {\em pole\/} of $\varphi$ of order $-\ord_{p}\varphi$ if 
$\ord_{p}\varphi<0$. Zeros and poles on the border $\partial S$ are referred to 
as {\em border zeros\/} and {\em border poles}, respectively. If 
$n:=\ord_{p}\varphi \geqq -1$, then there is a local coordinate $\zeta$ around 
$p$ with $\zeta(p)=0$ such that $\varphi=(n/2+1)^{2}\zeta^{n}\,d\zeta^{2}$ (see 
Strebel \cite[Theorem~6.1]{Strebel1984}). Such a local coordinate $\zeta$ will 
be called a {\em natural parameter\/} of $\varphi$ around $p$. If 
$\|\varphi\|_{U}<+\infty$ for an open set $U$, then $\ord_{p}\varphi \geqq -1$ 
for all $p \in U$, that is, all poles of $\varphi$ in $U$ are simple. 

A meromorphic quadratic differential $\varphi=\varphi(z)\,dz^{2}$ is said to be 
{\em positive\/} along a smooth curve $a:I \to S$, where $I$ is an interval in 
$\mathbb{R}$ which may be closed, half-closed, open or infinite, if 
$\varphi((z \circ a)(t))(z \circ a)'(t)^{2}$ is finite and positive whenever 
$a(t)$ is in the domain of a local coordinate $z$. A {\em horizontal 
trajectory\/} of $\varphi$ is, by definition, a maximal smooth curve along 
which $\varphi$ is positive. A horizontal trajectory can be a loop called a 
{\em closed\/} horizontal trajectory. Note that any horizontal trajectory of 
$\varphi$  contains no critical points of $\varphi$. By a {\em horizontal 
arc\/} we mean a simple subarc, not a loop, of a horizontal trajectory of 
$\varphi$. For example, if $\zeta$ is a natural parameter of $\varphi$ around 
$p \in S^{\circ}$ with $n:=\ord_{p}\varphi \geqq -1$, then for sufficiently 
small $\delta>0$ the arcs 
$(0,\delta) \ni t \mapsto \zeta^{-1}(te^{2k\pi i/(n+2)})$, $k=0,1,\ldots,n+1$, 
are horizontal arcs of $\varphi$ emanating from $p$ though $p$ does not lie on 
the arcs. If $p$ is a noncritical point of $\varphi$, that is, if $n=0$, then 
$(-\delta,\delta) \ni t \mapsto \zeta^{-1}(t)$ is a horizontal arc of $\varphi$ 
passing through $p$. 

Suppose that $S$ is a bordered Riemann surface. A nonzero element $\varphi$ of 
$M(S)$ is called {\em positive\/} if $\partial S$ consists of horizontal 
trajectories and, possibly, zeros of $\varphi$; thus $\varphi$ is holomorphic 
on $\partial S$. The set of positive quadratic differentials in $M(S)$ is 
denoted by $M_{+}(S)$. Define $A_{+}(S)=A(S) \cap M_{+}(S)$. If $\hat{S}$ 
denotes the double of $S$ (for the definition, see \cite[II.3E]{AS1960}), then 
the reflection principle enables us to extend every element in $M_{+}(S)$ 
(resp.\ $A_{+}(S)$) to a quadratic differential in $M(\hat{S})$ (resp.\ 
$A(\hat{S})$). Note that the orders of border zeros of elements in $M_{+}(S)$ 
are even. For the sake of convenience we set $M_{+}(S)=M(S) \setminus \{0\}$ 
and $A_{+}(S)=A(S) \setminus \{0\}$ for {\em closed\/} Riemann surfaces $S$. 

Let $R_{0}$ be a finite open Riemann surface, and let 
$(\breve{R}_{0},\breve{\iota}_{0})$ be a natural compact continuation of 
$R_{0}$. Then $\breve{R}_{0}$ is a compact Riemann surface. Set 
$M_{+}(R_{0})=\breve{\iota}_{0}^{*}M_{+}(\breve{R}_{0})$ and 
$A_{+}(R_{0})=\breve{\iota}_{0}^{*}A_{+}(\breve{R}_{0})$. These spaces do not 
depend of a particular choice of $(\breve{R}_{0},\breve{\iota}_{0})$, and are 
determined solely by $R_{0}$. 

\begin{prop}
\label{prop:positive_quadratic_differential:closed}
Let $R_{0}$ be a finite open Riemann surface, and let $\varphi$ be a nonzero 
meromorphic quadratic differential on $R_{0}$. If there is a sequence 
$\{\varphi_{n}\}$ in $A_{+}(R_{0})$ such that 
$\|\varphi_{n}-\varphi\|_{U} \to 0$ as $n \to \infty$ for some open subset $U$ 
of $R_{0}$, then $\varphi \in A_{+}(R_{0})$ and 
$\|\varphi_{n}-\varphi\|_{R_{0}} \to 0$ as $n \to \infty$. 
\end{prop}

\begin{proof}
Let $(\breve{R}_{0},\breve{\iota}_{0})$ be a natural compact continuation of 
$R_{0}$. If $\breve{R}_{0}$ is bordered, then let $R$ denote its double. 
Otherwise, set $R=\breve{R}_{0}$. In either case $R$ is a closed Riemann 
surface. Identifying $R_{0}$ with $\breve{\iota}(R_{0})$, we consider $R_{0}$ 
as a subdomain of $R$, and regard $A_{+}(R_{0})$ as a subset of $A(R)$. Both 
$\|\cdot\|_{R}$ and $\|\cdot\|_{U}$ are norms on $A(R)$; note that 
$\|\psi\|_{U}=0$ implies $\psi=0$ by the identity theorem. Since $A(R)$ is of 
finite dimension, there is $M>0$ such that 
$\|\psi\|_{U} \leqq \|\psi\|_{R} \leqq M\|\psi\|_{U}$ for all $\psi \in A(R)$. 

Now, since $\{\varphi_{n}\}$ is a Cauchy sequence in the Banach space 
$(A(R),\|\cdot\|_{U})$, it is also a Cauchy sequence in $(A(R),\|\cdot\|_{R})$ 
and hence there is $\psi \in A(R)$ for which $\|\varphi_{n}-\psi\|_{R} \to 0$ 
as $n \to \infty$. The identity theorem implies that $\psi=\varphi$ on $R_{0}$. 
Thus $\psi$ is nonzero and belongs to $A_{+}(R_{0})$. 
\end{proof}

\begin{defn}[self-welding]
\label{defn:self-welding}
Let $S$ be a compact bordered Riemann surface. A {\em self-welding\/} of $S$ is 
a pair $\langle R,\iota\rangle$, where 

\begin{list}{{(\roman{claim})}}{\usecounter{claim}
\setlength{\topsep}{0pt}
\setlength{\itemsep}{0pt}
\setlength{\parsep}{0pt}
\setlength{\labelwidth}{\leftmargin}}
\item $R$ is a compact Riemann surface, 

\item $\iota \in \Cont(S,R)$ and $\iota|_{S^{\circ}} \in \CEmb(S^{\circ},R)$ with 
$\iota(S)=R$, and 

\item there are $\varphi \in M_{+}(S)$ and $\psi \in M_{+}(R)$ with 
$(\iota|_{S^{\circ}})^{*}\psi=\varphi$ on $S^{\circ}$ such that 
$R \setminus \iota(S^{\circ})$ consists of finitely many horizontal arcs and, 
possibly, critical points of $\psi$. 
\end{list}
\end{defn}

The quadratic differentials $\varphi$ and $\psi$ are referred to as a 
{\em welder\/} of the welding $\langle R,\iota\rangle$ and the 
{\em co-welder\/} of the welder $\varphi$, respectively. If $S$ and $R$ are of 
the same genus, then $\langle R,\iota\rangle$ is said to be 
{\em genus-preserving}. If $R$ is closed, then $\langle R,\iota\rangle$ is 
called {\em closed}. 

Two self-weldings $\langle R_{1},\iota_{1}\rangle$ and 
$\langle R_{2},\iota_{2}\rangle$ of $S$ are defined to be equivalent to each 
other if there is $\kappa \in \CHomeo(R_{1},R_{2})$ such that 
$\iota_{2}=\kappa \circ \iota_{1}$. Here, we do not force the self-weldings to 
have a common welder. 

\begin{rem}
A positive meromorphic quadratic differential on $S$ can induce two or more 
inequivalent self-weldings of $S$ (see Examples~\ref{self-welding:Joukowski} 
and~\ref{exmp:full_self-welding}). Also, a self-welding of $S$ can have 
linearly independent welders (see Examples~\ref{self-welding:Joukowski} 
and~\ref{exmp:inducing_same_self-welding}). 
\end{rem}

\begin{exmp}
\label{self-welding:Joukowski}
Consider the meromorphic functions $\iota_{j}$, $j=0,1$, on $\bar{\mathbb{D}}$ 
defined by $\iota_{j}(z)=z+(-1)^{j}/z$. Then 
$\langle\hat{\mathbb{C}},\iota_{j}\rangle$, $j=0,1$, are genus-preserving 
closed self-weldings of $\bar{\mathbb{D}}$, which are inequivalent to each 
other. The meromorphic quadratic differential 
$\varphi_{0}:=(1-1/z^{2})^{2}\,dz^{2}$ on $\bar{\mathbb{D}}$ is a common welder 
of the self-weldings. The co-welders of $\varphi_{0}$ are $\psi_{0}:=dw^{2}$ 
for $\langle\hat{\mathbb{C}},\iota_{0}\rangle$ and 
$\psi_{1}:=w^{2}\,dw^{2}/(w^{2}+4)^{2}$ for 
$\langle\hat{\mathbb{C}},\iota_{1}\rangle$. Note that $\varphi_{0}$ has double 
zeros at $\pm 1 \in \partial\mathbb{D}$ while $\psi_{0}$ has no critical points 
on $\iota_{0}(\partial\mathbb{D})=[-2,2]$. Also, 
$\varphi_{1}:=-(1+1/z^{2})^{2}\,dz^{2}$ is a welder of 
$\langle\hat{\mathbb{C}},\iota_{1}\rangle$. Its co-welder is $-\psi_{0}$. 
\end{exmp}

\begin{defn}[self-welding continuation]
\label{defn:self-welding_continuation}
Let $R_{0}$ be a nonanalytically finite open Riemann surface. A compact 
continuation $(R,\iota)$ of $R_{0}$ is called a {\em self-welding 
continuation\/} of $R_{0}$ if there are a natural compact continuation 
$(\breve{R}_{0},\breve{\iota}_{0})$ of $R_{0}$ and a self-welding 
$\langle R,\kappa\rangle$ of $\breve{R}_{0}$ such that 
$\iota=\kappa \circ \breve{\iota}_{0}$. 
\end{defn}

If $\breve{\varphi} \in M_{+}(\breve{R}_{0})$ is a welder of the self-welding 
$\langle R,\kappa\rangle$, then 
$\varphi:=\breve{\iota}^{*}\breve{\varphi} \in M_{+}(R_{0})$ is said to be a 
{\em welder\/} of the continuation $(R,\iota)$. The co-welder of 
$\breve{\varphi}$ is also referred to as the {\em co-welder\/} of $\varphi$. 

Let $S$ be a compact bordered Riemann surface. If $\langle R,\iota\rangle$ is a 
self-welding of $S$, then $\iota^{-1}(\partial R)$ is included in $\partial S$. 
Moreover, $(R,\iota|_{S^{\circ}})$ is a self-welding continuation of 
$S^{\circ}$. If the self-welding is closed, then so is the induced 
continuation. 

Let $R_{0}$ be a nonanalytically finite open Riemann surface, and let 
$(\breve{R}_{0},\breve{\iota}_{0})$ be a natural compact continuation of 
$R_{0}$. Equivalent self-weldings of $\breve{R}_{0}$ induce equivalent 
continuations of $R_{0}$. Conversely, let $(R,\iota)$ be a self-welding 
continuation of $R_{0}$, and let $(R',\iota')$ be a continuation of $R_{0}$ 
equivalent to $(R,\iota)$. Then there is a self-welding 
$\langle R,\kappa\rangle$ of $\breve{R}_{0}$ such that 
$\iota=\kappa \circ \breve{\iota}_{0}$. Also, there is 
$\kappa' \in \CHomeo(R,R')$ such that $\iota'=\kappa' \circ \iota$. Thus 
$\langle R',\kappa' \circ \kappa\rangle$ is a self-welding of $\breve{R}_{0}$ 
equivalent to $\langle R,\kappa\rangle$ and induces the continuation 
$(R',\iota')$ of $R_{0}$. 

\begin{prop}
\label{prop:continuation:self-welding}
Let $S$ be a compact bordered Riemann surface, and let $(R,\iota)$ be a dense 
compact continuation of $S^{\circ}$. Suppose that there are 
$\varphi \in M_{+}(S)$ and $\psi \in M_{+}(R)$ such that 
$\varphi=\iota^{*}\psi$. Then $\iota$ is extended to 
$\bar{\iota} \in \Cont(S,R)$ with $\bar{\iota}(S)=R$ for which 
$\langle R,\bar{\iota}\rangle$ is a self-welding of $S$ with welder $\varphi$. 
\end{prop}

\begin{proof}
Take an arbitrary point $p_{0}$ on the border $\partial S$, and let $\zeta$ be 
a natural parameter of $\varphi$ around $p_{0}$. We choose $\zeta$ so that it 
maps a neighborhood $U$ of $p_{0}$ conformally onto a half-disk 
$\{\zeta \in \bar{\mathbb{H}} \mid |\zeta|<r\}$. If $\ord_{p_{0}}\varphi=2n$, 
then $\varphi=(n+1)^{2}\,\zeta^{2n}\,d\zeta^{2}$ on $U$; recall that the order 
of $\varphi$ at $p_{0}$ is even. Set 
$U_{k}=\{p \in U \mid k\pi/(n+1)<\arg\zeta(p)<(k+1)\pi/(n+1)\}$ for 
$k=0,1,\dots,n$. Thus a single-valued branch of the integral $\Phi$ of 
$\sqrt{\varphi\,}$ on $U \cap S^{\circ}$ is represented as 
$\Phi(\zeta)=\zeta^{n+1}$ and maps each $U_{k}$ onto a half disk of radius 
$r^{n+1}$ centered at $0$. Let $\{p_{\nu}\}$ be a sequence in a sector $U_{k}$ 
converging to $p_{0}$. Since $R$ is compact, a subsequence of 
$\{\iota(p_{\nu})\}$, to be denoted again by $\{\iota(p_{\nu})\}$, converges to 
a point $q_{0}$ in $R$. As $\varphi=\iota^{*}\psi$, the function $\Psi$ on 
$\iota(U \cap S^{\circ})$ defined by $\Psi \circ \iota=\Phi$ is a single-valued 
branch of the integral of $\sqrt{\psi\,}$, and $\iota$ maps horizontal arcs of 
$\varphi$ in $U_{k}$ onto those of $\psi$ in $\iota(U_{k})$. Since 
$\Psi(\iota(p_{\nu}))=\Phi(p_{\nu}) \to 0$ as $\nu \to 0$, we know that 
$m:=\ord_{q_{0}}\psi \geqq -1$. Let $\omega$ be a natural parameter of $\psi$ 
around $q_{0}$ so that $\psi=(m/2+1)^{2}\omega^{m}\,d\omega^{2}$, and set 
$V_{l}=\{q \mid 0<|\omega(p)|<r^{(2n+2)/(m+2)},
2l\pi/(m+2)<\arg\omega(p)<2(l+1)\pi/(m+2)\}$ for $l=0,1,\dots,m+1$. Then 
$\iota$ maps $U_{k}$ conformally onto some $V_{l}$. It follows that $\iota$ 
extends continuously to $U \cap \partial S$ so that $\iota$ maps each of the 
two components of $(U \cap \partial S) \setminus \{p_{0}\}$ onto a horizontal 
arc of $\psi$. 

We have shown that $\iota$ is extended to a continuous mapping $\bar{\iota}$ of 
$S$ into $R$ so that $\bar{\iota}(\partial S)$ consists of finitely many 
horizontal arcs and critical points of $\psi$. Since $\bar{\iota}(S^{\circ})$ 
is dense in $R$ and $\bar{\iota}(S)$ is compact, we infer that $\bar{\iota}$ is 
surjective and that $R \setminus \bar{\iota}(S^{\circ})$ is included in 
$\bar{\iota}(\partial S)$. Consequently, $\langle R,\bar{\iota}\rangle$ is a 
self-welding of $S$ with welder $\varphi$. 
\end{proof}

\begin{cor}
\label{cor:dense_continuation:self-welding}
Let $R_{0}$ be a nonanalytically finite open Riemann surface, and let 
$(R,\iota)$ be a dense compact continuation of $R_{0}$. If there are 
$\varphi \in M_{+}(R_{0})$ and $\psi \in M_{+}(R)$ such that 
$\varphi=\iota^{*}\psi$, then $(R,\iota)$ is a self-welding continuation of 
$R_{0}$ with welder $\varphi$. 
\end{cor}

\begin{proof}
Let $(\breve{R}_{0},\breve{\iota}_{0})$ be a natural compact continuation of 
$R_{0}$. Since punctures are removable singularities for conformal embeddings 
into compact Riemann surfaces, there is 
$\kappa \in \CEmb((\breve{R}_{0})^{\circ},R)$ such that 
$\iota=\kappa \circ \breve{\iota}_{0}$. Clearly, $(R,\kappa)$ is a dense 
compact continuation of $(\breve{R}_{0})^{\circ}$. Moreover, 
$\breve{\varphi}=\kappa^{*}\psi$ on $(\breve{R}_{0})^{\circ}$ for some 
$\breve{\varphi} \in M_{+}(\breve{R}_{0})$, for, 
$\breve{\iota}_{0}^{*}\kappa^{*}\psi=\iota^{*}\psi=\varphi \in M_{+}(R_{0})$. 
It follows from Proposition~\ref{prop:continuation:self-welding} that $\kappa$ 
is extended to $\bar{\kappa} \in \Cont(\breve{R}_{0},R)$ with 
$\bar{\kappa}(\breve{R}_{0})=R$ for which 
$\langle R,\bar{\kappa}\rangle$ is a self-welding of $\breve{R}_{0}$ with 
welder $\breve{\varphi}$. As $\iota=\bar{\kappa} \circ \breve{\iota}_{0}$ and 
$\varphi=\breve{\iota}_{0}^{*}\breve{\varphi}$, we obtain the corollary. 
\end{proof}

\begin{cor}
\label{cor:continuation:self-welding}
Let $R_{0}$ be a nonanalytically finite open Riemann surface, and let 
$(R,\iota)$ be a compact continuation of $R_{0}$. If there is 
$\psi \in M_{+}(R)$ such that $R \setminus \iota(R_{0})$ consists of finitely 
many horizontal arcs of $\psi$ together with finitely many points, then 
$(R,\iota)$ is a self-welding continuation of $R_{0}$ with welder 
$\iota^{*}\psi$. 
\end{cor}

\begin{proof}
Let $(\breve{R}_{0},\breve{\iota}_{0})$ and $\kappa$ be as in the proof of the 
preceding corollary. Thus $(R,\kappa)$ is a dense compact continuation of 
$(\breve{R}_{0})^{\circ}$ with $\iota=\kappa \circ \breve{\iota}_{0}$. Since 
$R \setminus \kappa((\breve{R}_{0})^{\circ})$ consists of finitely many 
horizontal arcs of $\psi$ together with finitely many points, it follows from a 
theorem of Carat\'{e}dory that $\kappa$ is extended to a continuous mapping 
$\bar{\kappa}$ of $\breve{R}_{0}$ onto $R$, which is holomorphic on 
$\partial\breve{R}_{0}$ off finitely points by the reflection principle, and 
hence the pull-back $\kappa^{*}\psi$ is extended to a quadratic differential 
$\breve{\varphi}$ in $M_{+}(\breve{R}_{0})$. 
Corollary~\ref{cor:dense_continuation:self-welding} shows that $(R,\iota)$ is 
a self-welding continuation of $R$ with welder 
$\iota^{*}\psi=\breve{\iota}_{0}^{*}\breve{\varphi} \in M_{+}(R_{0})$. 
\end{proof}

We now introduce a constructive procedure to obtain self-weldings of a compact 
bordered Riemann surface $S$. It is different from the welding procedure 
introduced in Ahlfors-Sario~\cite[II.3C--D]{AS1960} because $\iota$ is not 
required to be holomorphic on the border $\partial S$. In addition, we need to 
take into account quadratic differentials on $S$ and the corresponding 
quadratic differentials on the resulting Riemann surfaces. In the following, 
for a curve $c:I \to X$ on a topological space $X$, where $I$ is an interval on 
$\mathbb{R}$, by abuse of notation we sometimes use the same letter $c$ to 
denote its image $c(I)$. Let $\varphi \in M_{+}(S)$, and let $a_{1}$ and 
$a_{2}$ be simple arcs, not loops, of the same $\varphi$-length, say, $L$, on 
the border $\partial S$. They may pass through zeros of $\varphi$. Suppose that 
each of the arcs contains its endpoints and that the arcs are nonoverlapping, 
that is, $a_{1}^{\circ}$ and $a_{2}^{\circ}$ are disjoint, where $c^{\circ}$ 
stands for the part obtained from a simple arc $c$ by deleting its endpoints; 
we allow $a_{1}$ and $a_{2}$ to have a common endpoint. Then we can identify 
the arcs to obtain a new Riemann surface so that $\varphi$ induces a 
meromorphic quadratic differential on the new surface. 

Specifically, parametrize $a_{k}$, $k=1,2$, with $\varphi$-length parameter. 
Thus the $\varphi$-length of the subarc $a_{k}|_{[0,s]}$ of the arc 
$a_{k}:[0,L] \to \partial S$ is identical with $s$. There are two natural ways 
of identifying $a_{1}$ and $a_{2}$. One is to identify $a_{1}(s)$ with 
$a_{2}(s)$ for $s \in [0,L]$, and the other is to identify $a_{1}(s)$ with 
$a_{2}(L-s)$ for $s \in [0,L]$. Exactly one of the identifications leads us to 
an orientable topological surface $R$, possibly with border. For the sake of 
definiteness we assume that the first identification has been chosen. The genus 
of $R$ is the same as that of $S$ if and only if $a_{1}$ and $a_{2}$ lie on the 
same component of $\partial S$. Let $\iota:S \to R$ be the natural continuous 
mapping. Then $\iota \circ a_{1}$ is identical with $\iota \circ a_{2}$, which 
defines a simple arc $a:[0,L] \to R$. Note that $a^{\circ}$ lies in the 
interior $R^{\circ}$ though the endpoints of $a$ may or may not be on the 
border $\partial R$. 

We endow $R^{\circ}$ with a conformal structure as follows. We begin with 
adopting local parameters at each point of $R^{\circ}$ off $a$ so that $\iota$ 
is a conformal homeomorphism of $S^{\circ}$ onto $\iota(S^{\circ})$. We then 
apply an extension of the welding procedure described in 
\cite[II.3C--D]{AS1960} to choose local parameters at each point on 
$a \cap R^{\circ}$. To be more precise, taking $s_{0} \in (0,L)$, let 
$\zeta_{k}=\xi_{k}+i\eta_{k}$ be a natural parameter of $\varphi$ around 
$p_{k}:=a_{k}(s_{0})$, where $S$ is considered as a subsurface of its double 
$\hat{S}$. A small neighborhood $U_{k}$ of $p_{k}$ in $\hat{S}$ is mapped 
conformally onto a neighborhood of $0$ in $\mathbb{C}$ by $\zeta_{k}$. We may 
suppose that $S^{\circ}$ lies on the left and right of $a_{1}$ and $a_{2}$, 
respectively, so that $\zeta_{1}(U_{1} \cap S^{\circ}) \subset \mathbb{H}$ 
while $\zeta_{2}(U_{2} \cap S^{\circ}) \cap \bar{\mathbb{H}}=\varnothing$. Note 
that 
$$
s=\sgn(\xi_{k} \circ a_{k})(s) \cdot |(\xi_{k} \circ a_{k})(s)|^{n_{k}+1}+s_{0}
$$
for $s$ sufficiently near $s_{0}$, where $2n_{k}=\ord_{p_{k}}\varphi$ with 
nonnegative integer $n_{k}$. Take a neighborhood $U$ of 
$a(s_{0})=\iota(p_{1})=\iota(p_{2})$ included in 
$\iota((U_{1} \cup U_{2}) \cap S)$, and define a mapping $z:U \to \mathbb{C}$ 
by 
$$
z \circ \iota(p)=
\begin{cases}
\zeta_{1}(p)^{2(n_{1}+1)/(n_{1}+n_{2}+2)} & \text{if $p \in U_{1} \cap S$}, \\
\zeta_{2}(p)^{2(n_{2}+1)/(n_{1}+n_{2}+2)} & \text{if $p \in U_{2} \cap S$}
\end{cases}
$$
with $z \circ a(s)>0$ for $s>s_{0}$. Then $z$ is a well-defined element of 
$\TEmb(U,\mathbb{C})$, which is holomorphic off $a$. With the aid of these 
parameters $z$ we make $R^{\circ} \setminus \{a(0),a(l)\}$ a Riemann surface. 
If the endpoints of $a$ lie on the border $\partial R$, then we are done. If 
$a(0) \in R^{\circ}$, then $a_{1}(0)=a_{2}(0)$. Let $\zeta$ be a natural 
parameter of $\varphi$ around the common endpoint. Then the function $z$ 
defined by $z \circ \iota(p)=\zeta(p)^{2}$ is homeomorphic on a neighborhood of 
$a(0)$ and holomorphic on a punctured neighborhood of $a(0)$. Thus $z$ works as 
an analytic local parameter around $a(0)$. We deal with the case where 
$a(L) \in R^{\circ}$ in a similar manner. We have thus endowed $R^{\circ}$ with 
conformal structure. 

If $(\breve{R},\breve{\iota})$ is a natural compact continuation of 
$R^{\circ}$, then $\breve{\iota}$ is extended to a homeomorphism of $R$ onto 
$\breve{R}$. We introduce a conformal structure on $R$ so that the 
homeomorphism is actually a conformal homeomorphism of $R$ onto $\breve{R}$. We 
have thus obtained a compact Riemann surface $R$. 

The continuous mapping $\iota$ is holomorphic on 
$S \setminus (a_{1} \cup a_{2})$. Let $p$ be a point on 
$a_{1} \cup a_{2}$. If $p$ lies on $a_{1}^{\circ} \cup a_{2}^{\circ}$, then 
$\iota$ is holomorphic at $p$ if and only if $\ord_{p}\varphi=\ord_{q}\varphi$, 
where $q$ is the other point than $p$ projected to $\iota(p)$. If $p$ is an 
endpoint of $a_{1}$ or $a_{2}$, then $\iota$ is holomorphic at $p$ if and only 
if $p$ is a common endpoint of $a_{1}$ and $a_{2}$, or equivalently, $\iota(p)$ 
lies in $R^{\circ}$. 

The quadratic differential $\iota_{*}\varphi$ on $\iota(S^{\circ})$ induced by 
$\iota$ from $\varphi$ is extended to a meromorphic quadratic differential 
$\psi$ on $R$. Observe that 
\begin{equation}
\label{eq:welding:order}
\ord_{a(s)}\psi=
\begin{cases}
\dfrac{\,\ord_{a_{1}(s)}\varphi+\ord_{a_{2}(s)}\varphi\,}{2} &
\text{if $0<s<L$}, \\[2ex]
\dfrac{\,\ord_{a_{1}(s)}\varphi+\ord_{a_{2}(s)}\varphi\,}{4}-1 & 
\text{if ``$s=0$ or $L$" and $a(s) \in R^{\circ}$}, \\[2ex]
\ord_{a_{1}(s)}\varphi+\ord_{a_{2}(s)}\varphi+2 &
\text{if ``$s=0$ or $L$" and $a(s) \in \partial R$}.
\end{cases}
\end{equation}
In particular, $\psi$ has a pole on $a$ if and only if $a_{1}$ and $a_{2}$ has 
a common endpoint at which $\varphi$ does not vanish. If this is the case, then 
the pole is simple and lies in the interior $R^{\circ}$. In any case, $\psi$ is 
holomorphic and positive on $\partial R$ and hence belongs to 
$M_{+}(R)$. Clearly, the simple arc $a$ is composed of finitely many horizontal 
arcs and, possibly, critical points of $\psi$. Therefore, 
$\langle R,\iota\rangle$ is a self-welding of $S$ with welder $\varphi$, and 
$\psi$ is the co-welder of $\varphi$. We say that $\langle R,\iota\rangle$ is 
the self-welding of $S$ with welder $\varphi$ along $(a_{1},a_{2})$. 

More generally, let $a_{k}$, $k=1,\ldots,2N$, be nonoverlapping simple arcs on 
$\partial S$ with $L_{j}:=L_{\varphi}(a_{2j-1})=L_{\varphi}(a_{2j})$ for 
$j=1,\ldots,N$. Identifying $a_{2j-1}$ with $a_{2j}$ as above for 
$j=1,\ldots,N$ leads us to 
\begin{list}{{(\roman{claim})}}{\usecounter{claim}
\setlength{\topsep}{0pt}
\setlength{\itemsep}{0pt}
\setlength{\parsep}{0pt}
\setlength{\labelwidth}{\leftmargin}}
\item a compact Riemann surface $R$, 

\item a continuous mapping $\iota$ of $S$ onto $R$ with 
$\iota|_{S^{\circ}} \in \CEmb(S^{\circ},R)$ such that 
$\iota \circ a_{2j-1}(s)=\iota \circ a_{2j}(s) \in R^{\circ}$ or 
$\iota \circ a_{2j-1}(s)=\iota \circ a_{2j}(L_{j}-s) \in R^{\circ}$ for 
$s \in (0,L_{j})$, and 

\item a quadratic differential $\psi$ in $M_{+}(R)$ with 
$\varphi=(\iota|_{S^{\circ}})^{*}\psi$ on $S^{\circ}$. 
\end{list}
By subdividing $a_{k}$'s if necessary, it is no harm to assume from the outset 
that each pair $(a_{2j-1},a_{2j})$ induces a simple arc, not a loop, on $R$. 
The pair $\langle R,\iota\rangle$ will be called the self-welding of $S$ with 
welder $\varphi$ along $(a_{2j-1},a_{2j})$, $j=1,\ldots,N$. 

\begin{rem}
For more general procedures and applications of weldings we refer the reader to 
Bishop \cite{Bishop1988,Bishop1990}, Hamilton \cite{Hamilton1991}, Ishida 
\cite{Ishida1974}, Maitani \cite{Maitani1998} and Nishikawa-Maitani 
\cite{NM1997}, Oikawa \cite{Oikawa1961}, Semmes \cite{Semmes1985} and Williams 
\cite{Williams2012}. 
\end{rem}

Let $\langle R_{1},\iota_{1}\rangle$ be the self-welding of $S$ with welder 
$\varphi$ along $(a_{2j-1},a_{2j})$, $j=1,\ldots,N-1$, and let 
$\psi_{1} \in M_{+}(R_{1})$ be the co-welder of $\varphi$. Then the image arcs 
$(\iota_{1})_{*}a_{2N-1}$ and $(\iota_{1})_{*}a_{2N}$ are nonoverlapping simple 
arcs of the same $\psi_{1}$-length on the border $\partial R_{1}$. Hence we can 
construct the self-welding $\langle R',\iota'\rangle$ of $R_{1}$ with welder 
$\psi_{1}$ along $((\iota_{1})_{*}a_{2N-1},(\iota_{1})_{*}a_{2N})$. It is easy 
to verify that $\langle R',\iota' \circ \iota_{1}\rangle$ is a self-welding of 
$S$ equivalent to $\langle R,\iota\rangle$. Therefore, the self-welding 
$\langle R,\iota\rangle$ is also obtained by consecutive applications of 
self-welding procedures along {\em one\/} pair of arcs. 

\begin{exmp}
\label{exmp:self-welding:torus}
The quadratic differential $\varphi=(-1/z^{2})\,dz^{2}$ belongs to 
$M_{+}(\bar{\mathbb{D}})$, and the arcs $a_{1},\ldots,a_{8}$ on 
$\partial\mathbb{D}$ defined by $a_{j}(t)=e^{(t+j-1)\pi i/4}$, $t \in [0,1]$, 
are of the same $\varphi$-length. Let $\langle R,\iota\rangle$ be the 
self-welding of $\bar{\mathbb{D}}$ with welder $\varphi$ along $(a_{1},a_{6})$, 
$(a_{2},a_{5})$, $(a_{3},a_{8})$ and $(a_{4},a_{7})$. Then $R$ is a torus. 
Hence the self-welding is closed but not genus-preserving. 
\end{exmp}

Any self-welding of $S$ is a self-welding of $S$ along finitely many pairs of 
simple arcs on $\partial S$. To see this, let $\langle R,\iota\rangle$ be an 
arbitrary self-welding of $S$ with welder $\varphi$, and denote by $\psi$ the 
co-welder of $\varphi$. Then $R \setminus \iota(S^{\circ})$ consists of 
finitely many nonoverlapping simple arcs $e_{1},\ldots,e_{N}$. Note that every 
$e_{j}$ consists of horizontal arcs and, possibly, critical points of $\psi$. 
Observe that there are nonoverlapping simple arcs $a_{1},\ldots,a_{2N}$ on 
$\partial S$ such that $\iota_{*}a_{2j-1}=\iota_{*}a_{2j}=e_{j}$ for each $j$. 
Then $\langle R,\iota\rangle$ is the self-welding of $S$ with welder $\varphi$ 
along $(a_{2j-1},a_{2j})$, $j=1,\ldots,N$. 

\begin{exmp}
\label{self-welding:Joukowski_2}
Let $\langle\hat{\mathbb{C}},\iota_{j}\rangle$ be the self-weldings of 
$\bar{\mathbb{D}}$ in Example~\ref{self-welding:Joukowski}. The arcs $a_{01}$ 
and $a_{02}$ on $\partial\mathbb{D}$ defined by $a_{0k}(t)=e^{(t+k-1)\pi i}$, 
$t \in [0,1]$, are of the same $\varphi_{0}$-length, where 
$\varphi_{0}=(1-1/z^{2})^{2}\,dz^{2}$. Then 
$\langle\hat{\mathbb{C}},\iota_{0}\rangle$ is the self-welding of 
$\bar{\mathbb{D}}$ with welder $\varphi_{0}$ along $(a_{01},a_{02})$. If 
$a_{1k}:=ia_{0k}$, then $\langle\hat{\mathbb{C}},\iota_{1}\rangle$ is the 
self-welding of $\bar{\mathbb{D}}$ with welder $\varphi_{0}$ along 
$(a_{11},a_{12})$.
\end{exmp}

\section{Genus-preserving closed regular self-weldings}
\setcounter{equation}{0}
\label{sec:self-welding:closed_regular}

Let $S$ be a compact bordered Riemann surface. We introduce important classes 
of self-weldings of $S$. 

\begin{defn}[full self-welding]
\label{defn:full_self-welding}
Let $C$ be a union of connected components of $\partial S$. A genus-preserving 
self-welding $\langle R,\iota\rangle$ of $S$ is called {\em $C$-full\/} if 
$\iota(C) \subset R^{\circ}$ while $\iota(\partial S \setminus C)=\partial R$. 
\end{defn}

As is easily verified, a self-welding of $S$ is $\partial S$-full if and only 
if it is genus-preserving and closed. Note that the self-welding of $S$ with 
welder $\varphi$ along $(a_{2j-1},a_{2j})$, $j=1,\ldots,N$, is $C$-full if and 
only if the arcs $a_{k}$, $k=1,\ldots,2N$, exhaust $C$ and $a_{2j-1}$ and 
$a_{2j}$ lie on the same component of $C$ without separating any other pairs 
for $j=1,\ldots,N$. 

\begin{exmp}
\label{exmp:full_self-welding}
Each $\varphi \in M_{+}(S)$ induces a $C$-full self-welding of $S$. To obtain 
an example let $C_{1},\dots,C_{N}$ be the connected components of $C$. Divide 
each $C_{j}$ into two subarcs $a_{2j-1}$ and $a_{2j}$ of the same 
$\varphi$-length. Then the self-welding of $S$ with welder $\varphi$ along 
$(a_{2j-1},a_{2j})$, $j=1,\ldots,N$, is $C$-full. 
\end{exmp}

Let $\langle R,\iota\rangle$ be a $C$-full self-welding of $S$ with welder 
$\varphi$, and set $G_{\iota}=\iota(C)$. We consider $G_{\iota}$ as a graph, 
called the {\em weld graph} of $\langle R,\iota\rangle$, by declaring 
the points $p \in G_{\iota}$ with $\card\iota^{-1}(p) \neq 2$ to be the 
vertices of $G_{\iota}$. Each component of $G_{\iota}$ is in fact a tree. Note 
that each component of $G_{\iota}$ contains more than one vertex. If $v$ is a 
vertex of $G_{\iota}$, then 
\begin{equation}
\label{eq:welding:vertex}
\card\iota^{-1}(v)=\deg_{G_{\iota}}v \leqq \ord_{v}\psi+2,
\end{equation}
where $\psi$ is the co-welder of $\varphi$. An end-vertex of $G_{\iota}$ is a 
vertex $v$ with $\deg_{G_{\iota}} v=1$. It is characterized as a point on 
$G_{\iota}$ whose preimage by $\iota$ is a singleton 
(see~\eqref{eq:welding:vertex}). 

\begin{defn}[regular self-welding]
\label{defn:regular_self-welding}
A self-welding $\langle R,\iota\rangle$ of $S$ is called 
{\em $\varphi$-regular\/} if $\varphi$ is a welder of the self-welding and the 
co-welder of $\varphi$ is holomorphic on $R$. A self-welding of $S$ is said to 
be {\em regular\/} if it is $\varphi$-regular for some welder $\varphi$ of the 
self-welding. 
\end{defn}

\begin{rem}
If a self-welding $\langle R,\iota\rangle$ of $S$ is $\varphi$-regular, then 
$\varphi$ necessarily belongs to $A_{+}(S)$. Note that if $\psi$ is the 
co-welder of $\varphi$, then $\|\psi\|_{R}=\|\varphi\|_{S}$. 
\end{rem}

The next example is one of our motivations for introducing the notion of a 
genus-preserving closed regular self-welding. It also shows that an element of 
$A_{+}(S)$ can be a common welder of infinitely many inequivalent 
genus-preserving closed regular self-weldings of $S$. 

\begin{exmp}
\label{exmp:hydrodynamic_continuation}
Let $\omega=\omega(z)\,dz$ be a nonzero holomorphic semiexact 1-form on $S$ 
whose imaginary part vanishes along $\partial S$ (for the definition of a 
semiexact 1-form see \cite[V.5B]{AS1960}). Let $(R,\iota)$ be a 
{\em hydrodynamic\/} continuation of $S^{\circ}$ with respect to $\omega$ 
introduced by Shiba-Shibata \cite{SS1987}. Thus $R$ is a closed Riemann surface 
of the same genus as $S$, and $\iota$ is a conformal embedding of $S^{\circ}$ 
into $R$ such that $\omega=\iota^{*}\varsigma$ for some holomorphic 1-form 
$\varsigma$ on $R$ for which $R \setminus \iota(S^{\circ})$ consists of 
finitely many arcs along which $\im\varsigma$ vanishes. 
Corollary~\ref{cor:continuation:self-welding} implies that $\iota$ is extended 
to a continuous mapping $\bar{\iota}$ of $S$ onto $R$ such that 
$\langle R,\bar{\iota}\rangle$ is a genus-preserving closed 
$\omega^{2}$-regular self-welding of $S$; note that $\omega^{2}$ is a positive 
holomorphic quadratic differential on $S$. As is remarked in 
\cite[Section~26]{Shiba1984}, for some $S$ there is $\omega$ inducing 
uncountably many hydrodynamic continuations of $S^{\circ}$, or uncountably many 
inequivalent genus-preserving closed $\omega^{2}$-regular self-weldings of $S$. 
\end{exmp}

If $\varphi$ is a welder of a self-welding of $S$, then so is $r\varphi$ for 
any positive real number $r$. The next example shows that linearly independent 
quadratic differentials can induce the same genus-preserving closed regular 
self-welding. 

\begin{exmp}
\label{exmp:inducing_same_self-welding}
Let $R$ be a closed Riemann surface of genus $g>1$, and assume that it is 
symmetric in the sense that it admits an anti-conformal involution $J$ whose 
set $F$ of fixed points is composed of finitely many analytic simple loops on 
$R$. Take a simple arc $c$ on $F$, and set $R_{0}=R \setminus c$. The 
subsurface $R_{0}$ is a nonanalytically finite open Riemann surface of genus 
$g$. Let $(\breve{R}_{0},\breve{\iota}_{0})$ be a natural compact continuation 
of $R_{0}$. The element of $\CEmb(\breve{\iota}_{0}(R_{0}),R)$ defined by 
$\breve{\iota}_{0}(p) \mapsto p$, $p \in R_{0}$, is extended to a continuous 
mapping $\iota$ of $\breve{R}_{0}$ onto $R$. Let $\Omega$ be the vector space 
over $\mathbb{R}$ of holomorphic 1-forms $\omega$ on $R$ such that 
$\im\omega=0$ along $F$. Recall that $\dim_{\mathbb{R}} \Omega=g$. For any 
nonzero $\omega \in \Omega$ the pull-back $\iota^{*}(\omega^{2})$ of its square 
belongs to $A_{+}(\breve{R}_{0})$ and the genus-preserving closed self-welding 
$\langle R,\iota\rangle$ of $\breve{R}_{0}$ is $\iota^{*}(\omega^{2})$-regular. 
\end{exmp}

Let $C$ be a union of connected components of $\partial S$. Each positive 
meromorphic quadratic differential on $S$ defines infinitely many $C$-full 
inequivalent self-weldings (see Example~\ref{exmp:full_self-welding}). We now 
ask which quadratic differentials $\varphi$ on $S$ make $C$-full 
$\varphi$-regular self-weldings. 

\begin{prop}
\label{prop:regular_self-welding}
Let $\varphi \in A_{+}(S)$, and let $C$ be a union of connected components of 
$\partial S$. A $C$-full self-welding $\langle R,\iota\rangle$ of $S$ is 
$\varphi$-regular if and only if every end-vertex of its weld graph is the 
image of a border zero of $\varphi$ by $\iota$. 
\end{prop}

\begin{proof}
Let $\langle R,\iota\rangle$ be a $C$-full self-welding of $S$ with welder 
$\varphi$, and let $\psi$ be the co-welder of $\varphi$. If $p$ is a point of 
$\partial S$ for which $\iota(p)$ is an end-vertex of the weld graph 
$G_{\iota}$, then 
\begin{equation}
\label{eq:end-vertex:order}
\ord_{\iota(p)} \psi=\frac{\,\ord_{p} \varphi\,}{2}-1
\end{equation}
by~\eqref{eq:welding:order}. If the self-welding is $\varphi$-regular, then 
$\psi$ is holomorphic on $R$. Therefore, $\ord_{\iota(p)} \psi \geqq 0$ so that 
$\ord_{p} \varphi \geqq 2$ by~\eqref{eq:end-vertex:order}. Hence $p$ is a 
border zero of $\varphi$. 

If the self-welding is not $\varphi$-regular, then $\psi$ has a pole $q$ on 
$R$. Since $\varphi \in A_{+}(S)$, the co-welder $\psi$ is holomorphic on 
$\iota(S^{\circ})$ so that the pole $q$ lies on $G_{\iota}$, which must be an 
end-vertex of $G_{\iota}$ by~\eqref{eq:welding:vertex}. Another application 
of~\eqref{eq:end-vertex:order} shows that the point in $\iota^{-1}(q)$ is not a 
zero of $\varphi$. This completes the proof. 
\end{proof}

The following corollary is a generalization of Shiba-Shibata 
\cite[Lemma~3]{SS1987}. 

\begin{cor}
\label{cor:regular_self-welding}
Let $\varphi \in A_{+}(S)$, and let $C$ be a union of connected components of 
$\partial S$. If there is a $C$-full $\varphi$-regular self-welding of $S$, 
then each component of $C$ contains two or more zeros of $\varphi$. 
\end{cor}

\begin{proof}
If $\langle R,\iota\rangle$ is a $C$-full $\varphi$-regular self-welding of 
$S$, then each component of $C$ contains at least two points $p_{1}$ and 
$p_{2}$ whose images by $\iota$ are end-vertices of the weld graph. Since the 
self-welding is $\varphi$-regular, $p_{1}$ and $p_{2}$ must be zeros of 
$\varphi$ by Proposition~\ref{prop:regular_self-welding}. 
\end{proof}

\begin{defn}[border length condition]
\label{defn:border_length_condition}
Let $C$ be a union of connected components of $\partial S$. A positive 
meromorphic quadratic differential $\varphi$ on $S$ is said to satisfy the 
{\em border length condition\/} on $C$ if 
$$
L_{\varphi}(a) \leqq \frac{1}{\,2\,}L_{\varphi}(C')
$$
for any component $C'$ of $C$ and any horizontal trajectory $a$ of $\varphi$ 
included in $C'$. 
\end{defn}

Denote by $M_{L}(S,C)$ the set of positive meromorphic quadratic differentials 
on $S$ satisfying the border length condition on $C$, and set 
$A_{L}(S,C)=M_{L}(S,C) \cap A(S)$. We abbreviate $M_{L}(S,\partial S)$ and 
$A_{L}(S,\partial S)$ to $M_{L}(S)$ and $A_{L}(S)$, respectively. 

\begin{thm}
\label{thm:border_length_condition}
Let $S$ be a compact bordered Riemann surface, and let $C$ be a union of 
connected components of $\partial S$. For $\varphi \in A_{+}(S)$ there exists a 
$C$-full $\varphi$-regular self-welding of $S$ if and only if $\varphi$ 
satisfies the border length condition on $C$. 
\end{thm}

\begin{cor}
\label{cor:border_length_condition}
For $\varphi \in A_{+}(S)$ there is a genus-preserving closed $\varphi$-regular 
self-welding of $S$ if and only if $\varphi \in A_{L}(S)$. 
\end{cor}

For the proof Theorem~\ref{thm:border_length_condition} we make use of the 
following lemma. A topological space is said to be {\em I-shaped\/} (resp.\ 
{\em Y-shaped\,}) if it is homeomorphic to $[0,1]$ (resp.\ 
$\{z \in \mathbb{C} \mid z^{3} \in [0,1]\}$). For example, each component of 
the weld graph of the self-welding in Example~\ref{exmp:full_self-welding} is 
I-shaped. 

\begin{lem}
\label{lem:C-full:three_points}
Let $\varphi \in M_{+}(S)$. Let $p_{j}$, $j=1,2,3$, be points on a component 
$C$ of $\partial S$ which divide $C$ into three arcs of $\varphi$-lengths less 
than $L_{\varphi}(C)/2$. Then there is a $C$-full self-welding 
$\langle R,\iota\rangle$ of $S$ with welder $\varphi$ such that its weld graph 
is a Y-shaped tree with end-vertices $\iota(p_{j})$, $j=1,2,3$. If 
$\langle R',\iota'\rangle$ is a $C$-full self-welding of $S$ whose weld graph 
is Y-shaped with end-vertices $\iota'(p_{j})$, $j=1,2,3$, then 
$\langle R',\iota'\rangle$ is equivalent to $\langle R,\iota\rangle$. 
\end{lem}

\begin{proof}
We denote by $c_{j}$, $j=1,2,3$, the arcs on $C$ obtained from $C$ by cutting 
$C$ at $p_{j}$, $j=1,2,3$, where we label them so that $p_{j} \not\in c_{j}$. 
By assumption we have 
$$
L_{\varphi}(c_{1})+L_{\varphi}(c_{2})+L_{\varphi}(c_{3})=L_{\varphi}(C)
$$
and 
$$
L_{\varphi}(c_{j})<\frac{\,L_{\varphi}(C)\,}{2}, \qquad j=1,2,3.
$$
Take a point $q_{j}$ on $c_{j}^{\circ}$ to divide $c_{j}$ into two arcs, and 
name the resulting six arcs $a_{j},a'_{j}$, $j=1,2,3$, as follows: 
\begin{gather*}
c_{1}=a_{2} \cup a'_{3}, \quad c_{2}=a_{3} \cup a'_{1}, \quad
c_{3}=a_{1} \cup a'_{2}, \\
a_{j} \cap a'_{j}=\{p_{j}\}, \quad j=1,2,3.
\end{gather*}
If we choose $q_{j}$ so that 
$$
L_{\varphi}(a_{j})=\frac{\,L_{\varphi}(C)\,}{2}-L_{\varphi}(c_{j}),\quad
j=1,2,3,
$$
then $L_{\varphi}(a_{j})=L_{\varphi}(a'_{j})$ for $j=1,2,3$ and hence we can 
make a $C$-full self-welding $\langle R,\iota\rangle$ of $S$ with welder 
$\varphi$ along $(a_{j},a'_{j})$, $j=1,2,3$. Its weld graph $G_{\iota}$ is a 
Y-shaped tree and the end-vertices of $G_{\iota}$ are exactly $\iota(p_{j})$, 
$j=1,2,3$. 

Let $\langle R',\iota'\rangle$ be another $C$-full self-welding of $S$ with 
welder $\varphi$ whose weld graph $G_{\iota'}$ is Y-shaped with end-vertices 
$v'_{j}:=\iota'(p_{j})$, $j=1,2,3$. Since $G_{\iota'}$ is Y-shaped, it has one 
more vertex $v'_{4}$ with $\deg_{G_{\iota'}} v'_{4}=3$ and three edges 
$e'_{j}$, $j=1,2,3$, where $v'_{j} \in e'_{j}$. As 
$$
L_{\psi'}(e'_{1})+L_{\psi'}(e'_{2})=L_{\varphi}(c_{3}), \quad
L_{\psi'}(e'_{2})+L_{\psi'}(e'_{3})=L_{\varphi}(c_{1}), \quad
L_{\psi'}(e'_{3})+L_{\psi'}(e'_{1})=L_{\varphi}(c_{2}),
$$
where $\psi' \in M_{+}(R')$ is the co-welder of $\varphi$, the lengths 
$L_{\psi'}(e'_{j})$, $j=1,2,3$, are uniquely determined, and hence 
$(\iota')^{-1}(v'_{4})$ depends only on $\varphi$. Consequently, 
$\langle R',\iota'\rangle$ is equivalent to $\langle R,\iota\rangle$. 
\end{proof}

\begin{rem}
The three points $q_{j}$, $j=1,2,3$, in the above proof are projected to the 
same point on $R$. It is a zero of the co-welder of $\varphi$ even if none of 
$q_{j}$ is a zero of $\varphi$. 
\end{rem}

\begin{proof}[Proof of Theorem~\ref{thm:border_length_condition}]
Let $\langle R,\iota\rangle$ be an arbitrary $C$-full self-welding of $S$ with 
welder $\varphi$. Take nonoverlapping arcs $a_{k}$, $k=1,\ldots,2N$, on $C$ so 
that $\langle R,\iota\rangle$ is the self-welding of $S$ with welder $\varphi$ 
along $(a_{2j-1},a_{2j})$, $j=1,\dots,N$. By subdividing the arcs if necessary 
we may assume that every zero of $\varphi$ on $C$ is an endpoint of some 
$a_{k}$. Then the closure of each horizontal trajectory of $\varphi$ on $C$ is 
a union of $a_{k}$'s. Denote by $\psi$ the co-welder of $\varphi$. 

If $\varphi \not\in A_{L}(S,C)$, then some component $C'$ of $C$ includes a 
horizontal trajectory $a$ of $\varphi$ with $L_{\varphi}(a)>L_{\varphi}(C')/2$. 
Since $a_{2j-1}$ and $a_{2j}$ are of the same $\varphi$-length and $a_{2j}$ 
lies on $C'$ if $a_{2j-1}$ does, some pair, which may be assumed to be $a_{1}$ 
and $a_{2}$, lie on $a$. If they have a common endpoint on $a$, then it is not 
a zero of $\varphi$ but is mapped to an end-vertex of the weld graph 
$G_{\iota}$, and hence the self-welding $\langle R,\iota\rangle$ is not 
$\varphi$-regular by Proposition~\ref{prop:regular_self-welding}. Otherwise, 
the closure of $C' \setminus (a_{1}^{\circ} \cup a_{2}^{\circ})$ has a 
component $c_{0}$ included in $a$. Let $\langle R_{1},\iota_{1}\rangle$ be the 
self-welding of $S$ with welder $\varphi$ along $(a_{1},a_{2})$. Observe that 
$\iota_{1}(c_{0})$ is a component of the border $\partial R_{1}$ and contains 
exactly one zero of the co-welder $\psi_{1} \in A_{+}(R_{1})$ of $\varphi$. It 
follows from Corollary~\ref{cor:regular_self-welding} that the self-welding 
$\langle R,\iota'\rangle$ of $R_{1}$ with welder $\psi_{1}$ along 
$((\iota_{1})_{*}a_{2j-1},(\iota_{1})_{*}a_{2j})$, $2 \leqq j \leqq N$, is not 
$\psi_{1}$-regular. Therefore $\psi$ has a pole on $R$ as it is the co-welder 
of $\psi_{1}$. This means that the self-welding $\langle R,\iota\rangle$ of $S$ 
is not $\varphi$-regular, either. We have thus proved that if some $C$-full 
self-welding of $S$ is $\varphi$-regular, then $\varphi$ satisfies the border 
length condition on $C$. 

To prove the converse suppose that $\varphi \in A_{L}(S,C)$. Let $C'$ be a 
component of $C$. We will take one or three pairs of arcs of equal 
$\varphi$-lengths on $C'$ to obtain a $C'$-full $\varphi$-regular self-welding 
$\langle R_{1},\iota_{1}\rangle$ of $S$. Then the co-welder $\psi_{1}$ 
satisfies the border length condition on $C_{1}:=\iota_{1}(C \setminus C')$, or 
$\psi_{1} \in A_{L}(R_{1},C_{1})$. Since the number of components of $C_{1}$ is 
smaller by one than that of $C$, repeating this process leads us to a $C$-full 
$\varphi$-regular self-welding of $S$. 

Let $Z$ be the set of zeros of $\varphi$ on $C'$. The border length condition 
implies that $\card Z \geqq 2$. If there are two points in $Z$ which divide 
$C'$ into two arcs $a_{1}$ and $a_{2}$ of the same $\varphi$-length, then the 
self-welding of $S$ with welder $\varphi$ along $(a_{1},a_{2})$ is 
$\varphi$-regular by Proposition~\ref{prop:regular_self-welding}. Otherwise, 
$Z$ contains at least three points. Two distinct points $p,q \in Z$ divide $C'$ 
into two arcs. Let $\overline{pq}$ denote the one with shorter 
$\varphi$-length. Thus $\overline{pq}$ is the simple arc on $C'$ joining $p$ 
and $q$ whose $\varphi$-length $L_{\varphi}(\overline{pq})$ is less than 
$L_{\varphi}(C')/2$. Choose two distinct points $p_{1},p_{2} \in Z$ so that 
$L_{\varphi}(\overline{p_{1}p_{2}})$ is the largest among 
$L_{\varphi}(\overline{pq})$, where $p,q \in Z$ with $p \neq q$. The border 
length condition implies that $Z$ contains a point $p_{3}$ that does not lie on 
$\overline{p_{1}p_{2}}$. Since $L_{\varphi}(\overline{p_{k}p_{k+1}}) \leqq
L_{\varphi}(\overline{p_{1}p_{2}})<L_{\varphi}(C')/2$ for $k=1,2,3$, where 
$p_{4}=p_{1}$, it follows from Lemma~\ref{lem:C-full:three_points} that there 
is a $C$-full self-welding $\langle R_{1},\iota_{1}\rangle$ of $S$ with welder 
$\varphi$ whose weld graph is Y-shaped with end-vertices $\iota_{1}(p_{j})$, 
$j=1,2,3$. Thus the self-welding is $\varphi$-regular by 
Proposition~\ref{prop:regular_self-welding}, for, the points $p_{j}$, 
$j=1,2,3$, are zeros of $\varphi$. This completes the proof. 
\end{proof}

\begin{exmp}
\label{exmp:g=1:regular}
We consider the case of genus one. Let $\langle R,\iota\rangle$ be a 
genus-preserving closed $\varphi$-regular self-welding of $S$, and let $\psi$ 
be the co-welder of $\varphi$. Since $R$ is a torus, $\psi=\varsigma^{2}$ for 
some holomorphic $1$-form $\varsigma$ on $R$. Set $\omega=\iota^{*}\varsigma$. 
Since $\varphi=\omega^{2}$ is positive, the imaginary part of $\omega$ vanishes 
along $\partial S$. Moreover, $\omega$ is semiexact as 
$\int_{C} \omega=\int_{\iota_{*}C} \varsigma=0$ for all components $C$ of 
$\partial S$. Therefore, $(R,\iota|_{S^{\circ}})$ is a hydrodynamic 
continuation of $S^{\circ}$ with respect to $\omega$. As $\varsigma$ is free of 
zeros, it follows from~\eqref{eq:welding:order} that $\omega$ has exactly two 
zeros on each component of $\partial S$, and these points are projected to 
end-vertices of the weld graph $G_{\iota}$ by 
Proposition~\ref{prop:regular_self-welding}. Since $G_{\iota}$ has no other 
end-vertices, each component of $G_{\iota}$ is I-shaped. 
\end{exmp}

Let $R_{0}$ be a finite open Riemann surface, and let 
$(\breve{R}_{0},\breve{\iota}_{0})$ denote a natural compact continuation of 
$R_{0}$. Assuming that $R_{0}$ is nonanalytically finite, let $(R,\iota)$ be a 
genus-preserving closed self-welding continuation of $R_{0}$. For 
$\varphi \in A_{+}(R_{0})$ the self-welding continuation is said to be 
{\em $\varphi$-regular\/} if $\varphi$ is a welder of $(R,\iota)$ and the 
co-welder of $\varphi$ is holomorphic on $R$. It is called {\em regular\/} if 
it is $\varphi$-regular for some $\varphi$. By 
Corollary~\ref{cor:border_length_condition} there is a genus-preserving closed 
$\varphi$-regular self-welding continuation of $R_{0}$ if and only if $\varphi$ 
belongs to $A_{L}(R_{0}):=\breve{\iota}_{0}^{*}A_{L}(\breve{R}_{0})$, which is 
a proper subset of $A_{+}(R_{0})$. In the case where $R_{0}$ is analytically 
finite, we define $A_{L}(R_{0})=\breve{\iota}_{0}^{*}A_{+}(\breve{R}_{0})$. In 
either case $A_{L}(R_{0})$ is closed in $A_{+}(R_{0})$. 

As will be remarked in \S\ref{sec:continuation_to_Riemann_surface_on_boundary}, 
a closed continuation $(R,\iota)$ of a finite open Riemann surface $R_{0}$ is a 
regular self-welding continuation if and only if $\iota$ is a Teichm\"{u}ller 
conformal embedding. If this is the case, then an welder of $(R,\iota)$ and its 
co-welder are initial and terminal differentials of $\iota$, respectively. 

\section{An extremal property of closed regular self-welding continuations}
\setcounter{equation}{0}
\label{sec:closed_regular_self-welding}

We begin this section with introducing spaces of quadratic differentials on 
marked Riemann surfaces of genus $g$. Let ${\bm S} \in \mathfrak{F}_{g}$, and 
consider all pairs $(\varphi,\eta)$, where $(S,\eta)$ represents $\bm S$ and 
$\varphi$ is a quadratic differential on $S$. Two such pairs 
$(\varphi_{1},\eta_{1})$ and $(\varphi_{2},\eta_{2})$, where 
$(S_{j},\eta_{j}) \in {\bm S}$, are defined to be equivalent to each other if 
$\varphi_{1}=\kappa^{*}\varphi_{2}$ for some $\kappa \in \CHomeo(S_{1},S_{2})$ 
with $\kappa \circ \eta_{1} \simeq \eta_{2}$. Each equivalence class 
${\bm \varphi}=[\varphi,\eta]$ is called a {\em quadratic differential\/} on 
$\bm S$. If $\varphi$ possesses some conformally invariant properties, then we 
say that $\bm \varphi$ has the same properties. For example, $\bm \varphi$ is 
called meromorphic if $\varphi$ is. If $\varphi$ is measurable on $S$, then we 
set $\|{\bm \varphi}\|_{\bm S}=\|\varphi\|_{S}$. If $\|{\bm \varphi}\|_{\bm S}$ 
is finite, then $\bm \varphi$ is called {\em integrable}. Also, $\bm \varphi$ 
is said to coincide with $\bm \psi$ almost everywhere on $\bm S$ if 
$\varphi=\psi$ almost everywhere on $S$ for {\em some\/} 
$(\varphi,\eta) \in {\bm \varphi}$ and $(\psi,\eta) \in {\bm \psi}$, where 
$(S,\eta) \in {\bm S}$. If $\bm S$ is not a marked torus, then for quadratic 
differentials $\bm \varphi$ and $\bm \psi$ on $\bm S$ and complex numbers $c$ 
we define ${\bm \varphi}+{\bm \psi}$ and $c{\bm \varphi}$ in the obvious 
manner: 
$$
{\bm \varphi}+{\bm \psi}=[\varphi+\psi,\eta] \quad \text{and} \quad
c{\bm \varphi}=[c\varphi,\eta],
$$
where $(S,\eta) \in {\bm S}$ and $\varphi$ and $\psi$ are quadratic 
differentials on $S$. If $\bm S$ is a marked torus, then the addition operator 
$+$ is available only if $\varphi$ or $\psi$ is invariant under conformal 
automorphisms homotopic to the identity, that is, $\varphi$ or $\psi$ is 
holomorphic on $S$. 

Let ${\bm \iota}=[\iota,\eta_{1},\eta_{2}] \in \CEmb({\bm S}_{1},{\bm S}_{2})$, 
where ${\bm S}_{j}=[S_{j},\eta_{j}]$, $j=1,2$, and let 
${\bm \varphi}=[\varphi,\eta_{2}]$ be a quadratic differential on 
${\bm S}_{2}$. If ${\bm S}_{2}$ is not a marked torus, then the 
{\em pull-back\/} ${\bm \iota}^{*}{\bm \varphi}$ of $\bm \varphi$ is defined by 
${\bm \iota}^{*}{\bm \varphi}=[\iota^{*}\varphi,\eta_{1}]$. This definition 
does not depend on a particular choice of representatives. In the case where 
${\bm S}_{2}$ is a marked torus, the pull-back operator ${\bm \iota}^{*}$ 
applies only to holomorphic quadratic differentials on ${\bm S}_{2}$. 

The sets of meromorphic and holomorphic quadratic differentials on $\bm S$ are 
denoted by $M({\bm S})$ and $A({\bm S})$, respectively. We define 
$M_{+}({\bm S})$, $A_{+}({\bm S})$, $M_{L}({\bm S})$ and $A_{L}({\bm S})$ to be 
the subsets consisting of those ${\bm \varphi}=[\varphi,\eta]$ with $\varphi$ 
belonging to $M_{+}(S)$, $A_{+}(S)$, $M_{L}(S)$ and $A_{L}(S)$, respectively, 
provided that those spaces are meaningful. If $\bm S$ is closed, that is, if 
${\bm S} \in \mathfrak{T}_{g}$, then $A({\bm S})$ is a complex Banach space of 
dimension $d_{g}=\max\{g,3g-3\}$ with norm $\|\cdot\|_{\bm S}$. 

Now, let $\bm S=[S,\eta]$ be a marked compact bordered Riemann surface of genus 
$g$, and let $\langle R,\iota\rangle$ be a genus-preserving self-welding of 
$S$. Set $\theta=\iota \circ \eta$, which is a $g$-handle mark of $R$ as 
$\eta(\dot{\Sigma}_{g}) \subset S^{\circ}$, to obtain 
${\bm R}:=[R,\theta] \in \mathfrak{F}_{g}$ and 
${\bm \iota}:=[\iota,\eta,\theta] \in \Cont_{\mathrm{hc}}({\bm S},{\bm R})$. If 
$\langle R',\iota'\rangle$ is a self-welding of $S$ equivalent to 
$\langle R,\iota\rangle$ and $\theta'$ is a $g$-handle mark of $R'$ with 
$\iota' \circ \eta \simeq \theta'$, then $(R',\theta') \in {\bm R}$ and 
$(\iota',\eta,\theta') \in {\bm \iota}$. Also, if $(S',\eta') \in {\bm S}$ and 
$[\kappa,\eta,\eta'] \in \CHomeo_{\mathrm{hc}}({\bm S},{\bm S}')$, then 
$\langle R,\iota \circ \kappa^{-1}\rangle$ is a self-welding of $S'$ with 
$\kappa \circ \iota^{-1} \circ \theta \simeq \eta'$ and 
$(\iota \circ \kappa^{-1},\eta',\theta)$ represents $\bm \iota$. Thus 
$\langle{\bm R},{\bm \iota}\rangle$ is a well-defined pair, which we call a 
{\em self-welding\/} of $\bm S$. If $\varphi \in M_{+}(S)$ is a welder of 
$\langle R,\iota\rangle$ and $\psi \in M_{+}(R)$ is its co-welder, then we call 
${\bm \varphi}:=[\varphi,\eta] \in M_{+}({\bm S})$ and 
${\bm \psi}:=[\psi,\theta] \in M_{+}({\bm R})$ a {\em welder\/} of 
$\langle{\bm R},{\bm \iota}\rangle$ and the {\em co-welder\/} of $\bm \varphi$, 
respectively. If ${\bm \psi} \in A_{+}({\bm R})$, then 
$\langle{\bm R},{\bm \iota}\rangle$ is said to be {\em $\bm \varphi$-regular}. 
If this is the case, then ${\bm \varphi} \in A_{L}({\bm S})$ and 
${\bm \varphi}={\bm \iota}^{*}{\bm \psi}$. A self-welding 
$\langle{\bm R},{\bm \iota}\rangle$ is called {\em regular\/} if it is 
$\bm \varphi$-regular for some ${\bm \varphi} \in A_{L}({\bm R})$. Also, if 
$\langle R,\iota\rangle$ has some additional properties, then we say that 
$\langle{\bm R},{\bm \iota}\rangle$ possesses the same properties. For example, 
If $\langle R,\iota\rangle$ is closed, then so is 
$\langle{\bm R},{\bm \iota}\rangle$. 

Let us return to our investigations on $\mathfrak{M}({\bm R}_{0})$. We are 
exclusively concerned with the case where ${\bm R}_{0}$ is a marked finite open 
Riemann surface of positive genus $g$. Let 
$(\breve{\bm R}_{0},\breve{\bm \iota}_{0})$ denote the natural compact 
continuation of ${\bm R}_{0}$. Since punctures are removable singularities for 
conformal embeddings of a Riemann surface into closed Riemann surfaces, we have 
$\mathfrak{M}((\breve{\bm R}_{0})^{\circ})=\mathfrak{M}({\bm R}_{0})$. Without 
causing any trouble we sometimes assume, if necessary, that ${\bm R}_{0}$ is 
the interior of a marked compact bordered Riemann surface. 

Suppose that ${\bm R}_{0}$ is nonanalytically finite. Then $\breve{\bm R}_{0}$ 
is a marked compact bordered Riemann surface. A continuation 
$({\bm R},{\bm \iota})$ of ${\bm R}_{0}$ is said to be a {\em self-welding 
continuation\/} of ${\bm R}_{0}$ if 
${\bm \iota}={\bm \kappa} \circ \breve{\bm \iota}_{0}$ for some self-welding 
$\langle{\bm R},{\bm \kappa}\rangle$ of $\breve{\bm R}_{0}$. The pull-back 
${\bm \varphi}:=\breve{\bm \iota}_{0}^{*}\breve{\bm \varphi}$ of a welder 
$\breve{\bm \varphi}$ of the self-welding $\langle{\bm R},{\bm \kappa}\rangle$ 
is referred to as a {\em welder\/} of the self-welding continuation 
$({\bm R},{\bm \iota})$, and the co-welder of $\breve{\bm \varphi}$ is also 
called the {\em co-welder\/} of $\bm \varphi$. If the self-welding 
$\langle{\bm R},{\bm \kappa}\rangle$ is $\breve{\bm \varphi}$-regular, then the 
continuation $({\bm R},{\bm \iota})$ is said to be {\em $\bm \varphi$-regular}. 
A self-welding continuation is called {\em regular\/} if it is 
$\bm \varphi$-regular for some welder $\bm \varphi$. 

\begin{prop}
\label{prop:self-welding_continuation}
Let ${\bm R}_{0}$ be a marked nonanalytically finite open Riemann surface, and 
let $({\bm R},{\bm \iota})$ be a dense compact continuation of ${\bm R}_{0}$. 
If there are ${\bm \varphi} \in A_{+}({\bm R}_{0})$ and 
${\bm \psi} \in A_{+}({\bm R})$ such that 
${\bm \varphi}={\bm \iota}^{*}{\bm \psi}$, then $({\bm R},\bm \iota)$ is a 
self-welding continuation of ${\bm R}_{0}$ with welder ${\bm \varphi}$. 
\end{prop}

This is an immediate consequence of 
Corollary~\ref{cor:dense_continuation:self-welding}. Recall that the pull-back 
${\bm \iota}^{*}{\bm \psi}$ is well-defined for 
${\bm \psi} \in A_{+}({\bm R})$ even if $\bm R$ is a marked torus. 

Let ${\bm R}_{0}$ be a marked nonanalytically finite open Riemann surface of 
genus $g$. Note that Theorem~\ref{thm:main:A_L(R_0)} follows at once from 
Corollary~\ref{cor:border_length_condition}. Let 
${\bm \varphi} \in A_{L}({\bm R}_{0})$. For 
${\bm R} \in \mathfrak{M}({\bm R}_{0})$ let 
$\CEmb_{\bm \varphi}({\bm R}_{0},{\bm R})$ be the set of 
${\bm \iota} \in \CEmb_{\mathrm{hc}}({\bm R}_{0},{\bm R})$ such that 
$({\bm R},{\bm \iota})$ is a closed ${\bm \varphi}$-regular self-welding 
continuation of ${\bm R}_{0}$. It may be an empty set. Let 
$\mathfrak{M}_{\bm \varphi}({\bm R}_{0})$ denote the set of 
${\bm R} \in \mathfrak{M}({\bm R}_{0})$ for which 
$\CEmb_{\bm \varphi}({\bm R}_{0},{\bm R}) \neq \varnothing$. We are interested 
in the set 
$$
\mathfrak{M}_{L}({\bm R}_{0}):=\bigcup_{{\bm \varphi} \in A_{L}({\bm R}_{0})}
\mathfrak{M}_{\bm \varphi}({\bm R}_{0}).
$$ 

If ${\bm R}_{0}$ is analytically finite, then $\mathfrak{M}({\bm R}_{0})$ is 
exactly the singleton $\{\breve{\bm R}_{0}\}$. In this case we set 
$\mathfrak{M}_{L}({\bm R}_{0})=\mathfrak{M}({\bm R}_{0})$ for the sake of 
convenience. 

For the investigation of $\mathfrak{M}_{L}({\bm R}_{0})$ we recall the 
definition and some properties of measured foliations on surfaces. Let 
$\mathscr{S}(\Sigma_{g})$ denote the set of free homotopy classes of 
homotopically nontrivial simple loops on $\Sigma_{g}$. The set of nonnegative 
functions on $\mathscr{S}(\Sigma_{g})$ is identified with the product space 
$\mathbb{R}_{+}^{\mathscr{S}(\Sigma_{g})}$, where $\mathbb{R}_{+}=[0,+\infty)$. 
We endow it with the topology of pointwise convergence. Following 
\cite{AMO2016}, we define $\mathscr{MF}(\Sigma_{g})$ to be the closure of the 
set of functions of the form $\mathscr{S}(\Sigma_{g}) \ni \gamma \mapsto
ri(\alpha,\gamma) \in \mathbb{R}_{+}$ with $r \in \mathbb{R}_{+}$ and 
$\alpha \in \mathscr{S}(\Sigma_{g})$, where $i(\alpha,\gamma)$ denotes the 
geometric intersection number of $\alpha$ and $\gamma$, that is, 
$i(\alpha,\gamma)$ is the minimum of the numbers of common points of loops in 
$\alpha$ and $\gamma$. Every element of $\mathscr{MF}(\Sigma_{g})$ is called a 
{\em measured foliation\/} on $\Sigma_{g}$. If 
$\mathcal{F} \in \mathscr{MF}(\Sigma_{g})$ and $r \in \mathbb{R}_{+}$, then 
$r\mathcal{F} \in \mathscr{MF}(\Sigma_{g})$. 

Important examples of measured foliations are those induced by holomorphic 
quadratic differentials defined as follows. Let ${\bm R} \in \mathfrak{T}_{g}$. 
Choose a closed Riemann surface $R$ of genus $g$ together with 
$\theta \in \Homeo^{+}(\Sigma_{g},R)$ so that $(R,\dot{\theta}) \in {\bm R}$, 
where $\dot{\theta}=\theta|_{\dot{\Sigma}_{g}}$. For measurable quadratic 
differentials ${\bm \psi}=[\psi,\dot{\theta}]$ on $\bm R$ define a mapping 
$\mathcal{H}_{\bm R}({\bm \psi}):\mathscr{S}(\Sigma_{g}) \to \mathbb{R}_{+}$ 
by 
$$
\mathcal{H}_{\bm R}({\bm \psi})(\gamma)=
\inf_{c \in \gamma} H_{\psi}(\theta_{*}c)=
\inf_{c \in \gamma} \int_{\theta_{*}c} |\im\sqrt{\psi\,}|.
$$
This definition does not depend on a particular choice of representatives even 
if $g=1$. In the case where ${\bm \psi} \in A({\bm R})$, the mapping 
$\mathcal{H}_{\bm R}({\bm \psi})$ is a measured foliation on $\Sigma_{g}$ 
called the {\em horizontal foliation\/} of $\bm \psi$. Note that 
$\mathcal{H}_{\bm R}(r{\bm \psi})=\sqrt{r\,}\mathcal{H}_{\bm R}({\bm \psi})$ 
for $r \in \mathbb{R}_{+}$. 

\begin{prop}[Hubbard-Masur \cite{HM1979}]
\label{prop:Hubbard-Masur's_theorem}
For any ${\bm R} \in \mathfrak{T}_{g}$ the correspondence 
$$
\mathcal{H}_{\bm R}:{\bm \psi} \mapsto \mathcal{H}_{\bm R}({\bm \psi})
$$
is a homeomorphism of $A({\bm R})$ onto $\mathscr{MF}(\Sigma_{g})$. 
\end{prop}

For the proof see also Gardiner \cite[Theorem~6]{Gardiner1984}. The inverse of 
the homeomorphism $\mathcal{H}_{\bm R}$ will be denoted by 
${\bm Q}_{\bm R}:\mathscr{MF}(\Sigma_{g}) \to A({\bm R})$. Thus 
${\bm Q}_{\bm R}(\mathcal{F})$ stands for the holomorphic quadratic 
differential on $\bm R$ such that 
$\mathcal{H}_{\bm R}({\bm Q}_{\bm R}(\mathcal{F}))=\mathcal{F}$. In fact, the 
correspondence  $({\bm R},\mathcal{F}) \mapsto {\bm Q}_{\bm R}(\mathcal{F})$ 
defines a homeomorphism of $\mathfrak{T}_{g} \times \mathscr{MF}(\Sigma_{g})$ 
onto the complex vector bundle of holomorphic quadratic differentials over 
$\mathfrak{T}_{g}$. The bundle is canonically identified with the cotangent 
bundle $T^{*}\mathfrak{T}_{g}$ of the Teichm\"{u}ller space through the 
bilinear form $(\mu,\psi) \mapsto \re\int_{R} \mu\psi$ on the space of pairs of 
bounded measurable $(-1,1)$-forms $\mu$ on $R$ and holomorphic quadratic 
differentials $\psi$ on $R$. Note that 
${\bm Q}_{\bm R}(r\mathcal{F})=r^{2}{\bm Q}_{\bm R}(\mathcal{F})$ for 
$r \in \mathbb{R}_{+}$. 

For $\mathcal{F} \in \mathscr{MF}(\Sigma_{g})$ and 
${\bm R} \in \mathfrak{T}_{g}$ the {\em extremal length\/} 
$\Ext({\bm R},\mathcal{F})$ of $\mathcal{F}$ on $\bm R$ is defined by 
$$
\Ext({\bm R},\mathcal{F})=\|{\bm Q}_{\bm R}(\mathcal{F})\|_{\bm R}.
$$
Set $\Ext_{\mathcal{F}}({\bm R})=\Ext({\bm R},\mathcal{F})$ to obtain a 
nonnegative function $\Ext_{\mathcal{F}}$ on $\mathfrak{T}_{g}$. 

\begin{thm}
\label{thm:maximal_norm_property}
Let ${\bm R}_{0}$ be a marked nonanalytically finite open Riemann surface. Let 
$({\bm R},{\bm \iota})$ be a closed $\bm \varphi$-regular self-welding 
continuation of ${\bm R}_{0}$, and let $\bm \psi$ be the co-welder of 
$\bm \varphi$. Set $\mathcal{F}=\mathcal{H}_{\bm R}({\bm \psi})$. Then 
$$
\Ext_{\mathcal{F}}({\bm R}') \leqq \Ext_{\mathcal{F}}({\bm R})
$$
for all ${\bm R}' \in \mathfrak{M}({\bm R}_{0})$.
\end{thm}

In other words, the function $\Ext_{\mathcal{F}}$ attains its maximum on 
$\mathfrak{M}({\bm R}_{0})$ at $\bm R$. To prove 
Theorem~\ref{thm:maximal_norm_property} we first show the following proposition 
and lemma. The lemma will be also applied when we investigate the maximal sets 
for measured foliations on $\mathfrak{M}({\bm R}_{0})$ as well as uniqueness of 
conformal embeddings (see~\S\S\ref{sec:Ioffe_rays} and~\ref{sec:maximal_set}). 

In general, let ${\bm \varphi}=[\varphi,\dot{\theta}]$ be a measurable 
quadratic differential on ${\bm R}$. Define a mapping 
$\mathcal{H}'_{\bm R}({\bm \varphi})$ of $\mathscr{S}(\Sigma_{g})$ into 
$\mathbb{R}_{+}$ by $\mathcal{H}'_{\bm R}({\bm \varphi})(\gamma)=
\inf_{c} H_{\varphi}(\theta_{*}c)$, where the infimum is taken over all 
$c \in \gamma$ for which $\theta_{*}c$ is a piecewise analytic simple loop on 
$R$. 

\begin{prop}
\label{prop:second_minimal_norm_property}
Let ${\bm R} \in \mathfrak{T}_{g}$ and ${\bm \psi} \in A({\bm R})$, and let 
$\bm \varphi$ be an integrable quadratic differential on $\bm R$. If 
$\mathcal{H}_{\bm R}({\bm \psi})(\gamma) \leqq
\mathcal{H}'_{\bm R}({\bm \varphi})(\gamma)$ for all 
$\gamma \in \mathscr{S}(\Sigma_{g})$, then 
$$
\|{\bm \psi}\|_{\bm R} \leqq \|{\bm \varphi}\|_{\bm R}.
$$
The sign of equality occurs if and only if ${\bm \varphi}={\bm \psi}$ almost 
everywhere on $\bm R$. 
\end{prop}

\begin{proof}
Let $(R,\dot{\theta}) \in {\bm R}$ and 
$(\varphi,\dot{\theta}) \in {\bm \varphi}$. Though $\varphi$ is not continuous, 
we can apply the arguments in the proof of the second minimal norm property 
\cite[Theorem~9 in \S2.6]{Gardiner1987}. In fact, to estimate the integrals 
over spiral domains $D$ of $\psi$, where $(\psi,\dot{\theta}) \in {\bm \psi}$, 
we choose a horizontal arc $a$ of $\psi$ in $D$ so that $\sqrt{|\varphi|\,}$ is 
integrable on $a$ as in~\cite[Lemma~4 in \S12.7]{GL2000}, where the roles of 
$\psi$ and $\varphi$ are interchanged and the letter $\alpha$ is used for $a$. 
Then the reasoning in \cite{Gardiner1987} works without any further 
modifications. 
\end{proof}

Let $S'$ be a subsurface of a Riemann surface $S$, and let $\varphi'$ be a 
quadratic differential on $S'$. By the {\em zero-extension\/} of $\varphi'$ to 
$S$ we mean the quadratic differential $\varphi$ on $S$ defined by 
$\varphi=\varphi'$ on $S'$ and $\varphi=0$ on $S' \setminus S$. If 
$({\bm S},{\bm \iota})=([S,\eta],[\iota,\eta',\eta])$ is a continuation of 
${\bm S}'=[S',\eta']$ and ${\bm \varphi}'=[\varphi',\eta']$ is a quadratic 
differential on ${\bm S}'$, then the 
{\em $({\bm S},{\bm \iota})$-zero-extension\/} of ${\bm \varphi}'$ means the 
quadratic differential $[\varphi,\eta]$ on $\bm S$, where $\varphi$ is the 
zero-extension of the quadratic differential $\iota_{*}\varphi'$ on 
$\iota(S')$ to $S$. 

\begin{lem}
\label{lem:height_comparison}
Let ${\bm R}_{0}$, $({\bm R},{\bm \iota})$, $\bm \varphi$ and $\mathcal{F}$ be 
as in Theorem~\ref{thm:maximal_norm_property}, and let 
$({\bm R}',{\bm \iota}')$ be a closed continuation of ${\bm R}_{0}$. If 
${\bm \varphi}'$ is the $({\bm R}',{\bm \iota}')$-zero-extension of 
$\bm \varphi$, then for any $\gamma \in \mathscr{S}(\Sigma_{g})$ the inequality 
\begin{equation}
\label{eq:minimal_norm:claim}
\mathcal{H}_{{\bm R}'}({\bm \psi}')(\gamma) \leqq
\mathcal{H}'_{{\bm R}'}({\bm \varphi'})(\gamma)
\end{equation}
holds, where ${\bm \psi}'={\bm Q}_{{\bm R}'}(\mathcal{F})$. 
\end{lem}

\begin{proof}
Take representatives $(R_{0},\theta_{0})$, $(R,\dot{\theta})$, 
$(\iota,\theta_{0},\dot{\theta})$ and $(\varphi,\theta_{0})$ of ${\bm R}_{0}$, 
$\bm R$, $\bm \iota$ and $\bm \varphi$, respectively, where $\dot{\theta}$ 
stands for the restriction of $\theta \in \Homeo^{+}(\Sigma_{g},R)$ to 
$\dot{\Sigma}_{g}$. Let $(\breve{\bm R}_{0},\breve{\bm \iota}_{0})=
([\breve{R}_{0},\breve{\theta}_{0}],
[\breve{\iota}_{0},\theta_{0},\breve{\theta}_{0}])$ be the natural compact 
continuation of ${\bm R}_{0}$. There is a closed $\breve{\bm \varphi}$-regular 
self-welding $\langle{\bm R},\breve{\bm \iota}\rangle$ of $\breve{\bm R}_{0}$ 
such that ${\bm \iota}=\breve{\bm \iota} \circ \breve{\bm \iota}_{0}$ and 
${\bm \varphi}=\breve{\bm \iota}_{0}^{*}\breve{\bm \varphi}$ with 
$\breve{\bm \varphi}=[\breve{\varphi},\breve{\theta}_{0}] \in
A_{L}(\breve{\bm R}_{0})$. We may assume that $R_{0}$ has no punctures so that 
$\breve{\iota}_{0}(R_{0})=(\breve{R}_{0})^{\circ}$. Let $\breve{\Phi}$ be the 
integral of $\sqrt{\breve{\varphi}\,}$, which is a multi-valued function and 
may have finitely many singular points on $\breve{R}_{0}$. Let 
$\breve{C}_{1},\ldots,\breve{C}_{n_{0}}$ be the components of 
$\partial \breve{R}_{0}$. For each $\breve{C}_{k}$ choose a doubly connected 
closed neighborhood $\breve{U}_{k}$ in $\breve{R}_{0}$ of $\breve{C}_{k}$ so 
that 
\begin{list}{{(\roman{claim})}}{\usecounter{claim}
\setlength{\topsep}{0pt}
\setlength{\itemsep}{0pt}
\setlength{\parsep}{0pt}
\setlength{\labelwidth}{\leftmargin}}
\item $\breve{U}_{1},\ldots,\breve{U}_{n_{0}}$ are mutually disjoint, 

\item each $\breve{U}_{k}$ is divided into finitely many simply connected 
closed domains $\breve{D}_{1}^{(k)},\ldots,\breve{D}_{m_{k}}^{(k)}$, 

\item a branch of $\breve{\Phi}$ maps each $\breve{D}_{j}^{(k)}$ 
homeomorphically onto a closed rectangle in $\mathbb{C}$ with sides parallel to 
the real and imaginary axes, and 

\item the vertical sides of the rectangles $\breve{\Phi}(\breve{D}_{j}^{(k)})$, 
$1 \leqq j \leqq m_{k}$, $1 \leqq k \leqq n_{0}$, are of the same length. 
\end{list}
Note that $(\breve{R}_{0})^{\circ} \cap \partial \breve{U}_{k}$ is composed of 
finitely many horizontal and vertical arcs of $\breve{\varphi}$. Set 
$\breve{D}_{0}=\breve{R}_{0} \setminus \bigcup_{k} \breve{U}_{k}$ and 
$D=\breve{\iota}(\breve{D}_{0})$, where 
$\breve{\bm \iota}=[\breve{\iota},\breve{\theta}_{0},\theta]$. Observe that 
$\breve{\iota}(\breve{U}_{k})$ is a component of $R \setminus D$, which is a 
topological closed disk on $R$, even though $\breve{\iota}$ is not injective on 
$\breve{U}_{k}$. Also, each $\breve{\iota}(\breve{D}_{j}^{(k)})$ meets the weld 
graph $G_{\breve{\iota}}$ since 
$\breve{D}_{j}^{(k)} \cap \breve{C}_{k} \neq \varnothing$. 

Choose a closed Riemann surface $R'$ of genus $g$ together with 
$\theta' \in \Homeo^{+}(\Sigma_{g},R')$ so that 
$(R',\dot{\theta}') \in {\bm R}'$, where 
$\dot{\theta}'=\theta'|_{\dot{\Sigma}_{g}}$, and take a representative 
$(\iota',\theta_{0},\dot{\theta}') \in {\bm \iota}'$. Let $\gamma$ be an 
arbitrary element of $\mathscr{S}(\Sigma_{g})$, and let 
$c' \in \theta'_{*}\gamma$ be a piecewise analytic simple loop on $R'$. We may 
assume that the initial (and terminal) point of $c'$ is in the domain 
$D':=\iota' \circ \breve{\iota}_{0}^{-1}(\breve{D}_{0})$. Divide $c'$ into 
subarcs to obtain $c'=c'_{1}d'_{1} \cdots c'_{m-1}d'_{m-1}c'_{m}$ so that 
$c'_{1},\ldots,c'_{m}$ lie in $\bar{D}'$ while $d'_{1},\ldots,d'_{m-1}$ lie in 
$R' \setminus D'$. For $j=1,\ldots,m$ let $c_{j}$ be the image arc of $c'_{j}$ 
by $\iota \circ (\iota')^{-1}$. They are piecewise analytic simple arcs on 
$\bar{D}$. Note that the terminal point $q_{j}$ of $c_{j}$ and the initial 
point $p_{j+1}$ of $c_{j+1}$ lie on the same component, say 
$\breve{\iota}(\breve{U}_{k})$, of $R \setminus D$. We choose a piecewise 
analytic simple arc $d_{j}$ joining $q_{j}$ with $p_{j+1}$ within 
$\breve{\iota}(\breve{U}_{k})$ as follows. Let 
${\bm \psi}=[\psi,\dot{\theta}] \in A({\bm R})$ be the co-welder of 
$\breve{\bm \varphi}$, and let $\varepsilon>0$. Take $\breve{D}_{\nu}^{(k)}$ 
and $\breve{D}_{\mu}^{(k)}$ so that $q_{j}$ and $p_{j+1}$ belong to 
$\breve{\iota}(\partial \breve{D}_{\nu}^{(k)})$ and 
$\breve{\iota}(\partial \breve{D}_{\mu}^{(k)})$, respectively. We then let 
$d''_{j}$ be a simple arc on $\breve{\iota}(\breve{U}_{k})$ joining $q_{j}$ 
with $p_{j+1}$ composed of horizontal and vertical arcs, and possibly, zeros of 
$\psi$, where the horizontal and vertical arcs should lie on 
$\breve{\iota}(\partial\breve{D}_{\nu}^{(k)} \cup \partial\breve{D}_{\mu}^{(k)})
\cup G_{\breve{\iota}}$ and 
$\breve{\iota}(\breve{D}_{\nu}^{(k)} \cup \breve{D}_{\mu}^{(k)})$, 
respectively. We require that each $d''_{j}$ should include at most two 
vertical arcs, which implies $H_{\psi}(d''_{j}) \leqq H_{\varphi'}(d'_{j})$. 
Some of $d''_{j}$'s may have common points. We modify them by slightly shifting 
their horizontal and vertical arcs and going around the zeros of $\psi$ to 
obtain simple arcs $d_{j}$ without changing the endpoints so that 
$H_{\psi}(d_{j}) \leqq H_{\psi}(d''_{j})+\varepsilon/m$ and that 
$d_{1},\ldots,d_{m-1}$ are mutually disjoint. The simple loop 
$c:=c_{1}d_{1}c_{2}d_{2} \cdots d_{m-1}c_{m}$ belongs to the homotopy class 
$\theta_{*}\gamma$, and satisfies 
$H_{\psi}(c) \leqq H_{\varphi'}(c')+\varepsilon$ and hence 
$\mathcal{H}_{\bm R}({\bm \psi})(\gamma) \leqq H_{\varphi'}(c')+\varepsilon$, 
which leads us to $\mathcal{H}_{\bm R}({\bm \psi})(\gamma) \leqq
\mathcal{H}'_{{\bm R}'}({\bm \varphi'})(\gamma)$. Since 
$\mathcal{H}_{{\bm R}'}({\bm \psi}')(\gamma)=\mathcal{F}(\gamma)=
\mathcal{H}_{\bm R}({\bm \psi})(\gamma)$, we 
obtain~\eqref{eq:minimal_norm:claim}. 
\end{proof}

\begin{proof}[Proof of Theorem~\ref{thm:maximal_norm_property}]
Let $({\bm R}',{\bm \iota}')$ be an arbitrary closed continuation of 
${\bm R}_{0}$, and let ${\bm \varphi}'$ be the 
$({\bm R}',{\bm \iota}')$-zero-extension of $\bm \varphi$. Set 
${\bm \psi}'={\bm Q}_{{\bm R}'}(\mathcal{F})$. We apply 
Lemma~\ref{lem:height_comparison} and 
Proposition~\ref{prop:second_minimal_norm_property} to obtain 
$\|{\bm \psi}'\|_{{\bm R}'} \leqq \|{\bm \varphi}'\|_{{\bm R}'}=
\|{\bm \varphi}\|_{{\bm R}_{0}}=
\|{\bm \psi}\|_{\bm R}$. This completes the proof as 
$\Ext_{\mathcal{F}}({\bm R}')=\|{\bm \psi}'\|_{{\bm R}'}$ and 
$\Ext_{\mathcal{F}}({\bm R})=\|{\bm \psi}\|_{\bm R}$. 
\end{proof}

Let ${\bm R}_{j}$, $j=1,2$, be distinct points of $\mathfrak{T}_{g}$. Then 
$d_{T}({\bm R}_{1},{\bm R}_{2})=(\log K({\bm h}))/2$, where $\bm h$ is the 
{\em Teichm\"{u}ller quasiconformal homeomorphism\/} of ${\bm R}_{1}$ onto 
${\bm R}_{2}$. Thus $\bm h$ belongs to 
$\QCHomeo_{\mathrm{hc}}({\bm R}_{1},{\bm R}_{2})$ and there are nonzero 
${\bm \psi}_{j} \in A({\bm R}_{j})$, $j=1,2$, such that for some 
$(h,\theta_{1},\theta_{2}) \in {\bm h}$ and 
$(\psi_{j},\theta_{j}) \in {\bm \psi}_{j}$ with 
$(R_{j},\theta_{j}) \in {\bm R}_{j}$ 
\begin{list}{{(\roman{claim})}}{\usecounter{claim}
\setlength{\topsep}{0pt}
\setlength{\itemsep}{0pt}
\setlength{\parsep}{0pt}
\setlength{\labelwidth}{\leftmargin}}
\item the Beltrami differential of $h$ is exactly $k|\psi_{1}|/\psi_{1}$, where 
$k=(K({\bm h})-1)/(K({\bm h})+1)$, 

\item $h$ maps every noncritical point of $\psi_{1}$ to a noncritical point of 
$\psi_{2}$, and 

\item $h$ is represented as 
$$
\zeta_{2}=K({\bm h})\re\zeta_{1}+i\im\zeta_{1}=
\frac{\,(K({\bm h})+1)\zeta_{1}+(K({\bm h})-1)\bar{\zeta}_{1}\,}{2}
$$
with respect to some natural parameter $\zeta_{1}$ (resp.\ $\zeta_{2}$) of 
$\psi_{1}$ (resp.\ $\psi_{2}$) around any noncritical point $p$ (resp.\ 
$h(p)$). 
\end{list}
In other words, $h$ is a uniform stretching along horizontal trajectories of 
$\psi_{1}$. The quadratic differentials ${\bm \psi}_{1}$ and ${\bm \psi}_{2}$ 
are called {\em initial\/} and {\em terminal\/} quadratic differentials of 
$\bm h$, respectively. Note that ${\bm h}^{-1}$ is also the Teichm\"{u}ller 
quasiconformal homeomorphism of ${\bm R}_{2}$ onto ${\bm R}_{1}$, whose initial 
and terminal quadratic differentials are $-{\bm \psi}_{2}$ and 
$-K^{2}{\bm \psi}_{1}$, respectively. 

Let ${\bm R} \in \mathfrak{T}_{g}$ and 
${\bm \psi} \in A({\bm R}) \setminus \{{\bm 0}\}$. For each $t>0$ there 
uniquely exists ${\bm r}(t) \in \mathfrak{T}_{g}$ with 
$d_{T}({\bm r}(t),{\bm R})=t$ such that $\bm \psi$ is an initial quadratic 
differential of the Teichm\"{u}ller quasiconformal homeomorphism of $\bm R$ 
onto ${\bm r}(t)$. Set ${\bm r}(0)={\bm R}$. Then the mapping 
${\bm r}:\mathbb{R}_{+} \to \mathfrak{T}_{g}$ is a (simple) ray emanating from 
$\bm R$, called a {\em Teichm\"{u}ller geodesic ray}. Note that 
\begin{equation}
\label{eq:Teichmuller_ray:parametrization}
d_{T}({\bm r}(t_{1}),{\bm r}(t_{2}))=|t_{1}-t_{2}|
\end{equation}
for $t_{1},t_{2} \in \mathbb{R}_{+}$. For $t_{1},t_{2} \in [0,+\infty]$ with 
$t_{1}<t_{2}$ we denote by ${\bm r}(t_{1},t_{2})$ the image of the interval 
$(t_{1},t_{2})$ by $\bm r$: 
\begin{equation}
\label{eq:Teichmulle_geodesic_ray:subarc}
{\bm r}(t_{1},t_{2})=\{{\bm r}(t) \mid t_{1}<t<t_{2}\}.
\end{equation}

If we need to refer to the initial point $\bm R$ and the quadratic differential 
$\bm \psi$ we use the notation ${\bm r}_{\bm R}[{\bm \psi}]$ for $\bm r$. Thus 
$\bm \psi$ is an initial quadratic differential of the Teichm\"{u}ller 
quasiconformal homeomorphism ${\bm h}_{t}$ of $\bm R$ onto 
${\bm r}_{\bm R}[{\bm \psi}](t)$. Let ${\bm \psi}_{t}$ be the corresponding 
terminal quadratic differential of ${\bm h}_{t}$. If we set 
$\mathcal{F}=\mathcal{H}_{{\bm R}}({\bm \psi})$, then we have 
${\bm Q}_{{\bm r}_{\bm R}[{\bm \psi}](t)}(\mathcal{F})={\bm \psi}_{t}$ as 
${\bm h}_{t}$ is homotopically consistent. It follows that 
\begin{equation}
\label{eq:extremal_length:Teichmuller_qc}
\Ext_{\mathcal{F}}({\bm r}_{\bm R}[{\bm \psi}](t))=
e^{2t}\Ext_{\mathcal{F}}({\bm R}), \quad
\mathcal{F}=\mathcal{H}_{{\bm R}}({\bm \psi}).
\end{equation}

A Teichm\"{u}ller geodesic ray is uniquely determined by the initial 
point $\bm R$ and a point ${\bm R}' \neq {\bm R}$ that the ray passes through. 
We write ${\bm r}_{\bm R}[{\bm R}']$ to denote such a ray. Note that we 
parametrize each Teichm\"{u}ller geodesic ray with respect to the distance from 
its initial point. Therefore, 
${\bm r}_{\bm R}[r{\bm \psi}]={\bm r}_{\bm R}[{\bm \psi}]$ for $r>0$. 

As an application of Theorem~\ref{thm:maximal_norm_property} we show the 
following proposition. It proves a half of 
Theorem~\ref{thm:main:CEmb(R_0,R)}~(i). 

\begin{prop}
\label{prop:boundary:regular_self-welding}
Let ${\bm R}_{0}$ be a marked finite open Riemann surface. Then 
$\mathfrak{M}_{L}({\bm R}_{0})$ is included in the boundary 
$\partial\mathfrak{M}({\bm R}_{0})$. 
\end{prop}

\begin{proof}
We have only to consider the case where ${\bm R}_{0}$ is nonanalytically 
finite. Let $({\bm R},{\bm \iota})$ be a closed $\bm \varphi$-regular 
self-welding continuation of ${\bm R}_{0}$, and denote  by $\bm \psi$ the 
co-welder of $\bm \varphi$. We know that $\bm R$ belongs to 
$\mathfrak{M}({\bm R}_{0})$.  If $\mathcal{F}=\mathcal{H}_{\bm R}({\bm \psi})$, 
then it follows from~\eqref{eq:extremal_length:Teichmuller_qc} that 
$\Ext_{\mathcal{F}}({\bm r}_{\bm R}[{\bm \psi}](t))=
e^{2t}\Ext_{\mathcal{F}}({\bm R})>
\Ext_{\mathcal{F}}({\bm R})$ for $t>0$. Theorem~\ref{thm:maximal_norm_property} 
then assures us that ${\bm r}_{\bm R}[{\bm \psi}](t)$ lies outside of 
$\mathfrak{M}({\bm R}_{0})$. Since ${\bm r}_{\bm R}[{\bm \psi}](t)$ tends to 
$\bm R$ as $t \to 0$, we conclude that $\bm R$ is certainly on the boundary of 
$\mathfrak{M}({\bm R}_{0})$. 
\end{proof}

\begin{rem}
We could apply Kahn-Pilgrim-Thurston \cite{KPT2022} to prove 
Proposition~\ref{prop:boundary:regular_self-welding}. Let $\bm R$, 
$\bm \varphi$ and $\bm \psi$ be as in the proof of the proposition. 
Approximating $\bm \varphi$ with Jenkins-Strebel quadratic differentials in 
$A_{+}({\bm R}_{0})$, we know that the stretch factor of homotopically 
consistent topological embeddings of ${\bm R}_{0}$ into 
${\bm r}_{\bm R}[{\bm \psi}](t)$ is exactly $e^{2t}$, which implies that 
${\bm r}_{\bm R}[{\bm \psi}](t) \not\in \mathfrak{M}({\bm R}_{0})$ for $t>0$ by 
\cite[Theorem~1]{KPT2022}. 
\end{rem}

\begin{exmp}
\label{exmp:g=1:extremal_length}
We consider the case of genus one, and use the notations in 
Examples~\ref{exmp:conformal_automorphism:torus} and~\ref{exmp:marked_torus}. 
The holomorphic quadratic differential $dz^{2}$ 
on $\mathbb{C}$ is projected to a holomorphic quadratic differential 
$\psi_{\tau}$ on $T_{\tau}$ through the natural projection 
$\Pi_{\tau}:\mathbb{C} \to T_{\tau}=\mathbb{C}/\Gamma_{\tau}$. Set 
${\bm \psi}_{\tau}=[\psi_{\tau},\dot{\eta}_{\tau}] \in A({\bm T}_{\tau})$. 
Recall that $\Sigma_{1}=T_{\sqrt{-1}}$. If $\mathcal{F}=
\mathcal{H}_{{\bm T}_{\sqrt{-1}}}({\bm \psi}_{\sqrt{-1}}) \in
\mathscr{MF}(\Sigma_{1})$, then 
${\bm Q}_{{\bm T}_{\tau}}(\mathcal{F})={\bm \psi}_{\tau}/(\im\tau)^{2}$ so that 
$\Ext_{\mathcal{F}}({\bm T}_{\tau})=1/\im\tau$. Now, let 
${\bm R}_{0}=[R_{0},\theta_{0}]$ be a marked finite open Riemann surface of 
genus one, and let $(\breve{\bm R}_{0},\breve{\bm \iota}_{0})$ be the natural 
compact continuation. Let 
$(\breve{R}_{0},\breve{\theta}_{0}) \in \breve{\bm R}_{0}$ and 
$A_{0}=(\breve{\theta}_{0})_{*}A_{\sqrt{-1}}$. For $t \in \mathbb{R}$ there 
exists a (unique) holomorphic semiexact 1-form $\omega_{t}$ on $\breve{R}_{0}$ 
with $\int_{A_{0}} \omega_{t}=1$ such that the imaginary part of 
$e^{-i t \pi/2}\omega_{t}$ vanishes along the border $\partial\breve{R}_{0}$. 
Set ${\bm \varphi}_{t}=[e^{-it\pi}\omega_{t}^{2},\breve{\theta}_{0}]$, which is 
an element of $A_{L}(\breve{\bm R}_{0})$. Let 
$\langle{\bm T}_{\tau(t)},{\bm \iota}_{t}\rangle$ be a closed 
${\bm \varphi}_{t}$-regular self-welding of $\breve{\bm R}_{0}$. Then 
${\bm \psi}_{\tau(t)}$ is the co-welder of the welder ${\bm \varphi}_{t}$. 
Restricting ourselves to the case $t=0$, we conclude 
from Theorem~\ref{thm:maximal_norm_property} that $\im\tau \geqq \im\tau(0)$ 
if ${\bm T}_{\tau} \in \mathfrak{M}({\bm R}_{0})$. This is nothing but the 
first inequality of \cite[Theorem~2~(I)]{Shiba1987}. We remark that 
$\Ext_{\mathcal{F}}({\bm T}_{\tau(0)})$ is equal to the extremal length of the 
weak homology class of $(\theta_{0})_{*}A_{\sqrt{-1}}$ by 
\cite[Proposition~1]{Masumoto1992}. Thus we have obtained an alternative proof 
of \cite[Lemma~1]{Masumoto1994}. 
\end{exmp}

\section{Sequences of continuations}
\setcounter{equation}{0}
\label{sec:continuation:sequence}

The present section is of preparatory character. In this article we need to 
consider sequences of continuations of different Riemann surfaces. To deal with 
their convergence properties we pass to universal covering Riemann surfaces. As 
an application, we show, in the next section, that the boundary points obtained 
through closed regular self-welding continuations of ${\bm R}_{0}$ actually 
exhaust $\partial\mathfrak{M}({\bm R}_{0})$. 

Let $\chi$  be a $1$-handle mark of a Riemann surface $S$ of positive genus, 
and let $\Pi_{\chi}:\mathbb{U}_{S} \to S$ be a holomorphic universal covering 
map with covering transformation group $\Gamma_{\chi}:=\Aut(\Pi_{\chi})$. We 
normalize $\Pi_{\chi}$ and $\Gamma_{\chi}$ as follows. Set 
$A=\chi_{*}A_{\sqrt{-1}}$ and $B=\chi_{*}B_{\sqrt{-1}}$. 

If $S$ is not a torus, then we may assume that 
$\mathbb{H} \subset \mathbb{U}_{S} \subset \bar{\mathbb{H}}$. Then 
$\Gamma_{\chi}$ is a torsion-free Fuchsian group keeping $\mathbb{U}_{S}$ 
invariant. We can choose $\Pi_{\chi}$ and $\Gamma_{\chi}$ so that 
$\Gamma_{\chi}$ contains hyperbolic transformations $\gamma_{\chi}^{(A)}$ and 
$\gamma_{\chi}^{(B)}$ such that the repelling  and attracting fixed points of 
$\gamma_{\chi}^{(A)}$ (resp.\ $\gamma_{\chi}^{(B)}$) are $0$ and $\infty$ 
(resp.\ $1$ and some negative real number), respectively, and that the axes of 
$\gamma_{\chi}^{(A)}$ and $\gamma_{\chi}^{(B)}$ oriented from the repelling 
fixed points to the attracting fixed points are projected onto closed 
hyperbolic geodesics $A_{\chi}$ and $B_{\chi}$ on $S^{\circ}$ freely homotopic 
to $A$ and $B$, respectively. 

If $S$ is a torus, then we may set $\mathbb{U}_{S}=\mathbb{C}$. Let 
$\omega_{\chi}$ be the unique holomorphic 1-form on $S$ for which 
$\int_{A} \omega_{\chi}=1$. Choose $\Pi_{\chi}$ so that 
$\Pi_{\chi}^{*}\omega_{\chi}=dz$. If we set $\tau=\int_{B} \omega_{\chi}$, 
then, using the notations in Example~\ref{exmp:conformal_automorphism:torus}, 
we have $S=T_{\tau}$ and $\Gamma_{\chi}=\Gamma_{\tau}$. Denote by 
$\gamma_{\chi}^{(A)}$ and $\gamma_{\chi}^{(B)}$ the transformations in 
$\Gamma_{\chi}$ defined by $\gamma_{\chi}^{(A)}:z \mapsto z+1$ and 
$\gamma_{\chi}^{(B)}:z \mapsto z+\tau$. 

In either case, $\Gamma_{\chi}$ is uniquely determined and is referred to as 
the {\em universal $\chi$-covering transformation group}. Also, 
$(\gamma_{\chi}^{(A)},\gamma_{\chi}^{(B)})$ is said to be the {\em standard 
$\chi$-pair\/} in $\Gamma_{\chi}$. We call $\Pi_{\chi}$ a {\em holomorphic 
universal $\chi$-covering map}. Unless $S$ is a torus, it is uniquely decided. 

Let $\chi'$ be another $1$-handle mark of $S$. Set $A'=\chi'_{*}A_{\sqrt{-1}}$ 
and $B'=\chi'_{*}B_{\sqrt{-1}}$, and let $\gamma^{(A')}$ and $\gamma^{(B')}$ be 
covering transformations in $\Gamma_{\chi}$ corresponding to $A'$ and $B'$. 
This means, in the case where $S$ is not a torus, that $\gamma^{(A')}$ and 
$\gamma^{(B')}$ are hyperbolic transformations and that their axes have one 
point in common and are projected to the closed hyperbolic geodesics 
$A_{\chi'}$ and $B_{\chi'}$ freely homotopic to $A'$ and $B'$. Let $\delta$ be 
the unique element in $\Aut(\mathbb{H})$ that maps $0$, $\infty$ and $1$ to the 
repelling fixed point of $\gamma^{(A')}$, the attracting fixed point of 
$\gamma^{(A')}$ and the repelling fixed point of $\gamma^{(B')}$, respectively. 
Then $\Pi_{\chi} \circ \delta$ is the holomorphic universal $\chi'$-covering 
map. If $S$ is a torus, then $\gamma^{(A')}$ is of the form 
$\gamma_{A'}(z)=z+a'$ with $a' \neq 0$. Taking $\delta \in \Aut(\mathbb{C})$ 
defined by $\delta(z)=a'z$, we obtain a holomorphic universal $\chi'$-covering 
map $\Pi_{\chi} \circ \delta:\mathbb{C} \to S$. In either case, 
$\delta^{-1} \circ \Gamma_{\chi} \circ \delta$ is the universal 
$\chi'$-covering transformation group and the standard $\chi'$-pair is 
$(\delta^{-1} \circ \gamma^{(A')} \circ \delta,
\delta^{-1} \circ \gamma^{(B')} \circ \delta)$. 

Now, let $S_{1}$ and $S_{2}$ be Riemann surfaces of positive genera, and let 
$f \in \TEmb^{+}(S_{1},S_{2})$. If $\chi_{1}$ is a $1$-handle mark of $S_{1}$, 
then $\chi_{2}:=f \circ \chi_{1}$ is a $1$-handle mark of $S_{2}$, and $f$ is 
lifted to a locally homeomorphic mapping $\tilde{f}$ of $\mathbb{U}_{S_{1}}$ 
into $\mathbb{U}_{S_{2}}$ such that 
$$
\Pi_{\chi_{2}} \circ \tilde{f}=f \circ \Pi_{\chi_{1}}.
$$
Note that $\tilde{f}$ is not necessarily injective. There is a group 
homomorphism $\rho:\Gamma_{\chi_{1}} \to \Gamma_{\chi_{2}}$ such that 
$$
\tilde{f} \circ \gamma=\rho(\gamma) \circ \tilde{f}, \quad
\gamma \in \Gamma_{\chi_{1}}.
$$
We can choose $\tilde{f}$ so that if 
$(\gamma_{\chi}^{(A)},\gamma_{\chi}^{(B)})$ is the standard $\chi_{1}$-pair in 
$\Gamma_{\chi_{1}}$, then 
$(\rho(\gamma_{\chi}^{(A)}),\rho(\gamma_{\chi}^{(B)}))$ is the standard 
$\chi_{2}$-pair in $\Gamma_{\chi_{2}}$. If $S_{2}$ is not a torus, then nor is 
$S_{1}$, and those requirements determine $\tilde{f}$ and $\rho$ uniquely. If 
$S_{2}$ is a torus but $S_{1}$ is not, then we adjust $\Pi_{\chi_{2}}$ in 
addition so that $\tilde{f}(i)=0$, where $i=\sqrt{-1\,}$. Again, $\tilde{f}$ 
and $\rho$ are uniquely determined. In the case where $S_{1}$ is a torus, so is 
$S_{2}$ and $f$ is a member of $\Homeo^{+}(S_{1},S_{2})$. Moreover, $\rho$ is 
uniquely decided and we can choose $\Pi_{\chi_{j}}$ and $\tilde{f}$ so that 
$\tilde{f}(0)=0$ though we cannot assert the uniqueness of $\tilde{f}$. In any 
case we say that $\tilde{f}$ and $\rho$ are a {\em $\chi_{1}$-lift\/} of $f$ 
and the {\em $\chi_{1}$-homomorphism} induced by $f$, respectively, where we 
should call $\tilde{f}$ {\em the\/} $\chi_{1}$-lift of $f$ unless $S_{1}$ is a 
torus. Note that if $f \in \QCHomeo(S_{1},S_{2})$, then $\tilde{f}$ is extended 
to a quasiconformal homeomorphism of $\hat{\mathbb{C}}$ onto itself fixing $0$, 
$1$ and $\infty$. 

Let $f_{0},f_{1} \in \TEmb^{+}(S_{1},S_{2})$. Then $f_{0} \simeq f_{1}$ if and 
only if $f_{0}$ and $f_{1}$ induce the same $\chi_{1}$-homomorphism. 

Next, we define $1$-handle marks of marked Riemann surfaces of genus $g$. Let 
${\bm S} \in \mathfrak{F}_{g}$. Consider all pairs $(\chi,\eta)$, where 
$(S,\eta)$ represents $\bm S$ and $\chi$ is a $1$-handle mark of $S$. Two such 
pairs $(\chi_{1},\eta_{1})$ and $(\chi_{2},\eta_{2})$ are defined to be 
equivalent to each other if $\kappa \circ \chi_{1} \simeq \chi_{2}$ for some 
$\kappa \in \CHomeo(S_{1},S_{2})$ with $\kappa \circ \eta_{1} \simeq \eta_{2}$, 
where $(S_{j},\eta_{j}) \in {\bm S}$, $j=1,2$. Each equivalence class 
${\bm \chi}=[\chi,\eta]$ is called a {\em $1$-handle mark\/} of $\bm S$. All 
representatives $(S,\eta) \in {\bm S}$ and $(\chi,\eta) \in {\bm \chi}$ have 
the universal $\chi$-covering transformation group and the standard $\chi$-pair 
in common. Thus we can speak of the {\em universal $\bm \chi$-covering 
transformation group\/} $\Gamma_{\bm \chi}$ and the {\em standard 
$\bm \chi$-pair\/} in $\Gamma_{\bm \chi}$. 

Let ${\bm S}_{j}=[S_{j},\eta_{j}] \in \mathfrak{F}_{g}$ for $j=1,2$, and take 
${\bm f}=[f,\eta_{1},\eta_{2}] \in \TEmb^{+}({\bm S}_{1},{\bm S}_{2})$. If 
${\bm \chi}_{1}=[\chi_{1},\eta_{1}]$ is a $1$-handle mark of ${\bm S}_{1}$, 
then we denote by ${\bm f} \circ {\bm \chi}_{1}$ the $1$-handle mark of 
${\bm S}_{2}$ represented by $(f \circ \chi_{1},\eta_{2})$. If ${\bm S}_{1}$ is 
not a marked torus, then the $\chi_{1}$-lift $\tilde{f}$ of $f$ and the 
$\chi_{1}$-homomorphism $\rho$ induced by $f$ are independent of a particular 
choice of representatives $(S_{j},\eta_{j})$, $(\chi_{1},\eta_{1})$ and 
$(f,\eta_{1},\eta_{2})$. We call $\tilde{f}$ and $\rho$ the 
{\em $\bm \chi_{1}$-lift\/} of $\bm f$ and the 
{\em ${\bm \chi}_{1}$-homomorphism\/} induced by $\bm f$, respectively. If 
${\bm S}_{1}$ is a marked torus, then the ${\bm \chi}_{1}$-homomorphism induced 
by $\bm f$ is well-defined while there are infinitely many $\chi_{1}$-lifts of 
$f$. Any one of them is referred to as a ${\bm \chi}_{1}$-lift of $\bm f$. 

Let $\bm S$ be a marked Riemann surface of genus $g$ without border. Then 
$\mathbb{U}_{S}$, $(S,\eta) \in {\bm S}$, are identical with one another and it 
is no harm to denote it by $\mathbb{U}_{\bm S}$. Thus 
$\mathbb{U}_{\bm S}=\mathbb{C}$ if $\bm S$ is a marked torus while 
$\mathbb{U}_{\bm S}=\mathbb{H}$ otherwise. 

\begin{defn}[convergence of sequence of homeomorphisms]
\label{defn:convergence_of_sequence_of_homeomorphisms}
Let $\bm S$, ${\bm S}_{n}$, $n \in \mathbb{N}$, be marked Riemann surfaces of 
genus $g$ without border homeomorphic to one another, and let 
${\bm h}_{n} \in \Homeo^{+}({\bm S},{\bm S}_{n})$. In the case where $\bm S$ is 
not a marked torus, the sequence $\{{\bm h}_{n}\}$ is said to {\em converge to 
${\bm 1}_{\bm S}$} if for some $1$-handle mark $\bm \chi$ of $\bm S$ the 
sequence of $\bm \chi$-lifts of ${\bm h}_{n}$ converges to $\id_{\mathbb{H}}$ 
locally uniformly on $\mathbb{H}$. If $\bm S$ is a marked torus, then for 
$\{{\bm h}_{n}\}$ to converge to ${\bm 1}_{\bm S}$ we require {\em some\/} 
sequence of $\bm \chi$-lifts of ${\bm h}_{n}$ to converge to $\id_{\mathbb{C}}$ 
locally uniformly on $\mathbb{C}$. 
\end{defn}

We claim that the above definition does not depend on $\bm \chi$. Set 
${\bm \chi}_{n}={\bm h}_{n} \circ {\bm \chi}$, and let 
$\tilde{h}_{n}:\mathbb{U} \to \mathbb{U}$ be a $\bm \chi$-lift of 
${\bm h}_{n}$, where $\mathbb{U}=\mathbb{U}_{\bm S}=\mathbb{U}_{{\bm S}_{n}}$. 
Suppose that $\{\tilde{h}_{n}\}$ converges to $\id_{\mathbb{U}}$ locally 
uniformly on $\mathbb{U}$. As 
$\tilde{h}_{n} \in \Homeo^{+}(\mathbb{U},\mathbb{U})$, the 
$\bm \chi$-homomorphism 
$\rho_{n}:\Gamma_{\bm \chi} \to \Gamma_{{\bm \chi}_{n}}$ induced by 
${\bm h}_{n}$ is given by 
$\rho_{n}(\gamma)=\tilde{h}_{n} \circ \gamma \circ \tilde{h}_{n}^{-1}$, 
$\gamma \in \Gamma_{\bm \chi}$. Since $\{\tilde{h}_{n}^{-1}\}$ converges to 
$\id_{\mathbb{U}}$ locally uniformly on $\mathbb{U}$ (see 
Lemma~\ref{lem:Polish_space} below), it follows that 
$\rho_{n}(\gamma) \to \gamma$ as $n \to \infty$ for 
$\gamma \in \Gamma_{\bm \chi}$. Now, let ${\bm \chi}'=[\chi',\eta]$ be another 
$1$-handle mark of ${\bm S}=[S,\eta]$, and set 
${\bm \chi}'_{n}={\bm h}_{n} \circ {\bm \chi}'$. Take 
$\gamma^{(A')},\gamma^{(B')} \in \Gamma_{\bm \chi}$ corresponding to 
$A':=\chi'_{*}A_{\sqrt{-1}}$ and $B':=\chi'_{*}B_{\sqrt{-1}}$. Since 
$\rho_{n}(\gamma^{(A')}) \to \gamma^{(A')}$ and 
$\rho_{n}(\gamma^{(B')}) \to \gamma^{(B')}$ as $n \to \infty$, we can choose 
$\delta,\delta_{n} \in \Aut(\mathbb{U})$ with $\Gamma_{{\bm \chi}'}=
\delta^{-1} \circ \Gamma_{\bm \chi} \circ \delta$ and $\Gamma_{{\bm \chi}'_{n}}=
\delta_{n}^{-1} \circ \Gamma_{{\bm \chi}_{n}} \circ \delta_{n}$ so that 
$\delta_{n} \to \delta$. Consequently, the sequence 
$\{\delta_{n}^{-1} \circ \tilde{h}_{n} \circ \delta\}$ tends to 
$\id_{\mathbb{U}}$ locally uniformly on $\mathbb{U}$. Since 
$\delta_{n}^{-1} \circ \tilde{h}_{n} \circ \delta$ is a ${\bm \chi}'$-lift of 
${\bm h}_{n}$, our claim has been established. 

The following lemma, which is well-known among those who are familiar with 
descriptive set theory, follows from Arens \cite[Theorem~3]{Arens1946}. We 
include a direct proof for the sake of convenience. 

\begin{lem}
\label{lem:Polish_space}
Let $X$ and $Y$ be locally compact metric spaces homeomorphic to each other. If 
a sequence $\{h_{n}\}$ in $\Homeo(X,Y)$ converges to $h \in \Homeo(X,Y)$ 
uniformly on every compact subset of $X$, then $\{h_{n}^{-1}\}$ converges to 
$h^{-1}$ uniformly on every compact subset of $Y$. 
\end{lem}

\begin{proof}
Let $L$ be an arbitrary compact subset of $Y$, and let $\varepsilon>0$. Take a 
compact subset $U$ of $X$ including the $\varepsilon$-neighborhood of 
$K:=h^{-1}(L)$. Since $h^{-1}$ is uniformly continuous on $V:=h(U)$, there is 
$\delta>0$ such that
$$
d_{X}(h^{-1}(y_{1}),h^{-1}(y_{2}))<\varepsilon
$$
whenever $y_{1},y_{2} \in V$ and $d_{Y}(y_{1},y_{2})<\delta$, or equivalently, 
$$
d_{Y}(h(x_{1}),h(x_{2})) \geqq \delta
$$
whenever $x_{1},x_{2} \in U$ and $d_{X}(x_{1},x_{2}) \geqq \varepsilon$, where 
$d_{X}$ and $d_{Y}$ denote the distance functions on $X$ and $Y$, respectively. 
Replacing $\delta$ with a smaller one if necessary, we may assume that the 
$\delta$-neighborhood of $L$ is included in $V$. 

Since $\{h_{n}\}$ converges to $h$ uniformly on $U$, we can find 
$N \in \mathbb{N}$ for which 
$$
d_{Y}(h_{n}(x),h(x))<\frac{\delta}{\,3\,}
$$
for $n>N$ and $x \in U$. Note that if $x_{1},x_{2} \in U$ and 
$d_{X}(x_{1},x_{2}) \geqq \varepsilon$, then 
\begin{align*}
d_{Y}(h_{n}(x_{1}),h_{n}(x_{2})) & \geqq d_{Y}(h(x_{1}),h(x_{2}))-
d_{Y}(h_{n}(x_{1}),h(x_{1}))-d_{Y}(h_{n}(x_{2}),h(x_{2})) \\
 & \geqq \delta-\frac{\delta}{\,3\,}-\frac{\delta}{\,3\,}=\frac{\delta}{3}
\end{align*}
for $n>N$. In other words, 
$$
d_{X}(h_{n}^{-1}(y_{1}),h_{n}^{-1}(y_{2}))<\varepsilon
$$
for $n>N$ if $y_{1},y_{2} \in V$ and $d_{Y}(y_{1},y_{2})<\delta/3$. 

We now claim that 
$$
d_{X}(h_{n}^{-1}(y),h^{-1}(y))<\varepsilon
$$
for $n>N$ and $y \in L$, which will prove the lemma. To verify the claim take 
$y \in L$ and $n>N$ arbitrarily, and set $x=h^{-1}(y)$ and $y_{n}=h_{n}(x)$. 
Then $x$ is a point of $K \subset U$ so that $y_{n}$ belongs to the 
$(\delta/3)$-neighborhood of $L$ and hence to $V$, for, 
$d_{Y}(y,y_{n})=d_{Y}(h(x),h_{n}(x))<\delta/3$. We then have 
$$
d_{X}(h_{n}^{-1}(y),h^{-1}(y))=d_{X}(h_{n}^{-1}(y),h_{n}^{-1}(y_{n}))<
\varepsilon,
$$
as claimed. 
\end{proof}

\begin{defn}[convergence of sequence of topological embeddings]
\label{defn:convergence_of_sequence_of_homeomorphi_injections}
Let $\bm S$ and ${\bm S}'$ be marked Riemann surfaces of genus $g$ without 
border, and let ${\bm f} \in \TEmb^{+}({\bm S},{\bm S}')$. Also, let 
$\{{\bm S}_{n}\}$ and $\{{\bm S}'_{n}\}$ be sequences in $\mathfrak{F}_{g}$, 
where ${\bm S}_{n}$ and ${\bm S}'_{n}$ are supposed to be homeomorphic to 
$\bm S$ and ${\bm S}'$, respectively, and let 
${\bm f}_{n} \in \TEmb^{+}({\bm S}_{n},{\bm S}'_{n})$. Then we say that 
$\{{\bm f}_{n}\}$ converges to $\bm f$ if there are 
${\bm h}_{n} \in \Homeo^{+}({\bm S},{\bm S}_{n})$ and 
${\bm h}'_{n} \in \Homeo^{+}({\bm S}',{\bm S}'_{n})$ together with a $1$-handle 
mark $\bm \chi$ of $\bm S$ such that 
\begin{list}{{(\roman{claim})}}{\usecounter{claim}
\setlength{\topsep}{0pt}
\setlength{\itemsep}{0pt}
\setlength{\parsep}{0pt}
\setlength{\labelwidth}{\leftmargin}}
\item $\{{\bm h}_{n}\}$ and $\{{\bm h}'_{n}\}$ converge to ${\bm 1}_{\bm S}$ and 
${\bm 1}_{{\bm S}'}$, respectively, and 

\item the sequence of ${\bm h}_{n} \circ {\bm \chi}$-lifts of ${\bm f}_{n}$ 
converges to the $\bm \chi$-lift of $\bm f$ locally uniformly on $\mathbb{H}$, 
\end{list}
provided that $\bm S$ is not a marked torus. In the case where $\bm S$ is a 
marked torus, we replace condition~(ii) with 
\begin{list}{{(\roman{claim}$'$)}}{\usecounter{claim}
\setlength{\topsep}{0pt}
\setlength{\itemsep}{0pt}
\setlength{\parsep}{0pt}
\setlength{\labelwidth}{\leftmargin}}
\stepcounter{claim}
\item some sequence of ${\bm h}_{n} \circ {\bm \chi}$-lifts of ${\bm f}_{n}$ 
converges to a $\bm \chi$-lift of $\bm f$ locally uniformly on $\mathbb{C}$. 
\end{list}
\end{defn}

\begin{lem}
\label{lem:compact_continuation:sequence}
Let $\bm S$ and ${\bm S}'$ be marked Riemann surfaces of genus $g$ without 
border. Let $\{{\bm S}_{n}\}$ be a sequence in $\mathfrak{F}_{g}$ and let 
$({\bm S}'_{n},{\bm \iota}_{n})$ be a continuation of ${\bm S}_{n}$. Assume 
that there are ${\bm h}_{n} \in \Homeo^{+}({\bm S},{\bm S}_{n})$ and 
${\bm h}'_{n} \in \Homeo^{+}({\bm S}',{\bm S}'_{n})$ such that 
$({\bm h}'_{n})^{-1} \circ {\bm \iota}_{n} \circ {\bm h}_{n}$, 
$n \in \mathbb{N}$, are mutually homotopic. If $\{{\bm h}_{n}\}$ and 
$\{{\bm h}'_{n}\}$ converge to ${\bm 1}_{\bm S}$ and ${\bm 1}_{{\bm S}'}$, 
respectively, then there is a continuation $({\bm S}',{\bm \iota})$ of $\bm S$ 
such that a subsequence of $\{{\bm \iota}_{n}\}$ converges to ${\bm \iota}$. 
\end{lem}

\begin{proof}
Let $\bm \chi$ be a $1$-handle mark of $\bm S$, and set 
${\bm \chi}_{n}={\bm h}_{n} \circ {\bm \chi}$, 
${\bm \chi}'_{n}={\bm \iota}_{n} \circ {\bm \chi}_{n}$ and 
${\bm \chi}'=({\bm h}'_{1})^{-1} \circ {\bm \chi}'_{1}$. Let $\tilde{h}_{n}$ 
and $\tilde{h}'_{n}$ be $\bm \chi$- and ${\bm \chi}'$-lifts of ${\bm h}_{n}$ 
and ${\bm h}'_{n}$, respectively. By assumption we can choose 
$\{\tilde{h}_{n}\}$ and $\{\tilde{h}'_{n}\}$ so that they converge locally 
uniformly to $\id_{\mathbb{U}_{\bm S}}$ and $\id_{\mathbb{U}_{{\bm S}'}}$, 
respectively. 

If $\bm S$ is closed, then so are ${\bm S}'$, ${\bm S}_{n}$ and ${\bm S}'_{n}$, 
and hence 
${\bm \iota}_{n} \in \CHomeo_{\mathrm{hc}}({\bm S}_{n},{\bm S}'_{n})$. Thus 
${\bm S}_{n}$ coincides with ${\bm S}'_{n}$, and $\id_{\mathbb{U}_{\bm S}}$ is 
a ${\bm \chi}_{n}$-lift of ${\bm \iota}_{n}$. The convergence properties of 
$\{\tilde{h}_{n}\}$ and $\{\tilde{h}'_{n}\}$ imply that $\{{\bm S}_{n}\}$ and 
$\{{\bm S}'_{n}\}$ converge to $\bm S$ and ${\bm S}'$ in $\mathfrak{T}_{g}$, 
respectively. In particular, $\bm S$ is identical with ${\bm S}'$, and 
$\{{\bm \iota}_{n}\}$ converges to ${\bm 1}_{\bm S}$. 

Suppose now that $\bm S$ is open. Then $\mathbb{U}_{\bm S}=\mathbb{H}$. Set 
$\mathbb{U}=\mathbb{U}_{{\bm S}'}$ for the sake of simplicity. If $g=1$ and 
${\bm S}' \in \mathfrak{T}_{1}$, then $\mathbb{U}=\mathbb{C}$. Otherwise, 
$\mathbb{U}=\mathbb{H}$. 

Take $(S,\eta) \in {\bm S}$ and $(S_{n},\eta_{n}) \in {\bm S}_{n}$, and let 
$(h_{n},\eta,\eta_{n}) \in {\bm h}_{n}$, 
$(\chi,\eta) \in {\bm \chi}$ and $(\chi_{n},\eta_{n}) \in {\bm \chi}_{n}$. Then 
we have 
$$
\Pi_{\chi_{n}} \circ \tilde{h}_{n}=h_{n} \circ \Pi_{\chi},
$$
and the $\chi$-homomorphism of $\Gamma_{\chi}$ onto $\Gamma_{\chi_{n}}$ induced 
by $h_{n}$ is given by 
$$
\gamma \mapsto \tilde{h}_{n} \circ \gamma \circ \tilde{h}_{n}^{-1}, \quad
\gamma \in \Gamma_{\chi}.
$$
Since $\{\tilde{h}_{n}\}$ converges to $\id_{\mathbb{H}}$ locally uniformly in 
$\mathbb{H}$, so does $\{\tilde{h}_{n}^{-1}\}$ by Lemma~\ref{lem:Polish_space}. 

Similarly, take $(S',\eta') \in {\bm S}'$ and 
$(S'_{n},\eta'_{n}) \in {\bm S}'_{n}$, and let 
$(h'_{n},\eta',\eta'_{n}) \in {\bm h}'_{n}$, $(\chi',\eta') \in {\bm \chi}'$  
and $(\chi'_{n},\eta'_{n}) \in {\bm \chi}'_{n}$. 
If $S'$ is not a torus, then $\Pi_{\chi'}$ and $\Pi_{\chi'_{n}}$ are uniquely 
determined and satisfy 
\begin{equation}
\label{eq:lift_condition}
\Pi_{\chi'_{n}} \circ \tilde{h}'_{n}=h'_{n} \circ \Pi_{\chi'}.
\end{equation}
If $S'$ is a torus, then we can adjust $\Pi_{\chi'}$, $\Pi_{\chi_{n}}$ and 
$\tilde{h}'_{n}$ to obtain~\eqref{eq:lift_condition} for all $n$. The 
$\chi'_{n}$-homomorphism of $\Gamma_{\chi'}$ onto $\Gamma_{\chi'_{n}}$ induced 
by $h'_{n}$ is given by 
$$
\gamma \mapsto \tilde{h}'_{n} \circ \gamma \circ (\tilde{h}'_{n})^{-1}, \quad
\gamma \in \Gamma_{\chi'}.
$$
By Lemma~\ref{lem:Polish_space} again, the sequences $\{\tilde{h}'_{n}\}$ and 
$\{(\tilde{h}'_{n})^{-1}\}$ converge to $\id_{\mathbb{U}}$ locally uniformly on 
$\mathbb{U}$. In the case where $\mathbb{U}=\mathbb{C}$, we have 
$\Gamma_{\chi'}=\Gamma_{\tau'}$ and $\Gamma_{\chi'_{n}}=\Gamma_{\tau'_{n}}$ for 
some $\tau',\tau'_{n} \in \mathbb{H}$ (for the notations see 
Example~\ref{exmp:conformal_automorphism:torus}). Note that 
$\tau'_{n} \to \tau'$ as $n \to \infty$. 

Let $\tilde{\iota}_{n}$ and $\rho_{n}$ be the $\chi_{n}$-lift of $\iota_{n}$ 
and the $\chi_{n}$-homomorphism induced by $\iota_{n}$, respectively, where 
$(\iota_{n},\eta_{n},\eta'_{n}) \in {\bm \iota}_{n}$. Then we have 
$$
\Pi_{\chi'_{n}} \circ \tilde{\iota}_{n}=\iota_{n} \circ \Pi_{\chi_{n}}
$$
and 
$$
\tilde{\iota}_{n} \circ \gamma=\rho_{n}(\gamma) \circ \tilde{\iota}_{n}, \quad
\gamma \in \Gamma_{\chi_{n}}.
$$
Observe that 
$(\tilde{h}'_{n})^{-1} \circ \tilde{\iota}_{n} \circ \tilde{h}_{n}$ is the 
$\chi$-lift of $(h'_{n})^{-1} \circ \iota_{n} \circ  h_{n}$. Since 
$(h'_{n})^{-1} \circ \iota_{n} \circ h_{n} \in \TEmb^{+}(S,S')$, 
$n \in \mathbb{N}$, are homotopic to one another, they induce the same 
$\chi$-homomorphism $\rho:\Gamma_{\chi} \to \Gamma_{\chi'}'$. It follows that 
\begin{equation}
\label{eq:homomorphism_iota_n}
\tilde{\iota}_{n} \circ
\bigl(\tilde{h}_{n} \circ \gamma \circ \tilde{h}_{n}^{-1}\bigr)=
\bigl(\tilde{h}'_{n} \circ \rho(\gamma) \circ (\tilde{h}'_{n})^{-1}\bigr) \circ
\tilde{\iota}_{n}, \quad \gamma \in \Gamma_{\chi},
\end{equation}
for all $n$. 
$$
\begin{CD}
\mathbb{H} @>\tilde{h}_{n}>> \mathbb{H} @>\tilde{\iota}_{n}>> \mathbb{U}
@<\tilde{h}'_{n}<< \mathbb{U} \\
@V{\Pi_{\chi}} VV @V{\Pi_{\chi_{n}}}VV @V{\Pi_{\chi'_{n}}}VV @V{\Pi_{\chi'}}VV
\\
S @>h_{n}>> S_{n} @>\iota_{n}>> S'_{n} @<h'_{n}<< S'
\end{CD}
$$

We claim that $\{\tilde{\iota}_{n}\}$ has a subsequence converging locally 
uniformly on $\mathbb{H}$ to $\tilde{\iota}$, which is a holomorphic function 
or the constant $\infty$. This is clear if $\mathbb{U}=\mathbb{H}$. If 
$\mathbb{U}=\mathbb{C}$, then there is a constant $\tilde{c}_{n}$ such that 
$\tilde{\iota}_{n}-\tilde{c}_{n}$ omits $0$, for, $\iota_{n}$ is not 
surjective. Then it also omits $1$, which is in the $\Gamma_{\chi'_{n}}$-orbit 
of $0$. Therefore, $\{\tilde{\iota}_{n}-\tilde{c}_{n}\}$ forms a normal family. 
Since $\{\tilde{c}_{n}\}$ can be chosen to be bounded, we conclude that 
$\{\tilde{\iota}_{n}\}$ includes a desired subsequence. 

For typographical reason we assume that $\{\tilde{\iota}_{n}\}$ converges to 
$\tilde{\iota}$ locally uniformly on $\mathbb{H}$. It then follows 
from~\eqref{eq:homomorphism_iota_n} that 
\begin{equation}
\label{eq:homomorphism_iota}
\tilde{\iota} \circ \gamma=\rho(\gamma) \circ \tilde{\iota}, \quad
\gamma \in \Gamma_{\chi}.
\end{equation}
If $\tilde{\iota}$ were a constant function, then the constant 
$w_{0} \in \hat{\mathbb{C}}$ should satisfy 
$$
\rho(\gamma)(w_{0})=\rho(\gamma) \circ \tilde{\iota}(z)=
\tilde{\iota} \circ \gamma(z)=w_{0}
$$
for $z \in \mathbb{H}$ and $\gamma \in \Gamma_{\chi}$. Thus $w_{0}$ is a common 
fixed point of elements of $\rho(\Gamma_{\chi})$, which implies that 
$\rho(\Gamma_{\chi})$ is abelian (see, for example, Beardon 
\cite[Theorem~5.1.2]{Beardon1983}). This is impossible if 
$\mathbb{U}=\mathbb{H}$. If $\mathbb{U}=\mathbb{C}$, then 
$w_{0}=0$ as $\tilde{\iota}_{n}(i)=0$ for all $n$. However, no finite point is 
fixed by any nontrivial elements of $\Gamma_{\chi'}$ and we again reach a 
contradiction. Therefore, $\tilde{\iota}$ is nonconstant and holomorphic, and 
hence is locally univalent as well as $\tilde{\iota}_{n}$. 

It follows from~\eqref{eq:homomorphism_iota} that $\tilde{\iota}$ induces a 
holomorphic and locally homeomorphic mapping $\iota:S \to S'$ satisfying 
$\Pi_{\chi'} \circ \tilde{\iota}=\iota \circ \Pi_{\chi}$. We claim that it is 
injective. Otherwise, we could find two distinct points $p_{1},p_{2} \in S$ 
mapped to the same point $q_{0}$ of $S'$ by $\iota$. Let 
$z_{j} \in \Pi_{\chi}^{-1}(p_{j})$, $j=1,2$, and set 
$w_{j}=\tilde{\iota}(z_{j})$. Since 
$$
\Pi_{\chi'}(w_{j})=\Pi_{\chi'} \circ \tilde{\iota}(z_{j})=
\iota \circ \Pi_{\chi}(z_{j})=\iota(p_{j})=q_{0},
$$
we have $w_{2}=\gamma'(w_{1})$ for some $\gamma' \in \Gamma_{\chi'}$. Take a 
relatively compact neighborhood $U_{j}$ of $z_{j}$ such that 
$\Pi_{\chi}(U_{1}) \cap \Pi_{\chi}(U_{2})=\varnothing$ and 
$\gamma' \circ \tilde{\iota}(U_{1})=\tilde{\iota}(U_{2})$. Since 
$\{\tilde{\iota}_{n}\}$ (resp.\ $\{\tilde{h}_{n}\}$, $\{\tilde{h}'_{n}\}$) 
converges to $\tilde{\iota}$ (resp.\ $\id_{\mathbb{H}}$, $\id_{\mathbb{U}}$) 
locally uniformly on $\mathbb{H}$ (resp.\ $\mathbb{H}$, $\mathbb{U}$), there 
are neighborhoods $V_{j}$ of $z_{j}$ with $\bar{V}_{j} \subset U_{j}$ and 
elements $\gamma'_{n}$ in $\Gamma_{\chi'_{n}}$ such that 
$V_{j} \subset \tilde{h}_{n}(U_{j})$ and 
$\gamma'_{n} \circ \tilde{\iota}_{n}(V_{1}) \cap \tilde{\iota}_{n}(V_{2}) \neq
\varnothing$ for sufficiently large $n$. Choose $z_{jn} \in V_{j}$ so that 
$\gamma'_{n} \circ \tilde{\iota}_{n}(z_{1n})=\tilde{\iota}_{n}(z_{2n})$. Then 
we obtain 
$$
\iota_{n} \circ \Pi_{\chi_{n}}(z_{1n})=
\Pi_{\chi'_{n}} \circ \tilde{\iota}_{n}(z_{1n})=
\Pi_{\chi'_{n}} \circ \gamma'_{n}\circ \tilde{\iota}_{n}(z_{1n})=
\Pi_{\chi'_{n}} \circ \tilde{\iota}_{n}(z_{2n})=
\iota_{n} \circ \Pi_{\chi_{n}}(z_{2n}),
$$
which implies that $\Pi_{\chi_{n}}(z_{1n})=\Pi_{\chi_{n}}(z_{2n})$, or, 
$z_{2n}=\tilde{h}_{n} \circ \gamma_{n} \circ \tilde{h}_{n}^{-1}(z_{1n})$ for 
some $\gamma_{n} \in \Gamma_{\chi}$, for, $\iota_{n}$ is injective. Then 
$\gamma_{n} \circ \tilde{h}_{n}^{-1}(z_{1n})=\tilde{h}_{n}^{-1}(z_{2n}) \in
\gamma_{n}(U_{1}) \cap U_{2}=\varnothing$, which is absurd. This proves that 
$\iota$ is injective, as claimed. If we set ${\bm \iota}=[\iota,\eta,\eta']$, 
then $({\bm S}',{\bm \iota})$ is a continuation of $\bm S$ and 
$\{{\bm \iota}_{n}\}$ converges to ${\bm \iota}$. 
\end{proof}

\begin{prop}
\label{prop:closed_continuation:sequence}
If ${\bm R}_{0}$ is a marked open Riemann surface of genus $g$, then any 
sequence $\{({\bm R}_{n},{\bm \iota}_{n})\}$ of closed continuations of 
${\bm R}_{0}$ contains a subsequence 
$\{({\bm R}_{n_{k}},{\bm \iota}_{n_{k}})\}$ such that $\{{\bm R}_{n_{k}}\}$ and 
$\{{\bm \iota}_{n_{k}}\}$ converge to some ${\bm R} \in \mathfrak{T}_{g}$ and 
${\bm \iota} \in \CEmb_{\mathrm{hc}}({\bm R}_{0},{\bm R})$, respectively. 
\end{prop}

\begin{proof}
Take representatives $(R_{n},\theta_{n}) \in {\bm R}_{n}$, $n \geqq 0$, and 
$(\iota_{n},\theta_{0},\theta_{n}) \in {\bm \iota}_{n}$, $n>0$. Note that 
$\iota_{n} \circ \theta_{0} \simeq \theta_{n}$. If $g>1$, then $R_{n}$, 
$n \geqq 0$, carry hyperbolic metrics. Since $\iota_{n}$ decreases the 
hyperbolic metrics, for any simple loop $c$ on $\dot{\Sigma}_{g}$ the 
hyperbolic length of $(\theta_{n})_{*}c$ does not exceed that of 
$(\theta_{0})_{*}c$. With the aid of the Fenchel-Nielsen coordinates on 
$\mathfrak{T}_{g}$ (see, for example, Abikoff 
\cite[Chapter~II~\S3]{Abikoff1980}) we deduce that $\{{\bm R}_{n}\}$ is 
relatively compact in $\mathfrak{T}_{g}$. If $g=1$, then choosing a point 
$p_{n} \in R_{n} \setminus \iota_{n}(R_{0})$ for $n>0$ and making use of 
hyperbolic metrics on $R_{0}$ and $R_{n} \setminus \{p_{n}\}$, we reach the 
same conclusion. In either case $\{{\bm R}_{n}\}$ includes a convergent 
subsequence. We may assume that $\{{\bm R}_{n}\}$ itself converges to some 
$\bm R$ in $\mathfrak{T}_{g}$. If ${\bm h}_{n}$ denotes the Teichm\"{u}ller 
quasiconformal homeomorphism of $\bm R$ onto ${\bm R}_{n}$, then 
$\{{\bm h}_{n}\}$ converges to ${\bm 1}_{\bm R}$. We apply 
Lemma~\ref{lem:compact_continuation:sequence} to obtain desired subsequences 
$\{{\bm R}_{n_{k}}\}$ and $\{{\bm \iota}_{n_{k}}\}$ and a continuation 
$({\bm R},{\bm \iota})$ of ${\bm R}_{0}$. 
\end{proof}

\begin{cor}[Oikawa \cite{Oikawa1957}]
\label{cor:Oikawa:compact}
If ${\bm R}_{0}$ is a marked open Riemann surface of genus $g$, then 
$\mathfrak{M}({\bm R}_{0})$ is compact. 
\end{cor}

\begin{rem}
Note that ${\bm R}_{0}$ is not assumed to be finite. An alternative proof of 
the corollary is found in \cite[Proposition~5.3]{Masumoto2000}. In 
\cite{Oikawa1957} Oikawa also proved that $\mathfrak{M}({\bm R}_{0})$ is 
connected. Theorem~\ref{thm:main:M(R_0)} gives an alternative proof of this 
fact in the case where ${\bm R}_{0}$ is finite. The connectedness of 
$\mathfrak{M}({\bm R}_{0})$ for general ${\bm R}_{0}$ easily follows from this 
special case, as was shown in \cite{Oikawa1957}. 
\end{rem}

\begin{prop}
\label{prop:self-welding:sequence}
Let ${\bm R}_{0}$ be a marked finite open Riemann surface of genus $g$, and let 
$\{({\bm R}_{n},{\bm \iota}_{n})\}$ be a sequence of closed regular 
self-welding continuations of ${\bm R}_{0}$. If $\{{\bm R}_{n}\}$ converges to 
a point $\bm R$ in\/ $\mathfrak{T}_{g}$, then there is a closed regular 
self-welding continuation $({\bm R},{\bm \iota})$ of ${\bm R}_{0}$ such that a 
subsequence of $\{{\bm \iota}_{n}\}$ converges to $\bm \iota$. 
\end{prop}

\begin{proof}
It follows from Proposition~\ref{prop:closed_continuation:sequence} that a 
subsequence of $\{{\bm \iota}_{n}\}$ converges to some 
${\bm \iota} \in \CEmb_{\mathrm{hc}}({\bm R}_{0},{\bm R})$. We may assume that 
$\{{\bm \iota}_{n}\}$ converges to $\bm \iota$. We need to show that 
$({\bm R},{\bm \iota})$ is a regular self-welding continuation of 
${\bm R}_{0}$. 

We know that $\mathbb{H}=\mathbb{U}_{{\bm R}_{0}}$. Set 
$\mathbb{U}=\mathbb{U}_{{\bm R}}=\mathbb{U}_{{\bm R}_{n}}$ for $n \geqq 1$. Fix 
a $1$-handle mark ${\bm \chi}_{0}$ of ${\bm R}_{0}$, and set 
${\bm \chi}_{n}={\bm \iota}_{n} \circ {\bm \chi}_{0}$. Let 
${\bm \varphi}_{n} \in A_{L}({\bm R}_{0})$ and 
${\bm \psi}_{n} \in A({\bm R}_{n})$ be a welder of the self-welding 
continuation $({\bm R}_{n},{\bm \iota}_{n})$ of ${\bm R}_{0}$ and its 
co-welder, respectively. They induce automorphic $2$-forms 
$\tilde{\varphi}_{n}$ and $\tilde{\psi}_{n}$ for $\Gamma_{{\bm \chi}_{0}}$ and 
$\Gamma_{{\bm \chi}_{n}}$, respectively. We adopt the normalization condition 
$\|{\bm \varphi}_{n}\|_{{\bm R}_{0}}=1$. As 
$\|{\bm \psi}_{n}\|_{{\bm R}_{n}}=\|{\bm \varphi}_{n}\|_{{\bm R}_{0}}=1$, by 
taking subsequences if necessary, we may assume that $\{\tilde{\varphi}_{n}\}$ 
and $\{\tilde{\psi}_{n}\}$ converge locally uniformly to 
$\tilde{\varphi} \in A(\mathbb{H})$ and $\tilde{\psi} \in A(\mathbb{U})$, 
respectively. Proposition~\ref{prop:positive_quadratic_differential:closed} 
implies that $\tilde{\varphi}$ is projected to some 
${\bm \varphi} \in A_{L}({\bm R}_{0})$ with 
$\|{\bm \varphi}\|_{{\bm R}_{0}}=1$. Also, $\tilde{\psi}$ is projected to 
some ${\bm \psi} \in A({\bm R})$ with $\|{\bm \psi}\|_{\bm R}=1$. 

Denote by $\tilde{\iota}_{n}$ and $\tilde{\iota}$ the ${\bm \chi}_{0}$-lifts of 
${\bm \iota}_{n}$ and $\bm \iota$, respectively. We know that 
$\{\tilde{\iota}_{n}\}$ converges to $\tilde{\iota}$ locally uniformly on 
$\mathbb{H}$. Owing to 
$\tilde{\iota}_{n}^{*}\tilde{\psi}_{n}=\tilde{\varphi}_{n}$, we have 
$\tilde{\iota}^{*}\tilde{\psi}=\tilde{\varphi}$, or equivalently, 
${\bm \iota}^{*}{\bm \psi}={\bm \varphi}$. The continuation 
$({\bm R},{\bm \iota})$ of ${\bm R}_{0}$ is dense as 
$\|{\bm \iota}^{*}{\bm \psi}\|_{{\bm R}_{0}}=\|{\bm \varphi}\|_{{\bm R}_{0}}=1=
\|{\bm \psi}\|_{\bm R}$. Proposition~\ref{prop:self-welding_continuation} thus 
implies that $({\bm R},{\bm \iota})$ is a closed $\bm \varphi$-regular 
self-welding continuation of ${\bm R}_{0}$. 
\end{proof}

\begin{rem}
We could apply Bourque~\cite[Lemma~5.4]{Bourque2018} to obtain 
Proposition~\ref{prop:self-welding:sequence}. It should be noted that our proof 
is direct and does not require any results on extremal quasiconformal 
embeddings. 
\end{rem}

\section{Continuations to Riemann surfaces on $\partial\mathfrak{M}({\bm R}_{0})$}
\setcounter{equation}{0}
\label{sec:continuation_to_Riemann_surface_on_boundary}

The aim of the present section is to prove the following proposition. It is 
claimed in Kahn-Pilgrim-Thurston \cite[Proposition~1.7 and Remark~1.5]{KPT2022} 
without proof. 

\begin{prop}
\label{prop:boundary_of_M(R_0)}
Let ${\bm R}_{0}$ be a marked nonanalytically finite open Riemann surface of 
genus $g$. Then for any ${\bm R} \in \partial\mathfrak{M}({\bm R}_{0})$ there 
is $\bm \iota \in \CEmb_{\mathrm{hc}}({\bm R}_{0},{\bm R})$ such that 
$({\bm R},{\bm \iota})$ is a closed regular self-welding continuation of 
${\bm R}_{0}$. 
\end{prop}

Propositions~\ref{prop:boundary:regular_self-welding} 
and~\ref{prop:boundary_of_M(R_0)} prove Theorem~\ref{thm:main:CEmb(R_0,R)}~(i). 
To prove Proposition~\ref{prop:boundary_of_M(R_0)} we employ extremal 
quasiconformal embeddings of finite open Riemann surfaces into closed Riemann 
surfaces. Following Ioffe \cite{Ioffe1975} (see also Kahn-Pilgrim-Thurston 
\cite[Definition~4.1]{KPT2022} and Bourque \cite[Definition~3.4]{Bourque2018}) 
we make the following definition. 

\begin{defn}[Teichm\"{u}ller quasiconformal embedding]
\label{defn:Teichmuller_embedding}
Let $R_{0}$ be a finite open Riemann surface of genus $g$ without border, and 
let $R$ be a closed Riemann surface of the same genus. Then 
$f \in \QCEmb(R_{0},R)$ is called a {\em Teichm\"{u}ller quasiconformal 
embedding\/} if 
\begin{list}{{(\roman{claim})}}{\usecounter{claim}
\setlength{\topsep}{0pt}
\setlength{\itemsep}{0pt}
\setlength{\parsep}{0pt}
\setlength{\labelwidth}{\leftmargin}}
\item its Beltrami differential $\mu_{f}$ is of the form $k|\varphi|/\varphi$ 
for some $\varphi \in A_{+}(R_{0})$ and $k \in [0,1)$, 

\item there is $\psi \in A(R)$ such that $f$ maps every noncritical point of $\varphi$ 
to that of $\psi$ and is represented as 
$$
\omega=K(f)\re\zeta+i\im\zeta=\frac{\,(K(f)+1)\zeta+(K(f)-1)\bar{\zeta}\,}{2}
$$
with respect to some natural parameter $\zeta$ (resp.\ $\omega$) of $\varphi$ 
(resp.\ $\psi$) around any noncritical point $p$ (resp.\ $f(p)$), and 

\item the complement $R \setminus f(R_{0})$ consists of finitely many horizontal 
arcs and, possibly, zeroes of $\psi$ together with finitely many isolated 
points. 
\end{list}
\end{defn}

The quadratic differentials $\varphi$ and $\psi$ are referred to as 
{\em initial\/} and {\em terminal\/} quadratic differentials of $f$, 
respectively. Note that $f$ maps each horizontal arc of $\varphi$ onto a 
horizontal arc of $\psi$ and is a uniform stretching along horizontal 
trajectories of $\varphi$. The Beltrami differential of the inverse 
$f^{-1}:f(R_{0}) \to R_{0}$ is equal to $-k|\psi|/\psi$. If $k=0$, then $f$ is 
conformal, in which case it is said to be a {\em Teichm\"{u}ller conformal 
embedding}. 

\begin{rem}
Teichm\"{u}ller conformal embeddings are called {\em slit mappings\/} in 
Bourque \cite{Bourque2018}, and {\em Teichm\"{u}ller embeddings with dilatation 
$1$} in Kahn-Pilgrim-Thurston \cite{KPT2022}. 
\end{rem}

Let $(R,\iota)$ be a genus-preserving closed continuation of $R_{0}$. If it is 
a closed regular self-welding continuation, then $\iota$ is a Teichm\"{u}ller 
conformal embedding, and a welder of $(R,\iota)$ and its co-welder are initial 
and terminal quadratic differentials of $\iota$, respectively. Conversely, if 
$\iota$ is a Teichm\"{u}ller conformal embedding, then $(R,\iota)$ is a closed 
regular self-welding continuation of $R_{0}$ by 
Corollary~\ref{cor:continuation:self-welding}. More generally, we have the 
following lemma. 

\begin{lem}
\label{lem:Teichmuller_embedding}
Let $R_{0}$ be a nonanalytically finite open Riemann surface of genus $g$, and 
let $R$ be a closed Riemann surface of genus $g$. If $f$ is a Teichm\"{u}ller 
quasiconformal embedding of $R_{0}$ into $R$ with initial quadratic 
differential $\varphi$, then there are a closed $\varphi$-regular self-welding 
continuation $(R',\iota)$ of $R_{0}$ and a Teichm\"{u}ller quasiconformal 
homeomorphism $h$ of $R'$ onto $R$ such that 
\begin{list}{{\rm (\roman{claim})}}{\usecounter{claim}
\setlength{\topsep}{0pt}
\setlength{\itemsep}{0pt}
\setlength{\parsep}{0pt}
\setlength{\labelwidth}{\leftmargin}}
\item the co-welder of $\varphi$ is an initial quadratic differential of $h$, and 

\item $f=h \circ \iota$. 
\end{list}
\end{lem}

\begin{proof}
Let $\psi \in A(R)$ be the terminal quadratic differential of $f$ corresponding 
to $\varphi$. Thus near a noncritical point $p \in R_{0}$ of $\varphi$ the 
mapping $f$ is represented as $\omega=K\re\zeta+i\im\zeta$ with respect to 
natural parameters $\zeta$ and $\omega$ of $\varphi$ and $\psi$ around $p$ and 
$f(p)$, respectively, where $K=K(f)$. Let $h'$ be a Teichm\"{u}ller 
quasiconformal homeomorphism of $R$ onto a closed Riemann surface $R'$ for 
which $\mu_{h'}=\mu_{f^{-1}}=-k|\psi|/\psi$, where $k=(K-1)/(K+1)$, and set 
$\iota=h' \circ f$, which is a conformal embedding of $R_{0}$ into $R'$. If 
$\psi' \in A(R')$ denotes the terminal quadratic differential of $h'$ 
corresponding to $-\psi$ so that $h'$ is expressed as 
$\omega'=K\im\omega-i\re\omega$ with a natural parameter $\omega'$ of $\psi'$ 
around $\iota(p)=h'(f(p))$, then $\iota$ is given by $\omega'=-iK\zeta$ near 
$p$, which implies that $\varphi=\iota^{*}(-\psi'/K^{2})$. Since 
$R \setminus \iota(R_{0})=R \setminus h'(f(R_{0}))$ consists of finitely many 
horizontal arcs of $-\psi'/K^{2}$ together with finitely many points, from 
Corollary~\ref{cor:continuation:self-welding} we infer that $(R',\iota)$ is a 
closed $\varphi$-regular self-welding continuation of $R$. Finally, 
$h:=(h')^{-1}$ is a Teichm\"{u}ller quasiconformal homeomorphism of $R'$ onto 
$R$ with $f=h \circ \iota$, and $-\psi'/K^{2}$ is an initial quadratic 
differential of $h$. 
\end{proof}

\begin{rem}
Bourque \cite[Remark~3.8]{Bourque2018} gives a similar decomposition though the 
order of the factors is opposite. Our factorization fits with the following 
arguments. 

\end{rem}

The following corollary is an immediate consequence of 
Lemma~\ref{lem:Teichmuller_embedding} and 
Corollary~\ref{cor:border_length_condition}. 

\begin{cor}
\label{cor:Teichmuller_embedding:quadratic_differential}
Let $R_{0}$ be a nonanalytically finite open Riemann surface of genus $g$. Then 
initial quadratic differentials of Teichm\"{u}ller quasiconformal embeddings of 
$R_{0}$ into closed Riemann surfaces of genus $g$ belong to $A_{L}(R_{0})$. 
\end{cor}

Let ${\bm R}_{0}$ be a marked finite open Riemann surface, and let 
${\bm R} \in \mathfrak{T}_{g}$. A {\em Teichm\"{u}ller quasiconformal 
embedding\/} of ${\bm R}_{0}$ into $\bm R$ is, by definition, an element 
$\bm f \in \QCEmb_{\mathrm{hc}}({\bm R}_{0},{\bm R})$ for which $f:R_{0} \to R$ 
is a Teichm\"{u}ller quasiconformal embedding for some 
$(R_{0},\theta_{0}) \in {\bm R}_{0}$, $(R,\theta) \in {\bm R}$ and 
$(f,\theta_{0},\theta) \in {\bm f}$. This definition does depend on a 
particular choice of representatives of ${\bm R}_{0}$, $\bm R$ and $\bm f$. Any 
initial and terminal quadratic differentials $\varphi$ and $\psi$ of $f$ 
determine well-defined elements 
${\bm \varphi}=[\varphi,\theta_{0}] \in A_{+}({\bm R}_{0})$ and 
${\bm \psi}=[\psi,\theta] \in A({\bm R})$, which will be referred to as 
{\em initial\/} and {\em terminal quadratic differentials\/} of $\bm f$, 
respectively. 

The next lemma follows at once from Lemma~\ref{lem:Teichmuller_embedding} and 
Corollary~\ref{cor:Teichmuller_embedding:quadratic_differential}. 

\begin{lem}
\label{lem:marked_Teichmuller_embedding}
Let ${\bm R}_{0}$ be a marked nonanalytically finite open Riemann surface of 
genus $g$, and let ${\bm R} \in \mathfrak{T}_{g}$. If $\bm f$ is a 
Teichm\"{u}ller quasiconformal embedding of ${\bm R}_{0}$ into $\bm R$ with 
initial quadratic differential $\bm \varphi$, then there are a closed 
$\bm \varphi$-regular self-welding continuation $({\bm R}',{\bm \iota})$ of 
${\bm R}_{0}$ and a Teichm\"{u}ller quasiconformal homeomorphism $\bm h$ of 
${\bm R}'$ onto $\bm R$ such that 
\begin{list}{{\rm (\roman{claim})}}{\usecounter{claim}
\setlength{\topsep}{0pt}
\setlength{\itemsep}{0pt}
\setlength{\parsep}{0pt}
\setlength{\labelwidth}{\leftmargin}}
\item the co-welder of $\bm \varphi$  is an initial quadratic differential of $\bm h$, 
and 

\item ${\bm f}={\bm h} \circ {\bm \iota}$. 
\end{list}
Moreover, $\bm \varphi$ belongs to $A_{L}({\bm R}_{0})$. 
\end{lem}

If ${\bm R} \in \mathfrak{T}_{g} \setminus \mathfrak{M}({\bm R}_{0})$, then 
$\QCEmb_{\mathrm{hc}}({\bm R}_{0},{\bm R}) \neq \varnothing$ even though 
$\CEmb_{\mathrm{hc}}({\bm R}_{0},{\bm R})=\varnothing$. A homotopically 
consistent quasiconformal embedding of ${\bm R}_{0}$ into $\bm R$ is said to be 
{\em extremal\/} if it has the smallest maximal dilatation in 
$\QCEmb_{\mathrm{hc}}({\bm R}_{0},{\bm R})$. 

\begin{prop}[\mbox{Ioffe \cite[Theorem~0.1]{Ioffe1975}, %
Bourque \cite[Theorem~3.11]{Bourque2018}}]
\label{prop:extremal_qc_embedding}
Assume that ${\bm R}_{0}$ is a marked finite open Riemann surface of genus $g$. 
If ${\bm R}$ belongs to $\mathfrak{T}_{g} \setminus \mathfrak{M}({\bm R}_{0})$, 
then $\QCEmb_{\mathrm{hc}}({\bm R}_{0},{\bm R})$ has an extremal element, which 
is a Teichm\"{u}ller quasiconformal embedding. 
\end{prop}

\begin{rem}
In \cite[Theorem~0.1]{Ioffe1975} Ioffe claimed that 
$\QCEmb_{\mathrm{hc}}({\bm R}_{0},{\bm R})$ has a {\em unique\/} extremal 
element. Bourque has disproved the uniqueness by counterexamples in 
\cite[\S3.3]{Bourque2018}. The error affects the uniqueness assertion in 
Earle-Marden \cite{EM1978}. 
\end{rem}

With the aid of Lemma~\ref{lem:marked_Teichmuller_embedding} 
and Proposition~\ref{prop:self-welding:sequence} it is easy to prove 
Proposition~\ref{prop:boundary_of_M(R_0)}. 

\begin{proof}[Proof of Proposition~\ref{prop:boundary_of_M(R_0)}]
Let ${\bm R} \in \partial\mathfrak{M}({\bm R}_{0})$. Take a sequence 
$\{{\bm R}_{n}\}$ in $\mathfrak{T}_{g} \setminus \mathfrak{M}({\bm R}_{0})$ 
converging to $\bm R$, and choose an extremal 
${\bm f}_{n} \in \QCEmb_{\mathrm{hc}}({\bm R}_{0},{\bm R}_{n})$ for each $n$. 
It is a Teich\-m\"{u}l\-ler quasiconformal embedding by 
Proposition~\ref{prop:extremal_qc_embedding}. By 
Lemma~\ref{lem:marked_Teichmuller_embedding} there are a closed regular 
self-welding continuation $({\bm R}'_{n},{\bm \iota}'_{n})$ of ${\bm R}_{0}$ 
and a Teichm\"{u}ller quasiconformal homeomorphism ${\bm h}'_{n}$ of 
${\bm R}'_{n}$ onto ${\bm R}_{n}$ such that 
${\bm f}_{n}={\bm h}'_{n} \circ {\bm \iota}'_{n}$. Since 
${\bm R} \in \mathfrak{M}({\bm R}_{0})$, we can find 
$\bm \kappa \in \CEmb_{\mathrm{hc}}({\bm R}_{0},{\bm R})$. If ${\bm h}_{n}$ is 
the Teichm\"{u}ller quasiconformal homeomorphism of $\bm R$ onto 
${\bm R}_{n}$, then ${\bm h}_{n} \circ {\bm \kappa} \in
\QCEmb_{\mathrm{hc}}({\bm R}_{0},{\bm R}_{n})$. The fact that ${\bm f}_{n}$ is 
extremal yields 
$$
e^{2d_{T}({\bm R}'_{n},{\bm R})}=K({\bm h}'_{n})=K({\bm f}_{n}) \leqq
K({\bm h}_{n} \circ {\bm \kappa})=K({\bm h}_{n})=e^{2d_{T}({\bm R}_{n},{\bm R})}
\to 1
$$
as $n \to \infty$. Therefore, Proposition~\ref{prop:self-welding:sequence} 
guarantees the existence of a closed regular self-welding continuation 
$({\bm R},{\bm \iota})$ of ${\bm R}_{0}$, as desired. 
\end{proof}

In general, let ${\bm R}_{0}$ be a marked nonanalytically finite open Riemann 
surface of genus $g$. Two elements $\bm \varphi$ and ${\bm \varphi}'$ of 
$A_{L}({\bm R}_{0})$ are said to be {\em projectively equivalent\/} to each 
other if ${\bm \varphi}'=r{\bm \varphi}$ for some $r>0$. The set of projective 
equivalence classes $[{\bm \varphi}]$ is denoted by $PA_{L}({\bm R}_{0})$. Now, 
Theorem~\ref{thm:main:CEmb(R_0,R)}~(i) asserts that any point $\bm R$ on the 
boundary $\partial\mathfrak{M}({\bm R}_{0})$ is obtained as a closed 
${\bm \varphi}$-regular self-welding continuation of ${\bm R}_{0}$ for some 
${\bm \varphi} \in A_{L}({\bm R}_{0})$ and that each 
${\bm \varphi} \in A_{L}({\bm R}_{0})$ gives rise to a point $\bm R$ 
on $\partial\mathfrak{M}({\bm R}_{0})$ through a closed $\bm \varphi$-regular 
self-welding continuation of ${\bm R}_{0}$. Recall that projectively equivalent 
elements of $A_{L}({\bm R}_{0})$ induce the same closed regular self-welding 
continuation of ${\bm R}_{0}$ (see the paragraph precedent to 
Example~\ref{exmp:inducing_same_self-welding}). 
Examples~\ref{exmp:hydrodynamic_continuation} 
and~\ref{exmp:inducing_same_self-welding} show that the correspondence 
$\partial\mathfrak{M}({\bm R}_{0}) \ni {\bm R} \leftrightarrow
[{\bm \varphi}] \in PA_{L}({\bm R}_{0})$ is, in general, many-to-many. 

\begin{exmp}
\label{exmp:g=1:projective_equivalence_class}
We consider the case of genus one, and continue to use the notations in 
Example~\ref{exmp:g=1:extremal_length}. The mapping 
$t \mapsto {\bm T}_{\tau(t)}$ is a bijection of $(-1,1]$ onto 
$\partial\mathfrak{M}({\bm R}_{0})$ by \cite[Theorem~5]{Shiba1987}. Thus 
the above mentioned correspondence between 
$\partial\mathfrak{M}({\bm R}_{0})$ and $PA_{L}({\bm R}_{0})$ is in fact a 
bijection in the case where $g=1$ (see 
Example~\ref{exmp:g=1:exceptional_border_components} below). 
\end{exmp}

\section{Ioffe rays}
\setcounter{equation}{0}
\label{sec:Ioffe_rays}

Let ${\bm R}_{0}$ be a marked open Riemann surface of  genus $g$. For 
$K \geqq 1$ we denote by $\mathfrak{M}_{K}({\bm R}_{0})$ the set of 
${\bm R} \in \mathfrak{T}_{g}$ into which ${\bm R}_{0}$ can be mapped by a 
homotopically consistent $K$-quasiconformal embedding. In particular, we have 
$\mathfrak{M}_{1}({\bm R}_{0})=\mathfrak{M}({\bm R}_{0})$ since 
$1$-quasiconformal embeddings are conformal embeddings. If $K \leqq K'$, then 
$\mathfrak{M}_{K}({\bm R}_{0}) \subset \mathfrak{M}_{K'}({\bm R}_{0})$. Also, 
we have $\bigcup_{K \geqq 1} \mathfrak{M}_{K}({\bm R}_{0})=\mathfrak{T}_{g}$. 
The next proposition claims that $\mathfrak{M}_{K}({\bm R}_{0})$ is the closed 
$(\log K)/2$-neighborhood of $\mathfrak{M}({\bm R}_{0})$. 

\begin{prop}
\label{prop:space:qc_embedding}
Let $\bm R \in \mathfrak{T}_{g}$ and $K \geqq 1$. Then 
${\bm R} \in \mathfrak{M}_{K}({\bm R}_{0})$ if and only if 
\begin{equation}
\label{eq:space:qc_embedding}
d_{T}({\bm R},\mathfrak{M}({\bm R}_{0})) \leqq \frac{1}{\,2\,}\log K.
\end{equation}
\end{prop}

\begin{proof}
Let ${\bm R}=[R,\theta] \in \mathfrak{T}_{g}$, and suppose that there is 
${\bm f}=[f,\theta_{0},\theta] \in \QCEmb_{\mathrm{hc}}({\bm R}_{0},{\bm R})$ 
with $K({\bm f}) \leqq K$, where ${\bm R}_{0}=[R_{0},\theta_{0}]$. Take a 
quasiconformal homeomorphism $h$ of $R$ onto a closed Riemann surface $R'$ such 
that $\mu_{h}=\mu_{f^{-1}}$ on $f(R_{0})$ and $\mu_{h}=0$ on 
$R \setminus f(R_{0})$. Setting ${\bm R}'=[R',h \circ f \circ \theta_{0}]$ and 
${\bm h}'=[h \circ f,\theta_{0},h \circ f \circ \theta_{0}]$, we obtain a 
closed continuation $({\bm R}',{\bm h}')$ of ${\bm R}_{0}$. Thus 
${\bm R}' \in \mathfrak{M}({\bm R}_{0})$ and hence 
$$
d_{T}({\bm R},\mathfrak{M}({\bm R}_{0})) \leqq
d_{T}({\bm R},{\bm R}') \leqq \frac{1}{\,2\,}\log K({\bm h})=
\frac{1}{\,2\,}\log K({\bm f}) \leqq \frac{1}{\,2\,}\log K.
$$

Conversely, assume that inequality~\eqref{eq:space:qc_embedding} holds. Let 
${\bm R}'$ be a point of $\mathfrak{M}({\bm R}_{0})$ nearest to $\bm R$ (see 
Corollary~\ref{cor:Oikawa:compact}), and take 
${\bm \iota} \in \CEmb_{\mathrm{hc}}({\bm R}_{0},{\bm R}')$. Let $\bm h$ be the 
Teichm\"{u}ller quasiconformal homeomorphism of ${\bm R}'$ onto $\bm R$. Then 
${\bm h} \circ {\bm \iota}$ belongs to 
$\QCEmb_{\mathrm{hc}}({\bm R}_{0},{\bm R})$, and its maximal dilatation 
satisfies 
$$
\log K({\bm h} \circ {\bm \iota}) \leqq \log K({\bm h})=2d_{T}({\bm R},{\bm R}')
=2d_{T}({\bm R},\mathfrak{M}({\bm R}_{0})) \leqq \log K.
$$
Consequently, $\bm R$ lies in $\mathfrak{M}_{K}({\bm R}_{0})$. 
\end{proof}

For ${\bm R} \in \mathfrak{T}_{g}$ and $\rho>0$ let 
$\mathfrak{B}_{\rho}({\bm R})$ denote the open ball of radius $\rho$ centered 
at $\bm R$: 
$$
\mathfrak{B}_{\rho}({\bm R})=
\{{\bm S} \in \mathfrak{T}_{g} \mid d_{T}({\bm S},{\bm R})<\rho\}.
$$
Its closure will be denoted by $\bar{\mathfrak{B}}_{\rho}({\bm R})$, which is a 
compact subset of $\mathfrak{T}_{g}$. For the sake of convenience we define 
$\bar{\mathfrak{B}}_{0}({\bm R})=\{{\bm R}\}$. 

\begin{cor}
\label{cor:space:qc_embedding}
$\mathfrak{M}_{K}({\bm R}_{0})$ is compact for all $K \geqq 1$. 
\end{cor}

\begin{proof}
It follows from Corollary~\ref{cor:Oikawa:compact} that 
$\mathfrak{M}({\bm R}_{0})$ is included in $\bar{\mathfrak{B}}_{\rho}({\bm R})$ 
for some ${\bm R} \in \mathfrak{M}({\bm R}_{0})$ and $\rho>0$. 
Proposition~\ref{prop:space:qc_embedding} then implies that 
$\mathfrak{M}_{K}({\bm R}_{0})$ is a closed subset of the compact set 
$\bar{\mathfrak{B}}_{\rho_{K}}({\bm R})$, where $\rho_{K}=\rho+(\log K)/2$, 
and hence is compact. 
\end{proof}

\begin{defn}[Ioffe ray]
\label{defn:Ioffe_ray}
Let ${\bm R}_{0}$ be a marked nonanalytically finite open Riemann surface. An 
{\em Ioffe ray\/} of ${\bm R}_{0}$ is, by definition, a Teichm\"{u}ller 
geodesic ray of the form ${\bm r}_{\bm R}[{\bm \psi}]$, where 
${\bm R} \in \mathfrak{M}({\bm R}_{0})$ and ${\bm \psi} \in A({\bm R})$ such 
that for some ${\bm \iota} \in \CEmb_{\mathrm{hc}}({\bm R}_{0},{\bm R})$ and 
${\bm \varphi} \in A_{L}({\bm R}_{0})$ the continuation $({\bm R},{\bm \iota})$ 
is a closed $\bm \varphi$-regular self-welding continuation of ${\bm R}_{0}$ 
and that $\bm \psi$ is the co-welder of $\bm \varphi$. Let 
$\mathscr{I}({\bm R}_{0})$ denote the set of Ioffe rays of ${\bm R}_{0}$.
\end{defn}

If ${\bm R}_{0}$ is a marked analytically finite open Riemann surface, then 
$\mathfrak{M}({\bm R}_{0})$ consists of exactly one point, say $\bm R$. We call 
each Teichm\"{u}ller geodesic ray emanating from $\bm R$ an {\em Ioffe ray\/} 
of ${\bm R}_{0}$. 

\begin{rem}
By Theorem~\ref{thm:main:CEmb(R_0,R)}~(i) the initial point of each Ioffe ray 
of ${\bm R}_{0}$ lies on the boundary $\partial\mathfrak{M}({\bm R}_{0})$ and 
any boundary point of $\mathfrak{M}({\bm R}_{0})$ is the initial point of some 
Ioffe ray of ${\bm R}_{0}$. Note that two or more Ioffe rays of ${\bm R}_{0}$ 
can have a common initial point. See the paragraph precedent to 
Example~\ref{exmp:g=1:projective_equivalence_class}. 
\end{rem}

One of the purposes of the present section is to establish the following 
theorem. As an application, we give an alternative proof of Bourque 
\cite[Theorem~3.13]{Bourque2018} in the case where the target is a closed 
Riemann surface of the same genus as the domain surface of quasiconformal 
embeddings; it should be noted that Theorem~\ref{thm:Ioffe_ray:M(R_0)^c} 
follows easily from the theorem of Bourque. Geometric aspects of Ioffe rays 
should be emphasized. 

\begin{thm}
\label{thm:Ioffe_ray:M(R_0)^c}
Let ${\bm R}_{0}$ be a marked finite open Riemann surface of genus $g$. Then 
the correspondence 
$$
\mathscr{I}({\bm R}_{0}) \times (0,+\infty) \ni ({\bm r},t) \mapsto
{\bm r}(t) \in \mathfrak{T}_{g} \setminus \mathfrak{M}({\bm R}_{0})
$$
is bijective. 
\end{thm}

Biernacki \cite{Biernacki1936} called a plane domain $D$  {\em linearly 
accessible\/} in the strict sense if the compliment $\mathbb{C} \setminus D$ 
can be expressed as a union of mutually disjoint half lines except that the 
endpoint of one half line can lie on another half line. Lewandowski 
\cite{Lewandowski1958,Lewandowski1960} showed that the class of linearly 
accessible domains in the strict sense is precisely the class of 
close-to-convex domains of Kaplan \cite{Kaplan1952}. Following these 
terminologies, we may say that $\mathfrak{M}({\bm R}_{0})$ is linearly 
accessible in the strict sense or close-to-convex. 

Theorem~\ref{thm:Ioffe_ray:M(R_0)^c} implies in particular that each Ioffe ray 
never hits $\mathfrak{M}({\bm R}_{0})$ again after departure. We prove the 
theorem step by step. We begin with the following proposition. 
Theorem~\ref{thm:main:CEmb(R_0,R)}~(ii) is a corollary to the proposition. 

\begin{prop}
\label{prop:uniqueness:boundary_of_M(R_0)}
Let ${\bm R}_{0}$ be a marked nonanalytically finite open Riemann surface of 
genus $g$, and let ${\bm R} \in \partial\mathfrak{M}({\bm R}_{0})$. Take 
${\bm \varphi} \in A_{L}({\bm R}_{0})$ so that $\bm R$ is obtained via a closed 
$\bm \varphi$-regular self-welding continuation of ${\bm R}_{0}$ and let 
${\bm \psi} \in A({\bm R})$ be the co-welder of $\bm \varphi$. Then any element 
of $\CEmb_{\mathrm{hc}}({\bm R}_{0},{\bm R})$ defines a closed 
$\bm \varphi$-regular self-welding continuation of ${\bm R}_{0}$ for which 
$\bm \psi$ is the co-welder of $\bm \varphi$. 
\end{prop}

\begin{proof}
Let $({\bm R},{\bm \iota})$ be a closed $\bm \varphi$-regular self-welding 
continuation of ${\bm R}_{0}$ with ${\bm \varphi}={\bm \iota}^{*}{\bm \psi}$. 
For any closed continuation $({\bm R},{\bm \iota}')$ of ${\bm R}_{0}$ let 
${\bm \varphi}'$ be the $({\bm R},{\bm \iota}')$-zero-extension of 
$\bm \varphi$. It then follows from Lemma~\ref{lem:height_comparison} that for 
any $\gamma \in \mathscr{S}(\Sigma_{g})$ the inequality 
$\mathcal{H}_{\bm R}({\bm \psi})(\gamma) \leqq
{\mathcal H}'_{\bm R}({\bm \varphi}')(\gamma)$ holds. 
Proposition~\ref {prop:second_minimal_norm_property} then implies that 
$\|{\bm \psi}\|_{\bm R} \leqq \|{\bm \varphi}'\|_{\bm R}$. Actually, the sign 
of equality occurs because $\|{\bm \varphi}'\|_{\bm R}=
\|{\bm \varphi}\|_{{\bm R}_{0}}=\|{\bm \psi}\|_{\bm R}$. Another application of 
Proposition~\ref{prop:second_minimal_norm_property} gives us 
${\bm \varphi}'={\bm \psi}$ almost everywhere on $\bm R$. In particular, 
$({\bm R},{\bm \iota}')$ is a dense continuation of ${\bm R}_{0}$. Since 
${\bm \varphi}=({\bm \iota}')^{*}{\bm \psi}$, 
Proposition~\ref{prop:self-welding_continuation} implies that 
$({\bm R},{\bm \iota}')$ is a closed $\bm \varphi$-regular self-welding 
continuation of ${\bm R}_{0}$ for which $\bm \psi$ is the co-welder of 
$\bm \varphi$. 
\end{proof}

\begin{rem}
One might surmise that Proposition~\ref{prop:uniqueness:boundary_of_M(R_0)} 
assures us that $\CEmb_{\mathrm{hc}}({\bm R}_{0},{\bm R})$ is a singleton. This 
is not always the case, however, as Bourque \cite[\S3.3]{Bourque2018} pointed 
out. 
\end{rem}

\begin{prop}
\label{prop:Ioffe_ray:existence}
Suppose that ${\bm R}_{0}$ is a marked finite open Riemann surface of genus 
$g$. Let ${\bm S} \in \mathfrak{T}_{g} \setminus \mathfrak{M}({\bm R}_{0})$. If 
$\bm R$ is a point of\/ $\mathfrak{M}({\bm R}_{0})$ nearest to $\bm S$, then\/ 
${\bm r}_{\bm R}[{\bm S}]$ is an Ioffe ray of ${\bm R}_{0}$. 
\end{prop}

\begin{proof}
We may assume that ${\bm R}_{0}$ is nonanalytically finite. Take 
${\bm \iota} \in \CEmb_{\mathrm{hc}}({\bm R}_{0},{\bm R})$. Let $\bm h$ be the 
Teichm\"{u}ller quasiconformal homeomorphism of $\bm R$ onto $\bm S$, and set 
${\bm f}={\bm h} \circ {\bm \iota}$. If 
${\bm f}' \in \QCEmb_{\mathrm{hc}}({\bm R}_{0},{\bm S})$, then 
Proposition~\ref{prop:space:qc_embedding} yields that 
$$
\log K({\bm f}') \geqq 2d_{T}({\bm S},\mathfrak{M}({\bm R}_{0}))=
2d_{T}({\bm S},{\bm R})=\log K({\bm h}) \geqq \log K({\bm f}).
$$
This shows that $\bm f$ is extremal in 
$\QCEmb_{\mathrm{hc}}({\bm R}_{0},{\bm S})$, and hence is a Teichm\"{u}ller 
quasiconformal embedding by Proposition~\ref{prop:extremal_qc_embedding}. 

Let ${\bm \varphi}_{0} \in A_{L}({\bm R}_{0})$ and ${\bm \psi} \in A({\bm S})$ 
be initial and terminal quadratic differentials of $\bm f$, respectively (see 
Corollary~\ref{cor:Teichmuller_embedding:quadratic_differential}). Since 
$\bm R$ is a boundary point of $\mathfrak{M}({\bm R}_{0})$, 
Proposition~\ref{prop:uniqueness:boundary_of_M(R_0)} implies that 
$({\bm R},{\bm \iota})$ is a regular self-welding continuation of 
${\bm R}_{0}$. To show that $\bm \varphi_{0}$ is a welder of the continuation 
take representatives 
$(R_{0},\theta_{0}) \in {\bm R}_{0}$, $(R,\theta) \in {\bm R}$, 
$(S,\eta) \in {\bm S}$, $(\iota,\theta_{0},\theta) \in {\bm \iota}$, 
$(f,\theta_{0},\eta) \in {\bm f}$, $(h,\theta,\eta) \in {\bm h}$, 
$(\varphi_{0},\theta_{0}) \in {\bm \varphi}_{0}$ and 
$(\psi,\eta) \in {\bm \psi}$ so that $f=h \circ \iota$. The last identity shows 
that an initial quadratic differential $\varphi \in A(R)$ of $h$ satisfies 
$\bar{\varphi}/\varphi=\overline{\iota_{*}\varphi_{0}}/\iota_{*}\varphi_{0}$. 
Thus the meromorphic function $(\iota_{*}\varphi_{0})/\varphi$ on 
$\iota(R_{0})$ is real and hence constant. Since $h$ is a uniform stretch along 
horizontal trajectories of both $\iota_{*}\varphi_{0}$ and $\varphi$, the 
constant must be positive and hence may be supposed to be $1$. Then we obtain 
$\varphi_{0}=\iota^{*}\varphi$. As $\iota(R_{0})$ is dense in $R$, in view of 
Proposition~\ref{prop:self-welding_continuation} we see that 
$({\bm R},{\bm \iota})$ is a closed ${\bm \varphi}_{0}$-regular self-welding 
continuation of ${\bm R}_{0}$ and that ${\bm \varphi}:=[\varphi,\theta]$ is the 
co-welder of ${\bm \varphi}_{0}$. It follows that 
${\bm r}_{\bm R}[{\bm \varphi}]$ is an Ioffe ray of ${\bm R}_{0}$. Since 
$\bm \varphi$ is an initial quadratic differential of the Teichm\"{u}ller 
quasiconformal homeomorphism $\bm h$ of $\bm R$ onto $\bm S$, we infer that 
${\bm r}_{\bm R}[\bm S]={\bm r}_{\bm R}[{\bm \varphi}]$. 
\end{proof}

\begin{lem}
\label{lem:Teichmuller_ray:distance}
Let\/ $\mathfrak{K}$ be a compact subset of\/ $\mathfrak{T}_{g}$ and let 
$\bm r$ be a Teichm\"{u}ller geodesic ray with ${\bm r}(0) \in \mathfrak{K}$. 
If\/ $\mathfrak{B}_{\rho}({\bm r}(\rho)) \cap \mathfrak{K}=\varnothing$ for all 
$\rho>0$, then\/ 
$\bar{\mathfrak{B}}_{\rho}({\bm r}(\rho)) \cap \mathfrak{K}=\{{\bm r}(0)\}$ for 
$\rho>0$. 
\end{lem}

\begin{proof}
It follows from~\eqref{eq:Teichmuller_ray:parametrization} that ${\bm r}(0)$ 
belongs to $\bar{\mathfrak{B}}_{\rho}({\bm r}(\rho)) \cap \mathfrak{K}$. 
Suppose that $\bm R$ also belongs to both 
$\bar{\mathfrak{B}}_{\rho}({\bm r}(\rho))$ and $\mathfrak{K}$. Then 
\begin{equation}
\label{eq:Teichmuller_ray:neighborhood_of_K}
d_{T}({\bm R},{\bm r}(2\rho)) \leqq
d_{T}({\bm R},{\bm r}(\rho))+d_{T}({\bm r}(\rho),{\bm r}(2\rho)) \leqq
\rho+\rho=2\rho,
\end{equation}
and hence ${\bm R} \in \bar{\mathfrak{B}}_{2\rho}({\bm r}(2\rho))$. Since 
$\mathfrak{B}_{2\rho}({\bm r}(2\rho)) \cap \mathfrak{K}=\varnothing$, we 
obtain ${\bm R} \in \partial\mathfrak{B}_{2\rho}({\bm r}(2\rho))$, or 
equivalently, $d_{T}({\bm R},{\bm r}(2\rho))=2\rho$. Therefore, the signs of 
equality occur in~\eqref{eq:Teichmuller_ray:neighborhood_of_K}. Since the 
metric space $(\mathfrak{T}_{g},d_{T})$ is a straight space in the sense of 
Busemann (see Abikoff \cite[Section~(3.2)]{Abikoff1980}), the three points 
$\bm R$, ${\bm r}(\rho)$ and ${\bm r}(2\rho)$ lie on the same Teichm\"{u}ller 
geodesic line, which must include the Teichm\"{u}ller geodesic ray $\bm r$. It 
then follows from~\eqref{eq:Teichmuller_ray:neighborhood_of_K} that 
${\bm R}={\bm r}(0)$, which finishes the proof. 
\end{proof}

\begin{thm}
\label{thm:Ioffe_ray:distance}
Let ${\bm R}_{0}$ be a marked finite open Riemann surface. If 
${\bm r} \in \mathscr{I}({\bm R}_{0})$, then 
\begin{equation}
\label{eq:Ioffe_ray:ball}
\bar{\mathfrak{B}}_{\rho}({\bm r}(\rho+\tau)) \cap
\mathfrak{M}_{e^{2\tau}}({\bm R}_{0})=\{{\bm r}(\tau)\}
\end{equation}
and 
\begin{equation}
\label{eq:Ioffe_ray:distance}
d_{T}({\bm r}(t),\mathfrak{M}_{e^{2\tau}}({\bm R}_{0}))=\max\{t-\tau,0\}
\end{equation}
for all $\rho>0$, $\tau \geqq 0$ and $t \geqq 0$. 
\end{thm}

To prove Theorem~\ref{thm:Ioffe_ray:distance} we need the following lemma, 
which implies that $\log\Ext_{\mathcal{F}}$ is a Lipschitz continuous function 
on $\mathfrak{T}_{g}$. For the proof see Gardiner \cite[Lemma~4]{Gardiner1984} 
or Gardiner-Lakic \cite[Lemma~12.5]{GL2000}. 

\begin{lem}
\label{lem:extremal_length:continuity}
For ${\bm R}_{1},{\bm R}_{2} \in \mathfrak{T}_{g}$ and 
$\mathcal{F} \in \mathscr{MF}(\Sigma_{g}) \setminus \{0\}$, 
$$
\log\Ext_{\mathcal{F}}({\bm R}_{1})-2d_{T}({\bm R}_{1},{\bm R}_{2})
\leqq \log\Ext_{\mathcal{F}}({\bm R}_{2}) \leqq
\log\Ext_{\mathcal{F}}({\bm R}_{1})+2d_{T}({\bm R}_{1},{\bm R}_{2}).
$$
\end{lem}

\begin{proof}[Proof of Theorem~\ref{thm:Ioffe_ray:distance}]
We can write ${\bm r}={\bm r}_{\bm R}[{\bm \psi}]$ with 
${\bm \psi} \in A({\bm R})$, where $\bm \psi$ is the co-welder of a welder of 
some closed regular self-welding continuation $({\bm R},{\bm \iota})$ of 
${\bm R}_{0}$. Set 
$\mathcal{F}=\mathcal{H}_{\bm R}({\bm \psi})$. 

To prove~\eqref{eq:Ioffe_ray:ball} let ${\bm S} \in \mathfrak{T}_{g}$. If 
${\bm R}'$ is a point of $\mathfrak{M}({\bm R}_{0})$ nearest to $\bm S$, then 
Lemma~\ref{lem:extremal_length:continuity} and 
Theorem~\ref{thm:maximal_norm_property} imply that 
$$
\log\Ext_{\mathcal{F}}({\bm S}) \leqq
\log\Ext_{\mathcal{F}}({\bm R}')+2d_{T}({\bm S},{\bm R}') \leqq
\log\Ext_{\mathcal{F}}({\bm R})+2d_{T}({\bm S},\mathfrak{M}({\bm R}_{0})).
$$
Consider the Teichm\"{u}ller geodesic ray ${\bm r}'$ defined by 
${\bm r}'(t)={\bm r}(t+\tau)$, $t \geqq 0$. Then another application of 
Lemma~\ref{lem:extremal_length:continuity} together 
with~\eqref{eq:extremal_length:Teichmuller_qc} yields that 
\begin{align*}
\log\Ext_{\mathcal{F}}({\bm S}) & \geqq
\log\Ext_{\mathcal{F}}({\bm r}(\rho+\tau))-2d_{T}({\bm S},{\bm r}(\rho+\tau)) \\
 &=\log\Ext_{\mathcal{F}}({\bm R})+2(\rho+\tau)-2d_{T}({\bm S},{\bm r}'(\rho)).
\end{align*}
With the aid of Proposition~\ref{prop:space:qc_embedding} we thus obtain 
$$
d_{T}({\bm S},\mathfrak{M}_{e^{2\tau}}({\bm R}_{0})) \geqq
d_{T}({\bm S},\mathfrak{M}({\bm R}_{0}))-\tau \geqq
\rho-d_{T}({\bm S},{\bm r}'(\rho)).
$$
Therefore, $\mathfrak{B}_{\rho}({\bm r}'(\rho)) \cap
\mathfrak{M}_{e^{2\tau}}({\bm R}_{0})=\varnothing$ for all $\rho>0$. Since 
$\mathfrak{M}_{e^{2\tau}}({\bm R}_{0})$ is compact by 
Corollary~\ref{cor:space:qc_embedding} and contains 
${\bm r}'(0)={\bm r}(\tau)$, identity~\eqref{eq:Ioffe_ray:ball} follows at once 
from Lemma~\ref{lem:Teichmuller_ray:distance}. 

To show~\eqref{eq:Ioffe_ray:distance} we first remark that if 
$0 \leqq t \leqq \tau$, then 
$$
d_{T}({\bm r}(t),\mathfrak{M}({\bm R}_{0})) \leqq d_{T}({\bm r}(t),{\bm r}(0))=
t \leqq \tau
$$
as ${\bm r}(0)={\bm R} \in \mathfrak{M}({\bm R}_{0})$. Therefore, by 
Proposition~\ref{prop:space:qc_embedding} we know that 
${\bm r}(t) \in \mathfrak{M}_{e^{2\tau}}({\bm R}_{0})$, or equivalently, 
$d_{T}({\bm r}(t),\mathfrak{M}_{e^{2\tau}}({\bm R}_{0}))=0$. If $t>\tau$, then 
$$
\bar{\mathfrak{B}}_{t-\tau}({\bm r}(t)) \cap
\mathfrak{M}_{e^{2\tau}}({\bm R}_{0})=\{{\bm r}(\tau)\}
$$
by~\eqref{eq:Ioffe_ray:ball}, which implies that 
$d_{T}({\bm r}(t),\mathfrak{M}_{e^{2\tau}}({\bm R}_{0}))=t-\tau$. 
\end{proof}

\begin{cor}
\label{cor:Ioffe_ray:outside}
If ${\bm r} \in \mathscr{I}({\bm R}_{0})$ and $t>\tau \geqq 0$, then ${\bm r}(t)
\in \mathfrak{T}_{g} \setminus \mathfrak{M}_{e^{2\tau}}({\bm R}_{0})$. 
\end{cor}

\begin{proof}
By~\eqref{eq:Ioffe_ray:distance} we have 
$d_{T}({\bm r}(t),\mathfrak{M}_{e^{2\tau}}({\bm R}_{0}))>0$. Hence 
${\bm r}(t) \not\in \mathfrak{M}_{e^{2\tau}}({\bm R}_{0})$. 
\end{proof}

\begin{cor}
\label{cor:M_K(R_0):nearest_point}
Let $K \geqq 1$ and ${\bm R} \in \mathfrak{T}_{g}$. Then a point of\/ 
$\mathfrak{M}_{K}({\bm R}_{0})$ nearest to $\bm R$ is uniquely determined. 
\end{cor}

\begin{proof}
We have only to consider the case 
${\bm R} \not\in \mathfrak{M}_{K}({\bm R}_{0})$. By 
Proposition~\ref{prop:Ioffe_ray:existence} there is  
${\bm r} \in \mathscr{I}({\bm R}_{0})$ for which 
${\bm r}(t_{0})={\bm R}$ for some $t_{0} \geqq 0$. Since $\bm R$ does not 
belong to the compact set $\mathfrak{M}_{K}({\bm R}_{0})$, it follows 
from~\eqref{eq:Ioffe_ray:distance} that $t_{0}>(\log K)/2$. Setting 
$t_{1}=t_{0}-(\log K)/2$, we then infer from~\eqref{eq:Ioffe_ray:ball} that 
$\bar{\mathfrak{B}}_{t_{1}}({\bm R}) \cap \mathfrak{M}_{K}({\bm R}_{0})$ is a 
singleton, which proves the corollary. 
\end{proof}

We are now ready to prove Theorem~\ref{thm:Ioffe_ray:M(R_0)^c}

\begin{proof}[Proof of Theorem~\ref{thm:Ioffe_ray:M(R_0)^c}]
If $({\bm r},t) \in \mathscr{I}({\bm R}_{0}) \times (0,+\infty)$, then 
${\bm r}(t) \in \mathfrak{T}_{g} \setminus \mathfrak{M}({\bm R}_{0})$ by 
Corollary~\ref{cor:Ioffe_ray:outside}. Thus $({\bm r},t) \mapsto {\bm r}(t)$ 
defines a mapping of $\mathscr{I}({\bm R}_{0}) \times (0,+\infty)$ into 
$\mathfrak{T}_{g} \setminus \mathfrak{M}({\bm R}_{0})$. 

Let ${\bm S} \in \mathfrak{T}_{g} \setminus \mathfrak{M}({\bm R}_{0})$. As we 
already know that $\bm S$ lies on some Ioffe ray $\bm r$ (see 
Proposition~\ref{prop:Ioffe_ray:existence}), we need to show the uniqueness. 
Set ${\bm R}={\bm r}(0)$ and $\rho=d_{T}({\bm S},\mathfrak{M}({\bm R}_{0}))$. 
If ${\bm r}(t_{0})={\bm S}$, then 
$\rho=d_{T}({\bm r}(t_{0}),\mathfrak{M}({\bm R}_{0}))=t_{0}$ 
by~\eqref{eq:Ioffe_ray:distance}. Therefore, by~\eqref{eq:Ioffe_ray:ball} we 
have 
$$
\{{\bm R}\}=\{{\bm r}(0)\}=
\bar{\mathfrak{B}}_{\rho}({\bm r}(\rho)) \cap \mathfrak{M}({\bm R}_{0})=
\bar{\mathfrak{B}}_{d({\bm S})}({\bm S}) \cap \mathfrak{M}({\bm R}_{0}),
$$
where $d({\bm S})=d_{T}({\bm S},\mathfrak{M}({\bm R}_{0}))$. Hence $\bm R$ is 
determined solely by $\bm S$. Consequently, $\bm S$ lies on exactly one Ioffe 
ray since ${\bm r}={\bm r}_{\bm R}[{\bm S}]$. 
\end{proof}

The following proposition describes the boundary and the exterior of 
$\mathfrak{M}_{K}({\bm R}_{0})$ in terms of Ioffe rays. 

\begin{prop}
\label{prop:partial_M_K(R_0):Ioffe_ray}
If ${\bm R}_{0}$ is a marked finite open Riemann surface of genus $g$, then 
\begin{align*}
\partial\mathfrak{M}_{e^{2\tau}}({\bm R}_{0}) & =
\{{\bm r}(\tau) \mid {\bm r} \in \mathscr{I}({\bm R}_{0})\}, \quad \text{and} \\
\mathfrak{T}_{g} \setminus \mathfrak{M}_{e^{2\tau}}({\bm R}_{0}) & =
\{{\bm r}(t) \mid {\bm r} \in \mathscr{I}({\bm R}_{0}),t>\tau\}.
\end{align*}
for $\tau \geqq 0$. 
\end{prop}

\begin{proof}
By the definition of Ioffe rays the initial points of Ioffe rays lie on 
$\partial\mathfrak{M}({\bm R}_{0})$, and 
Theorem~\ref{thm:main:CEmb(R_0,R)}~(i) implies that any point on 
$\partial\mathfrak{M}({\bm R}_{0})$ is the initial point of some Ioffe ray. 
Hence, with the aid of Theorem~\ref{thm:Ioffe_ray:M(R_0)^c} we know that the 
proposition is valid for $\tau=0$. 

Assume that $\tau>0$ and ${\bm r} \in \mathscr{I}({\bm R}_{0})$. It follows 
from~\eqref{eq:Ioffe_ray:ball} that 
${\bm r}(\tau) \in \partial\mathfrak{M}_{e^{2\tau}}({\bm R}_{0})$. Also, 
${\bm r}(t)
\in \mathfrak{T}_{g} \setminus \mathfrak{M}_{e^{2\tau}}({\bm R}_{0})$ for 
$t>\tau$ by Corollary~\ref{cor:Ioffe_ray:outside}. If $0 \leqq t<\tau$, then 
$\mathfrak{B}_{\tau-t}({\bm r}(t))$ is included in 
$\mathfrak{M}_{e^{2\tau}}({\bm R}_{0})$ by 
Proposition~\ref{prop:space:qc_embedding} as 
$d_{T}({\bm r}(t),\mathfrak{M}({\bm R}_{0}))=t$ 
by~\eqref{eq:Ioffe_ray:distance}. 
Hence ${\bm r}(t)$ is an interior point of 
$\mathfrak{M}_{e^{2\tau}}({\bm R}_{0})$. Now, the proposition follows at once. 
\end{proof}

We conclude the section with the following proposition due to Bourque. Recall 
that $\QCEmb_{\mathrm{hc}}({\bm R}_{0},{\bm R})$ includes extremal elements for 
any ${\bm R} \in \mathfrak{T}_{g} \setminus \mathfrak{M}({\bm R}_{0})$, which 
are Teichm\"{u}ller quasiconformal embeddings. The proposition asserts that 
Teichm\"{u}ller quasiconformal embeddings are extremal. Note that the 
uniqueness does not hold in general.

\begin{prop}[Bourque \mbox{\cite[Theorem~3.13]{Bourque2018}}]
\label{prop:extremal_quasiconformal_embedding:uniqueness}
Let ${\bm R}_{0}$ be a marked finite open Riemann surface of genus $g$. Then 
for ${\bm R} \in \mathfrak{T}_{g} \setminus \Int\mathfrak{M}({\bm R}_{0})$ 
Teichm\"{u}ller quasiconformal embeddings of ${\bm R}_{0}$ into $\bm R$ are 
extremal in $\QCEmb_{\mathrm{hc}}({\bm R}_{0},{\bm R})$ and have initial and 
terminal quadratic differentials in common. 
\end{prop}

\begin{proof}
Let ${\bm R} \in \mathfrak{T}_{g} \setminus \Int\mathfrak{M}({\bm R}_{0})$, and 
take a Teichm\"{u}ller quasiconformal embedding $\bm f$ of ${\bm R}_{0}$ into 
$\bm R$ with initial quadratic differential $\bm \varphi$. If 
${\bm R} \in \partial\mathfrak{M}({\bm R})$, then the proposition follows from 
Proposition~\ref{prop:uniqueness:boundary_of_M(R_0)}. 

If ${\bm R} \in \mathfrak{T}_{g} \setminus \mathfrak{M}({\bm R}_{0})$, then by 
Lemma~\ref{lem:marked_Teichmuller_embedding} we obtain a closed 
$\bm \varphi$-regular self-welding continuation $({\bm R}',{\bm \iota})$ of 
${\bm R}_{0}$ and a Teichm\"{u}ller quasiconformal homeomorphism $\bm h$ of 
${\bm R}'$ onto $\bm R$ such that the co-welder ${\bm \varphi}'$ of 
$\bm \varphi$ is an initial quadratic differential of $\bm h$ and that 
${\bm f}={\bm h} \circ {\bm \iota}$. Thus $\bm R$ lies on the Ioffe ray 
${\bm r}':={\bm r}_{{\bm R}'}[{\bm \varphi}']$, and if ${\bm R}={\bm r}'(t)$, 
then $d_{T}({\bm R},\mathfrak{M}({\bm R}_{0}))=t=(\log K({\bm h}))/2=
(\log K({\bm f}))/2$ by~\eqref{eq:Ioffe_ray:distance}, and hence $\bm f$ is 
extremal by Proposition~\ref{prop:partial_M_K(R_0):Ioffe_ray}. Since ${\bm R}'$ 
is uniquely determined by Corollary~\ref{cor:M_K(R_0):nearest_point}, it 
follows that $\bm h$ is also uniquely determined. Since $\bm f$ and $\bm h$ 
have a terminal quadratic differential in common, this completes the proof of 
the proposition. 
\end{proof}

\section{Uniqueness of closed regular self-weldings}
\setcounter{equation}{0}
\label{sec:self-welding:uniqueness}

In this section we study uniqueness of closed regular self-weldings of a 
compact bordered Riemann surface $S$. In 
Example~\ref{exmp:hydrodynamic_continuation} we have remarked that some 
positive holomorphic quadratic differential on $S$ satisfying the border length 
condition on $\partial S$ can induce two or more inequivalent closed regular 
self-weldings of $S$. With this in mind we make the following definition. 

\begin{defn}[exceptional quadratic differential]
\label{defn:self-welding:exceptional}
Let $C$ be a union of connected components of $\partial S$, and let 
$\varphi \in A_{L}(S,C)$. If there are two $C$-full $\varphi$-regular 
self-weldings of $S$ that are inequivalent to each other, then $\varphi$ is 
called {\em exceptional\/} for $C$. 
\end{defn}

It is then natural to ask which elements of $A_{L}(S,C)$ are exceptional for 
$C$. To answer the question we introduce a subset $A_{E}(S,C)$ of $A_{L}(S,C)$ 
as follows. 

\begin{defn}[class $A_{E}(S,C)$]
\label{defn:self-welding:A_E(S,C)}
Let $C$ be a union of connected components of $\partial S$. Denote by 
$A_{E}(S,C)$ the set of $\varphi \in A_{L}(S,C)$ such that some component $C'$ 
of $C$ includes four or more horizontal trajectories of $\varphi$ and that 
\begin{equation}
\label{eq:length_assumption}
L_{\varphi}(a)<L_{\varphi}(C')/2
\end{equation}
for all horizontal trajectories $a$ of $\varphi$ on $C'$. Set 
$A_{E}(S)=A_{E}(S,\partial S)$. 
\end{defn}

Clearly, $A_{E}(S)=\bigcup_{C} A_{E}(S,C)$, where the union is taken over all 
components $C$ of $\partial S$. The following theorem asserts in particular 
that for $\varphi \in A_{L}(S) \setminus A_{E}(S)$ there is exactly one closed 
$\varphi$-regular self-welding of $S$ up to equivalence. 

\begin{thm}
\label{thm:exceptional_border_component}
Let $S$ be a compact bordered Riemann surface, and let $C$ be a union of 
components of $\partial S$. Then an element $\varphi$ of $A_{L}(S,C)$ is 
exceptional for $C$ if and only if it belongs to $A_{E}(S,C)$. If this is the 
case, then the family of equivalence classes of $C$-full $\varphi$-regular 
self-weldings of $S$ has the cardinality of the continuum. 
\end{thm}

\begin{proof}
We may suppose from the outset that $C$ is a connected component of 
$\partial S$. Assume first that $L_{\varphi}(a)=L_{\varphi}(C)/2$ for some 
horizontal trajectory $a$ of $\varphi$ on $C$. The endpoints of $a$, which are 
zeros of $\varphi$, divide $C$ into two arcs $a_{1}$ and $a_{2}$ of the same 
$\varphi$-length, where we label them so that $a_{1}^{\circ}=a$. Let 
$\langle R,\iota\rangle$ be a $C$-full $\varphi$-regular self-welding of $S$, 
and let $\psi$ denote the co-welder of $\varphi$. Since $\varphi$ does not 
vanish on $a_{1}^{\circ}$, it follows from 
Proposition~\ref{prop:regular_self-welding} that $\iota_{*}a_{1}$ is a simple 
arc on the weld graph $G_{\iota}$. As the sum of the $\psi$-lengths of edges of 
$G_{\iota}$ is identical with $L_{\varphi}(C)/2$, we deduce that 
$\iota_{*}a_{1}$ exhausts $G_{\iota}$, or $G_{\iota}$ is I-shaped. Therefore, 
$\langle R,\iota\rangle$ is equivalent to the self-welding of $S$ with welder 
$\varphi$ along $(a_{1},a_{2})$. Hence $\varphi$ is nonexceptional for $C$. 

In the rest of the proof we suppose that 
inequalities~\eqref{eq:length_assumption} with $C'$ replaced with $C$ hold for 
all horizontal trajectories $a$ of $\varphi$ on $C$. Then $C$ carries at least 
three horizontal trajectories of $\varphi$ and hence the set $Z$ of zeros of 
$\varphi$ on $C$ contains more than two points. 

Assume that $\card Z=3$ so that $Z=\{p_{1},p_{2},p_{3}\}$, say. The points of 
$Z$ divide $C$ into three arcs $a_{1}$, $a_{2}$ and $a_{3}$, where they are 
labeled so that $p_{k} \not\in a_{k}$ for $k=1,2,3$. Let 
$\langle R,\iota\rangle$ be a $C$-full $\varphi$-regular self-welding of $S$, 
and let $G_{\iota}$ be its weld graph. Since $a_{k}^{\circ}$ contains no zeros 
of $\varphi$ and the self-welding is $\varphi$-regular, it follows 
from~\eqref{eq:length_assumption} that the points $v_{k}:=\iota(p_{k})$, 
$k=1,2,3$, are the end-vertices of $G_{\iota}$ by 
Proposition~\ref{prop:regular_self-welding}, and hence $G_{\iota}$ is Y-shaped. 
Lemma~\ref{lem:C-full:three_points} assures us that $\langle R,\iota\rangle$ is 
determined up to equivalence. This means that $\varphi$ is nonexceptional for 
$C$. 

Suppose next that $\card Z>3$. We consider two cases. If $Z$ contains two 
points $p_{1}$ and $p_{2}$ that divide $C$ into two arcs $a_{1}$ and $a_{2}$ of 
the same $\varphi$-length $L_{\varphi}(C)/2$, then the self-welding 
$\langle R,\iota\rangle$ of $S$ with welder $\varphi$ along $(a_{1},a_{2})$ is 
$C$-full and $\varphi$-regular, and its weld graph is I-shaped. On the other 
hand, inequality~\eqref{eq:length_assumption} implies that each $a_{k}^{\circ}$ 
contains a point $q_{k}$ of $Z$. Take short subarcs $a_{k1}$ and $a_{k2}$ of 
$a_{k}$ emanating from $q_{k}$ so that $L_{\varphi}(a_{11})=L_{\varphi}(a_{12})=
L_{\varphi}(a_{21})=L_{\varphi}(a_{22})$, and let $\langle R',\iota'\rangle$ be 
the welding of $S$ with welder $\varphi$ along $(a_{k1},a_{k2})$, $k=1,2$. The 
two points $p'_{k}:=\iota'(p_{k})$, $k=1,2$, divide the component $C'$ of 
$\partial R'$ containing $p'_{k}$ into two arcs $a'_{1}$ and $a'_{2}$ of the 
same $\psi'$-length, where $\psi' \in A_{+}(R')$ is the co-welder of $\varphi$. 
If $\langle R'',\iota''\rangle$ denotes the self-welding of $R'$ with welder 
$\psi'$ along $(a'_{1},a'_{2})$, then $\langle R'',\iota'' \circ \iota'\rangle$ 
is a $C$-full $\varphi$-regular self-welding of $S$. Since its weld graph 
$G_{\iota'' \circ \iota'}$ has four end-vertices, the self-welding is not 
equivalent to $\langle R,\iota\rangle$. Thus $\varphi$ is exceptional for $C$. 

Finally, assume that no two points of $Z$ divide $C$ into arcs of the same 
$\varphi$-length, where $\card Z>3$. Then we can choose three points $p_{k}$, 
$k=1,2,3$, of $Z$ to divide $C$ into three arcs $a_{k}$, $k=1,2,3$, such that 
inequality~\eqref{eq:length_assumption} holds for $a=a_{k}$, $k=1,2,3$. We then 
apply Lemma~\ref{lem:C-full:three_points} to obtain a $C$-full 
$\varphi$-regular self-welding $\langle R,\iota\rangle$ of $S$ whose weld graph 
$G_{\iota}$ is Y-shaped. Let $p_{4}$ be a point of $Z$ other than $p_{k}$, 
$k=1,2,3$. Take two simple arcs $a_{41}$ and $a_{42}$ of the same 
$\varphi$-length on $C$ emanating from $p_{4}$. We choose them so that neither 
$a_{41}$ nor $a_{42}$ contains any of $p_{k}$, $k=1,2,3$. Let 
$\langle R',\iota'\rangle$ be the self-welding of $S$ with welder $\varphi$ 
along $(a_{41},a_{42})$, and let $\psi' \in A_{+}(R')$ be the co-welder of 
$\varphi$. If the $\varphi$-length of $a_{41}$ is sufficiently small, then the 
three points $p'_{k}:=\iota'(p_{k})$, $k=1,2,3$, divide the component $C'$ of 
$\partial R'$ containing these points into three arcs with $\psi'$-length less 
than $L_{\psi'}(C')/2$. Another application of 
Lemma~\ref{lem:C-full:three_points} yields a $C'$-full $\psi'$-regular 
self-welding $\langle R'',\iota''\rangle$ of $S$ whose weld graph $G_{\iota''}$ 
is Y-shaped. Then $\langle R'',\iota'' \circ \iota'\rangle$ is a $C$-full 
$\varphi$-regular self-welding of $S$. It is not equivalent to 
$\langle R,\iota\rangle$ because the weld graph of 
$\langle R'',\iota'' \circ \iota'\rangle$ is not Y-shaped. Therefore, $\varphi$ 
is exceptional for $C$. 

The last assertion of the theorem is clear form our constructions of 
inequivalent $C$-full $\varphi$-regular self-weldings of $S$. The proof is 
complete. 
\end{proof}

The following is a corollary to the proof of the above theorem. 

\begin{cor}
\label{cor:exceptional_border_component}
Let $S$ be a compact bordered Riemann surface, and let $C$ be a component of 
$\partial S$. Let $\varphi \in A_{E}(S,C)$. 
\begin{list}{{\rm (\roman{claim})}}{\usecounter{claim}
\setlength{\topsep}{0pt}
\setlength{\itemsep}{0pt}
\setlength{\parsep}{0pt}
\setlength{\labelwidth}{\leftmargin}}
\item  For any closed $\varphi$-regular self-welding $\langle R,\iota\rangle$ of $S$ 
the image $\iota(C)$ contains a zero of the co-welder of $\varphi$. 

\item  For some closed $\varphi$-regular self-welding $\langle R,\kappa\rangle$ of 
$S$ the image $\kappa(C)$ is neither I-shaped nor Y-shaped. 

\end{list}
\end{cor}

\begin{exmp}
\label{exmp:g=1:self-welding:uniqueness}
Let $S$ be a compact bordered Riemann surface of genus one. Since nonzero 
holomorphic quadratic differentials on a torus have no zeros, 
Corollary~\ref{cor:exceptional_border_component} implies that 
$A_{E}(S)=\varnothing$. Thus every element $\varphi$ of $A_{L}(S)$ yields 
exactly one closed $\varphi$-regular self-welding of $S$. 
\end{exmp}

\begin{rem}
Obstruct problems raised by Fehlmann and Gardiner in \cite{FG1995} are closely 
related to continuations of Riemann surfaces. The uniqueness theorem in 
\cite{FG1995} is false in general due to the existence of exceptional quadratic 
differentials as Sasai \cite{Sasai2005} pointed out. In \cite{Sasai2006} she 
gave a sufficient condition for positive quadratic differentials to be 
exceptional. 
\end{rem}

Now, let ${\bm S}=[S,\eta]$ be a marked compact bordered Riemann surface of 
positive genus. An element ${\bm \varphi}=[\varphi,\eta]$ of $A_{L}({\bm S})$ 
is called {\em exceptional\/} if $\varphi \in A_{E}(S)$. Let $A_{E}({\bm S})$ 
stand for the set of exceptional elements of $A_{L}({\bm S})$. 

Let ${\bm R}_{0}=[R_{0},\theta_{0}]$ be a marked nonanalytically finite open 
Riemann surface of positive genus. Set 
$A_{E}({\bm R}_{0})=\breve{\bm \iota}_{0}^{*}A_{E}(\breve{\bm R}_{0})$, where 
$(\breve{\bm R}_{0},\breve{\bm \iota}_{0})$ denotes the natural compact 
continuation of ${\bm R}_{0}$. Elements of $A_{E}({\bm R}_{0})$ are called 
{\em exceptional}. 

\begin{exmp}
\label{exmp:g=1:exceptional_border_components}
We consider the case of genus one. We know that $A_{E}({\bm R}_{0})$ is 
empty (see Example~\ref{exmp:g=1:self-welding:uniqueness}). Thus we have a 
well-defined mapping of $PA_{L}({\bm R}_{0})$ onto 
$\partial\mathfrak{M}({\bm R}_{0})$. To see that it is injective take 
${\bm \varphi}_{j} \in A_{L}({\bm R}_{0})$, $j=1,2$, and suppose that they 
induce homotopically consistent conformal embeddings ${\bm \iota}_{j}$ of 
${\bm R}_{0}$ into the same marked torus 
${\bm R} \in \partial\mathfrak{M}({\bm R}_{0})$. It then follows from 
Proposition~\ref{prop:uniqueness:boundary_of_M(R_0)} that 
${\bm \varphi}_{j}={\bm \iota}_{j}^{*}{\bm \psi}_{j}$ for some 
${\bm \psi}_{j} \in A({\bm R})$. Since $A({\bm R})$ is one-dimensional, there 
is a nonzero complex number $c$ such that ${\bm \psi}_{2}=c{\bm \psi}_{1}$ and 
hence ${\bm \varphi}_{2}=c{\bm \varphi}_{1}$. As 
${\bm \varphi}_{j} \in A_{+}({\bm R}_{0})$, $j=1,2$, we conclude that $c>0$ so 
that ${\bm \varphi}_{1}$ and ${\bm \varphi}_{2}$ represent the same element of 
$PA_{L}({\bm R}_{0})$. We have shown that the correspondence between 
$PA_{L}({\bm R}_{0})$ onto $\partial\mathfrak{M}({\bm R}_{0})$ is bijective. In 
particular, since $A_{E}({\bm R}_{0})=\varnothing$, we deduce that for each 
${\bm R} \in \partial\mathfrak{M}({\bm R}_{0})$ there exists exactly one 
homotopically consistent conformal embedding of ${\bm R}_{0}$ into $\bm R$. 
We have thus given alternative proofs of \cite[Theorems~4 and~5]{Shiba1987}. 
\end{exmp}

If we set 
$$
\CEmb_{\bm \varphi}({\bm R}_{0})=
\bigcup_{{\bm R} \in \mathfrak{M}({\bm R}_{0})}
\CEmb_{\bm \varphi}({\bm R}_{0},{\bm R}),
$$
we see that ${\bm \varphi} \in A_{L}({\bm R}_{0})$ is exceptional if and only 
if $\card\CEmb_{\bm \varphi}({\bm R}_{0})>1$. Also, set 
$$
\CEmb_{L}({\bm R}_{0})=
\bigcup_{{\bm \varphi} \in A_{L}({\bm R}_{0})} \CEmb_{\bm \varphi}({\bm R}_{0})
\quad \text{and} \quad
\CEmb_{E}({\bm R}_{0})=
\bigcup_{{\bm \varphi} \in A_{E}({\bm R}_{0})} \CEmb_{\bm \varphi}({\bm R}_{0}).
$$

\begin{thm}
\label{thm:exceptional:approximation}
Let ${\bm R}_{0}$ be a marked nonanalytically finite open Riemann surface of 
positive genus. If ${\bm \varphi} \in A_{E}({\bm R}_{0})$, then some element in 
$\CEmb_{\bm \varphi}({\bm R}_{0})$ cannot be approximated by any sequence in 
$\CEmb_{L}({\bm R}_{0}) \setminus \CEmb_{E}({\bm R}_{0})$. 
\end{thm}

\begin{proof}
Let ${\bm R}_{0}=[R_{0},\theta_{0}]$ and ${\bm \varphi}=[\varphi,\theta_{0}]$. 
Take a natural compact continuation $(\breve{R}_{0},\breve{\iota}_{0})$ of 
$R_{0}$ together with $\breve{\varphi} \in A_{E}(\breve{R}_{0})$ for which 
$\varphi=\breve{\iota}_{0}^{*}\breve{\varphi}$. By 
Corollary~\ref{cor:exceptional_border_component}~(ii) the weld graph of some 
closed $\breve{\varphi}$-regular self-welding $\langle R,\breve{\iota}\rangle$ 
of $\breve{R}_{0}$ has a component that is not I-shaped or Y-shaped. On the 
other hand, if 
$\breve{\varphi}' \in A_{L}(\breve{R}_{0}) \setminus A_{E}(\breve{R}_{0})$, 
then each component of the weld graph of any closed $\breve{\varphi}'$-regular 
self-welding of $\breve{R}_{0}$ is I-shaped or Y-shaped by 
Theorem~\ref{thm:exceptional_border_component} and 
Lemma~\ref{lem:C-full:three_points}. Consequently, no sequences in 
$\CEmb_{L}({\bm R}_{0}) \setminus \CEmb_{E}({\bm R}_{0})$ converge to 
${\bm \iota}:=[\breve{\iota} \circ \breve{\iota}_{0},\theta_{0},
\breve{\iota} \circ \breve{\iota}_{0} \circ \theta_{0}] \in
\CEmb_{\bm \varphi}({\bm R}_{0},{\bm R})$, where 
${\bm R}=[R,\breve{\iota} \circ \breve{\iota}_{0} \circ \theta_{0}]$. 
\end{proof}

As an application of the theorem, we prove 
Theorem~\ref{thm:main:CEmb(R_0,R)}~(iii). To this end we prepare two lemmas. 
Let $R$ be a closed Riemann surface of positive genus. For a nontrivial complex 
vector space $V$ of holomorphic $1$-forms on $R$ set 
$$
\ord_{p}V=\{\ord_{p}\omega \mid \omega \in V \setminus \{0\}\}
$$
for $p \in R$, where $\ord_{p}\omega$ denotes the order of $\omega$ at $p$. If 
$\ord_{p}V \neq \{0,1,\dots,d-1\}$, where $d=\dim V$, then $p$ is called a 
{\em Weierstrass point\/} for $V$. Otherwise, $p$ is said to be a 
{\em non-Weierstrass point}. 

\begin{lem}
\label{lem:1-forms:Weierstrass_point}
Let $V$ be a nontrivial complex vector space of holomorphic $1$-forms on a 
closed Riemann surface $R$ of genus $g>1$. Then the set of Weierstrass points 
for $V$ is nonempty and finite. 
\end{lem}

\begin{proof}
Let $\omega_{j}=\omega_{j}(z)\,dz$, $j=1,2,\dots,d$, be a basis of $V$. If 
$W(z)$ is the Wronskian of $\omega_{j}(z)$, that is, 
$W(z)=\det(\omega_{j}^{(k-1)}(z))_{j,k=1,\dots,d}$, then 
$W=W(z)\,dz^{d(d+1)/2}$ is a holomorphic $d(d+1)/2$-differential on $R$. Let 
$p \in R$ and take a local coordinate $z$ around $p$ with $z(p)=0$. If 
$\ord_{p}V=\{\nu_{1},\dots,\nu_{d}\}$ and $\nu=\sum_{j}\nu_{j}-d(d-1)/2$, then 
the power series expansion of $W$ is of the form $W(z)=cz^{\nu}+\cdots$ with 
$c \neq 0$. In particular, $W$ is nonzero. Moreover, $p$ is a Weierstrass point 
for $V$ if and only if $p$ is a zero of $W$. This proves the lemma. 
\end{proof}

\begin{lem}
\label{lem:special_1-form}
Let $S$ be a compact bordered Riemann surface of genus $g \geqq 3$, and let 
$\hat{S}$ denote its double. Then there is a holomorphic $1$-form $\omega$ on 
$\hat{S}$ such that 
\begin{list}{{\rm (\roman{claim})}}{\usecounter{claim}
\setlength{\topsep}{0pt}
\setlength{\itemsep}{0pt}
\setlength{\parsep}{0pt}
\setlength{\labelwidth}{\leftmargin}}
\item $\im\omega=0$ along $\partial S$, 

\item $\int_{C} \omega=0$ for each component $C$ of $\partial S$,  and 

\item $\omega$ has a zero on $S^{\circ}$ and four zeros on some component of 
$\partial S$. 
\end{list}
\end{lem}

\begin{proof}
Let $V$ be the vector space of holomorphic $1$-forms $\omega$ on $\hat{S}$ 
satisfying condition~(ii). Then $\dim V=2g$; note that 
$\int_{\partial S} \omega=0$ for all holomorphic $1$-forms $\omega$ on 
$\hat{S}$. Also, let $V_{\mathbb{R}}$ be the {\em real\/} vector space of 
$\omega \in V$ possessing property~(i). Let $J$ be the anti-conformal 
involution of $\hat{S}$ fixing $\partial S$ pointwise. Then 
$\omega \mapsto \overline{J^{*}\omega}$ is an $\mathbb{R}$-linear isomorphism 
of $V$ onto itself and $\sigma(\omega):=(\omega+\overline{J^{*}\omega})/2$ 
belongs to $V_{\mathbb{R}}$ for $\omega \in V$. 

Take a non-Weierstrass point $p_{0} \in S^{\circ}$ for $V$, and let $V_{p_{0}}$ 
be the set of $1$-forms in $V$ vanishing at $p_{0}$. Then 
$\dim V_{p_{0}}=2g-1$. If $V_{0}$ denotes the set of $1$-forms in $V_{p_{0}}$ 
vanishing at $J(p_{0})$, then $d:=\dim V_{0} \geqq 2g-2$. Note that for any $
\omega \in V_{0}$ the point $p_{0}$ is a zero of $\sigma(\omega)$. 

Since $g \geqq 3$, we have $d \geqq 4$. Fix a component $C_{0}$ of 
$\partial S$. Let $p_{1} \in C_{0}$ be a non-Weierstrass point for $V_{0}$, and 
set $V_{1}^{(k)}=
\{\omega \in V_{0} \mid \ord_{p_{1}}\omega \geqq k \text{ or } \omega=0\}$ for 
$k=1,2$. Then $\dim V_{1}^{(k)}=d-k$. Take a non-Weierstrass point 
$p_{2} \in C_{0} \setminus \{p_{1}\}$ for both $V_{1}^{(1)}$ and $V_{1}^{(2)}$, 
and set $V_{2}^{(k)}=\{\omega \in V_{1}^{(1)} \mid
\ord_{p_{2}}\omega \geqq k \text{ or } \omega=0\}$ for $k=1,2$. Note that 
$\dim V_{2}^{(k)}=d-1-k$ and that $\dim(V_{1}^{(2)} \cap V_{2}^{(1)})=d-3$. 
Finally, choose a non-Weierstrass point 
$p_{3} \in C_{0} \setminus \{p_{1},p_{2}\}$ for $V_{2}^{(k)}$, $k=1,2$, and 
$V_{2}^{(1)} \cap V_{1}^{(2)}$, and define $V_{3}^{(k)}=\{\omega \in V_{2}^{(1)}
\mid \ord_{p_{3}}\omega \geqq k \text{ or } \omega=0\}$ for $k=1,2$. Then 
$\dim V_{3}^{(1)}=d-3>\dim V'$ for $V'=V_{3}^{(2)}$, 
$V_{2}^{(2)} \cap V_{3}^{(1)}$ and 
$V_{2}^{(1)} \cap V_{1}^{(2)} \cap V_{3}^{(1)}$. Consequently, 
$\ord_{p_{l}}\omega_{1}=1$, $l=1,2,3$, for some nonzero $\omega_{1} \in V_{0}$. 
Set $\omega=\sigma(c\omega_{1})$, where $c \in \partial\mathbb{D}$. Then 
$\omega$ satisfies~(i) and~(ii). We can choose $c$ so that $\omega$ has simple 
zeros at $p_{l}$, $l=1,2,3$. Since $\im\omega=0$ along $C_{0}$ and 
$\int_{C_{0}} \omega=0$, the $1$-form $\omega$ has one more zero of odd order 
on $C_{0}$. Hence $\omega$ satisfies~(i), (ii) and~(iii). 
\end{proof}

We are now ready to prove Theorem~\ref{thm:main:CEmb(R_0,R)}~(iii). 

\begin{proof}[Proof of Theorem~\ref{thm:main:CEmb(R_0,R)}~(iii)]
Let $(\breve{\bm R}_{0},\breve{\bm \iota}_{0})$, where 
$\breve{\bm R}_{0}=[\breve{R}_{0},\breve{\theta}_{0}]$ and 
$\breve{\bm \iota}_{0}=[\breve{\iota}_{0},\theta_{0},\breve{\theta}_{0}]$, be 
the natural compact continuation of ${\bm R}_{0}=[R_{0},\theta_{0}]$. Let 
$\hat{R}_{0}$ be the double of $\breve{R}_{0}$. We apply 
Lemma~\ref{lem:special_1-form} to obtain a holomorphic $1$-form $\omega$ on 
$\hat{R}_{0}$ that has a zero on $(\breve{R}_{0})^{\circ}$ and four zeros on a 
component of $\partial\breve{R}_{0}$. Set 
${\bm \varphi}=[\breve{\iota}_{0}^{*}(\omega^{2}),\theta_{0}]$. Then, since 
${\bm \varphi} \in A_{E}({\bm R}_{0})$ by 
Theorem~\ref{thm:exceptional_border_component}, 
Theorem~\ref{thm:exceptional:approximation} implies that there is a closed 
$\bm \varphi$-regular self-welding continuation $({\bm R},{\bm \iota})$ of 
${\bm R}_{0}$ such that there are no sequences 
$\{({\bm R}_{n},{\bm \iota}_{n})\}$ of closed regular self-welding 
continuations of ${\bm R}_{0}$ for which $\{{\bm \iota}_{n}\}$ is a sequence in 
$\CEmb_{L}({\bm R}_{0}) \setminus \CEmb_{E}({\bm R}_{0})$ converging to 
$\bm \iota$. 

We claim that $\bm R$ is an interior point $\mathfrak{M}_{E}({\bm R}_{0})$ in 
$\partial\mathfrak{M}({\bm R}_{0})$. If not, then some sequence 
$\{{\bm R}'_{n}\}$ in 
$\partial\mathfrak{M}({\bm R}_{0}) \setminus \mathfrak{M}_{E}({\bm R}_{0})$ 
would converge to $\bm R$. Take closed regular 
self-welding continuations $({\bm R}'_{n},{\bm \iota}'_{n})$ of ${\bm R}_{0}$. 
By Proposition~\ref{prop:self-welding:sequence} there is a closed 
regular self-welding continuation $({\bm R},{\bm \iota}')$ of ${\bm R}_{0}$ 
such that a subsequence of $\{{\bm \iota}'_{n}\}$ converges to ${\bm \iota}'$. 
Since $\breve{\iota}_{0}^{*}(\omega^{2})$ has a zero on $R_{0}$, it follows 
from Bourque \cite[Remark~4.3]{Bourque2018} that ${\bm \iota}'={\bm \iota}$, 
which is impossible by the choice of $({\bm R},{\bm \iota})$. This completes 
the proof. 
\end{proof}

\begin{rem}
Let $\bm \varphi$ be as in the above proof. Then 
$\CEmb_{\mathrm{hc}}({\bm R}_{0},{\bm R})$ is a singleton for any 
${\bm R} \in \mathfrak{M}_{\bm \varphi}({\bm R}_{0})$. Since 
$\card\CEmb_{\bm \varphi}({\bm R}_{0})=2^{\aleph_{0}}$, we know that 
$\card\mathfrak{M}_{\bm \varphi}({\bm R}_{0})=2^{\aleph_{0}}$. 
\end{rem}

\section{Fillings for marked open Riemann surfaces}
\setcounter{equation}{0}
\label{sec:fillings_for_R_0}

The aim of the present section is to introduce some key tools used in the 
following sections. Let $\{{\bm S}_{t}\}_{t \in [0,1]}$ be a one-parameter 
family of marked open Riemann surfaces of genus $g$. We say that 
$\mathfrak{M}({\bm S}_{t})$ {\em shrinks continuously\/} if 
\begin{list}{{\rm (\roman{claim})}}{\usecounter{claim}
\setlength{\topsep}{0pt}
\setlength{\itemsep}{0pt}
\setlength{\parsep}{0pt}
\setlength{\labelwidth}{\leftmargin}}
\item $\mathfrak{M}({\bm S}_{t_{1}}) \supset \mathfrak{M}({\bm S}_{t_{2}})$ for 
$t_{1}<t_{2}$, and 

\item for each $\varepsilon>0$ and $t \in [0,1]$ there is $\delta>0$ such that 
$\mathfrak{M}({\bm S}_{t_{1}})$ is included in the $\varepsilon$-neighborhood 
of $\mathfrak{M}({\bm S}_{t_{2}})$, where 
$t_{1}=\max\{t-\delta,0\}$ and $t_{2}=\min\{t+\delta,1\}$. 
\end{list}
If, in addition, $\mathfrak{M}({\bm S}_{1})$ is a singleton, then 
$\mathfrak{M}({\bm S}_{t})$ is said to shrink continuously {\em to a point}. 

\begin{lem}
\label{lem:continuity_of_M(R_0)}
Let ${\bm S}_{j}$, $j=1,2$, be marked open Riemann surfaces of genus $g$. If 
there is a homotopically consistent $K$-quasiconformal embedding of 
${\bm S}_{1}$ into ${\bm S}_{2}$, then 
$$
\mathfrak{M}({\bm S}_{2}) \subset \mathfrak{M}_{K}({\bm S}_{1}). 
$$
\end{lem}

\begin{proof}
Let ${\bm f} \in \QCEmb_{\mathrm{hc}}({\bm S}_{1},{\bm S}_{2})$ with 
$K({\bm f}) \leqq K$. If $({\bm R},{\bm \iota})$ is a closed continuation of 
${\bm S}_{2}$, then 
${\bm \iota} \circ {\bm f} \in \QCEmb_{\mathrm{hc}}({\bm S}_{1},{\bm R})$ with 
$K({\bm \iota} \circ {\bm f}) \leqq K$. Hence $\bm R$ belongs to 
$\mathfrak{M}_{K}({\bm S}_{1})$. 
\end{proof}

\begin{prop}
\label{prop:filling:interior}
Let $\{{\bm S}_{t}\}_{t \in [0,1]}$ be a one-parameter family of marked finite 
open Riemann surfaces of genus $g$. If\/ $\mathfrak{M}({\bm S}_{t})$ shrinks 
continuously, then 
$$
\Int\mathfrak{M}({\bm S}_{t})=
\bigcup_{u \in (t,1]} \Int\mathfrak{M}({\bm S}_{u})
$$
for $t \in [0,1)$. 
\end{prop}

\begin{proof}
Fix $t \in [0,1)$. Since 
$\Int\mathfrak{M}({\bm S}_{u}) \subset \Int\mathfrak{M}({\bm S}_{t})$ for 
$u \in (t,1]$, we know that 
$$
\bigcup_{u \in (t,1]}\Int\mathfrak{M}({\bm S}_{u}) \subset
\Int\mathfrak{M}({\bm S}_{t}).
$$
To show the converse inclusion relation take an arbitrary interior point 
$\bm R$ of $\mathfrak{M}({\bm S}_{t})$. Then for some $\varepsilon>0$ the 
neighborhood $\mathfrak{B}_{4\varepsilon}({\bm R})$ of $\bm R$ is included in 
$\mathfrak{M}({\bm S}_{t})$. We can choose $u \in (t,1]$ sufficiently near to 
$t$ so that $\mathfrak{M}({\bm S}_{t}) \subset
\mathfrak{M}_{e^{4\varepsilon}}({\bm S}_{u})$. We claim that $\bm R$ is an 
interior point of $\mathfrak{M}({\bm S}_{u})$. If not, then $\bm R$ would lie 
on an Ioffe ray $\bm r$ of ${\bm S}_{u}$ with ${\bm r}(\tau)={\bm R}$, where 
$\tau=d_{T}({\bm R},\mathfrak{M}({\bm S}_{u})) \leqq 2\varepsilon$ by 
Proposition~\ref{prop:space:qc_embedding}. Since 
$$
d_{T}({\bm r}(3\varepsilon),{\bm R})=d_{T}({\bm r}(3\varepsilon),{\bm r}(\tau))=
3\varepsilon-\tau<4\varepsilon,
$$
we have ${\bm r}(3\varepsilon) \in \mathfrak{B}_{4\varepsilon}({\bm R}) \subset
\mathfrak{M}({\bm S}_{t}) \subset \mathfrak{M}_{e^{4\varepsilon}}({\bm S}_{u})$ 
and hence 
$d_{T}({\bm r}(3\varepsilon),\mathfrak{M}({\bm S}_{u})) \leqq 2\varepsilon$. 
This contradicts the identity 
$d_{T}({\bm r}(3\varepsilon),\mathfrak{M}({\bm S}_{u}))=3\varepsilon$, and the 
proof is complete. 
\end{proof}

Let ${\bm R}_{0}=[R_{0},\theta_{0}] \in \mathfrak{F}_{g}$ be open and 
nonanalytically finite. We construct two one-parameter families of 
continuations of ${\bm R}_{0}$. In the following examples 
$(\breve{\bm R}_{0},\breve{\bm \iota}_{0})$ denotes the natural compact 
continuation of ${\bm R}_{0}$, where 
$\breve{\bm R}_{0}=[\breve{R}_{0},\breve{\theta}_{0}]$ and 
$\breve{\bm \iota}_{0}=[\breve{\iota}_{0},\theta_{0},\breve{\theta}_{0}]$. 

In general, for $r>0$ let $\mathbb{D}_{r}$ denote the open disk in $\mathbb{C}$ 
of radius $r$ centered at $0$. Its closure is denoted by 
$\bar{\mathbb{D}}_{r}$. Set $\bar{\mathbb{D}}_{0}=\{0\}$ for convenience. We 
abbreviate $\mathbb{D}_{1}$ and $\bar{\mathbb{D}}_{1}$ to $\mathbb{D}$ and 
$\bar{\mathbb{D}}$, respectively. 

\begin{exmp}[circular filling]
\label{exmp:circular_filling}
Let $C_{1},\ldots,C_{n_{0}}$ be the connected components of the border 
$\partial \breve{R}_{0}$. Take doubly connected domains $U_{j}$, 
$j=1,\ldots,n_{0}$, on $(\breve{R}_{0})^{\circ}$ such that $C_{j}$ is a 
component of $\partial U_{j}$ and that $U_{j}$ is mapped onto a fixed annulus 
$\mathbb{D}_{r_{0}} \setminus \bar{\mathbb{D}}$ by a conformal homeomorphism 
$z_{j}$ with $|z_{j}|=1$ on $C_{j}$, where $r_{0}>1$. We attach 
$\bar{\mathbb{D}}$ to each $C_{j}$ by identifying $p \in C_{j}$ with 
$z_{j}(p) \in \partial\mathbb{D}$ to obtain a closed Riemann surface $W$ of 
genus $g$. We consider $\breve{R}_{0}$ as a closed subdomain of $W$, and denote 
by $D_{j}$ the component of $W \setminus \breve{R}_{0}$ with 
$C_{j}=\partial D_{j}$. Then $z_{j}$ is extended to a conformal mapping of 
$U_{j} \cup \bar{D}_{j}$ onto $\mathbb{D}_{r_{0}}$. Considering 
$\breve{\theta}_{0}$ as a $g$-handle mark of $W$ as well, set 
${\bm W}=[W,\breve{\theta}_{0}]$. Then ${\bm W} \in \mathfrak{M}({\bm R}_{0})$. 

For each $t \in [0,1]$ we construct a subsurface $W_{t}$ of $W$ homeomorphic to 
$(\breve{R}_{0})^{\circ}$ as follows: 
$$
W_{t}=W \setminus\bigcup_{j=1}^{n_{0}} z_{j}^{-1}(\bar{\mathbb{D}}_{1-t})
$$
Define ${\bm W}_{t}=[W_{t},\breve{\theta}_{0}]$ for $t \in [0,1]$. Regarding 
$\breve{\bm \iota}_{0}$ as a homotopically consistent conformal embedding 
${\bm \epsilon}_{t}$ of ${\bm R}_{0}$ into ${\bm W}_{t}$, we call 
$\{({\bm W}_{t},{\bm \epsilon}_{t})\}_{t \in [0,1]}$ a {\em circular filling\/} 
for ${\bm R}_{0}$. Note that $W_{0} \setminus \breve{\iota}_{0}(R_{0})$ is a 
finite set. 
\end{exmp}

\begin{exmp}[linear filling]
\label{exmp:linear_filling}
As in Example~\ref{exmp:full_self-welding}, take 
$\breve{\varphi} \in M_{+}(\breve{R}_{0})$, and divide each border component 
$C_{j}$ into two subarcs $a_{2j-1}$ and $a_{2j}$ of the same 
$\breve{\varphi}$-length $L_{j}:=L_{\breve{\varphi}}(C_{j})/2$. Let 
$\langle W,\epsilon\rangle$ be the closed self-welding of $\breve{R}_{0}$ with 
welder $\breve{\varphi}$ along $(a_{2j-1},a_{2j})$, $j=1,\ldots,n_{0}$. and let 
$\psi \in M(W)$ be the co-welder of $\breve{\varphi}$. The arcs $a_{2j-1}$ and 
$a_{2j}$ are projected to a simple arc $a'_{j}$ on $W$, which is parametrized 
with respect to $\psi$-length. Set 
$$
W_{t}=W \setminus \bigcup_{j=1}^{n_{0}} a'_{j}([0,(1-t)L_{j}])
$$
for $t \in [0,1]$. Let $\epsilon_{t}$ denote the conformal embedding of $R_{0}$ 
into $W_{t}$ induced by $\epsilon \circ \breve{\iota}_{0}$. The continuation 
$(W_{t},\epsilon_{t})$ of $R_{0}$ yields a continuation 
$({\bm W}_{t},{\bm \epsilon}_{t})$ of ${\bm R}_{0}$. We call 
$\{({\bm W}_{t},{\bm \epsilon}_{t})\}_{t \in [0,1]}$ a {\em linear filling\/} 
for ${\bm R}_{0}$. Again, $W_{0} \setminus \breve{\iota}_{0}(R_{0})$ is a 
finite set.
\end{exmp}

Fillings introduced in the above examples provide us with continuity methods: 

\begin{prop}
\label{prop:filling:shrink_continuously}
Let ${\bm R}_{0}$ be a marked nonanalytically finite open Riemann surface of 
genus $g$. If $\{({\bm W}_{t},{\bm \epsilon}_{t})\}_{t \in [0,1]}$ is a 
circular or linear filling for ${\bm R}_{0}$, then\/ 
$\mathfrak{M}({\bm W}_{t})$ shrinks continuously to a point. 
\end{prop}

\begin{proof}
It is clear that ${\bm W}_{1}$ is analytically finite and hence 
$\mathfrak{M}({\bm W}_{1})$ is a singleton, say, $\{{\bm W}\}$. If 
$t_{1}<t_{2}$, then 
$\CEmb_{\mathrm{hc}}({\bm W}_{t_{1}},{\bm W}_{t_{2}}) \neq \varnothing$, and 
hence $\mathfrak{M}({\bm W}_{t_{2}}) \subset \mathfrak{M}_{1}({\bm W}_{t_{1}})=
\mathfrak{M}({\bm W}_{t_{1}})$ by Lemma~\ref{lem:continuity_of_M(R_0)}. 

To show that $\mathfrak{M}({\bm W}_{t})$ shrinks continuously take 
$\varepsilon>0$ and $t \in [0,1]$. By Proposition~\ref{prop:space:qc_embedding} 
we are required to find $\delta>0$ such that $\mathfrak{M}({\bm W}_{t_{1}})
\subset\mathfrak{M}_{e^{2\varepsilon}}({\bm W}_{t_{2}})$, where 
$t_{1}=\max\{t-\delta,0\}$ and $t_{2}=\min\{t+\delta,1\}$. For $t \in [0,1)$ we 
can choose a desired $\delta$ with the aid of 
Lemma~\ref{lem:continuity_of_M(R_0)}. To verify the existence of $\delta$ for 
$t=1$ we have only to show that 
$\mathfrak{M}({\bm W}_{u}) \subset \mathfrak{B}_{\varepsilon}({\bm W})$ for all 
$u$ sufficiently near $1$. If there were no such $\delta$, then  we could find 
a point $\bm R$ in $\bigcap_{u<1} \mathfrak{M}({\bm W}_{u}) \setminus
\mathfrak{B}_{\varepsilon}({\bm W})$ as $\{\mathfrak{M}({\bm W}_{u}) \setminus
\mathfrak{B}_{\varepsilon}({\bm W})\}_{u<1}$ would be a family of compact sets 
with finite intersection property. Choose 
${\bm \iota}_{u} \in \CEmb_{\mathrm{hc}}({\bm W}_{u},{\bm R})$ and 
${\bm h}_{u} \in \Homeo^{+}_{\mathrm{hc}}({\bm W}_{1},{\bm W}_{u})$ for each 
$u \in (0,1)$ so that $\{{\bm h}_{u}\}_{u}$ converges to 
${\bm 1}_{{\bm W}_{1}}$ as $u \to 1$. We apply 
Lemma~\ref{lem:compact_continuation:sequence} to obtain a sequence $\{u_{n}\}$ 
converging to $1$ for which the sequence $\{{\bm \iota}_{u_{n}}\}$ converges to 
a homotopically consistent conformal embedding of ${\bm W}_{1}$ into $\bm R$. 
Then ${\bm R} \in \mathfrak{M}({\bm W}_{1})$, or 
${\bm R}={\bm W} \in \mathfrak{B}_{\varepsilon}({\bm W})$, which is absurd. 
This finishes the proof of the proposition. 
\end{proof}

We give an application of our methods. Combinations of circular and linear 
fillings give varieties of conformal embeddings. 

\begin{prop}
\label{prop:uniqueness:Int_M(R_0)}
If ${\bm R} \in \Int\mathfrak{M}({\bm R}_{0})$, then 
$\card\CEmb_{\mathrm{hc}}({\bm R}_{0},{\bm R})=2^{\aleph_{0}}$. 
\end{prop}

\begin{proof}
We use the notations in Example~\ref{exmp:linear_filling}. Thus 
$\{({\bm W}_{t},{\bm \epsilon}_{t})\}_{t \in [0,1]}$ is a linear filling for 
${\bm R}_{0}=[R_{0},\theta_{0}]$. For each $t \in [0,1)$ take a circular 
filling $\{({\bm W}_{t}^{(u)},{\bm \epsilon}_{t}^{(u)})\}_{u \in [0,1]}$ for 
${\bm W}_{t}$. Let ${\bm R}=[R,\theta] \in \Int\mathfrak{M}({\bm R}_{0})=
\Int\mathfrak{M}({\bm W}_{0})$. Since $\mathfrak{M}({\bm W}_{t})$ shrinks 
continuously by Proposition~\ref{prop:filling:shrink_continuously}, it follows 
from Proposition~\ref{prop:filling:interior} that 
${\bm R} \in \Int\mathfrak{M}({\bm W}_{\delta})$ for some $\delta \in (0,1)$. 
Another application of Propositions~\ref{prop:filling:interior} 
and~\ref{prop:filling:shrink_continuously} assures us that for each 
$t \in (0,\delta)$ there is $u_{t} \in (0,1)$ such that 
${\bm R} \in \Int\mathfrak{M}({\bm W}_{t}^{(u_{t})})$. Choose 
${\bm \kappa}_{t} \in \CEmb_{\mathrm{hc}}({\bm W}_{t}^{(u_{t})},{\bm R})$ and 
set ${\bm \iota}_{t}=
{\bm \kappa}_{t} \circ {\bm \epsilon}_{t}^{(u_{t})} \circ {\bm \epsilon}_{t}$ 
to obtain continuations $({\bm R},{\bm \iota}_{t})$, $t \in (0,\delta)$, of 
${\bm R}_{0}$. Observe that through $\iota_{t}$, where 
$(\iota_{t},\theta_{0},\theta) \in{\bm \iota}_{t}$, each pair of arcs 
$a_{2j-1}^{(t)}:=a_{2j-1}([0,(1-t)L_{j}])$ and 
$a_{2j}^{(t)}:=a_{2j}([0,(1-t)L_{j}])$ on $C_{j}$ yields a simple arc on $R$ 
while $C_{j} \setminus (a_{2j-1}^{(t)} \cup a_{2j}^{(t)})$ gives rise to a 
simple loop on $R$. Consequently, the continuations 
$({\bm R},{\bm \iota}_{t})$, $t \in (0,\delta)$, of ${\bm R}_{0}$ are distinct 
from one another. This completes the proof. 
\end{proof}

\begin{rem}
In the case of genus one alternative proofs of the proposition are found in 
\cite[Corollary~2]{Masumoto1994} and \cite{SchS1998} based on area theorems in 
\cite{Shiba1993}. 
\end{rem}

We conclude this section by giving one more application of circular fillings. 
We apply it in \S\ref{sec:M(R_0):geometry}. 

\begin{prop}
\label{prop:circular_filling:inerior}
If $\{({\bm W}_{t},{\bm \epsilon}_{t})\}_{t \in [0,1]}$ is a circular filling 
for ${\bm R}_{0}$, then 
$$
\mathfrak{M}({\bm W}_{u}) \subset \Int\mathfrak{M}({\bm W}_{t})
$$
for $0 \leqq t<u \leqq 1$. 
\end{prop}

\begin{proof}
We use the notations in Example~\ref{exmp:circular_filling}. 
Let $({\bm R},{\bm \iota})=([R,\theta],[\iota,\breve{\theta_{0}},\theta])$ be a 
closed continuation of ${\bm W}_{u}$, where $u \in (0,1]$. If $0 \leqq t<u$, 
then $W_{t}$ is a subsurface of $W_{u}$. Since $W_{u} \setminus W_{t}$ is of 
positive area, so is $R \setminus \iota(W_{t})$. 
Proposition~\ref{prop:uniqueness:boundary_of_M(R_0)} thus implies that $\bm R$ 
cannot lie on the boundary of $\mathfrak{M}({\bm W}_{t})$. Hence 
${\bm R} \in \Int\mathfrak{M}({\bm W}_{t})$, as desired. 
\end{proof}

\section{Maximal sets for measured foliations}
\setcounter{equation}{0}
\label{sec:maximal_set}

Let ${\bm R}_{0}$ be a marked finite open Riemann surface of genus $g$. In this 
section we give a homeomorphism of $A_{L}({\bm R}_{0})$ onto 
$\mathscr{MF}(\Sigma_{g}) \setminus \{0\}$ explicitly to establish 
Theorem~\ref{thm:main:measured_foliation}. Then we prove 
Theorem~\ref{thm:main:Ext}. 

\begin{defn}[maximal set]
Let $\mathcal{F} \in \mathscr{MF}(\Sigma_{g})$, and let $\mathfrak{K}$ be a 
compact subset of $\mathfrak{T}_{g}$. A point of $\mathfrak{K}$ where the 
restriction of $\Ext_{\mathcal{F}}$ to $\mathfrak{K}$ attains its maximum is 
referred to as a {\em maximal point\/} for $\mathcal{F}$ on $\mathfrak{K}$. We 
call the set of maximal points for $\mathcal{F}$ on $\mathfrak{K}$ the 
{\em maximal set\/} for $\mathcal{F}$ on $\mathfrak{K}$. 
\end{defn}

Since $\mathfrak{M}({\bm R}_{0})$ is compact by 
Corollary~\ref{cor:Oikawa:compact}, we can speak of maximal points and sets for 
$\mathcal{F}$ on $\mathfrak{M}({\bm R}_{0})$. Denote by 
$\mathfrak{M}_{\mathcal{F}}({\bm R}_{0})$ the maximal set for $\mathcal{F}$ on 
$\mathfrak{M}({\bm R}_{0})$. Then we have the following proposition. 

\begin{prop}
\label{prop:extremal_length:maximal_point}
Let ${\bm R}_{0} \in \mathfrak{F}_{g}$ be finite and open, and let 
$\mathcal{F} \in \mathscr{MF}(\Sigma_{g}) \setminus \{0\}$. Then any point 
$\bm R \in \mathfrak{M}_{\mathcal{F}}({\bm R}_{0})$ lies on 
$\partial\mathfrak{M}({\bm R}_{0})$, and\/ ${\bm r}_{\bm R}[{\bm \psi}]$ is an 
Ioffe ray of ${\bm R}_{0}$, where ${\bm \psi}={\bm Q}_{{\bm R}}(\mathcal{F})$. 
Moreover, the pull-back ${\bm \iota}^{*}{\bm \psi}$ of $\bm \psi$ by 
${\bm \iota} \in \CEmb_{\mathrm{hc}}({\bm R}_{0},{\bm R})$ is determined solely 
by $\mathcal{F}$ and does not depend on $\bm R$ or $\bm \iota$. 
\end{prop}

\begin{proof}
Fix a maximal point $\bm R$ for $\mathcal{F}$ on $\mathfrak{M}({\bm R}_{0})$, 
and set ${\bm r}={\bm r}_{\bm R}[{\bm \psi}]$, where 
${\bm \psi}={\bm Q}_{\bm R}(\mathcal{F})$. If 
${\bm S} \in \mathfrak{B}_{\rho}({\bm r}(\rho))$ with $\rho>0$, then 
Lemma~\ref{lem:extremal_length:continuity} together 
with~\eqref{eq:extremal_length:Teichmuller_qc} shows that 
\begin{align*}
\log\Ext_{\mathcal{F}}({\bm S}) & \geqq
\log\Ext_{\mathcal{F}}({\bm r}(\rho))-2d_{T}({\bm S},{\bm r}(\rho)) \\
 & =\log\Ext_{\mathcal{F}}({\bm r}(0))+2\rho-2d_{T}({\bm S},{\bm r}(\rho))>
\log\Ext_{\mathcal{F}}({\bm r}(0)),
\end{align*}
which implies that ${\bm S} \not\in \mathfrak{M}({\bm R}_{0})$, for, 
${\bm R}={\bm r}(0)$ is a maximal point for $\mathcal{F}$ on 
$\mathfrak{M}({\bm R}_{0})$. We have proved that 
$\mathfrak{B}_{\rho}({\bm r}(\rho)) \cap \mathfrak{M}({\bm R}_{0})=\varnothing$ 
for all $\rho>0$. We then apply Lemma~\ref{lem:Teichmuller_ray:distance} to 
obtain $\bar{\mathfrak{B}}_{\rho}({\bm r}(\rho)) \cap \mathfrak{M}({\bm R}_{0})=
\{{\bm R}\}$. Hence $\bm R$ is a boundary point of $\mathfrak{M}({\bm R}_{0})$. 
Furthermore, $\bm R$ is the nearest point of $\mathfrak{M}({\bm R}_{0})$ to 
${\bm r}(\rho)$. Since ${\bm r}={\bm r}_{\bm R}[{\bm r}(\rho)]$, 
Proposition~\ref{prop:Ioffe_ray:existence} shows that $\bm r$ is an Ioffe ray 
emanating from $\bm R$. 

To prove the last assertion of the theorem we may assume that ${\bm R}_{0}$ is 
nonanalytically finite. Since ${\bm r}_{\bm R}[{\bm \psi}]$ is an Ioffe ray of 
$\mathfrak{M}({\bm R}_{0})$, for some ${\bm \varphi} \in A_{L}({\bm R}_{0})$ 
there is a ${\bm \varphi}$-regular self-welding 
continuation $({\bm R},{\bm \iota})$ of ${\bm R}_{0}$ such that $\bm \psi$ is 
the co-welder of $\bm \varphi$. Let ${\bm R}'$ be an arbitrary maximal point 
for $\mathcal{F}$ on $\mathfrak{M}({\bm R}_{0})$, and choose 
${\bm \iota}' \in \CEmb_{\mathrm{hc}}({\bm R}_{0},{\bm R}')$ arbitrarily. 
Denote by ${\bm \varphi}'$ the $({\bm R}',{\bm \iota}')$-zero-extension of 
$\bm \varphi$. By Lemma~\ref{lem:height_comparison} and 
Proposition~\ref {prop:second_minimal_norm_property} we obtain 
$\|{\bm \psi}'\|_{{\bm R}'} \leqq \|{\bm \varphi}'\|_{{\bm R}'}$, where 
${\bm \psi}'={\bm Q}_{{\bm R}'}(\mathcal{F})$. Actually, the sign of equality 
occurs because 
$$
\|{\bm \varphi}'\|_{{\bm R}'}=\|{\bm \varphi}\|_{{\bm R}_{0}}=
\|{\bm \psi}\|_{\bm R}=\Ext_{\mathcal{F}}({\bm R})=\Ext_{\mathcal{F}}({\bm R}')=
\|{\bm \psi}'\|_{{\bm R}'}.
$$
Another application of Proposition~\ref{prop:second_minimal_norm_property} 
gives us ${\bm \varphi}'={\bm \psi}'$ almost everywhere on ${\bm R}'$, which 
implies that $({\bm \iota}')^{*}{\bm \psi}'={\bm \varphi}$. This completes the 
proof. 
\end{proof}

Let ${\bm \varphi} \in A_{L}({\bm R}_{0})$. If ${\bm R}_{0}=[R_{0},\theta_{0}]$ 
is nonanalytically finite, then $\bm \varphi$ is a welder of some regular 
self-welding continuation 
$({\bm R},{\bm \iota})=([R,\theta],[\iota,\theta_{0},\theta])$ of ${\bm R}_{0}$ 
by Theorem~\ref{thm:main:A_L(R_0)}. Let 
${\bm \psi}=[\psi,\theta] \in A({\bm R})$ be the co-welder of ${\bm \varphi}$, 
and set 
$$
\mathcal{H}_{{\bm R}_{0}}({\bm \varphi})=\mathcal{H}_{\bm R}({\bm \psi}).
$$
The measured foliation $\mathcal{H}_{{\bm R}_{0}}({\bm \varphi})$ on 
$\Sigma_{g}$ does not depend on a particular choice of the self-welding 
$({\bm R},{\bm \iota})$, for, $R \setminus \iota(R_{0})$ consists of finitely 
many horizontal arcs of $\psi$ together with finitely many points and 
contributes nothing for evaluating the $\psi$-heights $H_{\psi}(c)$ of curves 
$c$ on $R$. We have thus obtained a mapping $\mathcal{H}_{{\bm R}_{0}}$ of 
$A_{L}({\bm R}_{0})$ into $\mathscr{MF}(\Sigma_{g}) \setminus \{0\}$, which is 
bijective by Proposition~\ref{prop:extremal_length:maximal_point}. We denote 
its inverse $\mathcal{H}_{{\bm R}_{0}}^{-1}$ by ${\bm Q}_{{\bm R}_{0}}$. 

If ${\bm R}_{0}$ is analytically finite, then it has a unique closed 
continuation $({\bm R},{\bm \iota})$, and $\bm \iota$ induces a bijection of 
$A_{L}({\bm R}_{0})$ onto $A({\bm R}) \setminus \{{\bm 0}\}$. We then define 
$\mathcal{H}_{{\bm R}_{0}}=\mathcal{H}_{\bm R} \circ {\bm \iota}_{*}$, which is 
a homeomorphism of $A_{L}({\bm R}_{0})$ onto 
$\mathscr{MF}(\Sigma_{g}) \setminus \{0\}$ by 
Proposition~\ref{prop:Hubbard-Masur's_theorem}. Again, we set 
${\bm Q}_{{\bm R}_{0}}=\mathcal{H}_{{\bm R}_{0}}^{-1}$. 

\begin{proof}[Proof of Theorem~\ref{thm:main:measured_foliation}]
We are required to show that $\mathcal{H}_{{\bm R}_{0}}$ and 
${\bm Q}_{{\bm R}_{0}}$ are continuous. To prove that 
$\mathcal{H}_{{\bm R}_{0}}$ is continuous let $\hat{R}_{0}$ be the double of 
$\breve{R}_{0}$, where 
$(\breve{\bm R}_{0},\breve{\bm \iota}_{0})=([\breve{R}_{0},\breve{\theta}_{0}],
[\breve{\iota}_{0},\theta_{0},\breve{\theta}_{0}])$ stands for the natural 
compact continuation of ${\bm R}_{0}$. If $n_{0}$ is the number of components 
of $\partial\breve{R}_{0}$, then $\hat{R}_{0}$ is a closed Riemann surface of 
genus $\hat{g}:=2g+n_{0}-1$. We consider $\breve{R}_{0}$ as a subdomain of 
$\hat{R}_{0}$, and denote by $\iota$ the inclusion mapping of $\breve{R}_{0}$ 
into $\hat{R}_{0}$. Then $\hat{\iota}_{0}:=\iota \circ \breve{\iota}_{0}$ is a 
conformal embedding of $R_{0}$ into $\hat{R}_{0}$. Fix 
$\hat{\theta}_{0} \in \Homeo^{+}(\Sigma_{\hat{g}},\hat{R}_{0})$ to obtain a 
marked closed Riemann surface 
$\hat{\bm R}_{0}=[\hat{R}_{0},\hat{\theta}_{0}|_{\dot{\Sigma}_{\hat{g}}}]$ of 
genus $\hat{g}$. 

Let ${\bm \varphi} \in A_{L}({\bm R}_{0})$, and take 
$\breve{\bm \varphi}=[\breve{\varphi},\breve{\theta}]$ in 
$A_{L}(\breve{\bm R}_{0})$ so that 
${\bm \varphi}=\breve{\bm \iota}_{0}^{*}\breve{\bm \varphi}$. The reflection 
principle enables us to extend $\breve{\varphi}$ to a holomorphic quadratic 
differential $\hat{\varphi}$ on $\hat{R}_{0}$. Set 
$\mathcal{F}=\mathcal{H}_{{\bm R}_{0}}({\bm \varphi})$ and 
$\hat{\mathcal{F}}=\mathcal{H}_{\hat{\bm R}_{0}}(\hat{\bm \varphi})$, where 
$\hat{\bm \varphi}=[\hat{\varphi},\hat{\theta}_{0}|_{\dot{\Sigma}_{\hat{g}}}]$. 
Let $\gamma \in \mathscr{S}(\Sigma_{g})$. Choose a simple loop $c_{0}$ in the 
homotopy class $\gamma$ so that $c_{0}$ stays in $\dot{\Sigma}_{g}$, and denote 
by $\hat{\gamma}$ the free homotopy class in $\mathscr{S}(\Sigma_{\hat{g}})$ 
for which $(\hat{\iota}_{0} \circ \theta_{0})_{*}c_{0}
\in (\hat{\theta}_{0})_{*}\hat{\gamma}$. 

We then claim that $\hat{\mathcal{F}}(\hat{\gamma})=\mathcal{F}(\gamma)$. To 
verify the claim take a point $\breve{p}_{0}$ in $(\breve{R}_{0})^{\circ}$, and 
construct a covering Riemann surface $S$ of $\hat{R}_{0}$ corresponding to the 
subgroup $\iota_{*}\pi_{1}(\breve{R}_{0},\breve{p}_{0})$ of the fundamental 
group $\pi_{1}(\hat{R}_{0},\breve{p}_{0})$. Let $\Pi:S \to \hat{R}_{0}$ 
denote the holomorphic covering map. For some component $S_{0}$ of 
$\Pi^{-1}(\breve{R}_{0})$ we have 
$\Pi|_{S_{0}} \in \CHomeo(S_{0},\breve{R}_{0})$. Observe that each component of 
$S \setminus S_{0}$ is a doubly connected planar domain on $S$. 
Therefore, if $c:[0,1] \to \hat{R}_{0}$ is a loop in 
$(\hat{\theta}_{0})_{*}\hat{\gamma}$ with $c(0) \in (\breve{R}_{0})^{\circ}$ 
and leaves from $\breve{R}_{0}$ across a component $C$ of the border 
$\partial\breve{R}_{0}$, that is, $c(t_{0}) \in C$ and 
$c((t_{0},t_{1})) \cap \breve{R}_{0}=\varnothing$ for some $t_{0},t_{1}$ with 
$0<t_{0}<t_{1}<1$, then $c$ returns to $\breve{R}_{0}$ across the same 
component $C$ in such a way that the subarc of $c$ between the leaving point 
and the returning point is homotopic to an arc on $C$ through a homotopy fixing 
the endpoints, or more precisely, there is $t_{2} \in [t_{1},1)$ together with 
$H \in \Cont([t_{0},t_{2}] \times [0,1],\hat{R}_{0})$ such that $H(t,0)=c(t)$, 
$H(t,1) \in C$  for $t \in [t_{0},t_{2}]$ and $H(t_{j},s)=c(t_{j})$, $j=0,2$, 
for $s \in [0,1]$. In other words, the loop $c':[0,1] \to \hat{R}_{0}$ defined 
by $c'(t)=H(t,1)$ for $t \in [t_{1},t_{2}]$ and $c'(t)=c(t)$ otherwise is 
freely homotopic to $c$ on $\hat{R}_{0}$. Because $\hat{\varphi}$ are positive 
along $C$, we obtain $H_{\hat{\varphi}}(c') \leqq H_{\hat{\varphi}}(c)$, which 
proves the claim. 

Now, let $\{{\bm \varphi}_{n}\}$ be a sequence in $A_{L}({\bm R}_{0})$ 
converging to ${\bm \varphi} \in A_{L}({\bm R}_{0})$, and set 
$\mathcal{F}_{n}=\mathcal{H}_{{\bm R}_{0}}({\bm \varphi}_{n})$ and 
$\mathcal{F}=\mathcal{H}_{{\bm R}_{0}}({\bm \varphi})$. The quadratic 
differentials ${\bm \varphi}_{n}$ and $\bm \varphi$ induce quadratic 
differentials $\hat{\bm \varphi}_{n}$ and $\hat{\bm \varphi}$ on 
$\hat{\bm R}_{0}$ and measured foliations $\hat{\mathcal{F}}_{n}$ and 
$\hat{\mathcal{F}}$ on $\Sigma_{\hat{g}}$. Since $\{\hat{\bm \varphi}_{n}\}$ 
converges to $\hat{\bm \varphi}$, it follows from 
Proposition~\ref{prop:Hubbard-Masur's_theorem} that 
$\mathcal{F}_{n}(\gamma)=\hat{\mathcal{F}}_{n}(\hat{\gamma}) \to
\hat{\mathcal{F}}(\hat{\gamma})=\mathcal{F}(\gamma)$ as $n \to \infty$ for 
$\gamma  \in \mathscr{S}(\Sigma_{g})$, which means that $\{\mathcal{F}_{n}\}$ 
converges to $\mathcal{F}$. We have proved that $\mathcal{H}_{{\bm R}_{0}}$ is 
continuous. 

Next, we show that ${\bm Q}_{{\bm R}_{0}}$ is also continuous. Fix 
${\bm S}_{0} \in \mathfrak{T}_{g}$. Since $\mathfrak{M}({\bm R}_{0})$ is 
compact by Corollary~\ref{cor:Oikawa:compact}, there is $d>0$ such that 
$d_{T}({\bm R},{\bm S}_{0}) \leqq d$ for all 
${\bm R} \in \mathfrak{M}({\bm R}_{0})$. It thus follows from 
Lemma~\ref{lem:extremal_length:continuity} that 
$$
e^{-2d}\Ext_{\mathcal{F}}({\bm S}_{0}) \leqq \Ext_{\mathcal{F}}({\bm R}) \leqq
e^{2d}\Ext_{\mathcal{F}}({\bm S}_{0})
$$
for ${\bm R} \in \mathfrak{M}({\bm R}_{0})$ and 
$\mathcal{F} \in \mathscr{MF}(\Sigma_{g})$. Let $\{\mathcal{F}_{n}\}$ be a 
sequence in $\mathscr{MF}(\Sigma_{g}) \setminus \{0\}$ converging to 
$\mathcal{F} \in \mathscr{MF}(\Sigma_{g}) \setminus \{0\}$, and set 
${\bm \varphi}_{n}={\bm Q}_{{\bm R}_{0}}(\mathcal{F}_{n})$. Let 
$({\bm R}_{n},{\bm \iota}_{n})$ be a closed ${\bm \varphi}_{n}$-regular 
self-welding continuation of ${\bm R}_{0}$. Since 
Proposition~\ref{prop:Hubbard-Masur's_theorem} implies that 
$$
\Ext_{\mathcal{F}_{n}}({\bm S}_{0})=
\|{\bm Q}_{{\bm S}_{0}}(\mathcal{F}_{n})\|_{{\bm S}_{0}} \to
\|{\bm Q}_{{\bm S}_{0}}(\mathcal{F})\|_{{\bm S}_{0}}=
\Ext_{\mathcal{F}}({\bm S}_{0})
\neq 0
$$
as $n \to \infty$, the sequence $\{\|{\bm \varphi}_{n}\|_{{\bm R}_{0}}\}=
\{\Ext_{\mathcal{F}_{n}}({\bm R}_{n})\}$ is bounded and bounded away from $0$. 
With the aid of Proposition~\ref{prop:positive_quadratic_differential:closed} 
together with normal family arguments we infer that any subsequence of 
$\{{\bm \varphi}_{n}\}$ contains a subsequence converging to some 
${\bm \varphi} \in A_{L}({\bm R}_{0})$. Since $\mathcal{H}_{{\bm R}_{0}}$ is 
continuous, we obtain $\mathcal{H}_{{\bm R}_{0}}({\bm \varphi})=\mathcal{F}$, 
or ${\bm \varphi}={\bm Q}_{{\bm R}_{0}}(\mathcal{F})$. Consequently, 
${\bm \varphi}_{n} \to {\bm \varphi}$ as $n \to \infty$, proving that 
${\bm Q}_{{\bm R}_{0}}$ is continuous. This completes the proof. 
\end{proof}

\begin{proof}[Proof of Theorem~\ref{thm:main:Ext}]
We begin with the proof of~(i). Let 
$\mathcal{F} \in \mathscr{MF}(\Sigma_{g}) \setminus \{0\}$, and set 
${\bm \varphi}={\bm Q}_{{\bm R}_{0}}(\mathcal{F})$. 
Equality~\eqref{eq:Ext:maximum} follows from 
Proposition~\ref{prop:extremal_length:maximal_point}. Any point of 
$\mathfrak{M}_{\bm \varphi}({\bm R}_{0})$ is a maximal point for $\mathcal{F}$ 
on $\partial\mathfrak{M}({\bm R}_{0})$ by 
Theorem~\ref{thm:maximal_norm_property}, and such points exhaust the maximal 
set for $\mathcal{F}$ on $\partial\mathfrak{M}({\bm R}_{0})$ by 
Proposition~\ref{prop:extremal_length:maximal_point}. This proves~(i). 

Assertion~(ii) follows from the remark following the proof of 
Theorem~\ref{thm:main:CEmb(R_0,R)}~(iii) (see 
\S\ref{sec:self-welding:uniqueness}). We have only to note that 
$\mathfrak{M}_{\mathcal{F}}({\bm R}_{0})=
\mathfrak{M}_{\bm \varphi}({\bm R}_{0})$ if 
$\mathcal{F}=\mathcal{H}_{{\bm R}_{0}}({\bm \varphi})$. 

To prove~(iii) let ${\bm R} \in \mathfrak{M}({\bm R}_{0})$. Then we infer from 
Theorem~\ref{thm:maximal_norm_property} that 
$$
\Ext_{\mathcal{F}}({\bm R}) \leqq 
\|{\bm Q}_{{\bm R}_{0}}(\mathcal{F})\|_{{\bm R}_{0}}=
\max\Ext_{\mathcal{F}}(\partial\mathfrak{M}({\bm R}_{0}))
$$
holds for all $\mathcal{F} \in \mathscr{MF}(\Sigma_{g}) \setminus \{0\}$. 

To show the converse assume that 
${\bm R} \in \mathfrak{T}_{g} \setminus \mathfrak{M}({\bm R}_{0})$. Let 
${\bm R}'$ be the point of $\mathfrak{M}({\bm R}_{0})$ nearest to $\bm R$. The 
point $\bm R$ lies on an Ioffe ray ${\bm r}_{{\bm R}'}[{\bm \psi}]$, where 
${\bm \psi} \in A({\bm R}')$ (see Proposition~\ref{prop:Ioffe_ray:existence}). 
Then $\bm \psi$ is the co-welder of a welder 
${\bm \varphi} \in A_{L}({\bm R}_{0})$ of some $\bm \varphi$-regular 
self-welding continuation $({\bm R}',{\bm \iota})$ of ${\bm R}_{0}$. Set 
$\mathcal{F}=\mathcal{H}_{{\bm R}_{0}}({\bm \varphi})$. Since 
${\bm \psi}=Q_{{\bm R}'}(\mathcal{F})$ and 
$\|{\bm \psi}\|_{{\bm R}'}=\|{\bm \varphi}\|_{{\bm R}_{0}}$, we 
apply~\eqref{eq:extremal_length:Teichmuller_qc} to obtain 
$$
\Ext_{\mathcal{F}}({\bm R})=
e^{2d_{T}({\bm R},{\bm R}')}\Ext_{\mathcal{F}}({\bm R}')=
e^{2d_{T}({\bm R},{\bm R}')}\|{\bm \psi}\|_{{\bm R}'}>\|{\bm \psi}\|_{{\bm R}'}=
\max\Ext_{\mathcal{F}}(\partial\mathfrak{M}({\bm R}_{0})).
$$
This completes the proof. 
\end{proof}

\begin{exmp}
\label{exmp:g=1:horocycle}
We consider the case of genus one. We identify $\mathfrak{T}_{1}$ with 
$\mathbb{H}$ as in Example~\ref{exmp:g=1:compact}, and use the notations in 
Example~\ref{exmp:g=1:extremal_length}. For $t \in (-1,1]$ set $\mathcal{F}_{t}=
\mathcal{H}_{{\bm R}_{0}}(\breve{\bm \iota}_{0}^{*}{\bm \varphi}_{t}) \in
\mathscr{MF}(\Sigma_{1})$, and define $c_{t}=-\cot(t\pi/2)$, where 
$c_{0}=\infty$. Then the level curves of the function $\Ext_{\mathcal{F}_{t}}$ 
are the horocycles in $\mathbb{H}$ with center at $c_{t}$, that is, the circles 
in $\mathbb{H}$ tangent to $\mathbb{R} \cup \{\infty\}$ at $c_{t}$. 
Theorem~\ref{thm:main:Ext}~(iii) implies that $\mathfrak{M}({\bm R}_{0})$ 
is the envelope of the family of horocycles $\Ext_{\mathcal{F}_{t}}({\bm S})=
\max\mathcal{F}_{t}(\mathfrak{M}({\bm R}_{0}))$, $t \in (-1,1]$. Such a 
description of $\mathfrak{M}({\bm R}_{0})$ first appeared in 
\cite[\S5]{Masumoto1997}. Since $\mathfrak{M}({\bm R}_{0})$ is a closed disk, 
three of the horocycles, or equivalently, three of the measured foliations 
$\mathcal{F}_{t}$ are sufficient to determine it; see 
\cite[Theorem~1]{Masumoto1994}, where 
$\max\mathcal{F}_{t}(\mathfrak{M}({\bm R}_{0}))$ is given by the extremal 
length of a weak homology class of $R_{0}$. 
\end{exmp}

\section{Shape of $\mathfrak{M}_{K}({\bm R}_{0})$ with $K>1$}
\setcounter{equation}{0}
\label{sec:M_K(R_0):shape}

Let ${\bm R}_{0}$ be a marked finite open Riemann surface of positive genus 
$g$. The aim of the present section is to prove 
Theorem~\ref{thm:main:M_K(R_0)}. 

For ${\bm S} \in \mathfrak{T}_{g}$ let ${\bm H}_{K}[{\bm R}_{0}]({\bm S})$ 
denote the point of $\mathfrak{M}_{K}({\bm R}_{0})$ nearest to $\bm S$. By 
Corollary~\ref{cor:M_K(R_0):nearest_point} this gives a well-defined mapping 
${\bm H}_{K}[{\bm R}_{0}]$ of $\mathfrak{T}_{g}$ onto 
$\mathfrak{M}_{K}({\bm R}_{0})$. It fixes $\mathfrak{M}_{K}({\bm R}_{0})$ 
pointwise and maps $\mathfrak{T}_{g} \setminus \mathfrak{M}_{K}({\bm R}_{0})$ 
into $\partial\mathfrak{M}_{K}({\bm R}_{0})$. We abbreviate 
${\bm H}_{1}[{\bm R}_{0}]$ to ${\bm H}[{\bm R}_{0}]$. 

\begin{prop}
\label{prop:partial_M(R_0),continuous_mapping}
If  ${\bm R}_{0}$ is a marked finite open Riemann surface of positive genus 
$g$, then 
${\bm H}_{K}[{\bm R}_{0}]:\mathfrak{T}_{g} \to \mathfrak{M}_{K}({\bm R}_{0})$ 
is a retraction. 
\end{prop}

\begin{proof}
Since ${\bm H}_{K}[{\bm R}_{0}]$ fixes $\mathfrak{M}_{K}({\bm R}_{0})$ 
pointwise, we have only to show that it is continuous on 
$\mathfrak{T}_{g} \setminus \Int\mathfrak{M}_{K}({\bm R}_{0})$. 
If 
${\bm S} \in \mathfrak{T}_{g} \setminus \Int\mathfrak{M}_{K}({\bm R}_{0})$, 
then there is an Ioffe ray $\bm r$ of ${\bm R}_{0}$ such that 
${\bm r}(d_{0})={\bm S}$ for some $d_{0}$; note that 
$d_{0}=d_{T}({\bm S},\mathfrak{M}({\bm R}_{0})) \geqq (\log K)/2$ 
by~\eqref{eq:Ioffe_ray:distance}. Then 
${\bm H}_{K}[{\bm R}_{0}]({\bm S})={\bm r}((\log K)/2)$ 
by~\eqref{eq:Ioffe_ray:ball} and 
Proposition~\ref{prop:partial_M_K(R_0):Ioffe_ray}. 

Let $\{{\bm S}_{n}\}$ be a sequence in 
$\mathfrak{T}_{g} \setminus \Int\mathfrak{M}_{K}({\bm R}_{0})$ 
converging to $\bm S$. Take ${\bm r}_{n} \in \mathscr{I}({\bm R}_{0})$ and 
$t_{n} \geqq (\log K)/2$ so that ${\bm r}_{n}(t_{n})={\bm S}_{n}$. Then 
${\bm H}_{K}[{\bm R}_{0}]({\bm S}_{n})={\bm r}_{n}((\log K)/2)$. The sequence 
$\{t_{n}\}$ converges to $d_{0}$, for, 
$$
\lim_{n \to \infty} t_{n}=
\lim_{n \to \infty} d_{T}({\bm S}_{n},\mathfrak{M}({\bm R}_{0}))=
d_{T}({\bm S},\mathfrak{M}({\bm R}_{0}))=d_{0}.
$$
Since $\partial\mathfrak{M}_{K}({\bm R}_{0})$ is compact, a subsequence 
$\{{\bm H}_{K}[{\bm R}_{0}]({\bm S}_{n_{k}})\}$ converges to some point 
${\bm S}'$ on $\partial\mathfrak{M}_{K}({\bm R}_{0})$. Then we obtain 
\begin{align*}
d_{T}({\bm S},{\bm S}') & =\lim_{k \to \infty}
d_{T}({\bm S}_{n_{k}},{\bm H}_{K}[{\bm R}_{0}]({\bm S}_{n_{k}}))=
\lim_{k \to \infty}
d_{T}({\bm r}_{n_{k}}(t_{n_{k}}),{\bm r}_{n_{k}}((\log K)/2)) \\
 & =d_{0}-(\log K)/2=d_{T}({\bm S},\mathfrak{M}_{K}({\bm R}_{0}))
\end{align*}
by~\eqref{eq:Ioffe_ray:distance}. Hence ${\bm S}'$ is the point of 
$\mathfrak{M}_{K}({\bm R}_{0})$ nearest to $\bm S$, or, 
${\bm H}_{K}[{\bm R}_{0}]({\bm S})={\bm S}'$. This completes the proof. 
\end{proof}

For $K'>K \geqq 1$ let ${\bm H}^{K'}_{K}[{\bm R}_{0}]$ denote the restriction 
of ${\bm H}_{K}[{\bm R}_{0}]$ to $\partial\mathfrak{M}_{K'}({\bm R}_{0})$. We 
consider it as a mapping of $\partial\mathfrak{M}_{K'}({\bm R}_{0})$ into 
$\partial\mathfrak{M}_{K}({\bm R}_{0})$. 

\begin{cor}
\label{cor:boundary_of_M_K(R_0):homeomorphic}
If ${\bm R}_{0}$ is a marked finite open Riemann surface of genus $g$, 
then ${\bm H}^{K'}_{K}[{\bm R}_{0}] \in
\Homeo(\partial\mathfrak{M}_{K'}({\bm R}_{0}),
\partial\mathfrak{M}_{K}({\bm R}_{0}))$ for $K'>K>1$. 
\end{cor}

\begin{proof}
Theorem~\ref{thm:Ioffe_ray:M(R_0)^c} implies that 
${\bm H}^{K'}_{K}[{\bm R}_{0}]$ is a bijection. Since 
$\partial\mathfrak{M}_{K'}({\bm R}_{0})$ is compact and 
$\partial\mathfrak{M}_{K}({\bm R}_{0})$ is Hausdorff, the continuous mapping 
${\bm H}^{K'}_{K}[{\bm R}_{0}]$ is in fact a homeomorphism of 
$\partial\mathfrak{M}_{K'}({\bm R}_{0})$ onto 
$\partial\mathfrak{M}_{K}({\bm R}_{0})$. 
\end{proof}

\begin{rem}
The continuous mapping ${\bm H}^{K}_{1}[{\bm R}_{0}]:
\partial\mathfrak{M}_{K}({\bm R}_{0}) \to \partial\mathfrak{M}({\bm R}_{0})$ is 
not necessarily injective (see the remark after 
Definition~\ref{defn:Ioffe_ray}) though it is always surjective. 
\end{rem}

\begin{prop}
\label{prop:M_K(R_0):ball}
If ${\bm R}_{0}$ is a marked finite open Riemann surface of genus $g$, then for 
each $K>1$ there is a homeomorphism of\/ $\mathfrak{T}_{g}$ onto 
$\mathbb{R}^{2d_{g}}$ that maps the compact sets 
$\mathfrak{M}_{K'}({\bm R}_{0})$, $K' \geqq K$, onto concentric closed balls, 
where $d_{g}=\max\{g,3g-3\}$. 
\end{prop}

The proof of the proposition requires two lemmas. For the proof of the first 
lemma we need the following result of Earle. For 
${\bm R},{\bm S} \in \mathfrak{T}_{g}$ with ${\bm R} \neq {\bm S}$ let 
${\bm Q}_{\bm R}[{\bm S}]$ denote the element of $A({\bm R})$ with 
$\|{\bm Q}_{\bm R}[{\bm S}]\|_{\bm R}=1$ such that 
${\bm r}_{\bm R}[{\bm Q}_{\bm R}[{\bm S}]]={\bm r}_{\bm R}[{\bm S}]$. 

\begin{prop}[Earle \cite{Earle1977}]
\label{prop:Earle's_theorem}
The Teichm\"{u}ller distance function $d_{T}$ is of class $C^{1}$ on 
$\mathfrak{T}_{g} \times \mathfrak{T}_{g}$ off the diagonal. The differential 
of the function ${\bm R} \mapsto d_{T}({\bm R},{\bm S})$ is 
$-{\bm Q}_{\bm R}[{\bm S}]$. 
\end{prop}

Thus if $(R,\theta)$ and $(\psi,\theta)$ represent $\bm R$ and 
${\bm Q}_{\bm R}[{\bm S}]$, respectively, then the differential of 
$d_{T}(\,\cdot\,,{\bm S})$ at $\bm R$ is the linear functional 
$$
\mu \mapsto -\re \int_{R} \mu\psi
$$
on the space of bounded measurable $(-1,1)$-forms $\mu$ on $R$. Note that the 
inequality 
$-\re \int_{R} \mu\psi \leqq \|\mu\|_{\infty}$ holds and that the sign of 
equality occurs if $\mu=-|\psi|/\psi$. Since the tangent vector to 
${\bm r}_{\bm R}[{\bm Q}_{\bm R}[{\bm S}]]$ at the initial point is represented 
by $|\psi|/\psi$, we see that the tangent vector varies continuously with 
${\bm R}$ and ${\bm S}$. 

We now turn to our first lemma. Recall that the image of the interval 
$(t_{1},t_{2})$ by a Teichm\"{u}ller geodesic ray $\bm r$ is denoted by 
${\bm r}(t_{1},t_{2})$ (see~\eqref{eq:Teichmulle_geodesic_ray:subarc}). 

\begin{lem}
\label{lem:Teichmuller_ray:points}
Let $\bm R$ and $\bm S$ be distinct points of\/ $\mathfrak{T}_{g}$. Then for 
$\rho>0$ there are neighborhoods\/ $\mathfrak{U}$ and\/ $\mathfrak{V}$ of 
$\bm R$ and $\bm S$, respectively, such that for any 
${\bm R}',{\bm R}'' \in \mathfrak{U}$ and ${\bm S}' \in \mathfrak{V}$ the 
inclusion relation 
\begin{equation}
\label{eq:Teichmuller_ray:points}
{\bm r}''(\tau'',\tau''+\rho) \subset \mathfrak{B}_{\rho}({\bm r}'(\tau'+\rho))
\end{equation}
holds, where ${\bm r}'={\bm r}_{{\bm R}'}[{\bm S}']$, 
${\bm r}''={\bm r}_{{\bm R}''}[{\bm S}']$, $\tau'=d_{T}({\bm R}',{\bm S}')$ and 
$\tau''=d_{T}({\bm R}'',{\bm S}')$. 
\end{lem}

\begin{proof}
Set ${\bm r}={\bm r}_{{\bm R}}[{\bm S}]$ and $\tau=d_{T}({\bm R},{\bm S})$. 
Since the derivative of the function 
$$
t \mapsto
d_{T}({\bm r}(t),{\bm r}(\tau+\rho))=|\tau+\rho-t|
$$
is $-1$ on the interval $(0,\tau+\rho)$, Proposition~\ref{prop:Earle's_theorem} 
implies that there are neighborhoods $\mathfrak{U}$ and $\mathfrak{V}$ of 
$\bm R$ and $\bm S$, respectively, together with $\delta \in (0,\rho)$ such 
that if ${\bm R}',{\bm R}'' \in \mathfrak{U}$ and ${\bm S}' \in \mathfrak{V}$, 
then the derivative of the function 
$$
t \mapsto d_{T}({\bm r}''(t),{\bm r}'(\tau'+\rho))
$$
is negative on the interval $(\tau''-\delta,\tau''+\delta)$, where 
${\bm r}'={\bm r}_{{\bm R}'}[{\bm S}']$, 
${\bm r}''={\bm r}_{{\bm R}''}[{\bm S}']$, $\tau'=d_{T}({\bm R}',{\bm S}')$ and 
$\tau''=d_{T}({\bm R}'',{\bm S}')$. Since 
${\bm r}''(\tau'')={\bm S}'={\bm r}'(\tau')$, it follows that 
$$
d_{T}({\bm r}''(t),{\bm r}'(\tau'+\rho))<
d_{T}({\bm r}''(\tau''),{\bm r}'(\tau'+\rho))=\rho
$$
for $t \in (\tau'',\tau''+\delta)$, or equivalently, that 
\begin{equation}
\label{eq:Teichmuller_ray:points_1}
{\bm r}''(\tau'',\tau''+\delta) \subset
\mathfrak{B}_{\rho}({\bm r}'(\tau'+\rho)).
\end{equation}
Since 
$$
{\bm r}(\tau+\delta,\tau+\rho) \subset \mathfrak{B}_{\rho}({\bm r}(\tau+\rho))
$$ 
and the Teichm\"{u}ller distance between 
${\bm r}(\tau+\delta,\tau+\rho)$ and 
$\partial\mathfrak{B}_{\rho}({\bm r}(\tau+\rho))$ is exactly $\delta>0$, we can 
replace the neighborhoods $\mathfrak{U}$ and $\mathfrak{V}$ with smaller ones 
to ensure that 
\begin{equation}
\label{eq:Teichmuller_ray:points_2}
{\bm r}''(\tau''+\delta,\tau''+\rho) \subset
\mathfrak{B}_{\rho}({\bm r}'(\tau'+\rho))
\end{equation}
for ${\bm R}',{\bm R}'' \in \mathfrak{U}$ and ${\bm S}' \in \mathfrak{V}$. Now 
inclusion relation~\eqref{eq:Teichmuller_ray:points} is an immediate 
consequence of~\eqref{eq:Teichmuller_ray:points_1} 
and~\eqref{eq:Teichmuller_ray:points_2}. 
\end{proof}

\begin{lem}
\label{lem:Teichmuller_ray:compact_sets}
Let\/ $\mathfrak{K}$ and\/ $\mathfrak{L}$ be disjoint compact sets in\/ 
$\mathfrak{T}_{g}$, and let $\rho>0$. Then there is $\varepsilon_{0}>0$ such 
that for any ${\bm R} \in \mathfrak{K}$, ${\bm S} \in \mathfrak{L}$ and 
${\bm R}' \in \mathfrak{K} \cap \mathfrak{B}_{\varepsilon_{0}}({\bm R})$ the 
inclusion relation 
\begin{equation}
\label{eq:Teichmuller_ray:compact_sets}
{\bm r}'(\tau',\tau'+\rho) \subset \mathfrak{B}_{\rho}({\bm r}(\tau+\rho)),
\end{equation}
holds, where ${\bm r}={\bm r}_{{\bm R}}[{\bm S}]$, 
${\bm r}'={\bm r}_{{\bm R}'}[{\bm S}]$, $\tau=d_{T}({\bm R},{\bm S})$ and 
$\tau'=d_{T}({\bm R}',{\bm S})$. 
\end{lem}

\begin{proof}
For each ${\bm R} \in \mathfrak{K}$ we choose a neighborhood 
$\mathfrak{W}_{\bm R}$ of $\bm R$ as follows. For ${\bm S} \in \mathfrak{L}$ 
take neighborhoods $\mathfrak{U}=\mathfrak{U}_{\bm S}$ of $\bm R$ and 
$\mathfrak{V}=\mathfrak{V}_{\bm S}$ of $\bm S$ as in 
Lemma~\ref{lem:Teichmuller_ray:points}. As $\mathfrak{L}$ is compact, there are 
finitely many points ${\bm S}_{1},\ldots,{\bm S}_{\nu} \in \mathfrak{L}$ such 
that the corresponding open sets 
$\mathfrak{V}_{{\bm S}_{1}},\ldots,\mathfrak{V}_{{\bm S}_{\nu}}$ cover 
$\mathfrak{L}$. Then we set 
$\mathfrak{W}_{\bm R}=\bigcap_{k=1}^{\nu} \mathfrak{U}_{{\bm S}_{k}}$. It is a 
neighborhood of $\bm R$ such that~\eqref{eq:Teichmuller_ray:points} holds for 
${\bm R}',{\bm R}'' \in \mathfrak{W}_{\bm R}$ and ${\bm S}' \in \mathfrak{L}$. 
Then any Lebesgue number $\varepsilon_{0}$ of the open covering 
$\{\mathfrak{W}_{\bm R}\}_{{\bm R} \in \mathfrak{K}}$ of $\mathfrak{K}$ 
possesses the required properties. 
\end{proof}

We are now ready to prove Proposition~\ref{prop:M_K(R_0):ball}. 

\begin{proof}[Proof of Proposition~\ref{prop:M_K(R_0):ball}]
If ${\bm R}_{0}$ is analytically finite, then 
$\mathfrak{M}({\bm R}_{0})=\{{\bm W}\}$ for some ${\bm W} \in \mathfrak{T}_{g}$ 
so that 
$\mathfrak{M}_{K}({\bm R}_{0})=\bar{\mathfrak{B}}_{(\log K)/2}({\bm W})$. Hence 
the proposition is valid in this case. 

Assume now that ${\bm R}_{0}$ is nonanalytically finite. Let 
$\{({\bm W}_{t},{\bm \epsilon}_{t})\}_{t \in [0,1]}$ be a circular filling for 
${\bm R}_{0}$. Then $\mathfrak{M}({\bm W}_{1})$ is a singleton, say, 
$\{{\bm W}\}$. Fix $\rho>0$ satisfying 
\begin{equation}
\label{eq:M(R_0)_subset_B(W,rho)}
\mathfrak{M}({\bm R}_{0}) \subset \mathfrak{B}_{\rho}({\bm W}),
\end{equation}
and set $\mathfrak{L}=\bar{\mathfrak{B}}_{4\rho}({\bm W}) \setminus
\mathfrak{B}_{2\rho}({\bm W})$. It is a compact subset of 
$\mathfrak{T}_{g}$ that does not meet $\mathfrak{M}({\bm R}_{0})$. By 
Lemma~\ref{lem:Teichmuller_ray:compact_sets} there is $\varepsilon_{0}>0$ such 
that for any ${\bm R} \in \mathfrak{M}({\bm R}_{0})$, 
${\bm S} \in \mathfrak{L}$ and ${\bm R}' \in
\mathfrak{M}({\bm R}_{0}) \cap \mathfrak{B}_{\varepsilon_{0}}({\bm R})$ 
inclusion relation~\eqref{eq:Teichmuller_ray:compact_sets} holds. 

We claim that for every $t \in [0,1]$ there exists $\delta(t)>0$ such that if 
$t_{j} \in [0,1]$ and $|t_{j}-t|<\delta(t)$ for $j=1,2$, then 
$$
d_{T}({\bm H}[{\bm W}_{t_{1}}]({\bm S}),{\bm H}[{\bm W}_{t_{2}}]({\bm S}))<
\varepsilon_{0}
$$
for all ${\bm S} \in \mathfrak{L}$; recall that ${\bm H}[{\bm W}_{t}]({\bm S})$ 
stands for the point of $\mathfrak{M}({\bm W}_{t})$ nearest to $\bm S$. If the 
claim were false, then there would exist sequences 
$\{t_{n}\},\{t'_{n}\} \subset [0,1]$ and $\{{\bm S}_{n}\} \subset \mathfrak{L}$ 
such that 
$$
\lim_{n \to \infty} t_{n}=\lim_{n \to \infty} t'_{n}=t \quad \text{and} \quad
d_{T}({\bm R}_{n},{\bm R}'_{n}) \geqq \varepsilon_{0},
$$
where ${\bm R}_{n}={\bm H}[{\bm W}_{t_{n}}]({\bm S}_{n})$ and 
${\bm R}'_{n}={\bm H}[{\bm W}_{t'_{n}}]({\bm S}_{n})$. By taking subsequences 
if necessary we may assume that the sequences $\{{\bm S}_{n}\}$, 
$\{{\bm R}_{n}\}$ and $\{{\bm R}'_{n}\}$ converge to $\bm S$, $\bm R$ and 
${\bm R}'$ in $\mathfrak{T}_{g}$, respectively. Then ${\bm S} \in \mathfrak{L}$ 
and ${\bm R},{\bm R}' \in \mathfrak{M}({\bm W}_{t})$ with 
\begin{equation}
\label{eq:M_K(R_0):ball}
d_{T}({\bm R},{\bm R}') \geqq \varepsilon_{0}.
\end{equation}
For any $\varepsilon>0$ we have 
$$
d_{T}({\bm S}_{n},{\bm S})<\varepsilon,\quad
d_{T}({\bm R}_{n},{\bm R})<\varepsilon,\quad
\mathfrak{M}({\bm W}_{t}) \subset
\mathfrak{M}_{e^{2\varepsilon}}({\bm W}_{t_{n}})
$$
for sufficiently large $n$. Set $d=d_{T}({\bm S},{\bm R})$ and $d_{n}=
d_{T}({\bm S}_{n},{\bm R}_{n})=
d_{T}({\bm S}_{n},\mathfrak{M}({\bm W}_{t_{n}}))$. Since 
$$
\mathfrak{B}_{d_{n}}({\bm S}_{n}) \cap \mathfrak{M}({\bm W}_{t_{n}})=\varnothing
\qquad \text{and} \qquad
\mathfrak{B}_{d-4\varepsilon}({\bm S}) \subset
\mathfrak{B}_{d_{n}-\varepsilon}({\bm S}_{n}),
$$
we obtain 
$\mathfrak{B}_{d-4\varepsilon}({\bm S}) \cap \mathfrak{M}({\bm W}_{t})=
\varnothing$, or equivalently, 
$$
d_{T}({\bm S},{\bm R})-4\varepsilon \leqq
d_{T}({\bm S},\mathfrak{M}({\bm W}_{t})).
$$
As $\varepsilon>0$ is arbitrary, we know that 
$d_{T}({\bm S},{\bm R}) \leqq d_{T}({\bm S},\mathfrak{M}({\bm W}_{t}))$. Since 
$\bm R$ belongs to $\mathfrak{M}({\bm W}_{t})$, we conclude that 
${\bm H}[{\bm W}_{t}]({\bm S})={\bm R}$. Similarly, we obtain 
${\bm H}[{\bm W}_{t}]({\bm S})={\bm R}'$, and hence ${\bm R}={\bm R}'$, 
contradicting~\eqref{eq:M_K(R_0):ball}. 

Let $\delta_{0}>0$ be a Lebesgue number of the open covering 
$\{(t-\delta(t),t+\delta(t))\}_{t \in [0,1]}$ of $[0,1]$. Take an integer 
$N>1/\delta_{0}$ and set $\sigma_{k}=1-k/N$ for $k=0,1,\ldots,N$. Note that 
\begin{equation}
\label{eq:M(W_k)_subset_rho-nbhd_of_M(W_k+1)}
\mathfrak{M}({\bm W}_{\sigma_{k}}) \subset \mathfrak{M}({\bm W}_{\sigma_{k+1}})
\subset \mathfrak{M}({\bm R}_{0}) \subset \mathfrak{B}_{\rho}({\bm W}) \subset
\Int\mathfrak{M}_{e^{2\rho}}({\bm W}_{\sigma_{k}}).
\end{equation}

For ${\bm S} \in \mathfrak{T}_{g} \setminus \mathfrak{M}({\bm R}_{0})$ and 
$k=0,\ldots,N$ denote by ${\bm r}_{k}[{\bm S}]$ the Ioffe ray of 
${\bm W}_{\sigma_{k}}$ passing through $\bm S$, and set 
$\tau_{k}({\bm S})=d_{T}({\bm S},\mathfrak{M}({\bm W}_{\sigma_{k}}))$. Thus 
${\bm r}_{k}[{\bm S}](0)={\bm H}[{\bm W}_{\sigma_{k}}]({\bm S})$ and 
${\bm r}_{k}[{\bm S}](\tau_{k}({\bm S}))={\bm S}$. We then claim that 
\begin{equation}
\label{eq:h:leave_for_rho}
{\bm r}_{k+1}[{\bm S}](\tau_{k+1}({\bm S}),\tau_{k+1}({\bm S})+\rho) \cap
\mathfrak{M}_{e^{6\rho}}({\bm W}_{\sigma_{k}})=\varnothing
\end{equation}
for ${\bm S} \in \partial\mathfrak{M}_{e^{6\rho}}({\bm W}_{\sigma_{k}})$ and 
$k=0,\ldots,N-1$. To show the claim note first that inequality 
\begin{equation}
\label{eq:d_T(S,h(S))}
d_{T}({\bm S},{\bm r}_{k+1}[{\bm S}](3\rho))<\rho
\end{equation}
holds for ${\bm S} \in \partial\mathfrak{M}_{e^{6\rho}}({\bm W}_{\sigma_{k}})$ 
due to~\eqref{eq:M(W_k)_subset_rho-nbhd_of_M(W_k+1)}. As 
$\tau_{k+1}({\bm S})<3\rho$ and 
$$
d_{T}({\bm S},{\bm r}_{k+1}[{\bm S}](3\rho))=
d_{T}({\bm r}_{k+1}[{\bm S}](\tau_{k+1}({\bm S})),{\bm r}_{k+1}[{\bm S}](3\rho))
=3\rho-\tau_{k+1}({\bm S}), 
$$
we have $\tau_{k+1}({\bm S})>2\rho$ by \eqref{eq:d_T(S,h(S))}. Thus 
\begin{equation}
\label{eq:range_of_exit_time}
\tau_{k+1}({\bm S}) \in (2\rho,3\rho)
\end{equation}
for ${\bm S} \in \partial\mathfrak{M}_{e^{6\rho}}({\bm W}_{\sigma_{k}})$. Since 
$|\sigma_{k}-\sigma_{k+1}|<\delta_{0}$, there exists $t_{k} \in [0,1]$ such 
that $|\sigma_{j}-t_{k}|<\delta(t_{k})$ for $j=k,k+1$. Hence 
$$
d_{T}({\bm r}_{k}[{\bm S}](0),{\bm r}_{k+1}[{\bm S}](0))=
d_{T}({\bm H}[{\bm W}_{\sigma_{k}}]({\bm S}),
{\bm H}[{\bm W}_{\sigma_{k+1}}]({\bm S}))<\varepsilon_{0},
$$
which implies that 
$$
{\bm r}_{k+1}[{\bm S}](\tau_{k+1}({\bm S}),\tau_{k+1}({\bm S})+\rho) \subset
\mathfrak{B}_{\rho}({\bm r}_{k}[{\bm S}](\tau_{k}({\bm S})+\rho))=
\mathfrak{B}_{\rho}({\bm r}_{k}[{\bm S}](4\rho)).
$$
Since $\mathfrak{B}_{\rho}({\bm r}_{k}[{\bm S}](4\rho))$ is disjoint from 
$\mathfrak{M}_{e^{6\rho}}({\bm W}_{\sigma_{k}})$ by 
Theorem~\ref{thm:Ioffe_ray:distance}, we obtain~\eqref{eq:h:leave_for_rho}. 

Now, we inductively show that for $K>1$ there is a homeomorphism of 
$\mathfrak{T}_{g}$ onto $\mathbb{R}^{2d_{g}}$ that maps 
$\mathfrak{M}_{K'}({\bm W}_{\sigma_{k}})$, $K' \geqq K$, onto concentric closed 
balls. This is true for $k=0$ since 
$\mathfrak{M}({\bm W}_{\sigma_{0}})=\{{\bm W}\}$. 

To complete the induction arguments suppose that the assertion is true for some 
$k$. Thus there is a homeomorphism of $\mathfrak{T}_{g}$ onto 
$\mathbb{R}^{2d_{g}}$ that maps 
$\mathfrak{M}_{e^{6\rho}}({\bm W}_{\sigma_{k}})$ onto a $2d_{g}$-dimensional 
euclidean closed ball $\bar{B}$. Define a mapping ${\bm h}:
\partial\mathfrak{M}_{e^{6\rho}}({\bm W}_{\sigma_{k}}) \to
\partial\mathfrak{M}_{e^{6\rho}}({\bm W}_{\sigma_{k+1}})$ by 
${\bm h}({\bm S})={\bm r}_{k+1}[{\bm S}](3\rho)$. It is surjective by 
Theorem~\ref{thm:Ioffe_ray:M(R_0)^c}. To show that it is injective suppose that 
${\bm h}({\bm S})={\bm h}({\bm S}')$ for some 
${\bm S}, {\bm S}' \in \partial\mathfrak{M}_{e^{6\rho}}({\bm W}_{\sigma_{k}})$. 
By the definition of $\bm h$ the point ${\bm S}'$ lies on the Ioffe ray of 
${\bm W}_{\sigma_{k+1}}$ passing through ${\bm h}({\bm S}')={\bm h}({\bm S})$, 
that is, on ${\bm r}_{k+1}[{\bm S}']={\bm r}_{k+1}[{\bm S}]$. Thus 
${\bm r}_{k+1}[{\bm S}](\tau_{k+1}({\bm S}'))={\bm S}' \in
\partial\mathfrak{M}_{e^{6\rho}}({\bm W}_{\sigma_{k}})$ and hence 
$\tau_{k+1}({\bm S}') \not\in (\tau_{k+1}({\bm S}),\tau_{k+1}({\bm S})+\rho)$ 
by~\eqref{eq:h:leave_for_rho}. Since 
$\tau_{k+1}({\bm S}),\tau_{k+1}({\bm S}') \in (2\rho,3\rho)$ 
by~\eqref{eq:range_of_exit_time}, we see that 
$\tau_{k+1}({\bm S}') \leqq \tau_{k+1}({\bm S})$. Interchanging the roles of 
$\bm S$ and ${\bm S}'$ we conclude that 
$\tau_{k+1}({\bm S}')=\tau_{k+1}({\bm S})$ and hence that 
${\bm S}'={\bm r}_{k+1}[{\bm S}'](\tau_{k+1}({\bm S}'))=
{\bm r}_{k+1}[{\bm S}](\tau_{k+1}({\bm S}))={\bm S}$, which proves that 
${\bm h}$ is injective. 

The bijection ${\bm h}$ is identical with 
$({\bm H}^{e^{6\rho}}_{e^{\rho}}[{\bm W}_{\sigma_{k+1}}])^{-1} \circ
{\bm H}_{e^{\rho}}[{\bm W}_{\sigma_{k+1}}]$ on 
$\partial\mathfrak{M}_{e^{6\rho}}({\bm W}_{\sigma_{k}})$ and hence continuous 
by Proposition~\ref{prop:partial_M(R_0),continuous_mapping} and 
Corollary~\ref{cor:boundary_of_M_K(R_0):homeomorphic}. It is homeomorphic 
since $\partial\mathfrak{M}_{e^{6\rho}}({\bm W}_{\sigma_{k}})$ is compact and 
$\partial\mathfrak{M}_{e^{6\rho}}({\bm W}_{\sigma_{k+1}})$ is Hausdorff. Thus 
$\partial\mathfrak{M}_{e^{6\rho}}({\bm W}_{\sigma_{k+1}})$ is homeomorphic to 
the sphere $\partial\bar{B}$. 

We extend $\bm h$ to a homeomorphism of $\mathfrak{T}_{g}$ onto itself step by 
step. We begin with extending it to a homeomorphism of 
$\mathfrak{M}_{e^{6\rho}}({\bm W}_{\sigma_{k}})\setminus
\Int\mathfrak{M}_{e^{4\rho}}({\bm W}_{\sigma_{k}})$ onto 
$\mathfrak{M}_{e^{6\rho}}({\bm W}_{\sigma_{k+1}})
\setminus \Int\mathfrak{M}_{e^{6\rho}}({\bm W}_{\sigma_{k}})$ as follows. Each 
point $\bm S$ of $\mathfrak{M}_{e^{6\rho}}({\bm W}_{\sigma_{k}})
\setminus \Int\mathfrak{M}_{e^{4\rho}}({\bm W}_{\sigma_{k}})$ is of the form 
${\bm r}_{k}[{\bm S}'](t)$, where 
${\bm S}' \in \partial\mathfrak{M}_{e^{6\rho}}({\bm W}_{\sigma_{k}})$ and 
$t \in [2\rho,3\rho]$. Then define 
${\bm h}({\bm S})={\bm r}_{k+1}[{\bm S}'](t')$, where 
$$
\frac{t'-\tau_{k+1}({\bm S}')}{\,3\rho-\tau_{k+1}({\bm S}')\,}=
\frac{\,t-2\rho\,}{\rho};
$$
note that ${\bm h}({\bm S})$ belongs to 
$\mathfrak{M}_{e^{6\rho}}({\bm W}_{\sigma_{k+1}})
\setminus \Int\mathfrak{M}_{e^{6\rho}}({\bm W}_{\sigma_{k}})$ 
by~\eqref{eq:h:leave_for_rho}. It is easy to extend $\bm h$ to a homeomorphism 
of $\mathfrak{M}_{e^{6\rho}}({\bm W}_{\sigma_{k}})$ onto 
$\mathfrak{M}_{e^{6\rho}}({\bm W}_{\sigma_{k+1}})$ because 
$\mathfrak{M}_{e^{4\rho}}({\bm W}_{\sigma_{k}})$ and 
$\mathfrak{M}_{e^{6\rho}}({\bm W}_{\sigma_{k}})$ are homeomorphic to the closed 
ball $\bar{B}$. 

Finally, for ${\bm S} \in \mathfrak{T}_{g} \setminus
\mathfrak{M}_{e^{6\rho}}({\bm W}_{\sigma_{k}})$ take 
${\bm S}' \in \partial\mathfrak{M}_{e^{3\rho}}({\bm W}_{\sigma_{k}})$ and 
$t \in (3\rho,+\infty)$ so that ${\bm S}={\bm r}_{k}[{\bm S}'](t)$. 
We then define ${\bm h}({\bm S})={\bm r}_{k+1}[{\bm S}'](t)$ to extend $\bm h$ 
to a homeomorphism of $\mathfrak{T}_{g}$ onto itself with 
${\bm h}(\mathfrak{M}_{K'}({\bm W}_{\sigma_{k}}))=
\mathfrak{M}_{K'}({\bm W}_{\sigma_{k+1}})$ for $K' \geqq e^{6\rho}$. 

With the aid of the induction hypothesis we can construct from ${\bm h}^{-1}$ 
a homeomorphism of $\mathfrak{T}_{g}$ onto $\mathbb{R}^{2d_{g}}$ that maps 
$\mathfrak{M}_{K'}({\bm W}_{\sigma_{k+1}})$, $K' \geqq e^{6\rho}$, onto 
concentric closed balls. Corollary~\ref{cor:boundary_of_M_K(R_0):homeomorphic} 
now guarantees the existence of desired homeomorphisms for 
$\mathfrak{M}_{K}({\bm W}_{\sigma_{k+1}})$, $K>1$. 
\end{proof}

We conclude this section with the following proposition. 
Theorem~\ref{thm:main:M_K(R_0)} follows at once from 
Propositions~\ref{prop:M_K(R_0):ball} and~\ref{prop:M_K(R_0):C^{1}}. 

\begin{prop}
\label{prop:M_K(R_0):C^{1}}
If ${\bm R}_{0}$ is a marked finite open Riemann surface of genus $g$, then the 
boundary $\partial\mathfrak{M}_{K}({\bm R}_{0})$ is a $C^{1}$-submanifold 
of\/ $\mathfrak{T}_{g}$ for $K>1$. 
\end{prop}

\begin{proof}
Let $\bm R$ be an arbitrary point of $\partial\mathfrak{M}_{K}({\bm R}_{0})$, 
and take the Ioffe ray ${\bm r}$ of ${\bm R}_{0}$ passing through $\bm R$. Set 
$\tau=(\log K)/2$. Thus ${\bm r}(\tau)={\bm R}$. Note that 
$\mathfrak{B}_{\tau}({\bm r}(2\tau)) \cap \mathfrak{M}_{K}({\bm R}_{0})=
\varnothing$ by~\eqref{eq:Ioffe_ray:ball} and that 
$\mathfrak{B}_{\tau/2}({\bm r}(\tau/2)) \subset \mathfrak{M}_{K}({\bm R}_{0})$ 
by Proposition~\ref{prop:space:qc_embedding}. Also, as $\mathfrak{T}_{g}$ is a 
straight space in the sense of Busemann, $\bm R$ is the unique common point of 
the closed balls $\bar{\mathfrak{B}}_{\tau}({\bm r}(2\tau))$ and 
$\bar{\mathfrak{B}}_{\tau/2}({\bm r}(\tau/2))$. 

Take a $C^{1}$ coordinate system $(\mathfrak{U},x)$ centered at $\bm R$. Thus 
$\mathfrak{U}$ is a connected open neighborhood of $\bm R$, and $x$ is a 
$C^{1}$ diffeomorphism of $\mathfrak{U}$ onto a domain of $\mathbb{R}^{2d_{g}}$ 
with $x({\bm R})=0$. Write $x=(x_{1},\ldots,x_{2d_{g}})=(x_{1},x')$. We choose 
the system so that $(x_{1}({\bm r}(t)),x'({\bm r}(t)))=(t-\tau,0)$ and 
$x(\mathfrak{U})=(-\delta,\delta) \times U'$ for some $\delta>0$ and some 
domain $U'$ in $\mathbb{R}^{2d_{g}-1}$. Also, we require that the points 
${\bm S} \in \mathfrak{U}$ with $x_{1}({\bm S})=\delta$ 
(resp.\ $x_{1}({\bm S})=-\delta$) should be contained in 
$\mathfrak{B}_{\tau}({\bm r}(2\tau))$ (resp.\ 
$\mathfrak{B}_{\tau/2}({\bm r}(\tau/2))$). Thus for each $\xi' \in U'$ there is 
${\bm S}_{\xi'} \in \partial\mathfrak{M}_{K}({\bm R}_{0}) \cap \mathfrak{U}$ 
such that $x'({\bm S}_{\xi'})=\xi'$. We show that ${\bm S}_{\xi'}$ is uniquely 
determined if $\xi'$ is sufficiently near $0$. Observe that even if 
${\bm S}_{\xi'}$ is not unique, it tends to $\bm R$ as $\xi' \to 0$ because 
${\bm S}_{\xi'} \not\in \mathfrak{B}_{\tau}({\bm r}(2\tau)) \cup
\mathfrak{B}_{\tau/2}({\bm r}(\tau/2))$. 
Since the partial derivatives of 
$d_{T}(\,\cdot\,,{\bm r}(2\tau))$ and $d_{T}(\,\cdot\,,{\bm r}(\tau/2))$ at 
$\bm R$ with respect to $x_{1}$ are $-1$ and $1$, respectively, we may assume 
that the partial derivative of $d_{T}(\,\cdot\,,{\bm S})$ with respect to 
$x_{1}$ is negative on $\mathfrak{U}$ for any ${\bm S}$ in a neighborhood 
$\mathfrak{V}_{1}$ of ${\bm r}(2\tau)$ and that the partial derivative of 
$d_{T}(\,\cdot\,,{\bm S})$ with respect to $x_{1}$ is positive on 
$\mathfrak{U}$ for any ${\bm S}$ in a neighborhood $\mathfrak{V}_{2}$ of 
${\bm r}(\tau/2)$. Let ${\bm r}_{\xi'}$ denote the Ioffe ray of ${\bm R}_{0}$ 
passing through ${\bm S}_{\xi'}$. If $\xi'$ is sufficiently near $0$ so that 
${\bm r}_{\xi'}(2\tau) \in \mathfrak{V}_{1}$ and 
${\bm r}_{\xi'}(\tau/2) \in \mathfrak{V}_{2}$, then 
$d_{T}({\bm S},{\bm r}_{\xi'}(2\tau))<
d_{T}({\bm S}_{\xi'},{\bm r}_{\xi'}(2\tau))=\tau$ for 
${\bm S} \in \mathfrak{U}$ with $x'({\bm S})=\xi'$ and 
$x_{1}({\bm S})>x_{1}({\bm S}_{\xi'})$, and hence 
${\bm S} \in \mathfrak{B}_{\tau}({\bm r}_{\xi'}(2\tau)) \subset
\mathfrak{T}_{g} \setminus \mathfrak{M}_{K}({\bm R}_{0})$. Also, we have 
$d_{T}({\bm S},{\bm r}_{\xi'}(\tau/2))<
d_{T}({\bm S}_{\xi'},{\bm r}_{\xi'}(\tau/2))=\tau/2$ for 
${\bm S} \in \mathfrak{U}$ with $x'({\bm S})=\xi'$ and 
$x_{1}({\bm S})<x_{1}({\bm S}_{\xi'})$, and hence 
${\bm S} \in \mathfrak{B}_{\tau/2}({\bm r}_{\xi'}(\tau/2)) \subset
\Int\mathfrak{M}_{K}({\bm R}_{0})$. Consequently, ${\bm S}_{\xi'}$ is uniquely 
determined for $\xi'$ sufficiently near $0$. Replacing $U'$ with a smaller one 
if necessary, we assume that ${\bm S}_{\xi'}$ is uniquely determined for 
$\xi' \in U'$. 

We have shown that there is a function $f:U' \to \mathbb{R}$ such that 
$f(\xi')=x_{1}({\bm S}_{\xi'})$. Since 
$d_{T}:\mathfrak{T}_{g} \times \mathfrak{T}_{g} \to \mathbb{R}_{+}$ is of class 
$C^{1}$ off the diagonal by Proposition~\ref{prop:Earle's_theorem}, the 
hypersurfaces 
$x(\partial\mathfrak{B}_{\tau}({\bm r}_{\xi'}(2\tau)) \cap \mathfrak{U})$ and 
$x(\partial\mathfrak{B}_{\tau/2}({\bm r}_{\xi'}(\tau/2)) \cap \mathfrak{U})$ 
have a common tangent hyperplane at $x({\bm S}_{\xi'})$. Since the graph of $f$ 
lies between these hypersurfaces, we infer that the hyperplane is also tangent 
to the graph. Consequently, $f$ is differentiable at $\xi'$. Since the 
hypersurfaces together with the common tangent hyperplanes move continuously 
with $\xi'$, we know that $f$ is of class $C^{1}$. The proof is complete. 
\end{proof}

\section{Geometric properties of $\mathfrak{M}({\bm R}_{0})$}
\setcounter{equation}{0}
\label{sec:M(R_0):geometry}

The final section is devoted to establishing Theorem~\ref{thm:main:M(R_0)}. 
Assuming that ${\bm R}_{0}$ is nonanalytically finite, fix a circular filling 
$\{({\bm W}_{t},{\bm \epsilon}_{t})\}_{t \in [0,1]}$ for ${\bm R}_{0}$. We use 
the same notations as in Example~\ref{exmp:circular_filling}. Each $W_{t}$ is 
the interior of a compact bordered Riemann surface of genus $g$ and includes 
$\breve{R}_{0}$ for $t \in (0,1)$. If $t_{0} \in (0,1)$ is sufficiently small, 
then every quadratic differential $\varphi$ in $A_{L}(W_{0})$ is extended to a 
holomorphic quadratic differential on $W_{t_{0}}$, denoted again by $\varphi$. 
Thus we consider $A_{L}(W_{0})$ as a subset of $A(W_{t_{0}})$. Let $E(W_{0})$ 
be the set of quadratic differentials $\varphi$ in 
$A_{L}(W_{0}) \subset A(W_{t_{0}})$ with $\|\varphi\|_{W_{0}}=1$. 

\begin{lem}
\label{lem:special_rectangle}
There exists $l_{0}>0$ such that each $\varphi \in E(W_{0})$ has a noncritical 
point $p \in \partial W_{0}$ together with a natural parameter 
$\zeta:U \to \mathbb{C}$ of $\varphi$ around $p$ satisfying 
$[-2l_{0},2l_{0}] \times [-tl_{0},l_{0}] \subset \zeta(U \cap W_{t})$ for 
$t \in (0,t_{0})$. 
\end{lem}

\begin{proof}
For $\varphi \in E(W_{0})$ choose a noncritical point $p \in \partial W_{0}$. 
There is a neighborhood $V_{\varphi}$ of $\varphi$ in $E(W_{0})$ such that 
$|(\psi/\varphi)|(p)|$, $\psi \in V_{\varphi}$, are bounded and bounded away 
from $0$; note that $\psi/\varphi$ is a meromorphic function on $W_{t_{0}}$. 
Then there are a positive number $l_{\varphi}$ and a neighborhood $U$ of $p$ 
such that a natural parameter $\zeta_{\psi}:U \to \mathbb{C}$ of $\psi$ around 
$p$ satisfies $[-2l_{\varphi},2l_{\varphi}] \times [-tl_{\varphi},l_{\varphi}] 
\subset \zeta_{\psi}(U \cap W_{t})$ for $t \in (0,t_{0})$. Since $E(W_{0})$ is 
compact, the open covering $\{V_{\varphi}\}$ of $E(W_{0})$ includes a finite 
subcovering $\{V_{\varphi_{1}},\ldots,V_{\varphi_{n}}\}$. Then 
$l_{0}:=\min_{j} l_{\varphi_{j}}$ possesses the required properties. 
\end{proof}

\begin{lem}
\label{lem:Ioffe_ray:approximation}
Let $\{t_{n}\}$ be a sequence of positive numbers converging to $0$. If a 
sequence $\{{\bm S}_{n}\}$ in\/ 
$\mathfrak{T}_{g} \setminus \mathfrak{M}({\bm R}_{0})$ converges to 
${\bm S} \in \mathfrak{T}_{g}$, then 
${\bm H}[{\bm W}_{t_{n}}]({\bm S}_{n}) \to {\bm H}[{\bm R}_{0}]({\bm S})$ as 
$n \to \infty$. 
\end{lem}

\begin{proof}
Define ${\bm R}_{n}={\bm H}[{\bm W}_{t_{n}}]({\bm S}_{n})$, 
${\bm R}'_{n}={\bm H}[{\bm R}_{0}]({\bm S}_{n})$ and 
${\bm R}={\bm H}[{\bm R}_{0}]({\bm S})$, and set 
$\rho_{n}=d_{T}({\bm S}_{n},{\bm R}_{n})=
d_{T}({\bm S}_{n},\mathfrak{M}({\bm W}_{t_{n}}))$, 
$\rho'_{n}=d_{T}({\bm S}_{n},{\bm R}'_{n})=
d_{T}({\bm S}_{n},\mathfrak{M}({\bm R}_{0}))$ and 
$\rho=d_{T}({\bm S},{\bm R})=d_{T}({\bm S},\mathfrak{M}({\bm R}_{0}))$. Note 
that ${\bm R}_{n} \in \partial\mathfrak{M}({\bm W}_{t_{n}}) \subset
\mathfrak{M}({\bm R}_{0})$ and 
${\bm R}'_{n},{\bm R} \in \partial\mathfrak{M}({\bm R}_{0})$ and that 
$\rho'_{n} \to \rho$ as $n \to \infty$. 

For any $\varepsilon>0$ we have $|\rho'_{n}-\rho|<\varepsilon$ and 
$\mathfrak{M}({\bm W}_{t_{n}}) \subset \mathfrak{M}({\bm R}_{0})
\subset \mathfrak{M}_{e^{2\varepsilon}}({\bm W}_{t_{n}})$ for sufficiently 
large $n$. Then $d_{T}({\bm R}'_{n},{\bm H}[{\bm W}_{t_{n}}]({\bm R}'_{n}))
\leqq \varepsilon$ and hence 
\begin{align*}
\rho-\varepsilon & <\rho'_{n} \leqq \rho_{n} \leqq
d_{T}({\bm S}_{n},{\bm H}[{\bm W}_{t_{n}}]({\bm R}'_{n})) \leqq
d_{T}({\bm S}_{n},{\bm R}'_{n})+
d_{T}({\bm R}'_{n},{\bm H}[{\bm W}_{t_{n}}]({\bm R}'_{n})) \\
& \leqq \rho'_{n}+\varepsilon<\rho+2\varepsilon.
\end{align*}
This proves that $\rho_{n} \to \rho$ as $n \to \infty$. 

We show that $\bm R$ is a unique accumulation point of $\{{\bm R}_{n}\}$. Let 
$\{{\bm R}_{n_{k}}\}$ be an arbitrary convergent subsequence of 
$\{{\bm R}_{n}\}$. If ${\bm R}'$ denotes its limit, then it belongs to 
$\mathfrak{M}({\bm R}_{0})$. As 
$$
d_{T}({\bm S},\mathfrak{M}({\bm R}_{0}))=\rho=\lim_{k \to \infty} \rho_{n_{k}}=
\lim_{k \to \infty} d_{T}({\bm S}_{n_{k}},{\bm R}_{n_{k}})=
d_{T}({\bm S},{\bm R}'),
$$
we know that ${\bm R}'$ is the nearest point of $\mathfrak{M}({\bm R}_{0})$ to 
$\bm S$ and hence ${\bm R}'={\bm H}[{\bm R}_{0}]({\bm S})={\bm R}$. This 
completes the proof. 
\end{proof}

\begin{lem}
\label{lem:Ext:estimate_of_derivative}
There exist positive numbers $c$ and $M$ such that 
$$
\frac{\,\max\Ext_{\mathcal{F}}(\mathfrak{M}({\bm W}_{t}))-
\max\Ext_{\mathcal{F}}(\mathfrak{M}({\bm R}_{0}))\,}{t} \leqq
(-c+Mt)\max\Ext_{\mathcal{F}}(\mathfrak{M}({\bm R}_{0}))
$$
for all $\mathcal{F} \in \mathscr{MS}(\Sigma_{g})$ and $t \in (0,t_{0})$. 
\end{lem}

\begin{proof}
Let $\mathcal{F} \in \mathscr{MS}(\Sigma_{g})$ with 
$\max\Ext_{\mathcal{F}}(\mathfrak{M}({\bm R}_{0}))=1$, and take 
${\bm \varphi}=[\varphi,\breve{\theta}_{0}] \in A_{L}({\bm W}_{0})$ for 
which $\mathcal{H}_{{\bm R}_{0}}({\bm \epsilon}_{0}^{*}{\bm \varphi})=
\mathcal{F}$. Note that $\varphi \in E(W_{0})$. Take a positive number $l_{0}$ 
as in Lemma~\ref{lem:special_rectangle}; there is a noncritical point 
$p \in \partial W_{0}$ of $\varphi$ together with a natural parameter 
$\zeta:U \to \mathbb{C}$, $\zeta=\xi+i\eta$, of $\varphi$ around $p$ such that 
the closed rectangle $E_{t}$ defined by $|\xi| \leqq 2l_{0}$ and 
$-tl_{0} \leqq \eta \leqq l_{0}$ is included in $U \cap W_{t}$ for 
$t \in (0,t_{0})$. Note that points in $U$ with $\eta>0$ belong to $W_{0}$. In 
the following we identify $U$ with $\zeta(U)$ in the obvious manner. 

We divide $E_{t}$ into three parts $E_{t}^{(1)}$, $E_{t}^{(2)}$ and 
$E_{t}^{(3)}$ as follows: 
\begin{align*}
E_{t}^{(1)} & =[-l_{0},l_{0}] \times [-tl_{0},l_{0}], \\
E_{t}^{(2)} & =\{\zeta \in E_{t} \mid l_{0}<|\xi| \leqq 2l_{0} \text{ and }
y_{t}(\xi) \leqq \eta \leqq l_{0}\}, \text{ and} \\
E_{t}^{(3)} & =E_{t} \setminus (E_{t}^{(1)} \cup E_{t}^{(2)}),
\end{align*}
where $y_{t}(\xi)=t(|\xi|-2l_{0})$. Define a function $v_{t}$ on $E_{t}$ so 
that for $y \in [0,l_{0}]$ the function $v_{t}$ assumes the value $y$ on the 
polygonal arc obtained by joining $-2l_{0}+iy$, $-l_{0}+i\{(1+t)y-tl_{0}\}$, 
$l_{0}+i\{(1+t)y-tl_{0}\}$ and $2l_{0}+iy$ successively. In other words, we 
define the function $v_{t}$ by 
$$
v_{t}(\zeta)=
\begin{cases}
\dfrac{\,\eta+tl_{0}\,}{\,1+t\,} & \text{if $\zeta \in E_{t}^{(1)}$}, \\[2ex]
l_{0}-\dfrac{l_{0}(\eta-l_{0})}{\,y_{t}(\xi)-l_{0}\,} &
\text{if $\zeta \in E_{t}^{(2)}$}, \\[2ex]
0 & \text{if $\zeta \in E_{t}^{(3)}$}.
\end{cases}
$$
For $t \in (0,t_{0})$ define a quadratic differential $\varphi_{t}$ on $W_{t}$ 
by 
$$
\varphi_{t}=
\begin{cases}
\biggl\{\dfrac{\partial v_{t}}{\partial\eta}(\zeta)+
i\dfrac{\partial v_{t}}{\partial\xi}(\zeta)\biggr\}^{2}\,d\zeta^{2} &
\text{on $E_{t}$}, \\
\varphi & \text{on $\bar{W}_{0} \setminus E_{t}$}, \\
0 & \text{on $W_{t} \setminus (\bar{W}_{0} \cup E_{t})$}.
\end{cases}
$$
Since 
\begin{align*}
\|\varphi_{t}\|_{E_{t}} & =
\iint_{E_{t}} \biggl\{\frac{\partial v_{t}}{\partial\xi}(\zeta)^{2}+
\frac{\partial v_{t}}{\partial\eta}(\zeta)^{2}\biggr\}\,d\xi\,d\eta=
\frac{2l_{0}^{2}}{\,1+t\,}+\frac{\,2l_{0}^{2}(t^{2}+3)\,}{3t}\log(1+t) \\
 & =\|\varphi\|_{W_{0} \cap E_{t}}-3l_{0}^{2}t+O(t^{2})
\end{align*}
as $t \to 0$, there is a positive number $M$ such that 
$$
\|\varphi_{t}\|_{W_{t}} \leqq \|{\bm \varphi}\|_{{\bm W}_{0}}-3l_{0}^{2}t+Mt^{2}
$$
for all $t \in (0,t_{0})$. 

Let ${\bm R} \in \mathfrak{M}_{\mathcal{F}}({\bm R}_{0})$.  It is induced by a 
$\bm \varphi$-regular self-welding continuation of ${\bm W}_{0}$ by 
Theorem~\ref{thm:main:Ext}~(i). Let ${\bm \psi} \in A({\bm R})$ be the 
co-welder of $\bm \varphi$. Thus $\mathcal{F}=\mathcal{H}_{\bm R}({\bm \psi})$ 
and $\|{\bm \varphi}\|_{{\bm W}_{0}}=\|{\bm \psi}\|_{\bm R}=
\Ext_{\mathcal{F}}({\bm R})=1$. For $t \in (0,t_{0})$ let 
$\tilde{\bm R}_{t}=[\tilde{R}_{t},\tilde{\theta}_{t}] \in
\mathfrak{M}_{\mathcal{F}}({\bm W}_{t})$, and set 
$\tilde{\bm \psi}_{t}={\bm Q}_{\tilde{\bm R}_{t}}(\mathcal{F})$. If 
$\tilde{\bm \iota}_{t}=[\tilde{\iota}_{t},\breve{\theta}_{0},\tilde{\theta}_{t}]
\in \CEmb_{\mathrm{hc}}({\bm W}_{t},\tilde{\bm R}_{t})$, then $\varphi_{t}$ 
induces a quadratic differential on 
$\tilde{\iota}_{t}(W_{t})$. Let $\tilde{\varphi}_{t}$ denote its zero-extension 
to $\tilde{R}_{t}$. Similarly, $\varphi$ induces a holomorphic differential 
{\em on\/} $\tilde{\iota}_{t}(W_{0})$. Denote by $\tilde{\varphi}$ its 
zero-extension to $\tilde{R}_{t}$. Since $|\im\sqrt{\varphi_{t}\,}|=|dv_{t}|$ 
on $E_{t}$, for $\gamma \in \mathscr{S}(\Sigma_{g})$ we have 
$H_{\tilde{\varphi}_{t}}(c)=H_{\tilde{\varphi}}(c)$ for any piecewise analytic 
simple loop $c$ on $\tilde{R}_{t}$ with $\tilde{\theta}_{t}^{*}c \in \gamma$. 
Hence $\mathcal{H}'_{\tilde{\bm R}_{t}}(\tilde{\bm \varphi}_{t})(\gamma)=
\mathcal{H}'_{\tilde{\bm R}_{t}}(\tilde{\bm \varphi})(\gamma)$, where 
$\tilde{\bm \varphi}_{t}=[\tilde{\varphi}_{t},\tilde{\theta}_{t}]$ and 
$\tilde{\bm \varphi}=[\tilde{\varphi},\tilde{\theta}_{t}]$. Since 
$$
\mathcal{H}_{\tilde{\bm R}_{t}}(\tilde{\bm \psi}_{t})(\gamma)=
\mathcal{F}(\gamma)=\mathcal{H}_{\bm R}({\bm \psi})(\gamma) \leqq
\mathcal{H}'_{\tilde{\bm R}_{t}}(\tilde{\bm \varphi})(\gamma)=
\mathcal{H}'_{\tilde{\bm R}_{t}}(\tilde{\bm \varphi}_{t})(\gamma),
$$
it follows from Proposition~\ref{prop:second_minimal_norm_property} that 
$$
\Ext_{\mathcal{F}}(\tilde{\bm R}_{t})=
\|\tilde{\bm \psi}_{t}\|_{\tilde{\bm R}_{t}} \leqq
\|\tilde{\bm \varphi}_{t}\|_{\tilde{\bm R}_{t}}=
\|\varphi_{t}\|_{W_{t}} \leqq \|{\bm \varphi}\|_{{\bm W}_{0}}-3l_{0}^{2}t+Mt^{2}
=1-3l_{0}^{2}t+Mt^{2}.
$$
Since $\tilde{\bm R}_{t}$ is a maximal point for $\mathcal{F}$ on 
$\mathfrak{M}({\bm W}_{t})$, we obtain 
$$
\max\Ext_{\mathcal{F}}(\mathfrak{M}({\bm W}_{t})) \leqq 1-3l_{0}^{2}t +Mt^{2}
$$
for all $t \in (0,t_{0})$, provided that 
$\max\Ext_{\mathcal{F}}(\mathfrak{M}({\bm R}_{0}))=1$. This proves the lemma. 
\end{proof}

In general, let $\mathfrak{K}$ be a compact set of $\mathfrak{T}_{g}$. For 
${\bm S} \in \mathfrak{K}$ denote by 
$\mathscr{MF}(\Sigma_{g};\mathfrak{K},{\bm S})$ the set of measured foliations 
$\mathcal{F}$ on $\Sigma_{g}$ such that $\bm S$ is a maximal point for 
$\mathcal{F}$ on $\mathfrak{K}$. For nonempty 
$\mathfrak{L} \subset \mathfrak{K}$ set 
$\mathscr{MF}(\Sigma_{g};\mathfrak{K},\mathfrak{L})=
\bigcup_{{\bm S} \in \mathfrak{L}}
\mathscr{MF}(\Sigma_{g};\mathfrak{K},{\bm S})$.

\begin{cor}
\label{cor:Ext:estimate_of_derivative}
There exist positive numbers $c$ and $M$ such that for all $t \in (0,t_{0})$ 
and ${\bm R} \in \partial\mathfrak{M}({\bm R}_{0})$ the inequality 
$$
\frac{\,\Ext_{\mathcal{F}}({\bm H}[{\bm W}_{t}]({\bm R}))-
\Ext_{\mathcal{F}}({\bm R})\,}
{d_{T}({\bm H}[{\bm W}_{t}]({\bm R}),{\bm R})} \leqq
(-c+Mt)\Ext_{\mathcal{F}}({\bm R})
$$
holds for all 
$\mathcal{F} \in \mathscr{MF}(\Sigma_{g};\mathfrak{M}({\bm R}_{0}),{\bm R})$. 
\end{cor}

\begin{proof}
Note that ${\bm H}[{\bm W}_{t}]({\bm R}) \neq {\bm R}$ by 
Proposition~\ref{prop:circular_filling:inerior}. Since there exists a 
homotopically consistent $(1-(\log(1-t))/\log r_{0})$-quasiconformal 
homeomorphism of ${\bm W}_{t}$ onto ${\bm W}_{0}$ (for the definition of 
$r_{0}$ see Example~\ref{exmp:circular_filling}), 
Lemma~\ref{lem:continuity_of_M(R_0)} and 
Proposition~\ref{prop:space:qc_embedding} imply that there is $c_{1}>0$ such 
that 
$$
d_{T}({\bm H}[{\bm W}_{t}]({\bm R}),{\bm R})=
d_{T}({\bm R},\mathfrak{M}({\bm W}_{t})) \leqq 
\frac{1}{\,2\,}\log\biggl\{1-\frac{\,\log(1-t)\,}{\log r_{0}}\biggr\}<c_{1}t
$$
for all $t \in (0,t_{0})$ and ${\bm R} \in
\partial\mathfrak{M}({\bm W}_{0})=\partial\mathfrak{M}({\bm R}_{0})$. 

Let ${\bm R} \in \partial\mathfrak{M}({\bm R}_{0})$. If $\mathcal{F}$ belongs 
to $\mathscr{MF}(\Sigma_{g};\mathfrak{M}({\bm R}_{0}),{\bm R})$, then 
$$
\Ext_{\mathcal{F}}({\bm H}[{\bm W}_{t}]({\bm R}))-\Ext_{\mathcal{F}}({\bm R})
\leqq \max\Ext_{\mathcal{F}}(\mathfrak{M}({\bm W}_{t}))-
\max\Ext_{\mathcal{F}}(\mathfrak{M}({\bm R}_{0}))<0,
$$
where the last inequality follows from the fact that 
$\mathfrak{M}({\bm W}_{t}) \subset \Int\mathfrak{M}({\bm R}_{0})$. 
Therefore, the corollary is an immediate consequence of 
Lemma~\ref{lem:Ext:estimate_of_derivative}. 
\end{proof}

In general, let ${\bm R}=[R,\theta] \in \mathfrak{T}_{g}$. For ${\bm \varphi}=
[\varphi,\theta],{\bm \psi}=[\psi,\theta] \in A({\bm R}) \setminus \{{\bm 0}\}$ 
define 
$$
D_{\bm R}[{\bm \varphi}]({\bm \psi})=\iint_{R} \frac{\,\varphi|\psi|\,}{\psi}
=\iint_{R}
\frac{\,\varphi(z)|\psi(z)|\,}{\psi(z)}\,\frac{\,dz \wedge d\bar{z}\,}{-2i},
$$
where $\varphi=\varphi(z)\,dz^{2}$ and $\psi=\psi(z)\,dz^{2}$. We need the 
following variational formula due to Gardiner, which implies that the extremal 
length function $\Ext_{\mathcal{F}}$ is continuously differentiable on 
$\mathfrak{T}_{g}$. For the proof see also Gardiner-Lakic 
\cite[Theorem~12.5]{GL2000}. 

\begin{prop}[\mbox{Gardiner \cite[Theorem~8]{Gardiner1984}}]
\label{prop:extremal_length:variation}
Let ${\bm R} \in \mathfrak{T}_{g}$ and 
${\bm \varphi} \in A({\bm R}) \setminus \{{\bm 0}\}$ and set 
$\mathcal{F}=\mathcal{H}_{\bm R}({\bm \varphi})$. Then for 
${\bm \psi} \in A({\bm R}) \setminus \{{\bm 0}\}$ 
$$
\log\Ext_{\mathcal{F}}({\bm r}_{\bm R}[{\bm \psi}](t))=
\log\Ext_{\mathcal{F}}(\bm R)+
\frac{2t}{\,\|{\bm \varphi}\|_{\bm R}\,}\re D_{\bm R}[{\bm \varphi}]({\bm \psi})
+o(t)
$$
as $t \to +0$. 
\end{prop}

In fact, the differential of $\log\Ext_{\mathcal{F}}$ at $\bm R$ is the linear 
functional 
$$
\mu \mapsto \frac{2}{\,\|{\bm \varphi}\|_{\bm R}\,}\re\int_{R} \mu\varphi
$$
on the space of bounded measurable $(-1,1)$-forms $\mu$ on $R$, where 
${\bm R}=[R,\theta]$ and ${\bm \varphi}=[\varphi,\theta]$. The inequality 
$(2/\|{\bm \varphi}\|_{\bm R})\re\int_{R} \mu\varphi \leqq 2\|\mu\|_{\infty}$ 
holds and the sign of equality occurs if $\mu=|\varphi|/\varphi$. 

For ${\bm R} \in \partial\mathfrak{M}({\bm R}_{0})$ and $t \in (0,1)$ we denote 
by ${\bm r}^{(t)}[{\bm R}]$ the Ioffe ray of ${\bm W}_{t}$ passing through 
$\bm R$. Thus its initial point ${\bm r}^{(t)}[{\bm R}](0)$ coincides with 
${\bm H}[{\bm W}_{t}]({\bm R})$. Also, set 
$\tau^{(t)}({\bm R})=d_{T}({\bm R},\mathfrak{M}({\bm W}_{t}))$ so that 
${\bm r}^{(t)}[{\bm R}](\tau^{(t)}({\bm R}))={\bm R}$. 

\begin{lem}
\label{lem:Ioffe_ray:circular_filling}
For each ${\bm R} \in \partial\mathfrak{M}({\bm R}_{0})$ there exist a 
neighborhood\/ $\mathfrak{U}$ of $\bm R$ and positive numbers $\tau_{0}$ and 
$\delta$ such that for 
${\bm R}' \in \partial\mathfrak{M}({\bm R}_{0}) \cap \mathfrak{U}$, 
$t \in (0,\tau_{0})$ and 
$\mathcal{F} \in \mathscr{MF}(\Sigma_{g};\mathfrak{M}({\bm R}_{0}),{\bm R}')$ 
the inequality 
$$
\Ext_{\mathcal{F}}({\bm r}^{(t)}[{\bm R}'](s))>\Ext_{\mathcal{F}}({\bm R}')
$$
holds whenever $\tau^{(t)}({\bm R}')<s<\tau^{(t)}({\bm R}')+\delta$. 
\end{lem}

\begin{proof}
Let ${\bm R} \in \partial\mathfrak{M}({\bm R}_{0})$. If there were no required 
$\mathfrak{U}$, $\tau_{0}$ or $\delta$, then there would exist a sequence 
$\{\mathfrak{U}_{n}\}$ of neighborhoods of $\bm R$ shrinking to $\{{\bm R}\}$ 
and sequences $\{\tau_{n}\}$ and $\{\delta_{n}\}$ of positive numbers 
converging to $0$ such that 
\begin{equation}
\label{eq:extremal_length:wrong_inequality}
\Ext_{\mathcal{F}_{n}}({\bm r}^{(\tau_{n})}[{\bm R}_{n}](s_{n})) \leqq
\Ext_{\mathcal{F}_{n}}({\bm R}_{n})
\end{equation}
for some 
${\bm R}_{n} \in \partial\mathfrak{M}({\bm R}_{0}) \cap \mathfrak{U}_{n}$, some 
$s_{n}$ with 
$$
\tau^{(\tau_{n})}({\bm R}_{n})<s_{n}<\tau^{(\tau_{n})}({\bm R}_{n})+\delta_{n}<1
$$
and some measured foliation $\mathcal{F}_{n}$ in 
$\mathscr{MF}(\Sigma_{g};\mathfrak{M}({\bm R}_{0}),{\bm R}_{n})$ with 
$\Ext_{\mathcal{F}_{n}}({\bm R}_{n})=1$. 

Set ${\bm S}_{n}={\bm r}^{(\tau_{n})}[{\bm R}_{n}](s_{n})$ and 
${\bm S}'_{n}={\bm r}^{(\tau_{n})}[{\bm R}_{n}](1)$. Taking a subsequence 
if necessary, we may suppose that $\{{\bm S}'_{n}\}$ and $\{\mathcal{F}_{n}\}$ 
converge to ${\bm S}' \in \mathfrak{T}_{g}$ and 
$\mathcal{F} \in \mathscr{MF}(\Sigma_{g})$, respectively. Note that 
${\bm R}_{n} \to {\bm R}$ as $n \to \infty$. Since 
$\Ext_{\mathcal{F}_{n}}({\bm R}') \leqq \Ext_{\mathcal{F}_{n}}({\bm R}_{n})=1$ 
for all ${\bm R}' \in \mathfrak{M}({\bm R}_{0})$, by letting $n \to \infty$ we 
obtain $\Ext_{\mathcal{F}}({\bm R}') \leqq \Ext_{\mathcal{F}}({\bm R})=1$. In 
particular, $\bm R$ is a maximal point for $\mathcal{F}$ on 
$\mathfrak{M}({\bm R}_{0})$. 

As ${\bm R}_{n}$ lies on the Teichm\"{u}ller geodesic segment connecting 
${\bm S}'_{n}={\bm r}^{(\tau_{n})}[{\bm R}_{n}](1)$ and 
${\bm r}^{(\tau_{n})}[{\bm R}_{n}](0)$, we have 
$$
d_{T}({\bm S}'_{n},{\bm r}^{(\tau_{n})}[{\bm R}_{n}](0)) \geqq
d_{T}({\bm S}'_{n},{\bm R}_{n}).
$$
Lemma~\ref{lem:Ioffe_ray:approximation} yields that 
${\bm r}^{(\tau_{n})}[{\bm R}_{n}](0)={\bm H}[{\bm W}_{\tau_{n}}]({\bm S}'_{n})
\to {\bm H}[{\bm R}_{0}]({\bm S}')$ as $n \to \infty$. Taking limits of the 
both sides of the above inequality, we obtain 
$$
d_{T}({\bm S}',\mathfrak{M}({\bm R}_{0}))=
d_{T}({\bm S}',{\bm H}[{\bm R}_{0}]({\bm S}')) \geqq
d_{T}({\bm S}',{\bm R}),
$$
which means that $\bm R$ is a point of $\mathfrak{M}({\bm R}_{0})$ nearest to 
${\bm S}'$. Therefore, ${\bm H}[{\bm R}_{0}]({\bm S}')$ is identical with 
$\bm R$ by Corollary~\ref{cor:M_K(R_0):nearest_point} and hence 
${\bm r}:={\bm r}_{\bm R}[{\bm S}']$ is an Ioffe ray of ${\bm R}_{0}$. 

If ${\bm \psi}={\bm Q}_{\bm R}[{\bm S}']$ and 
${\bm \psi}_{n}={\bm Q}_{{\bm R}_{n}}[{\bm S}'_{n}]$, then if follows from 
Proposition~\ref{prop:Earle's_theorem} that $\{{\bm \psi}_{n}\}$ converges to 
$\bm \psi$ in the complex vector bundle of holomorphic quadratic differentials 
over $\mathfrak{T}_{g}$. Set 
${\bm \varphi}_{n}={\bm Q}_{{\bm R}_{n}}(\mathcal{F}_{n})$ and 
$\sigma_{n}=d_{T}({\bm R}_{n},{\bm S}_{n})$. Since 
${\bm R}_{n}={\bm r}_{{\bm R}_{n}}[{\bm \psi}_{n}](0)$ and 
${\bm S}_{n}={\bm r}_{{\bm R}_{n}}[{\bm \psi}_{n}](\sigma_{n})$, 
Proposition~\ref{prop:extremal_length:variation} together 
with~\eqref{eq:extremal_length:wrong_inequality} implies 
$$
\re D_{{\bm R}_{n}}[{\bm \varphi}_{n}]({\bm \psi}_{n})=
\frac{\,\log\Ext_{\mathcal{F}_{n}}({\bm S}_{n})-
\log\Ext_{\mathcal{F}_{n}}({\bm R}_{n})\,}{2\sigma_{n}}+\varepsilon_{n} \leqq
\varepsilon_{n}
$$
with $\varepsilon_{n} \to 0$ as $n \to \infty$ and hence 
\begin{equation}
\label{eq:nonpositive_derivative}
\re D_{\bm R}[{\bm \varphi}]({\bm \psi}) \leqq 0,
\end{equation}
where ${\bm \varphi}={\bm Q}_{\bm R}(\mathcal{F})$. 

On the other hand, another application of 
Proposition~\ref{prop:extremal_length:variation} gives 
$$
\re D_{{\bm R}_{n}}[{\bm \varphi}_{n}](-{\bm \psi}_{n})=
\frac{\,\log\Ext_{\mathcal{F}_{n}}({\bm r}^{(\tau_{n})}[{\bm R}_{n}](0))-
\log\Ext_{\mathcal{F}_{n}}({\bm R}_{n})\,}{2\tau^{(\tau_{n})}({\bm R}_{n})}+
\varepsilon'_{n}
$$
with $\varepsilon'_{n} \to 0$ as $n \to \infty$. Since 
$\Ext_{\mathcal{F}_{n}}({\bm R}_{n})=1$ and 
${\bm r}^{(\tau_{n})}[{\bm R}_{n}](0)=
{\bm H}[{\bm W}_{\tau^{(n)}}]({\bm R}_{n})$, it follows from 
Corollary~\ref{cor:Ext:estimate_of_derivative} that 
$$
\re D_{{\bm R}_{n}}[{\bm \varphi}_{n}](-{\bm \psi}_{n}) \leqq
\frac{1}{\,2\tau^{(\tau_{n})}({\bm R}_{n})\,}
\log(1+(-c+M\tau_{n})\tau^{(\tau_{n})}({\bm R}_{n}))+\varepsilon'_{n}
$$
for some positive constants $c$ and $M$. Letting $n \to \infty$, we obtain 
$$
\re D_{{\bm R}}[{\bm \varphi}]({\bm \psi}) \geqq \frac{c}{\,2\,}>0,
$$
which contradicts~\eqref{eq:nonpositive_derivative}. 
\end{proof}

\begin{prop}
\label{prop:M(R_0):ball}
If ${\bm R}_{0}$ is nonanalytically finite, then there is a homeomorphism of\/ 
$\mathfrak{T}_{g}$ onto $\mathbb{R}^{2d_{g}}$ that maps\/ 
$\mathfrak{M}({\bm R}_{0})$ onto a closed ball. 
\end{prop}

\begin{proof}
Since $\partial\mathfrak{M}({\bm R}_{0})$ is compact, an application of 
Lemma~\ref{lem:Ioffe_ray:circular_filling} gives positive numbers $\tau_{0}$ 
and $\delta$ such that for ${\bm R} \in \partial\mathfrak{M}({\bm R}_{0})$, 
$t \in (0,\tau_{0})$ and 
$\mathcal{F} \in \mathscr{MF}(\Sigma_{g}; \mathfrak{M}({\bm R}_{0}),{\bm R})$ 
the inequality 
\begin{equation}
\label{eq:ray:outside}
\Ext_{\mathcal{F}}({\bm r}^{(t)}[{\bm R}](s))>\Ext_{\mathcal{F}}({\bm R})
\end{equation}
holds whenever $\tau^{(t)}({\bm R})<s<\tau^{(t)}({\bm R})+\delta$. Fix $t$ with 
$0<t<\min\{\tau_{0},\delta\}$ so that $\mathfrak{M}({\bm R}_{0}) \subset
\mathfrak{M}_{e^{2\delta}}({\bm W}_{t})$. By 
Proposition~\ref{prop:circular_filling:inerior} there is $\sigma>0$ such that 
$\mathfrak{M}_{e^{2\sigma}}({\bm W}_{t}) \subset
\Int\mathfrak{M}({\bm R}_{0})$. Inequality~\eqref{eq:ray:outside} implies 
that the Teichm\"{u}ller geodesic arc 
${\bm r}^{(t)}[{\bm R}](\tau^{(t)}({\bm R}),\tau^{(t)}({\bm R})+\delta)$ lies 
outside of $\mathfrak{M}({\bm R}_{0})$ since $\bm R$ is a maximal point for 
$\mathcal{F}$ on $\mathfrak{M}({\bm R}_{0})$. Note that 
$\tau^{(t)}({\bm R}) \leqq \delta$ as $\mathfrak{M}({\bm R}_{0})$ is 
included in $\mathfrak{M}_{e^{2\delta}}({\bm W}_{t})$. We thus deduce that the 
correspondence 
$$
\partial\mathfrak{M}({\bm R}_{0}) \ni {\bm R} \mapsto
{\bm r}^{(t)}[{\bm R}](\sigma) \in
\partial\mathfrak{M}_{e^{2\sigma}}({\bm W}_{t})
$$
is a continuous bijection, and hence is a homeomorphism as 
$\partial\mathfrak{M}({\bm R}_{0})$ is compact and 
$\partial\mathfrak{M}_{e^{2\sigma}}({\bm W}_{t})$ is Hausdorff. Observe that 
$\Int\mathfrak{M}_{e^{2\delta}}({\bm W}_{t}) \setminus
\mathfrak{M}_{e^{2\sigma}}({\bm W}_{t})$ is the disjoint union of 
Teich\-m\"{u}ller geodesic arcs ${\bm r}^{(t)}[{\bm R}](\sigma,\delta)$, 
${\bm R} \in \partial\mathfrak{M}({\bm R}_{0})$, and that 
${\bm r}^{(t)}[{\bm R}](\sigma,\delta) \cap \partial\mathfrak{M}({\bm R}_{0})=
\{{\bm R}\}$. Since ${\bm R} \mapsto \tau^{(t)}({\bm R})$ is continuous, there 
is a homeomorphism $\bm h$ of 
$\Int\mathfrak{M}_{e^{2\delta}}({\bm W}_{t}) \setminus
\mathfrak{M}_{e^{2\sigma}}({\bm W}_{t})$ onto itself for which 
${\bm h}({\bm R})={\bm r}^{(t)}[{\bm R}]((\sigma+\delta)/2)$ and 
${\bm h}({\bm r}^{(t)}[{\bm R}](\sigma,\delta))=
{\bm r}^{(t)}[{\bm R}](\sigma,\delta)$, 
${\bm R} \in \partial\mathfrak{M}({\bm R}_{0})$. With the aid of 
Theorem~\ref{thm:main:M_K(R_0)} we can extend it to a homeomorphism $\bm h$ of 
$\mathfrak{T}_{g}$ onto itself with ${\bm h}(\mathfrak{M}({\bm R}_{0}))=
\mathfrak{M}_{e^{\sigma+\delta}}({\bm W}_{t})$, completing the proof. 
\end{proof}

We have proved most assertions of Theorem~\ref{thm:main:M(R_0)}. The following 
proposition will finish the proof. 

\begin{prop}
\label{prop:M(R_0):outer_ball_inner_cone}
If ${\bm R}_{0}$ is nonanalytically finite, then\/ $\mathfrak{M}({\bm R}_{0})$ 
is a closed Lipschitz domain satisfying an outer ball condition. 
\end{prop}

\begin{proof}
The Teichm\"{u}ller distance function $d_{T}$ is of class $C^{2}$ on 
$\mathfrak{T}_{g} \times \mathfrak{T}_{g}$ off the diagonal by Rees 
\cite{Rees2002}. Hence Theorem~\ref{thm:Ioffe_ray:distance} shows that 
$\mathfrak{M}({\bm R}_{0})$ satisfies an outer ball condition. 

We verify that $\partial\mathfrak{M}({\bm R}_{0})$ is locally expressed as the 
graph of a Lipschitz function. For 
${\bm R} \in \partial\mathfrak{M}({\bm R}_{0})$ let $\mathscr{E}({\bm R})$ 
denote the set of 
$\mathcal{F} \in \mathscr{MF}(\Sigma_{g};\mathfrak{M}({\bm R}_{0}),{\bm R})$ 
such that $\Ext_{\mathcal{F}}({\bm R})=1$. For 
$\mathfrak{E} \subset \partial\mathfrak{M}({\bm R}_{0})$ set 
$\mathscr{E}(\mathfrak{E})=
\bigcup_{{\bm R} \in \mathfrak{E}} \mathscr{E}({\bm R})$. 

Now, take $\bm R \in \partial\mathfrak{M}({\bm R}_{0})$ arbitrarily. Choose a 
sequence $\{t_{n}\}$ of positive numbers with $t_{n} \to 0$ so that 
$\{{\bm R}'_{n}\}:=\{{\bm r}^{(t_{n})}[{\bm R}](1)\}$ converges to a point 
${\bm R}' \in \mathfrak{T}_{g}$. Then ${\bm r}_{\bm R}[{\bm R}']$ is an Ioffe 
ray of ${\bm R}_{0}$ by Lemma~\ref{lem:Ioffe_ray:approximation}. If 
$\mathcal{F} \in \mathscr{E}({\bm R})$, then 
Proposition~\ref{prop:extremal_length:variation} and 
Corollary~\ref{cor:Ext:estimate_of_derivative} imply 
\begin{align*}
\re D_{{\bm R}}[{\bm Q}_{\bm R}(\mathcal{F})](-{\bm Q}_{\bm R}[{\bm R}'_{n}]) &
=\frac{\,\log\Ext_{\mathcal{F}}({\bm r}^{(t_{n})}[{\bm R}](0))\,}
{2\tau^{(t_{n})}({\bm R})}+\varepsilon_{n} \\
 & \leqq \frac{1}{\,2\tau^{(t_{n})}({\bm R})\,}
\log(1+(-c+Mt_{n})\tau^{(t_{n})}({\bm R}))+\varepsilon_{n}
\end{align*}
for some positive $c$ and $M$ independent of $\mathcal{F}$, where 
$\varepsilon_{n} \to 0$ as $n \to \infty$. Letting $n \to \infty$, we obtain 
$$
\re D_{\bm R}[{\bm Q}_{\bm R}(\mathcal{F})]({\bm Q}_{\bm R}[{\bm R}']) \geqq
\frac{c}{\,2\,}.
$$
Set $\mathcal{G}=\mathcal{H}_{\bm R}({\bm Q}_{\bm R}[{\bm R}'])$. Since 
${\bm r}_{\bm R}[{\bm Q}_{\bm R}[{\bm R}']]={\bm r}_{\bm R}[{\bm R}']$ is an 
Ioffe ray of ${\bm R}_{0}$, we know that 
$\mathcal{G} \in \mathscr{E}({\bm R})$; note that 
$\Ext_{\mathcal{G}}({\bm R})=\|{\bm Q}_{\bm R}[{\bm R}']\|_{\bm R}=1$. The 
above inequality can be expressed as 
\begin{equation}
\label{eq:derivative:outer_vector}
\re D_{\bm R}[{\bm Q}_{\bm R}(\mathcal{F})]({\bm Q}_{\bm R}(\mathcal{G})) \geqq
\frac{c}{\,2\,}
\end{equation}
for $\mathcal{F} \in \mathscr{E}({\bm R})$. 

We examine the behavior of the function 
$F_{\mathcal{F}}:=(\log\Ext_{\mathcal{F}})/2$ along 
${\bm r}_{\bm S}[{\bm \varphi}]$. 
Proposition~\ref{prop:extremal_length:variation} shows that 
$$
\frac{d}{\,dt\,}F_{\mathcal{F}}({\bm r}_{\bm S}[{\bm \varphi}](t))=
\re D_{{\bm r}_{\bm S}[{\bm \varphi}](t)}
[{\bm Q}_{{\bm r}_{\bm S}[{\bm \varphi}](t)}(\mathcal{F})]
({\bm Q}_{{\bm r}_{\bm S}[{\bm \varphi}](t)}
[{\bm r}_{\bm S}[{\bm \varphi}](1)]), \quad 0<t<1.
$$
In general for ${\bm S} \in \mathfrak{T}_{g}$, 
${\bm \psi} \in A({\bm S}) \setminus \{0\}$, $\varepsilon>0$ and $\delta>0$ 
define 
\begin{align*}
U_{\bm S}(\varepsilon,{\bm \psi}) & =\{{\bm Q}_{\bm S}[{\bm S'}] \mid 
{\bm S}' \in \mathfrak{T}_{g} \setminus \{{\bm S}\} \text{ and }
d_{T}({\bm r}_{\bm S}[{\bm Q}_{\bm S}[{\bm S'}]](1),
{\bm r}_{\bm S}[{\bm \psi}](1))<\varepsilon\} \quad \text{and} \\
\mathfrak{C}_{\bm S}(\varepsilon,\delta) & =\{{\bm r}_{\bm S}[-{\bm \varphi}](t)
\mid {\bm \varphi} \in U_{\bm S}(\varepsilon,{\bm Q}_{\bm S}(\mathcal{G}))
\text{ and } 0<t<\delta\}.
\end{align*}
We claim that there is a neighborhood $\mathfrak{U}$ of $\bm R$ together with 
positive numbers $\varepsilon_{0}$ and $\delta_{0}$ such that 
\begin{equation}
\label{eq:derivative:outer_cone:wider}
\frac{d}{\,dt\,}F_{\mathcal{F}}({\bm r}_{\bm S}[{\bm \varphi}](t))>0, \quad
0<t<\delta_{0},
\end{equation}
for all ${\bm S} \in \mathfrak{U}$, $\mathcal{F} \in
\mathscr{E}(\partial\mathfrak{M}({\bm R}_{0}) \cap \mathfrak{U})$, 
${\bm \varphi} \in U_{\bm S}(\varepsilon_{0},{\bm Q}_{\bm S}(\mathcal{G}))$. If 
not, then there would exist sequences $\{{\bm S}_{n}\}$, $\{\mathcal{F}_{n}\}$, 
$\{{\bm \varphi}_{n}\}$ and $\{t_{n}\}$ such that 
\begin{list}{{\rm (\roman{claim})}}{\usecounter{claim}
\setlength{\topsep}{0pt}
\setlength{\itemsep}{0pt}
\setlength{\parsep}{0pt}
\setlength{\labelwidth}{\leftmargin}}
\item ${\bm S}_{n} \to {\bm R}$ as $n \to \infty$, 

\item there are $\mathcal{F} \in \mathscr{E}({\bm R})$ and 
${\bm S}'_{n} \in \partial\mathfrak{M}({\bm R}_{0})$ with 
$\mathcal{F}_{n} \in \mathscr{E}({\bm S}'_{n})$ such that 
$\mathcal{F}_{n} \to \mathcal{F}$ and ${\bm S}'_{n} \to {\bm R}$ as 
$n \to \infty$, 

\item ${\bm \varphi}_{n} \in A({\bm S}_{n})$ with 
$\|{\bm \varphi}_{n}\|_{{\bm S}_{n}}=1$, and 
${\bm \varphi}_{n} \to {\bm Q}_{\bm R}(\mathcal{G})$ as $n \to \infty$, 

\item $t_{n}>0$ and 
${\bm S}''_{n}:={\bm r}_{{\bm S}_{n}}[{\bm \varphi}_{n}](t_{n}) \to {\bm R}$ as 
$n \to \infty$, and 

\item $\re D_{{\bm S}''_{n}}[{\bm Q}_{{\bm S}''_{n}}(\mathcal{F}_{n})]
({\bm Q}_{{\bm S}''_{n}}[{\bm r}_{{\bm S}_{n}}[{\bm \varphi}_{n}](1)]) \leqq
0$. 
\end{list}
Letting $n \to \infty$ in the last inequality, we obtain 
$\re D_{\bm R}[{\bm Q}_{\bm R}(\mathcal{F})]({\bm Q}_{\bm R}(\mathcal{G})) \leqq
0$, contradicting~\eqref{eq:derivative:outer_vector}. 

Owing to Theorem~\ref{thm:main:Ext}~(iii) and 
Proposition~\ref{prop:M(R_0):ball} we can replace $\mathfrak{U}$ with a smaller 
one so that $\mathfrak{M}({\bm R}_{0}) \cap \mathfrak{U}$ is the set of 
${\bm S} \in \mathfrak{U}$ such 
that~\eqref{eq:extremal_length:description_of_M(R_0)} holds for all $\mathcal{F}
\in \mathscr{E}(\partial\mathfrak{M}({\bm R}_{0}) \cap \mathfrak{U})$. Since 
$F_{\mathcal{G}}$ is a $C^{1}$ function on $\mathfrak{T}_{g}$ with nonvanishing 
derivatives, its level hypersurfaces are $C^{1}$ submanifolds of 
$\mathfrak{T}_{g}$. Replacing $\mathfrak{U}$ with a smaller one if necessary, 
we take a $C^{1}$ coordinate system $(\mathfrak{U},x)$ centered at ${\bm R}$ so 
that if we write $x=(x_{1},\ldots,x_{2d_{g}})=(x_{1},x')$, then the level 
hypersurfaces $F_{\mathcal{G}}^{-1}(a)$ are represented as $x_{1}=a$ and the 
Teichm\"{u}ller geodesic rays ${\bm r}_{\bm S}[{\bm Q}_{\bm S}(\mathcal{G})]$ 
(resp.\ ${\bm r}_{\bm S}[-{\bm Q}_{\bm S}(\mathcal{G})]$) with 
$F_{\mathcal{G}}({\bm S})=a$ are represented as $(a+\xi,x'({\bm S}))$, 
$\xi \geqq 0$ (resp.\ $\xi \leqq 0$). Note that if $\bm S$ is in a small 
neighborhood $\mathfrak{V}$ of $\bm R$ and $\varepsilon_{0}$ and $\delta_{0}$ 
are sufficiently small, then $\mathfrak{C}_{\bm S}(\varepsilon_{0},\delta_{0})$ 
is included in $\mathfrak{U}$ and 
$x(\mathfrak{C}_{\bm S}(\varepsilon_{0},\delta_{0}))$ contains a cone with 
vertex at $x({\bm S})$ and axis parallel to the $x_{1}$-axis, where the cones 
can be chosen to be of the same shape regardless of $\bm S$. 

Suppose that ${\bm S} \in \partial\mathfrak{M}({\bm R}_{0}) \cap \mathfrak{V}$. 
From~\eqref{eq:derivative:outer_cone:wider} we infer that for $\mathcal{F} \in
\mathscr{E}(\partial\mathfrak{M}({\bm R}_{0}) \cap \mathfrak{U})$ the function 
$F_{\mathcal{F}}$ is decreasing along the Teichm\"{u}ller geodesic 
segment ${\bm r}_{\bm S}[-{\bm \varphi}](t)$, $0<t<\delta_{0}$, for 
${\bm \varphi} \in U_{\bm S}(\varepsilon_{0},{\bm Q}_{\bm S}(\mathcal{G}))$ and 
hence that 
$$
\Ext_{\mathcal{F}}({\bm r}_{\bm S}[-{\bm \varphi}](t)) \leqq 
\Ext_{\mathcal{F}}({\bm r}_{\bm S}[-{\bm \varphi}](0))=
\Ext_{\mathcal{F}}({\bm S}) \leqq
\max\Ext_{\mathcal{F}}(\partial\mathfrak{M}({\bm R}_{0})).
$$
Thus $\mathfrak{M}({\bm R}_{0})$ includes 
$\mathfrak{C}_{\bm S}(\varepsilon_{0},\delta_{0})$. In particular, 
$\mathfrak{M}({\bm R}_{0})$ includes 
$\mathfrak{C}_{\bm R}(\varepsilon_{0},\delta_{0})$. 

Take a neighborhood $\mathfrak{W}$ of $\bm R$ included in $\mathfrak{V}$ so 
that if ${\bm S} \in F_{\mathcal{G}}^{-1}(0) \cap \mathfrak{W}$, then the 
Teichm\"{u}ller geodesic ray ${\bm r}_{\bm S}[-{\bm Q}_{\bm S}(\mathcal{G})]$ 
hits $\mathfrak{C}_{\bm R}(\varepsilon_{0},\delta_{0})$. Set $\bm {M}({\bm S})=
{\bm r}_{\bm S}[-{\bm Q}_{\bm S}(\mathcal{G})](\sigma_{\bm S})$, where 
$\sigma_{\bm S}$ is the minimum of nonnegative $s$ for which 
${\bm r}_{\bm S}[-{\bm Q}_{\bm S}(\mathcal{G})](s) \in
\mathfrak{M}({\bm R}_{0})$. Then ${\bm M}$ is a mapping of 
$F^{-1}(0) \cap \mathfrak{W}$ into 
$\partial\mathfrak{M}({\bm R}_{0}) \cap \mathfrak{U}$, and 
$\mathfrak{C}_{{\bm M}({\bm S})}(\varepsilon_{0},\delta_{0})$ is included in 
$\mathfrak{M}({\bm R}_{0})$. Consequently, the function 
$x'({\bm S}) \mapsto x_{1}({\bm M}({\bm S}))$, 
${\bm S} \in F_{\mathcal{G}}^{-1}(0) \cap \mathfrak{V}$, is a Lipschitz 
function on a neighborhood of $0$ in the plane $x_{1}=0$, and 
$x(\partial\mathfrak{M}({\bm R}_{0}))$ is the graph of the function. 
This completes the proof. 
\end{proof}

\begin{thm}
\label{thm:M(R_0):nonsmoothness}
Suppose that ${\bm R}_{0}$ is nonanalytically finite. Let $\bm R$ be a boundary 
point of\/ $\mathfrak{M}({\bm R}_{0})$. If there are two Ioffe rays of\/ 
${\bm R}_{0}$ emanating from $\bm R$, then\/ 
$\partial\mathfrak{M}({\bm R}_{0})$ is not smooth at $\bm R$. 
\end{thm}

\begin{proof}
By assumption there are linearly independent 
${\bm \varphi}_{j} \in A_{L}({\bm R})$, $j=1,2$, such that 
${\bm r}_{\bm R}[{\bm \varphi}_{j}] \in \mathscr{I}({\bm R}_{0})$. Let 
$\mathfrak{E}_{j}$ be the set of ${\bm S} \in \mathfrak{T}_{g}$ for which 
$\Ext_{\mathcal{F}_{j}}({\bm S}) \leqq \Ext_{\mathcal{F}_{j}}({\bm R})$, where 
$\mathcal{F}_{j}=\mathcal{H}_{\bm R}({\bm \varphi}_{j})$. Then 
$\mathfrak{M}({\bm R}_{0})$ is included in 
$\mathfrak{E}_{1} \cap \mathfrak{E}_{2}$ by 
Theorem~\ref{thm:maximal_norm_property}. Since $\partial\mathfrak{E}_{1}$ and 
$\partial\mathfrak{E}_{2}$ meet transversally at $\bm R$, the boundary 
$\partial\mathfrak{M}({\bm R}_{0})$ cannot be smooth at $\bm R$. 
\end{proof}

See Example~\ref{exmp:inducing_same_self-welding} for the existence of 
${\bm R}_{0}$ satisfying the assumptions of the theorem. In the case of genus 
one $\mathfrak{M}({\bm R}_{0})$ is a closed ball with respect to the 
Teichm\"{u}ller distance provided that it is not a singleton. This is not 
always the case for $g>1$ because balls with respect to the Teichm\"{u}ller 
distance have smooth boundaries. 

We conclude the paper with the following proposition. It supplements 
Kahn-Pilgrim-Thurston \cite[Theorem~2]{KPT2022}. Note that for 
${\bm R} \in \partial\mathfrak{M}({\bm R}_{0})$ every element of 
$\CEmb_{\mathrm{hc}}({\bm R}_{0},{\bm R})$ has a dense image. 

\begin{prop}
\label{prop:int_M(R_{0}):nonuniqueness}
Suppose that ${\bm R}_{0}=[R_{0},\theta_{0}]$ is open and nonanalytically 
finite. If ${\bm R}=[R,\theta] \in \Int\mathfrak{M}({\bm R}_{0})$, then there 
are continuations $({\bm R},{\bm \iota}_{j})$, $j=1,2$, of ${\bm R}_{0}$ such 
that $R \setminus \iota_{1}({R}_{0})$ is of positive area while 
$R \setminus \iota_{2}({R}_{0})$ has a vanishing area, where 
${\bm \iota}_{j}=[\iota_{j},\theta_{0},\theta]$. 
\end{prop}

\begin{proof}
Fix ${\bm R} \in \Int\mathfrak{M}({\bm R}_{0})$. If 
$\{({\bm W}^{(1)}_{t},{\bm \epsilon}^{(1)}_{t})\}$ is a circular filling for 
${\bm R}_{0}$, then $\bm R$ belongs to $\mathfrak{M}({\bm W}^{(1)}_{t_{1}})$ 
for some positive $t_{1}$ by Proposition~\ref{prop:filling:interior}. Thus 
there is a homotopically consistent conformal embedding ${\bm \kappa}_{1}$ of 
${\bm W}^{(1)}_{t_{1}}$ into $\bm R$. Then 
${\bm \iota}_{1}:={\bm \kappa}_{1} \circ {\bm \epsilon}^{(1)}_{t_{1}}$ is a 
homotopically consistent conformal embedding of ${\bm R}_{0}$ into $\bm R$ and 
$({\bm R},{\bm \iota}_{1})$ is not a dense continuation of ${\bm R}_{0}$. 

Next let $\{({\bm W}^{(2)}_{t},{\bm \epsilon}^{(2)}_{t})\}_{t \in [0,1]}$ be a 
linear filling for ${\bm R}_{0}$. If 
$\mathfrak{M}({\bm W}^{(2)}_{1})=\{{\bm R}\}$, then ${\bm \epsilon}^{(2)}_{1}$ 
induces a required element ${\bm \iota}_{2}$ of 
$\CEmb_{\mathrm{hc}}({\bm R}_{0},{\bm R})$. Otherwise, 
${\bm R} \not\in \mathfrak{M}({\bm W}^{(2)}_{t})$ for some $t<1$. 
Let $t_{2}$ be the infimum of such $t$. For each $t \in (t_{2},1]$ let 
${\bm r}_{t}$ denote the Ioffe ray of ${\bm W}^{(2)}_{t}$ passing through 
$\bm R$. Then ${\bm r}_{t}(0) \to {\bm R}$ as $t \to t_{2}$. On the other hand, 
if we take a sequence $\{\tau_{n}\}$ in $(t_{2},1]$ so that 
$\tau_{n} \to t_{2}$ and 
${\bm S}_{n}:={\bm r}_{\tau_{n}}(1) \to {\bm S} \in \mathfrak{T}_{g}$ as 
$n \to \infty$, then Lemma~\ref{lem:Ioffe_ray:approximation} implies that 
${\bm r}_{\tau_{n}}(0)={\bm H}[{\bm W}^{(2)}_{\tau_{n}}]({\bm S}_{n}) \to 
{\bm H}[{\bm W}^{(2)}_{t_{2}}]({\bm S})$ as $n \to \infty$. Hence 
${\bm R}={\bm H}[{\bm W}^{(2)}_{t_{2}}]({\bm S}) \in
\partial\mathfrak{M}({\bm W}^{(2)}_{t_{2}})$. If ${\bm \kappa}_{2}$ is a 
homotopically consistent conformal embedding of ${\bm W}^{(2)}_{t_{2}}$ into 
$\bm R$, then 
${\bm \iota}_{2}:={\bm \kappa}_{2} \circ {\bm \epsilon}^{(2)}_{t_{2}}$ 
possesses the required properties. 
\end{proof}

%%%%%
%%%%%
%%%%%

\noindent
\footnotesize
\begin{tabular}{@{}l@{\ }l@{}}
Masumoto: & Department of Mathematics, Yamaguchi University, Yamaguchi 
753-8512, Japan \\
& E-mail: {\tt masumoto@yamaguchi-u.ac.jp}
\end{tabular} \\
\begin{tabular}{@{}l@{\ }l@{}}
Shiba: & Professor emeritus, Hiroshima University, Hiroshima 739-8511, Japan \\
 & E-mail: {\tt shiba.masakazu.h14@kyoto-u.jp}
\end{tabular}
\end{document}